\documentclass[reqno, 12pt]{amsart}

\usepackage{amsmath}
\usepackage{amsfonts}
\usepackage{amssymb}
\usepackage{amsthm}
\usepackage{mathrsfs}
\usepackage{bbm}
\usepackage{graphicx, color}
\usepackage{tikz}
\usepackage{arydshln}
\usetikzlibrary{calc,decorations.markings,patterns,angles,quotes}
\usetikzlibrary{arrows.meta}
\usetikzlibrary{positioning}
\usetikzlibrary{cd, intersections, calc, decorations.pathmorphing, arrows, decorations.pathreplacing}
\usepackage{stmaryrd}

\usepackage{pgfplots}
\pgfplotsset{compat=1.17}
\usetikzlibrary{pgfplots.fillbetween}

\DeclareTextFontCommand{\emph}{\bfseries\em}

\usepackage{adjustbox}

\usepackage{bmpsize}

\usepackage{mathtools}

\usepackage{enumerate}

\definecolor{refkey}{rgb}{0.9451,0.2706,0.4941}
\definecolor{labelkey}{rgb}{0.9451,0.2706,0.4941}

\definecolor{mygreen}{rgb}{0,0.7,0.3}
\definecolor{myblue}{rgb}{0,0.50,1.20}
\definecolor{myorange}{rgb}{1,0.5,0.1}

\definecolor{rbA}{rgb}{0,0.7,0.3}
\definecolor{rbB}{rgb}{0,1,0}

\usepackage{subfiles}
\usepackage[colorlinks, linkcolor=blue, citecolor=red, urlcolor=cyan,pagebackref]{hyperref}

\allowdisplaybreaks

\numberwithin{equation}{section}

\usepackage{cleveref}
\crefname{thm}{Theorem}{Theorems}
\crefname{cor}{Corollary}{Corollaries}
\crefname{lem}{Lemma}{Lemmas}
\crefname{lemdef}{Lemma-Definition}{Lemma-Definitions}
\crefname{prop}{Proposition}{Propositions}
\crefname{dfn}{Definition}{Definitions}
\crefname{defi}{Definition}{Definitions}
\crefname{ex}{Example}{Examples}
\crefname{claim}{Claim}{Claims}
\crefname{conj}{Conjecture}{Conjectures}
\crefname{conv}{Notation}{Notations}
\crefname{rem}{Remark}{Remarks}
\crefname{rmk}{Remark}{Remarks}
\crefname{prob}{Problem}{Problems}
\crefname{figure}{Figure}{Figures}
\crefname{table}{Table}{Tables}
\crefname{section}{Section}{Sections}
\crefname{subsection}{Section}{Sections}
\crefname{appendix}{Appendix}{Appendices}
\crefname{introthm}{Theorem}{Theorems}
\crefname{introcor}{Corollary}{Corollaries}
\crefname{introconj}{Conjecture}{Conjectures}

\newtheorem{thm}{Theorem}[section]
\newtheorem{prop}[thm]{Proposition}
\newtheorem{cor}[thm]{Corollary}
\newtheorem{lem}[thm]{Lemma}

\newtheorem{introthm}{Theorem}

\theoremstyle{definition}
\newtheorem{dfn}[thm]{Definition}

\newtheorem{ex}[thm]{Example}

\theoremstyle{remark}

\newtheorem{rem}[thm]{Remark}

\usepackage{version}

\newcommand{\bZ}{\mathbb{Z}}
\newcommand{\bQ}{\mathbb{Q}}
\newcommand{\bR}{\mathbb{R}}
\newcommand{\bC}{\mathbb{C}}

\newcommand{\bM}{\mathbb{M}}

\newcommand{\bD}{{\boldsymbol{\Delta}}}
\newcommand{\bP}{\mathbb{M}_\circ}

\newcommand{\bA}{\mathbb{A}}
\newcommand{\bB}{\mathbb{B}}

\newcommand{\bs}{{\boldsymbol{s}}}
\newcommand{\bi}{{\boldsymbol{i}}}
\newcommand{\bSigma}{{\boldsymbol{\Sigma}}}

\newcommand{\A}{\mathcal{A}}

\newcommand{\cC}{\mathcal{C}}

\newcommand{\cF}{\mathcal{F}}
\newcommand{\cG}{\mathcal{G}}

\newcommand{\cL}{\mathcal{L}}

\newcommand{\cO}{\mathcal{O}}

\def\P{{\mathcal{P}}}

\newcommand{\cR}{\mathcal{R}}

\newcommand{\X}{\mathcal{X}}

\newcommand{\sfa}{\mathsf{a}}
\newcommand{\bsfa}{\boldsymbol{\mathsf{a}}}

\newcommand{\sfs}{{\mathsf{s}}}
\newcommand{\sft}{\mathsf{t}}
\newcommand{\sfT}{\mathsf{T}}
\newcommand{\sfx}{\mathsf{x}}

\newcommand{\sfr}{\mathsf{r}}

\newcommand{\Hom}{\mathrm{Hom}}
\newcommand{\tri}{\triangle}
\newcommand{\sgn}{\mathrm{sgn}}

\newcommand{\Exch}{\mathrm{Exch}}

\newcommand{\uf}{\mathrm{uf}}
\newcommand{\f}{\mathrm{f}}
\newcommand{\Teich}{Teichm\"uller}
\newcommand{\fsl}{\mathfrak{sl}}
\newcommand{\fsp}{\mathfrak{sp}}
\newcommand{\fso}{\mathfrak{so}}
\renewcommand{\cL}{\mathcal{L}_{\fsp_4}}
\newcommand{\ve}{\varepsilon}
\newcommand{\pot}{\mathsf{w}}
\newcommand{\congr}{\mathrm{cong}}

\newcommand{\vect}{\mathcal{D}}

\newcommand{\splittri}{\triangle^{\mathrm{split}}}

\newcommand{\Bweb}{\mathsf{BWeb}}

\newcommand{\diag}{\mathsf{Diag}}
\newcommand{\Bdiag}{\mathsf{BDiag}}

\newcommand{\Blad}{\mathsf{BLad}}

\newcommand{\Skein}[2]{\mathscr{S}_{#1,#2}}
\newcommand{\greq}{\operatorname*{=}_{\mathrm{gr}}}
\newcommand{\diagshift}{\tilde{\sigma}}
\newcommand{\ladshift}{\sigma}

\DeclareMathOperator{\intersec}{\mathbin{\pitchfork}}

\DeclareMathOperator{\interior}{\mathrm{int}}

\makeatletter
\newcommand{\oset}[3][0ex]{%
  \mathrel{\mathop{#3}\limits^{
    \vbox to#1{\kern-2\ex@
    \hbox{$\scriptstyle#2$}\vss}}}}
\makeatother

\makeatletter
\newcommand{\osetnear}[3][0ex]{%
  \mathrel{\mathop{#3}\limits^{
    \vbox to#1{\kern-.3\ex@
    \hbox{$\scriptstyle#2$}\vss}}}}
\makeatother

\newcommand\dnode[2]{\filldraw[draw=#2,fill=#2!15,even odd rule](#1) circle(1.7pt) circle(3pt)}

\newcommand\qarrow[2]{\draw[-latex,shorten >=2pt,shorten <=2pt] (#1) -- (#2) [thick];} 
\newcommand\qsarrow[2]{\draw[-latex,shorten >=3pt,shorten <=3pt] (#1) -- (#2) [thick];} 
\newcommand\qsharrow[2]{\draw[-latex,shorten >=4pt,shorten <=2pt] (#1) -- (#2) [thick];} 
\newcommand\qstarrow[2]{\draw[-latex,shorten >=2pt,shorten <=4pt] (#1) -- (#2) [thick];} 
\newcommand{\uniarrow}[3]{\draw[-latex,thick,#3] (#1) to (#2);}

\newcommand\qshdarrow[2]{\draw[->,dashed,shorten >=4pt,shorten <=2pt] (#1) -- (#2) [thick];} 
\newcommand\qstdarrow[2]{\draw[->,dashed,shorten >=2pt,shorten <=4pt] (#1) -- (#2) [thick];} 

\def\centerarc(#1)(#2:#3:#4)
    {($(#1)+({#4*cos(#2)},{#4*sin(#2)})$) arc(#2:#3:#4) }

\tikzset{
  mid arrow/.style={postaction={decorate,decoration={
        markings,
        mark=at position .5 with {\arrow[#1]{stealth}}
      }}},
}

\newcommand{\markedpt}[1]{
\node[fill,circle,inner sep=1.2pt] at(#1) {};
}

\tikzset{
    partial ellipse/.style args={#1:#2:#3}{
        insert path={+ (#1:#3) arc (#1:#2:#3)}
    }
}

\newcommand{\quiverplusC}[3]{
    \begin{scope}[>=latex]
    {\color{mygreen}
    \path(#1) coordinate(x1);
    \path(#2) coordinate(x2);
    \path(#3) coordinate(x3);
    \foreach \j in {1,2,3}
    {
        \foreach \k in {1,2,3}
        {
            \foreach \l in {1,2}
            {
            \path($(x\j)!0.333*\l!(x\k)$) coordinate(x\j\k\l);
            }
        }
    }
    \dnode{x121}{mygreen};
    \draw(x122) circle(2pt);
    \dnode{x312}{mygreen};
    \draw(x311) circle(2pt);
    {\color{myblue}
        \draw(x231) circle(2pt);
        \dnode{x232}{myblue};
    }
    \dnode{$(x121)!0.5!(x312)$}{mygreen} coordinate(G2);
    \draw($(x122)!0.5!(x311)$) circle(2pt) coordinate(G1);
    \qsarrow{G2}{x121}
    \qsarrow{x312}{G2}
    \qarrow{G1}{x122}
    \qarrow{x311}{G1}
    \qstarrow{x121}{G1}
    \qsharrow{G1}{G2}
    \qstarrow{G2}{x311}
    {\color{myblue}
    \qarrow{x122}{x231}
    \qarrow{x231}{G1}
    \qsarrow{G2}{x232}
    \qarrow{x232}{x312}
    }
    }
    \draw($(x2)!0.5!(x3)$) coordinate(h);
    \draw($(x1)!0.1!(h)$) node{$\ast$};
    \end{scope}
}

\newcommand{\quiverminusC}[3]{
    \begin{scope}[>=latex]
    {\color{mygreen}
    \path(#1) coordinate(x1);
    \path(#2) coordinate(x2);
    \path(#3) coordinate(x3);
    \foreach \j in {1,2,3}
    {
        \foreach \k in {1,2,3}
        {
            \foreach \l in {1,2}
            {
            \path($(x\j)!0.333*\l!(x\k)$) coordinate(x\j\k\l);
            }
        }
    }
    \dnode{x121}{mygreen};
    \draw(x122) circle(2pt);
    \dnode{x312}{mygreen};
    \draw(x311) circle(2pt);
    {\color{myblue}
        \draw(x232) circle(2pt);
        \dnode{x231}{myblue};
    }
    \dnode{$(x121)!0.5!(x312)$}{mygreen} coordinate(G2);
    \draw($(x122)!0.5!(x311)$) circle(2pt) coordinate(G1);
    \qsarrow{G2}{x121}
    \qsarrow{x312}{G2}
    \qarrow{G1}{x122}
    \qarrow{x311}{G1}
    \qsharrow{x122}{G2}
    \qstarrow{G2}{G1}
    \qsharrow{G1}{x312}
    {\color{myblue}
    \qsarrow{x121}{x231}
    \qsarrow{x231}{G2}
    \qarrow{G1}{x232}
    \qarrow{x232}{x311}
    }
    }
    \draw($(x2)!0.5!(x3)$) coordinate(h);
    \end{scope}
}

\newcommand{\quiververtexC}[3]{
    \begin{scope}[>=latex]
    {\color{mygreen}
    \path(#1) coordinate(x1);
    \path(#2) coordinate(x2);
    \path(#3) coordinate(x3);
    \foreach \j in {1,2,3}
    {
        \foreach \k in {1,2,3}
        {
            \foreach \l in {1,2}
            {
            \path($(x\j)!0.333*\l!(x\k)$) coordinate(x\j\k\l);
            }
        }
    }
    \dnode{x121}{mygreen};
    \draw(x122) circle(2pt);
    \dnode{x312}{mygreen};
    \draw(x311) circle(2pt);
    {\color{myblue}
        \draw(x231) circle(2pt);
        \dnode{x232}{myblue};
    }
    \dnode{$(x121)!0.5!(x312)$}{mygreen} coordinate(G2);
    \draw($(x122)!0.5!(x311)$) circle(2pt) coordinate(G1);
    }
    \draw($(x2)!0.5!(x3)$) coordinate(h);
    \end{scope}
}

\newcommand{\quiverplus}[3]{
    \begin{scope}[>=latex]
    {\color{mygreen}
    \path(#1) coordinate(x1);
    \path(#2) coordinate(x2);
    \path(#3) coordinate(x3);
    \foreach \j in {1,2,3}
    {
        \foreach \k in {1,2,3}
        {
            \foreach \l in {1,2}
            {
            \path($(x\j)!0.333*\l!(x\k)$) coordinate(x\j\k\l);
            }
        }
    }
    \dnode{x121}{mygreen};
    \draw(x122) circle(2pt);
    \dnode{x312}{mygreen};
    \draw(x311) circle(2pt);
    {\color{mygreen}
        \draw(x231) circle(2pt);
        \dnode{x232}{mygreen};
    }
    \dnode{$(x121)!0.5!(x312)$}{mygreen} coordinate(G2);
    \draw($(x122)!0.5!(x311)$) circle(2pt) coordinate(G1);
    \qsarrow{G2}{x121}
    \qsarrow{x312}{G2}
    \qarrow{G1}{x122}
    \qarrow{x311}{G1}
    \qstarrow{x121}{G1}
    \qsharrow{G1}{G2}
    \qstarrow{G2}{x311}
    {\color{mygreen}
    \qarrow{x122}{x231}
    \qarrow{x231}{G1}
    \qarrow{G2}{x232}
    \qarrow{x232}{x312}
    }
    }
    \draw($(x2)!0.5!(x3)$) coordinate(h);
    \draw($(x1)!0.1!(h)$) node{$\ast$};
    \end{scope}
}

\newcommand{\quiverminus}[3]{
    \begin{scope}[>=latex]
    {\color{mygreen}
    \path(#1) coordinate(x1);
    \path(#2) coordinate(x2);
    \path(#3) coordinate(x3);
    \foreach \j in {1,2,3}
    {
        \foreach \k in {1,2,3}
        {
            \foreach \l in {1,2}
            {
            \path($(x\j)!0.333*\l!(x\k)$) coordinate(x\j\k\l);
            }
        }
    }
    \dnode{x121}{mygreen};
    \draw(x122) circle(2pt);
    \dnode{x312}{mygreen};
    \draw(x311) circle(2pt);
    {\color{mygreen}
        \draw(x232) circle(2pt);
        \dnode{x231}{mygreen};
    }
    \dnode{$(x121)!0.5!(x312)$}{mygreen} coordinate(G2);
    \draw($(x122)!0.5!(x311)$) circle(2pt) coordinate(G1);
    \qsarrow{G2}{x121}
    \qsarrow{x312}{G2}
    \qarrow{G1}{x122}
    \qarrow{x311}{G1}
    \qsharrow{x122}{G2}
    \qstarrow{G2}{G1}
    \qsharrow{G1}{x312}
    {\color{mygreen}
    \qarrow{x121}{x231}
    \qarrow{x231}{G2}
    \qarrow{G1}{x232}
    \qarrow{x232}{x311}
    }
    }
    \draw($(x2)!0.5!(x3)$) coordinate(h);
    \draw($(x1)!0.1!(h)$) node{$\ast$};
    \end{scope}
}

\newcommand{\quiversquare}[4]{
\begin{scope}[>=latex]
    {\color{mygreen}
    \path(#1) coordinate(x1);
    \path(#2) coordinate(x2);
    \path(#3) coordinate(x3);
				\path(#4) coordinate(x4);
				\foreach \i in {1,2}
				{
    \path($(x1)!\i/3!(x4)$) coordinate(x14\i);
				\path($(x2)!\i/3!(x3)$) coordinate(x23\i);
				}
				\foreach \j in {0,1,2,3,4}
				{
				\draw($(x141)!\j/4!(x231)$) circle(2pt) coordinate(v1\j);
				\draw($(x142)!\j/4!(x232)$) coordinate(v2\j);
				\dnode{v2\j}{mygreen};
				}
				\draw[myblue]($(x1)!1/4!(x2)$) circle(2pt) coordinate(yl);
				\draw[myblue]($(x1)!3/4!(x2)$) circle(2pt) coordinate(yr);
				\draw($(x4)!1/4!(x3)$) coordinate(zl);
				\dnode{zl}{myblue};
				\draw($(x4)!3/4!(x3)$) coordinate(zr);
				\dnode{zr}{myblue};
				}
\end{scope}
}

\newcommand{\CoG}[3]{
    \path(#1) coordinate(x1);
    \path(#2) coordinate(x2);
    \path(#3) coordinate(x3);
    \path($(x1)!0.5!(x2)$) coordinate(H);
    \path($(x3)!0.667!(H)$) circle(2pt) coordinate(G);}

\definecolor{basecolor}{gray}{1}
\tikzset{
    ->-/.style 2 args={
        postaction={decorate},
        decoration={markings, mark=at position #1 with {\arrow[thick, #2]{>}}}
    },
    ->-/.default={0.5}{}
}

\tikzset{
    -<-/.style 2 args={
        postaction={decorate},
        decoration={markings, mark=at position #1 with {\arrow[thick, #2]{<}}} 
    },
    -<-/.default={0.5}{}
}

\tikzset{
    uarc/.style={
        red, ultra thick
    }
}

\tikzset{
    oarc/.style={
        basecolor, double=red, double distance=1.6pt, line width=2.4pt
    }
}

\tikzset{
    uwarc/.style={
        line width=1.6pt, red!15, preaction={draw, line width=2.8pt, red}
    }
}

\tikzset{
    owarc/.style={
        line width=1.6pt, red!15, preaction={preaction={draw, line width=6.8pt, basecolor},draw, line width=2.8pt, red}
    }
}

\tikzset{
    oarcblack/.style={
        basecolor, double=black, double distance=1.6pt, line width=2.4pt
    }
}

\tikzset{
	uarcblack/.style={
		black, ultra thick
	}
}

\tikzset{
    uwarcblack/.style={
        line width=1.6pt, black!15, preaction={draw, line width=2.8pt}
    }
}

\tikzset{
	uarcgreen/.style={
		mygreen, ultra thick
	}
}

\tikzset{
    wuarcgreen/.style={
        line width=1.6pt, green!15, preaction={draw, line width=2.8pt, green}
    }
}

\newcommand{\Oweb}[1]{
    \mbox{
        \tikz[baseline=-.6ex, scale=.06]{
            \coordinate (N) at (0,10);
            \coordinate (S) at (0,-10);
            \coordinate (W) at (-10,0);
            \coordinate (E) at (10,0);
            \coordinate (NE) at (45:10);
            \coordinate (NW) at (135:10);
            \coordinate (SE) at (-45:10);
            \coordinate (SW) at (-135:10);
            \coordinate (C) at (0,0);
            \coordinate (CN) at (0,5);
            \coordinate (CS) at (0,-5);
            \coordinate (CW) at (-5,0);
            \coordinate (CE) at (5,0);
            \draw[dashed, fill=basecolor] (C) circle [radius=10cm];
            \draw[#1] (C) circle [radius=5cm];
        }
    }
}

\newcommand{\monoweb}[3][]{
    \mbox{
        \tikz[baseline=-.6ex, scale=.06,#1]{
            \coordinate (N) at (0,10);
            \coordinate (S) at (0,-10);
            \coordinate (W) at (-10,0);
            \coordinate (E) at (10,0);
            \coordinate (NE) at (45:10);
            \coordinate (NW) at (135:10);
            \coordinate (SE) at (-45:10);
            \coordinate (SW) at (-135:10);
            \coordinate (C) at (0,0);
            \coordinate (CN) at (0,5);
            \coordinate (CS) at (0,-5);
            \coordinate (CW) at (-5,0);
            \coordinate (CE) at (5,0);
            \draw[dashed, fill=basecolor] (C) circle [radius=10cm];
            \draw[#2] (CS) -- (S);
            \draw[#3] (C) circle [radius=5cm];
        }
    }
}

\newcommand{\biweb}[4]{
    \mbox{
        \tikz[baseline=-.6ex, scale=.06]{
            \coordinate (N) at (0,10);
            \coordinate (S) at (0,-10);
            \coordinate (W) at (-10,0);
            \coordinate (E) at (10,0);
            \coordinate (NE) at (45:10);
            \coordinate (NW) at (135:10);
            \coordinate (SE) at (-45:10);
            \coordinate (SW) at (-135:10);
            \coordinate (C) at (0,0);
            \coordinate (CN) at (0,5);
            \coordinate (CS) at (0,-5);
            \coordinate (CW) at (-5,0);
            \coordinate (CE) at (5,0);
            \draw[dashed, fill=basecolor] (C) circle [radius=10cm];
            \draw[#4] (CN) -- (N);
            \draw[#1] (CS) -- (S);
            \draw[#2] (CS) arc (-90:90:5cm);
            \draw[#3] (CN) arc (90:270:5cm);
        }
    }
}

\newcommand{\triweb}{
    \mbox{
        \tikz[baseline=-.6ex, scale=.06]{
            \coordinate (N) at (0,10);
            \coordinate (S) at (0,-10);
            \coordinate (W) at (-10,0);
            \coordinate (E) at (10,0);
            \coordinate (NE) at (45:10);
            \coordinate (NW) at (135:10);
            \coordinate (SE) at (-45:10);
            \coordinate (SW) at (-135:10);
            \coordinate (C) at (0,0);
            \coordinate (CN) at (0,5);
            \coordinate (CS) at (0,-5);
            \coordinate (CW) at (-5,0);
            \coordinate (CE) at (5,0);
            \draw[dashed, fill=basecolor] (C) circle [radius=10cm];
            \draw[uwarc] (CN) -- (N);
            \draw[uwarc] (210:5) -- (210:10);
            \draw[uwarc] (330:5) -- (330:10);
            \draw[uarc] (CN) -- (210:5) -- (330:5) -- cycle;
        }
    }
}

\newcommand{\cross}{
    \mbox{
        \tikz[baseline=-.6ex, scale=.06]{
            \draw[fill=red!30, thick] (0,0) circle [radius=30pt];
        }
    }
}

\newcommand{\noweb}{
    \mbox{
        \tikz[baseline=-.6ex, scale=.06]{
            \coordinate (N) at (0,10);
            \coordinate (S) at (0,-10);
            \coordinate (W) at (-10,0);
            \coordinate (E) at (10,0);
            \coordinate (NE) at (45:10);
            \coordinate (NW) at (135:10);
            \coordinate (SE) at (-45:10);
            \coordinate (SW) at (-135:10);
            \coordinate (C) at (0,0);
            \coordinate (CN) at (0,5);
            \coordinate (CS) at (0,-5);
            \coordinate (CW) at (-5,0);
            \coordinate (CE) at (5,0);
            \draw[dashed, fill=basecolor] (C) circle [radius=10cm];
        }
    }
}

\newcommand{\Xpos}[3][]{
    \mbox{
        \tikz[baseline=-.6ex, scale=.06,#1]{
            \coordinate (N) at (0,10);
            \coordinate (S) at (0,-10);
            \coordinate (W) at (-10,0);
            \coordinate (E) at (10,0);
            \coordinate (NE) at (45:10);
            \coordinate (NW) at (135:10);
            \coordinate (SE) at (-45:10);
            \coordinate (SW) at (-135:10);
            \coordinate (C) at (0,0);
            \coordinate (CN) at (0,5);
            \coordinate (CS) at (0,-5);
            \coordinate (CW) at (-5,0);
            \coordinate (CE) at (5,0);
            \draw[dashed, fill=basecolor] (C) circle [radius=10cm];
            \draw[#3] (SE) -- (NW);
            \draw[#2] (SW) -- (NE);
        }
    }
}

\newcommand{\Xneg}[3][]{
    \mbox{
        \tikz[baseline=-.6ex, scale=.06, #1]{
            \coordinate (N) at (0,10);
            \coordinate (S) at (0,-10);
            \coordinate (W) at (-10,0);
            \coordinate (E) at (10,0);
            \coordinate (NE) at (45:10);
            \coordinate (NW) at (135:10);
            \coordinate (SE) at (-45:10);
            \coordinate (SW) at (-135:10);
            \coordinate (C) at (0,0);
            \coordinate (CN) at (0,5);
            \coordinate (CS) at (0,-5);
            \coordinate (CW) at (-5,0);
            \coordinate (CE) at (5,0);
            \draw[dashed, fill=basecolor] (C) circle [radius=10cm];
            \draw[#3] (SW) -- (NE);
            \draw[#2] (SE) -- (NW);
        }
    }
}

\newcommand{\Xcross}{
    \mbox{
        \tikz[baseline=-.6ex, scale=.06]{
            \coordinate (N) at (0,10);
            \coordinate (S) at (0,-10);
            \coordinate (W) at (-10,0);
            \coordinate (E) at (10,0);
            \coordinate (NE) at (45:10);
            \coordinate (NW) at (135:10);
            \coordinate (SE) at (-45:10);
            \coordinate (SW) at (-135:10);
            \coordinate (C) at (0,0);
            \coordinate (CN) at (0,5);
            \coordinate (CS) at (0,-5);
            \coordinate (CW) at (-5,0);
            \coordinate (CE) at (5,0);
            \draw[dashed, fill=basecolor] (C) circle [radius=10cm];
            \draw[uarc] (SE) -- (NW);
            \draw[uarc] (SW) -- (NE);
            \node at (C) {\cross};
        }
    }
}

\newcommand{\XJcross}{
    \mbox{
        \tikz[baseline=-.6ex, scale=.06]{
            \coordinate (N) at (0,10);
            \coordinate (S) at (0,-10);
            \coordinate (W) at (-10,0);
            \coordinate (E) at (10,0);
            \coordinate (NE) at (45:10);
            \coordinate (NW) at (135:10);
            \coordinate (SE) at (-45:10);
            \coordinate (SW) at (-135:10);
            \coordinate (C) at (0,0);
            \coordinate (CN) at (0,5);
            \coordinate (CS) at (0,-5);
            \coordinate (CW) at (-5,0);
            \coordinate (CE) at (5,0);
            \draw[dashed, fill=basecolor] (C) circle [radius=10cm];
            \draw[uwarc] (NE) -- ($(NE)!.5!(C)$);
            \draw[uwarc] (NW) -- ($(NW)!.5!(C)$);
            \draw[uwarc] (SE) -- ($(SE)!.5!(C)$);
            \draw[uwarc] (SW) -- ($(SW)!.5!(C)$);
            \draw[uarc] ($(NE)!.5!(C)$) -- ($(NW)!.5!(C)$) -- ($(SW)!.5!(C)$) -- ($(SE)!.5!(C)$) -- cycle;
        }
    }
}

\newcommand{\Iweb}[2][]{
    \mbox{
        \tikz[baseline=-.6ex, scale=.06, #1]{
            \coordinate (N) at (0,10);
            \coordinate (S) at (0,-10);
            \coordinate (W) at (-10,0);
            \coordinate (E) at (10,0);
            \coordinate (NE) at (45:10);
            \coordinate (NW) at (135:10);
            \coordinate (SE) at (-45:10);
            \coordinate (SW) at (-135:10);
            \coordinate (C) at (0,0);
            \coordinate (CN) at (0,5);
            \coordinate (CS) at (0,-5);
            \coordinate (CW) at (-5,0);
            \coordinate (CE) at (5,0);
            \draw[dashed, fill=basecolor] (C) circle [radius=10cm];
            \draw[#2] (S) -- (N);
        }
    }
}

\newcommand{\IIweb}[3][]{
    \mbox{
        \tikz[baseline=-.6ex, scale=.06, #1]{
            \coordinate (N) at (0,10);
            \coordinate (S) at (0,-10);
            \coordinate (W) at (-10,0);
            \coordinate (E) at (10,0);
            \coordinate (NE) at (45:10);
            \coordinate (NW) at (135:10);
            \coordinate (SE) at (-45:10);
            \coordinate (SW) at (-135:10);
            \coordinate (C) at (0,0);
            \coordinate (CN) at (0,5);
            \coordinate (CS) at (0,-5);
            \coordinate (CW) at (-5,0);
            \coordinate (CE) at (5,0);
            \draw[dashed, fill=basecolor] (C) circle [radius=10cm];
            \draw[#2] (SW) to[bend right=60] (NW);
            \draw[#3] (SE) to[bend left=60] (NE);
        }
    }
}

\newcommand{\myHweb}[6][]{
    \mbox{
        \tikz[baseline=-.6ex, scale=.06,#1]{
            \coordinate (N) at (0,10);
            \coordinate (S) at (0,-10);
            \coordinate (W) at (-10,0);
            \coordinate (E) at (10,0);
            \coordinate (NE) at (45:10);
            \coordinate (NW) at (135:10);
            \coordinate (SE) at (-45:10);
            \coordinate (SW) at (-135:10);
            \coordinate (C) at (0,0);
            \coordinate (CN) at (0,5);
            \coordinate (CS) at (0,-5);
            \coordinate (CW) at (-5,0);
            \coordinate (CE) at (5,0);
            \draw[dashed, fill=basecolor] (C) circle [radius=10cm];
            \draw[#4] (CW) -- (CE);
            \draw[#2] (SW) -- (CW);
            \draw[#5] (CW) -- (NW);
            \draw[#3] (SE) -- (CE);
            \draw[#6] (CE) -- (NE);
        }
    }
}

\tikzset{
	overarc/.style={
		white, double=red, double distance=1.2pt, line width=2.4pt
	}
}

\tikzset{
  symbol/.style={
    draw=none,
    every to/.append style={
      edge node={node [sloped, allow upside down, auto=false]{$#1$}}}
  }
}

\newcommand{\bline}[3]{
    \path (#1)++(0,-#3) coordinate(m1);
    \path (#2)++(0,-#3) coordinate(m2);
    \filldraw[gray!30] (m1) -- (#1) -- (#2) -- (m2) --cycle;
    \draw[thick] (#1) -- (#2);
}

\tikzset{
    wline/.style={
        shorten >=0.3pt,shorten <=0.3pt,line width=1.2pt, red!15, preaction={draw, shorten >=0.3pt,shorten <=0.3pt,line width=2.5pt, red}
    }
}

\tikzset{
    wlineblack/.style={
        shorten >=0.3pt,shorten <=0.3pt,line width=1.2pt, black!15, preaction={draw, shorten >=0.3pt,shorten <=0.3pt,line width=2.5pt}
    }
}

\tikzset{
    wlinebdy/.style={
        shorten >=0.3pt,shorten <=0.3pt,line width=1.2pt, myblue!15, preaction={draw, shorten >=0.3pt,shorten <=0.3pt,line width=2.5pt, myblue}
    }
}

\tikzset{
	webline/.style={
		red, very thick
	}
}

\newcommand{\triv}[3]{
    \CoG{#1}{#2}{#3}
    \draw[wline] (#1) -- (G);
    \draw[webline] (#2) -- (G);
    \draw[webline] (#3) -- (G);
}

\tikzset{
    wlined/.style={
        shorten >=0.3pt,shorten <=0.3pt,line width=1.2pt, red!15, preaction={draw, shorten >=0.3pt,shorten <=0.3pt,line width=2.5pt, mygreen}
    }
}

\tikzset{
	weblined/.style={
		mygreen, very thick
	}
}

\newcommand{\trivd}[3]{
    \CoG{#1}{#2}{#3}
    \draw[weblined] (#1) -- (G);
    \draw[wlined] (#2) -- (G);
    \draw[wlined] (#3) -- (G);
}

\newcommand{\tribox}{
    \mbox{
        \tikz[baseline=-.6ex, scale=.1]{
            \fill[red!10] (0:5) -- (90:5) -- (180:5) -- cycle;
            \draw[wlineblack] (180:5) -- (0:5);
            \draw[thick] (0:5) -- (90:5) -- (180:5);
        }
    }
}

\newcommand{\sqbox}{
    \mbox{
        \tikz[baseline=-.6ex, scale=.1]{
            \fill[red!10] (45:5) -- (135:5) -- (225:5) -- (315:5) -- cycle;
            \draw[thick] (45:5) -- (135:5) -- (225:5) -- (315:5) -- cycle;
        }
    }
}

\newcommand{\clbox}{
    \mbox{
        \tikz[baseline=-.6ex, scale=.08]{
            \draw[fill=white] (-4,-1) rectangle (4,1);
        }
    }
}

\newcommand{\crossroad}{
    \tikz[baseline=-.6ex, scale=.1]{
        \draw[fill=pink, thick] (-5,0) circle [radius=20pt];
    }
}

\pgfdeclarelayer{bg}    
\pgfsetlayers{bg,main}  

\title[Bounded $\mathfrak{sp}_4$-laminations and their intersection coordinates]{Bounded $\mathfrak{sp}_4$-laminations and their intersection coordinates}

\author[Tsukasa Ishibashi]{Tsukasa Ishibashi}
\address{Tsukasa Ishibashi, Mathematical Institute, Tohoku University, 
6-3 Aoba, Aramaki, Aoba-ku, Sendai, Miyagi 980-8578, Japan.}
\email{tsukasa.ishibashi.a6@tohoku.ac.jp}

\author[Zhe Sun]{Zhe Sun}
\address{School of Mathematical Sciences, University of Science and Technology of China, 96 Jinzhai Road, 230026, Hefei, China}
\email{sunz@ustc.edu.cn}

\author[Wataru Yuasa]{Wataru Yuasa}
\address{International Institute for Sustainability with Knotted Chiral Meta Matter\\
Hiroshima University\\
2-313 Kagamiyama, Higashi-Hiroshima City, Hiroshima, 739-0046, Japan}
\email{wyuasa@hiroshima-u.ac.jp}
\urladdr{https://wataruyuasa.github.io/math/} 

\date{\today}

\begin{document}

\begin{abstract}
We introduce rational bounded $\mathfrak{sp}_4$-laminations on a marked surface $\boldsymbol{\Sigma}$ as a proposed topological model for the rational tropical points $\mathcal{A}_{Sp_4,\boldsymbol{\Sigma}}(\mathbb{Q}^{\mathsf{T}})$ of the Fock--Goncharov moduli space \cite{FG06}. Our space consists of certain equivalence classes of $\mathfrak{sp}_4$-webs introduced by Kuperberg \cite{Kuperberg}, together with rational measures. We define tropical coordinate systems using the $\mathfrak{sp}_4$-case of the intersection number of Shen--Sun--Weng \cite{SSW25}, and establish a bijection using the framework of the graded $\mathfrak{sp}_4$-skein algebra. This provides a topological perspective for Fock--Goncharov duality for $\mathfrak{sp}_4$. 
\end{abstract}

\maketitle

\setcounter{tocdepth}{1}
\tableofcontents

\section{Introduction}

\subsection{Measured geodesic laminations and tropical points of cluster varieties}
W. P. Thurston \cite{Thu} introduced the concept of measured geodesic laminations as a continuous generalization (or certain limits) of curves on a hyperbolic surface $\Sigma$. Since then, the measured geodesic laminations played significant roles in hyperbolic geometry and related subjects. Among others, their typical roles are:
 \begin{description}
    \item[(A) Boundary data at infinity] The points in the Thurston boundary of the \Teich\ space $\mathcal{T}(\Sigma)$ are represented by projective measured geodesic laminations. They encode the degenerations of hyperbolic structures.
    \item[(B) Trace function basis] For an integral lamination ({\it i.e.} a weighted multicurve), the associated trace function of monodromy defines a regular function on the $SL_2$-character variety $\mathrm{Ch}_{SL_2,\Sigma}=\Hom(\pi_1(\Sigma),SL_2)\sslash SL_2$. These functions form the basis of the coordinate ring $\cO(\mathrm{Ch}_{SL_2,\Sigma})$.
    \item[(C) Basis of the skein algebra] The coordinate ring has the Atiyah--Bott--Goldman Poisson structure, and its deformation quantization is given by the Kauffman bracket skein algebra \cite{PS00}. The integral laminations also provide a basis for this skein algebra. 
\end{description}
The measured geodesic laminations, together with these perspectives, also provide an effective tool for the study of mapping class groups. 

\paragraph{\textbf{Connection to the cluster theory}}
Investigation of measured geodesic laminations via certain coordinate systems \cite{FG07} has led to a fruitful connection to the theory of cluster algebras/cluster varieties. 
For a marked surface $\bSigma=(\Sigma,\bM)$, associated is the cluster ensemble $(\A_\bSigma,\X_\bSigma)$ \cite{FG09}, which is birationally isomorphic to the pair $(\A_{SL_2,\bSigma},\X_{PGL_2,\bSigma})$ of moduli spaces introduced in \cite{FG06}. These moduli spaces are two extensions of the $SL_2$-character variety $\mathrm{Ch}_{SL_2,\Sigma}$ with different decorations at the marked points.
The sets $(\A_\bSigma(\bR^\sfT),\X_\bSigma(\bR^\sfT))$ of real tropical points of the same cluster ensemble (which is Langlands self-dual in this case) are identified with extended spaces of measured geodesic laminations \cite{FG07} via certain coordinates associated with ideal triangulations. The weighted multicurves corresponds to the integral tropical points. See \cite{PP93,FG07} for more details. 

\paragraph{\textbf{Duality maps}}
In their seminal work \cite{FG09}, Fock--Goncharov formulated the \textit{(quantum) duality conjecture} for any cluster ensembles $(\A,\X)$. It asks for a parametrization
\begin{align}
    &\mathbb{I}_\A: \A^\vee(\bZ^\sfT) \to \cO(\X), \label{eq:duality_map_A} \\
    &\mathbb{I}_\X: \X^\vee(\bZ^\sfT) \to \cO(\A) \label{eq:duality_map_X}
\end{align}
of bases of regular function rings $(\cO(\A),\cO(\X))$ of cluster varieties by the sets $(\A^\vee(\bZ^\sfT)$, $\X^\vee(\bZ^\sfT))$ of integral tropical points of Langlands dual cluster varieties, together with certain good properties such as structure constant positivity. There is also the quantum version of conjecture, where we replace the target with $\cO_q(\X)$ or the quantum upper cluster algebra $\cO_q(\A)$ of full rank. A seminal work of Gross--Hacking--Keel--Kontsevich \cite{GHKK} establishes a general construction of duality maps under mild conditions. A quantum version is given by Davison--Mandel \cite{DM} in the skew-symmetric case, while the skew-symmetrizable case is not established so far.  

As the initial observation of Fock--Goncharov suggests, the integral laminations plays a key role in a topological construction of (quantum) duality maps for the cluster ensemble  $(\A_\bSigma,\X_\bSigma)$. Several works have been conducted around this, including \cite{FG06,MSW,Thu14,Muller16,AK,MQ,IKar}. These constructions use some extensions of the trace function basis and their quantization via skein algebras, based on the perspectives (B) and (C) above. 

\paragraph{\textbf{Tropical compactification}}
The classical perspective (A) is also generalized in the cluster theory, where we have the tropical compactification of the spaces $(\A(\bR_{>0}),\X(\bR_{>0}))$ of positive real points for any cluster variety \cite{FG16,Le16}, where the boundary at infinity is given by the projectivizations $(\mathbb{P} \A(\bR^\sfT), \mathbb{P} \X(\bR^\sfT))$. From this point of view, some studies for the mapping class groups are generalized to those for cluster modular groups \cite{Ish19,IK19,IK20}.

\subsection{Search for higher rank generalizations}
For any complex semisimple Lie algebra $\mathfrak{g}$ and a marked surface $\bSigma$, there is the cluster ensemble $(\A_{\mathfrak{g},\bSigma},\X_{\mathfrak{g},\bSigma})$ \cite{FG06,Le19,GS19}, which is birationally isomorphic to the Fock--Goncharov's moduli spaces $(\A_{G,\bSigma},\X_{G',\bSigma})$ with $G$ the simply-connected group, $G'$ the adjoint group having the Lie algebra $\mathfrak{g}$. The previous example corresponds to the rank one case $\mathfrak{g}=\fsl_2$. 

Based on the successful connection in the rank one case, search for a topological realization of real tropical points $(\A_{\mathfrak{g},\bSigma}(\bR^\sfT),\X_{\mathfrak{g},\bSigma}(\bR^\sfT))$ -- higher rank versions of measured geodesic laminations -- has been an active topic of research. It is expected that such a topological model is given by $\mathfrak{g}$-diagrams on $\bSigma$, certain graphs with edge coloring by representations of $\mathfrak{g}$. 
Let us call them bounded (resp. unbounded) if they end at the boundary intervals (resp. marked points). For later convenience, we distinguish them from the corresponding elements in $\mathfrak{g}$-skein algebras, latter being called the $\mathfrak{g}$-webs. 

\paragraph{\textbf{The case $\mathfrak{g}=\fsl_3$.}} In this case, a topological model for the tropical points $\A^+_{\fsl_3,\bSigma}(\bQ^\sfT)$ subordinate to certain inequalities is firstly given by Douglas and the second author \cite{DS20I,DS20II} by using the Kuperberg's $\fsl_3$-webs \cite{Kuperberg} ($\fsl_3$-diagrams in our terminology). 
%
%
They constructed certain coordinate systems for the space of bounded $\fsl_3$-diagrams associated with ideal triangulations of $\bSigma$, and proved that the coordinate transformations under the change of triangulations agrees with the tropical cluster transformations. These coordinate systems provide an equivariant identification of this space of bounded $\fsl_3$-diagrams with $\A^+_{\fsl_3,\bSigma}(\bZ^\sfT)$ under the natural mapping class group actions. The extension to the rational case is straightforward. Then Kim \cite{Kim21} extended such identification to the identification of the space of bounded rational $\fsl_3$-laminations (which allow peripheral loops/arcs with negative weights) with $\A_{\fsl_3,\bSigma}(\bQ^\sfT)$.
Based on this identification, Kim \cite{Kim21} constructed a candidate quantum duality map in the direction \eqref{eq:duality_map_A}, relating the target algebra $\cO_q(\X_{\fsl_3,\bSigma})$ to the stated $\fsl_3$-skein algebra via quantum trace maps. 

The tropical space $\X_{\fsl_3,\bSigma}(\bQ^\sfT)$ is also identified with a space of ``signed'' unbounded $\fsl_3$-laminations by the first author and Kano \cite{IK22,IK25}, where they have different behavior at punctures from the bounded case. When there are no punctures, 
Frohman--Sikora \cite{FS22} defined an $\fsl_3$-skein algebra, which can be viewed as the $\fsl_3$-version of the Muller skein algebra \cite{Muller16}. The first author and the third author gave a (partially conjectural) identification between this skein algebra with the quantum (upper) cluster algebra $\cO_q(\A_{\fsl_3,\bSigma})$ \cite{IYsl3}. Based on this identification, a candidate quantum duality map in the direction \eqref{eq:duality_map_X} is proposed in \cite{IK22}. When there are punctures, a version of $\fsl_3$-skein algebra is investigated in \cite{SSWa}.

\subsection{Main statement}
In this paper, we investigate the case $\mathfrak{g}=\fsp_4$ (Type $C_2$). This is the first non-trivial case of non-symmetric type. We use the Kuperberg's $\fsp_4$-webs \cite{Kuperberg} ($\fsp_4$-diagrams in our terminology), whose local pictures are

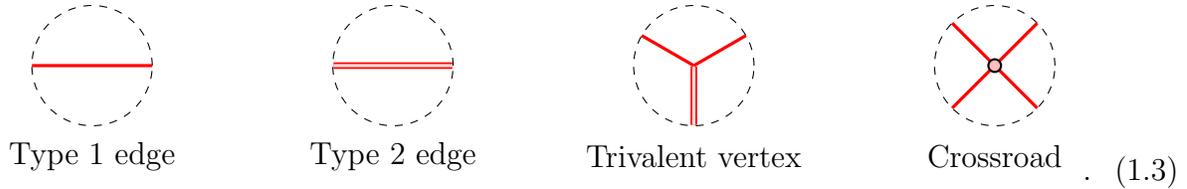
\begin{align}\label{eq:intro_local_sp4}
\begin{tikzpicture}[scale=0.8]
\draw[dashed] (0,0) circle(1cm);
\draw[webline] (-1,0) -- (1,0);
\node at (0,-1.5) {Type~1 edge};
\begin{scope}[xshift=5cm]
\draw[dashed] (0,0) circle(1cm);
\draw[wline] (-1,0) -- (1,0);
\node at (0,-1.5) {Type~2 edge};
\end{scope}
\begin{scope}[xshift=10cm]
\draw[dashed] (0,0) circle(1cm);
{\color{red} \triv{-90:1}{30:1}{150:1};}
\node at (0,-1.5) {Trivalent vertex};
\end{scope}
\begin{scope}[xshift=15cm]
\draw[dashed] (0,0) circle(1cm);
\draw[webline] (45:1) -- (-135:1);
\draw[webline] (-45:1) -- (135:1);
\draw[fill=pink,thick] (0,0) circle(3pt);
\node at (0,-1.5) {Crossroad};
\end{scope}
\end{tikzpicture}\ .
\end{align}

A version of $\fsp_4$-skein algebra, which is similar to Muller's \cite{Muller16} and Frohman--Sikora's \cite{FS22} ones, has been investigated by the first author and the third author \cite{IYsp4}, whose elements are represented by unbounded $\fsp_4$-diagrams. Let us call it the \emph{clasped $\fsp_4$-skein algebra}, as the marked points play the role of external clasps. They provided a (partially conjectural) identification between this skein algebra with the quantum (upper) cluster algebra $\cO_q(\A_{\fsp_4,\bSigma})$ for the unpunctured surface. 
In particular, they gave a natural basis $\Bweb_{\fsp_4,\bSigma}$; based on the perspective (C), the integral $\fsp_4$-laminations should be in bijection with the set $\Bweb_{\fsp_4,\bSigma}$. 

The elements of $\Bweb_{\fsp_4,\bSigma}$ are typically represented by diagrams with crossroads, the right-most 4-valent vertex in \eqref{eq:intro_local_sp4}. While this is a kind of canonical representative, we find that it is better to allow the \textit{ladder-resolutions} 
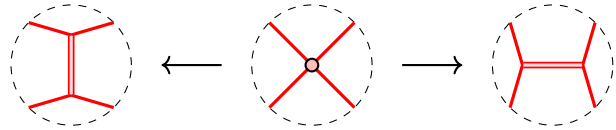
\begin{align*}
\begin{tikzpicture}[scale=0.8]
\draw[dashed] (0,0) circle(1cm);
\draw[webline] (45:1) -- (0,0.5);
\draw[webline] (135:1) -- (0,0.5);
\draw[webline] (-45:1) -- (0,-0.5);
\draw[webline] (-135:1) -- (0,-0.5);
\draw[wline] (0,0.5) -- (0,-0.5);
\begin{scope}[xshift=4cm]
\draw[dashed] (0,0) circle(1cm);
\draw[webline] (45:1) -- (-135:1);
\draw[webline] (-45:1) -- (135:1);
\draw[fill=pink,thick] (0,0) circle(3pt);
\end{scope}
\draw[thick,->] (2.5,0) -- (1.5,0);
\draw[thick,->] (5.5,0) -- (6.5,0);
\begin{scope}[xshift=8cm]
\draw[dashed] (0,0) circle(1cm);
\draw[webline] (45:1) -- (0.5,0);
\draw[webline] (-45:1) -- (0.5,0);
\draw[webline] (135:1) -- (-0.5,0);
\draw[webline] (-135:1) -- (-0.5,0);
\draw[wline] (0.5,0) -- (-0.5,0);
\end{scope}
\end{tikzpicture}
\end{align*}
in considering their coordinates. Thus we introduce the set $\Blad_{\fsp_4,\bSigma}$ of \textit{ladder-equivalence classes} of $\fsp_4$-diagrams, including the above move as an equivalence relation, and define the \textit{rational bounded $\fsp_4$-laminations} to be the elements of $\Blad_{\fsp_4,\bSigma}$ equipped with rational measures (\cref{def:lamination}). Let $\cL^a(\bSigma,\bQ)$ be the set formed by them. Note the following distinction:
\begin{align*}
\begin{tikzpicture}
    \node (A) at (0,0) {\{bounded $\fsp_4$-diagrams\}};
    \node (B) at (0,-2) {$\Blad_{\fsp_4,\bSigma}$};
    \node (C) at (6,-2) {$\Bweb_{\fsp_4,\bSigma}$};
    \node (A') at (6,0) {\{unbounded $\fsp_4$-diagrams\}};
    \draw[->] (A.south) --node[midway,left]{ladder-equiv.} (B.north);
    \draw[->] (A'.south) --node[midway,right]{skein rel.} (C.north);
    \draw[->] (A.east) --node[midway,above]{shift} (A'.west);
    \node[below] at (B.south) {(Lamination)};
    \node[below] at (C.south) {(Skein)};
\end{tikzpicture}
\end{align*}
Here the shift will be explained in \cref{dfn:shift}. 

In \cite{IYsp4}, the notion of \textit{web clusters} is introduced, which is expected to correspond to quantum clusters in the $\fsp_4$-skein algebra.\footnote{More precisely, we should consider web clusters formed by tree-type elementary webs, following \cite[Conjecture 4]{IYsp4}.} We can dualize this notion to be $\fso_5$-web clusters, where $\fso_5$ is the Langlands dual Lie algebra (Type $B_2$).\footnote{Here we remark that $\fsp_4 \cong \fso_5$ as Lie algebras, and the corresponding quantum cluster algebras are isomorphic via a certain permutation of quantum cluster variables. The skein algebras are isomorphic as well; only the graphical presentation is dualized. However, it is natural to consider the intersection pairing in view of $\fsp_4/\fso_5$ duality rather than working only with $\fsp_4$.}  
We introduce \textit{intersection coordinates} (see \cref{subsub:intersection} below) 
\begin{align}\label{eq:intro_coord_system}
    \bsfa^\cC: \cL^a(\bSigma,\bQ) \to \bQ^I
\end{align}
associated with an $\fso_5$-web cluster $\cC=\{V_i \}_{i \in I}$. In particular, for any decorated triangulation $\bD=(\tri,m_\tri,\bs_\tri)$ of $\bSigma$, there is an associated web cluster $\cC_\bD$ (corresponding to a quantum $\fso_5$-cluster) and the coordinate map $\bsfa^{\bD}:=\bsfa^{\cC_\bD}$. Our main theorem is the bijectivity of this map:

\begin{introthm}[\cref{thm:bijection}]\label{introthm:bijection}
For any decorated triangulation $\bD=(\tri,m_\tri,\bs_\tri)$ of $\bSigma$, the coordinate system 
\begin{align*}
    \bsfa^{\bD}: \cL^a
    (\bSigma,\bQ) \xrightarrow{\sim} \bQ^{I(\bD)}
\end{align*}
gives a bijection. 
\end{introthm}

In our sequel paper \cite{ISY}, we prove that the tropical seeds $(\ve^{\bD},\bsfa^{\bD})$ associated with any decorated triangulations $\bD$ are mutation-equivalent to each other. In particular, they combine to give a mapping class group equivariant PL isomorphism
\begin{align}\label{introeq:tropical_variety}
        \bsfa: \cL^a(\bSigma,\bQ) \xrightarrow{\sim} \A_{\fsp_4,\bSigma}(\bQ^\sfT). 
\end{align} 

Thus our space $\cL^a(\bSigma,\bQ)$ provides a topological model for the rational tropical points $\A_{\fsp_4,\bSigma}(\bQ^\sfT)$. We also introduce the subset $\cL^a(\bSigma,\bZ)_\congr \subset \cL^a(\bSigma,\bQ)$ of \emph{congruent laminations} (\cref{def:congruent}), which is identified with the set $\A_{\fsp_4,\bSigma}(\bZ^\sfT)$ of integral tropical points. 

\subsection{A recipe for quantum Fock--Goncharov duality}
The identification \eqref{introeq:tropical_variety} would lead to a topological construction of quantum Fock--Goncharov duality map
\begin{align*}
    \mathbb{I}_\A: \A_{\fsp_4,\bSigma}(\bZ^\sfT) \to \cO_q(\P_{PSO_5,\bSigma}),
\end{align*}
as follows. It will be discussed in detail in our sequel paper \cite{ISY}. 

\paragraph{\textbf{Step 1: Skein realization.}}
Given an integral bounded $\fsp_4$-lamination $L \in \cL^a(\bSigma,\bZ)$, take its representive with measure $\pm 1$ on each component.\footnote{We are going to give an $\fsp_4$-version of the Bangle basis for the $\fsl_2$-case rather than the Bracelet basis, as we do not know a correct analogue of the latter.} We define the corresponding element $S_\A(L)$ in the \emph{reduced stated $\fsp_4$-skein algebra} $\overline{\mathscr{S}}^q_{\!\fsp_4,\bSigma}(\bB)$ by taking the constant-elevation lift of each component, and assigning the lowest (resp. highest) state to each end on the boundary interval if the component has measure $+1$ (resp. $-1$). Then $S_\A(L)$ is defined to be the product of these elements, where the elevations are chosen appropriately and the Weyl normalization (cf. \cite[Definition 5.2]{IKar}) is applied. 
\paragraph{\textbf{Step 2: Splitting map}}
For any ideal triangulation $\tri$ of $\bSigma$, we define the splitting map, which injectively maps $\overline{\mathscr{S}}^q_{\!\fsp_4,\bSigma}(\bB)$ into the tensor product of the reduced stated skein algebras of the triangles $T$ of $\tri$. 

\paragraph{\textbf{Step 3: Dual quantum trace map}}
Given a decorated triangulation $\bD$, we define an embedding of $\overline{\mathscr{S}}^q_{\!\fsp_4,\bSigma}(\bB)$ into a quantum torus $\mathcal{Z}_{\bD}^\vee$ associated with the corresponding $\fso_5$-cluster, which we call the \emph{square-root $\fso_5$-Goncharov--Shen algebra}. It is constructed via splitting map, assigning each piece in the triangle with an appropriate quantum Laurent polynomial. 
In particular, each stated diagram with underlying diagram in \cref{fig:coordinate_corners,fig:coordinates_elementary} is send to a quantum Laurent monomial. 

These constructions are summarized in the diagram
\begin{align*}
    \mathbb{I}_{\A}^{\bD}: \A_{\fsp_4,\bSigma}(\bZ^\sfT)\subset \cL^a(\bSigma,\bZ) \xrightarrow{S_\A} \overline{\mathscr{S}}^q_{\!\fsp_4,\bSigma}(\bB) \xrightarrow{\mathrm{Tr}_{\bD}} \mathcal{Z}_{\bD}^\vee.
\end{align*}
We show that the image of $\mathbb{I}_{\A}^{\bD}$ lies in the $\fso_5$-Goncharov--Shen algebra $\X_\bD^\vee$ quantizing the Goncharov--Shen coordinates \cite{GS19}, and that the composite $\mathbb{I}_\A^{\bD'} \circ (\mathbb{I}_\A^{\bD})^{-1}$ coincides with the quantum cluster Poisson transformation for any decorated triangulations $\bD,\bD'$. Thus we get a candidate quantum duality map
\begin{align*}
    \widetilde{\mathbb{I}}_\A: \A_{\fsp_4,\bSigma}(\bZ^\sfT)\subset \cL^a(\bSigma,\bZ) \to \cO_q(\widetilde{\P}_{PSO_5,\bSigma}) := \bigcap_{\bD} \X_\bD^\vee,
\end{align*}
where $\bD$ runs over all the decorated triangulations.
We show that the lowest term of the quantum Laurent polynomial $\mathbb{I}_\A^{\bD}(L)$ coincides with the tropical coordinate $\bsfa^\bD(L)$. The mutation-equivalence of the coordinates $\bsfa^\bD$ is proved as a consequence. 

\subsection{Tools and techniques}
Here are our essential tools and techniques for the construction and the proof of \cref{introthm:bijection}. 

\subsubsection{Intersection pairing between $\fsp_4$- and $\fso_5$-diagrams}\label{subsub:intersection}
In \cite[Remark 2.11]{SSW25}, Shen, Sun and Weng introduced the \textit{intersection pairing} between a local $\mathfrak{g}$- and $\mathfrak{g}^\vee$-diagrams, 
which is given by the matrix entry $C^{st}$ of the inverse of Cartan matrix of $\mathfrak{g}$ when two local diagrams carry the colors $s$ and $t$, respectively. It is expected to give a topological description of the \textit{tropical pairing}
of \cite[Conjecture 4.3]{FG09}.

In the present work, we adopt this definition to our global diagrams, and define the intersection pairing $\bi_\bSigma(W,V) \in \frac 1 2 \bZ_{\geq 0}$ between a bounded $\fsp_4$-diagram $W$ and an unbounded $\fso_5$-diagram $V$ (\cref{sec:pairing}) as the weighted sum of local contributions at their intersection points. The local rule is the following:
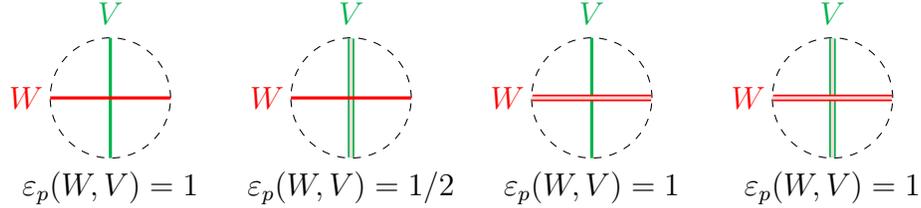
\begin{align*}
\begin{tikzpicture}[scale=0.8]\draw[dashed] (0,0) circle(1cm);
\draw[weblined] (0,-1) -- (0,1);
\draw[webline] (-1,0) -- (1,0);
\node at (0,-1.5) {$\varepsilon_p(W,V)=1$};
\node[red] at (-1.4,0) {$W$};
\node[mygreen] at (0,1.4) {$V$};
\begin{scope}[xshift=4cm]
\draw[dashed] (0,0) circle(1cm);
\draw[wlined] (0,-1) -- (0,1);
\draw[webline] (-1,0) -- (1,0);
\node at (0,-1.5) {$\varepsilon_p(W,V)=1/2$};
\node[red] at (-1.4,0) {$W$};
\node[mygreen] at (0,1.4) {$V$};
\end{scope}
\begin{scope}[xshift=8cm]
\draw[dashed] (0,0) circle(1cm);
\draw[weblined] (0,-1) -- (0,1);
\draw[wline] (-1,0) -- (1,0);
\node at (0,-1.5) {$\varepsilon_p(W,V)=1$};
\node[red] at (-1.4,0) {$W$};
\node[mygreen] at (0,1.4) {$V$};
\end{scope}
\begin{scope}[xshift=12cm]
\draw[dashed] (0,0) circle(1cm);
\draw[wlined] (0,-1) -- (0,1);
\draw[wline] (-1,0) -- (1,0);
\node at (0,-1.5) {$\varepsilon_p(W,V)=1$};
\node[red] at (-1.4,0) {$W$};
\node[mygreen] at (0,1.4) {$V$};
\end{scope}
\end{tikzpicture}
\end{align*}
We can upgrade the intersection pairing to a map
\begin{align*}
    \bi_\bSigma: \cL^a(\bSigma,\bQ) \times \Blad_{\fso_5,\Sigma}^\infty \to \bQ,
\end{align*}
where we take the minimum of intersection numbers among the representatives of the ladder-equivalence classes, and then consider the weight sum with respect to the rational measure of a given $\fsp_4$-lamination. 

For a given $\fso_5$-web cluster $\cC=\{V_i\}_{i \in I}$, we define the coordinate system \eqref{eq:intro_coord_system} by $\bsfa^\cC(L):=(\bi_\bSigma(L,V_i))_{i \in I}$ for $L \in \cL^a(\bSigma,\bQ)$. We expect that any such web cluster provides a bijection $\bsfa^\cC: \cL^a(\bSigma,\bQ) \to \bQ^I$.
The statement of \cref{introthm:bijection} is for the special case where $\cC$ is associated with a decorated triangulation $\bD$. 

\subsubsection{Hilbert basis and $\fsp_4$-diagrams on a triangle}
The triangle case $\bSigma=T$ is already non-trivial in the case of $\fsp_4$, since it should be equipped with a non-trivial cluster structure of finite type $C_2$. Indeed, we have $6=3 \times 2$ choices of decorations of the unique triangulation of $T$, namely the choices of a corner and a sign (which corresponds to a choice of reduced word, $1212$ or $2121$ of $w_0$). See the left picture in \cref{fig:intro_hexagon}. 
For each decorated triangle $\bD$, we have an associated quiver $Q_\bD$ and a tropical coordinate system $\bsfa_\bD$ (consisting of 8 vertices/coordinates). 

On the other hand, we have a classification of bounded $\fsl_4$-diagrams on a triangle up to symmetry of order 6, which involves 8 `geometric' parameters $(k_1,k_2,k_3,l_1,l_2,l_3,n_1,n_2)$ (\cref{lem:triangle_blad}):
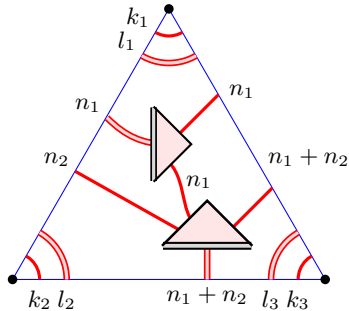
\begin{align*}
\begin{tikzpicture}[scale=.12]
    \coordinate (N) at (90:20);
    \coordinate (S) at (-90:20);
    \coordinate (W) at (180:20);
    \coordinate (E) at (0:20);
    \coordinate (NE) at (45:20);
    \coordinate (NW) at (135:20);
    \coordinate (SE) at (-45:20);
    \coordinate (SW) at (-135:20);
    \coordinate (C) at (0,0);
    \coordinate (CN) at ($(C)!.5!(N)$);
    \coordinate (CS) at ($(C)!.5!(S)$);
    \coordinate (CW) at ($(C)!.5!(W)$);
    \coordinate (CE) at ($(C)!.5!(E)$);
    \coordinate (P1) at (90:20);
    \coordinate (P2) at (210:20);
    \coordinate (P3) at (-30:20);
    \draw[blue] (P1) -- (P2) -- (P3)--cycle;
    \foreach \i in {-30,210,90} \fill(\i:20) circle(15pt);
    \begin{scope}
        \clip (P1) -- (P2) -- (P3)--cycle;
        \draw[webline] ($(C)!.4!(SE)$) --+ (45:10);
        \draw[webline] ($(C)!.3!(SE)$) to[out=west, in=south east] ($(C)!.3!(CN)$);
        \draw[webline] ($(C)!.3!(SE)+(0,-2)$) -- ($(90:20)!.6!(210:20)$);
        \draw[wline] ($(C)!.3!(SE)$) -- +(0,-10);
        \draw[webline] ($(C)!.5!(CN)$) --+ (45:10);
        \draw[wline] ($(C)!.5!(CN)$) to[out=west, in=south east] ($(90:20)!.4!(210:20)$);
        \node at ($(C)!.3!(SE)$) [scale=1.2]{$\tribox$};
        \node at ($(C)!.5!(CN)$) [rotate=-90]{$\tribox$};
        \foreach \i in {1,2,3}{
            \draw[webline] (P\i) circle [radius=3cm];
            \draw[wline] (P\i) circle [radius=6cm];
        }
    \end{scope}
    \node at ($(P1)!.4!(P2)$) [above left=-.1cm]{\scriptsize $n_1$};
    \node at ($(P1)!.6!(P2)$) [above left=-.1cm]{\scriptsize $n_2$};
    \node at ($(C)!.3!(SE)$) [below=.65cm]{\scriptsize $n_1+n_2$};
    \node at ($(P1)!.6!(P3)$) [above right=-.1cm]{\scriptsize $n_1+n_2$};
    \node at ($(P1)!.35!(P3)$) [above right=-.1cm]{\scriptsize $n_1$};
    \node at (C) [yshift=.1cm, xshift=.4cm]{\scriptsize $n_1$};
    \node at ($(P1)+(-120:3)$) [above left=-.1cm]{\scriptsize $k_1$};
    \node at ($(P1)+(-120:6)$) [above left=-.1cm]{\scriptsize $l_1$};
    \node at ($(P2)+(0:3)$) [below]{\scriptsize $k_2$};
    \node at ($(P2)+(0:6)$) [below]{\scriptsize $l_2$};
    \node at ($(P3)+(0:-3)$) [below]{\scriptsize $k_3$};
    \node at ($(P3)+(0:-6)$) [below]{\scriptsize $l_3$};
\end{tikzpicture}
\end{align*}
These parameters are non-negative and would span a non-trivial cone in the coordinate space. According to the 3 `directions' (corresponding to the choice of corners) and the 2 `chiralities' (corresponding to the choice of sign), we would obtain 6 cones. In \cref{subsec:triangle_case}, we investigate a coordinate change between the cluster coordinates and these geometric parameters. As a result, we establish the bijectivity of the coordinate system in the triangle case. 

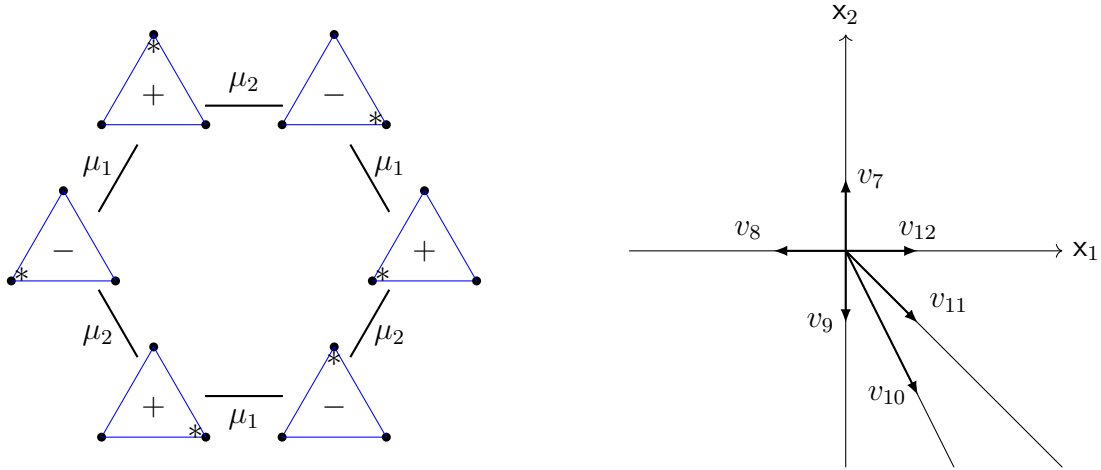
\begin{figure}[ht]
    \centering
\begin{tikzpicture}[scale=0.8]
\begin{scope}[xshift=3cm]
    \foreach \i in {90,210,330}
    {
    \markedpt{\i:1};
    \draw[blue] (\i:1) -- (\i+120:1);
    }
\node at (0,0) {$+$};
\node at (210:0.8) {$\ast$};
\end{scope}

\begin{scope}[xshift=1.5cm, yshift=1.5*1.732cm]
    \foreach \i in {90,210,330}
    {
    \markedpt{\i:1};
    \draw[blue] (\i:1) -- (\i+120:1);
    }
\node at (0,0) {$-$};
\node at (330:0.8) {$\ast$};
\end{scope}

\begin{scope}[xshift=-1.5cm, yshift=1.5*1.732cm]
    \foreach \i in {90,210,330}
    {
    \markedpt{\i:1};
    \draw[blue] (\i:1) -- (\i+120:1);
    }
\node at (0,0) {$+$};
\node at (90:0.8) {$\ast$};
\end{scope}

\begin{scope}[xshift=-3cm]
    \foreach \i in {90,210,330}
    {
    \markedpt{\i:1};
    \draw[blue] (\i:1) -- (\i+120:1);
    }
\node at (0,0) {$-$};
\node at (210:0.8) {$\ast$};
\end{scope}

\begin{scope}[xshift=-1.5cm, yshift=-1.5*1.732cm]
    \foreach \i in {90,210,330}
    {
    \markedpt{\i:1};
    \draw[blue] (\i:1) -- (\i+120:1);
    }
\node at (0,0) {$+$};
\node at (330:0.8) {$\ast$};
\end{scope}

\begin{scope}[xshift=1.5cm, yshift=-1.5*1.732cm]
    \foreach \i in {90,210,330}
    {
    \markedpt{\i:1};
    \draw[blue] (\i:1) -- (\i+120:1);
    }
\node at (0,0) {$-$};
\node at (90:0.8) {$\ast$};
\end{scope}

\foreach \k in {0,60,120,180,240,300}
\draw[thick] (\k+15:2.5) -- (\k+45:2.5);
\foreach \k in {30,150,270}
\node at (\k:2.8) {$\mu_1$};
\foreach \k in {90,210,-30}
\node at (\k:2.8) {$\mu_2$};






\begin{scope}[xshift=10cm,scale=1.2]
\draw[->](-3,0) -- (3,0) node[right]{$\sfx_1$};
\draw[->](0,-3) -- (0,3) node[above]{$\sfx_2$};
\draw (0,0) -- (3,-3);
\draw (0,0) -- (1.5,-3);
{\begin{scope}[>=latex]
    \draw[->,thick] (0,0) -- (0,1) node[right]{$v_7$};
    \draw[->,thick] (0,0) -- (-1,0) node[above left]{$v_8$};
    \draw[->,thick] (0,0) -- (0,-1) node[left]{$v_9$};
    \draw[->,thick] (0,0) -- (1,-2) node[left]{$v_{10}$};
    \draw[->,thick] (0,0) -- (1,-1) node[above right]{$v_{11}$};
    \draw[->,thick] (0,0) -- (1,0) node[above]{$v_{12}$};
\end{scope}}    
\end{scope}
\end{tikzpicture}
    \caption{The cluster structure of a triangle (Left) and the corresponding fan in the tropical $\X$-plane (Right).}
    \label{fig:intro_hexagon}
\end{figure}

The structure of constituted by the 6 cones is best viewed in the plane of tropical $\X$-coordinates $(\sfx_1,\sfx_2)$, shown in the right picture in \cref{fig:intro_hexagon}. The generators of the 6 rays separating the cones, together with the 6 corner arcs form the \emph{Hilbert basis}. The corresponding diagrams are shown in \cref{fig:coordinate_corners,fig:coordinates_elementary}. It turns out that these elements satisfy a \emph{sign-coherence} for the coordinates (\ref{lem:sign-coherence}). This fact significantly simplifies the verification of the mutation-equivalence between the coordinate systems, since it ensures that the piecewise-linear coordinate transformations (=tropical cluster transformations) are linear on each of the cones. 

\subsubsection{Graded $\fsp_4$-skein algebra}
In order to discuss the `gluing' of pieces obtained on each triangle, the \emph{graded $\fsp_4$-skein algebra} is useful. We introduce a filtration on the $\fsp_4$-skein algebra with respect to a disjoint collection of ideal arcs $\tri$, and consider the associated graded skein algebra (\cref{dfn:filtration}). 

Given a coordinate vector $\bsfa=(\sfa_i) \in \bQ^I$, we know a corresponding piece of $\fsp_4$-diagram on each triangle $T$ from the discussion in the previous subsection. These pieces certainly have the same number of ends on both sides of each edge shared by two triangles, and we can connect them to obtain an $\fsp_4$-diagram. However, typically we have many ways of doing this, and a direct approach to find a correct way to obtain a non-elliptic diagram yields a combinatorial complication, which is elaborated in \cite{DS20I} in the $\fsl_3$. We avoid this complication by using the graded skein algebra, basically following the approach of \cite{FS22}.


Indeed, such a way of connection is unique in the graded skein algebra.
For example, the following relations hold (\cref{lem:equality-in-gr} and \cref{rem:flat-braid}):
\begin{align*}
    q^{-1}
    \mbox{
    \tikz[baseline=-.6ex, scale=.1]{
        \draw[dashed, fill=white] (0,0) circle [radius=6];
        \coordinate (U) at ($(150:6)!.2!(30:6)$);
        \coordinate (D) at ($(-150:6)!.2!(-30:6)$);
        \coordinate (U2) at ($(150:6)!.6!(30:6)$);
        \coordinate (D2) at ($(-150:6)!.6!(-30:6)$);
        \draw[oarc, rounded corners] (-150:6) -- (U2) -- (30:6);
        \draw[oarc, rounded corners] (150:6) -- (D2) -- (-30:6);
        \draw[blue] (60:6) -- (-60:6);
    }
    }
    \greq
    \mbox{
    \tikz[baseline=-.6ex, scale=.1]{
        \draw[dashed, fill=white] (0,0) circle [radius=6];
        \coordinate (U) at ($(150:6)!.2!(30:6)$);
        \coordinate (D) at ($(-150:6)!.2!(-30:6)$);
        \coordinate (U2) at ($(150:6)!.6!(30:6)$);
        \coordinate (D2) at ($(-150:6)!.6!(-30:6)$);
        \draw[webline] (150:6) -- (30:6);
        \draw[webline] (-150:6) -- (-30:6);
        \draw[blue] (60:6) -- (-60:6);
    }
    },\quad
    q^{-2}
    \mbox{
    \tikz[baseline=-.6ex, scale=.1]{
        \draw[dashed, fill=white] (0,0) circle [radius=6];
        \coordinate (U) at ($(150:6)!.2!(30:6)$);
        \coordinate (D) at ($(-150:6)!.2!(-30:6)$);
        \coordinate (U2) at ($(150:6)!.6!(30:6)$);
        \coordinate (D2) at ($(-150:6)!.6!(-30:6)$);
        \draw[owarc, rounded corners] (-150:6) -- (U2) -- (30:6);
        \draw[owarc, rounded corners] (150:6) -- (D2) -- (-30:6);
        \draw[blue] (60:6) -- (-60:6);
    }
    }
    \greq
    \mbox{
    \tikz[baseline=-.6ex, scale=.1]{
        \draw[dashed, fill=white] (0,0) circle [radius=6];
        \coordinate (U) at ($(150:6)!.2!(30:6)$);
        \coordinate (D) at ($(-150:6)!.2!(-30:6)$);
        \coordinate (U2) at ($(150:6)!.6!(30:6)$);
        \coordinate (D2) at ($(-150:6)!.6!(-30:6)$);
        \draw[wline] (150:6) -- (30:6);
        \draw[wline] (-150:6) -- (-30:6);
        \draw[blue] (60:6) -- (-60:6);
    }
    },\quad
    q^{-2}
    \mbox{
    \tikz[baseline=-.6ex, scale=.1]{
        \draw[dashed, fill=white] (0,0) circle [radius=6];
        \coordinate (U) at ($(150:6)!.2!(30:6)$);
        \coordinate (D) at ($(-150:6)!.2!(-30:6)$);
        \coordinate (U2) at ($(150:6)!.6!(30:6)$);
        \coordinate (D2) at ($(-150:6)!.6!(-30:6)$);
        \draw[wline, rounded corners] (-150:6) to[in=west, out=east] (U2);
        \draw[oarc, rounded corners] (D2) -- ++(.2,0) to[in=west, out=north east] (30:6);
        \draw[oarc, rounded corners] (150:6) to[in=west, out=east] (D2);
        \draw[owarc, rounded corners] (U2) -- ++(.2,0) to[in=west, out=south east] (-30:6);
        \draw[blue] (80:6) -- (-80:6);
    }
    }
    \greq
    \mbox{
    \tikz[baseline=-.6ex, scale=.1]{
        \draw[dashed, fill=white] (0,0) circle [radius=6];
        \coordinate (U) at ($(150:6)!.2!(30:6)$);
        \coordinate (D) at ($(-150:6)!.2!(-30:6)$);
        \coordinate (U2) at ($(150:6)!.6!(30:6)$);
        \coordinate (D2) at ($(-150:6)!.6!(-30:6)$);
        \draw[webline] (150:6) -- (U);
        \draw[wline] (U) -- (U2);
        \draw[webline] (U2) -- (30:6);
        \draw[wline] (-150:6) -- (D);
        \draw[webline] (D) -- (D2);
        \draw[wline] (D2) -- (-30:6);
        \draw[webline] (U) -- (D);
        \draw[webline] (U2) -- (D2);
        \draw[blue] (60:6) -- (-60:6);
    }
    }
    \greq
    \mbox{
    \tikz[baseline=-.6ex, scale=.1]{
        \draw[dashed, fill=white] (0,0) circle [radius=6];
        \coordinate (U) at ($(150:6)!.2!(30:6)$);
        \coordinate (D) at ($(-150:6)!.2!(-30:6)$);
        \coordinate (U2) at ($(150:6)!.6!(30:6)$);
        \coordinate (D2) at ($(-150:6)!.6!(-30:6)$);
        \draw[webline] (150:6) -- (30:6);
        \draw[wline] (-150:6) -- (-30:6);
        \draw[blue] (60:6) -- (-60:6);
    }
    }.
\end{align*}
Here, the vertical blue line is an arc in $\tri$, and the symbol $\greq$ stands for an equality in the graded skein algebra. 
Using these relations, we can prove that the differences arising from deformations involving the birth-death of a bigon and the jumping over a double point have lower degrees (\cref{rem:flat-braid}). 
Then we can pick up the leading term corresponding to a minimal representative of the class in the graded skein algebra. The minimality of the degree implies that this representative is non-elliptic.

\smallskip

As the above discussion suggests, the elements of the graded skein algebra are very close to the integral $\fsp_4$-laminations. Morally, we would like to regard the graded skein algebra as a ``tropicalization'' of the skein algebra. Such a slogan seems to be also illustrated in the work of Farajzadeh-Tehrani and Frohman \cite{FTF}, where they study a graded version of Kauffman bracket skein algebra in relation with a projective compactification of the character variety.




\subsection*{Acknowledgements}
T. Ishibashi is supported by JSPS KAKENHI Grant Numbers JP20K22304 and 24K16914. 
Z. Sun is supported by the NSFC grant 12471068.
W. Yuasa is supported by JSPS KAKENHI Grant Numbers JP19J00252, JP19K14528, and the World Premier International Research Center Initiative Program, International Institute for Sustainability with Knotted Chiral Meta Matter (WPI-SKCM${}^2$).
\section{Bounded \texorpdfstring{$\fsp_4$}{sp(4)}-laminations}

\subsection{Marked surfaces and their triangulations}\label{subsec:notation_marked_surface}
A marked surface $\bSigma=(\Sigma,\bM)$ is a compact oriented surface $\Sigma$ together with a fixed non-empty finite set $\bM \subset \Sigma$ of \emph{marked points}. 
A marked point is called a \emph{puncture} if it lies in the interior of $\Sigma$, and a \emph{special point} otherwise. 
Let $\bP$ (resp. $\bM_\partial$) denote the set of punctures (resp. special points), so that $\bM=\bP \sqcup \bM_\partial$. 
Let $\Sigma^*:=\Sigma \setminus \bP$. 
We always assume the following conditions:
\begin{enumerate}
    \item[(S1)] Each boundary component (if exists) has at least one marked point.
    \item[(S2)] $-2\chi(\Sigma^*)+|\bM_\partial| >0$.
    \item[(S3)] $\bSigma$ is not a once-punctured disk with a single special point on the boundary.
\end{enumerate}
We call a connected component of the punctured boundary $\partial \Sigma^\ast=\partial\Sigma\setminus \bM_\partial$ a \emph{boundary interval}. The set of boundary intervals is denoted by $\bB$. 
We have $|\bM_\partial|=|\bB|$. 

Unless otherwise stated, an \emph{isotopy} in a marked surface $\bSigma$ means an ambient isotopy in $\Sigma$ relative to $\bM$, which preserves each boundary interval setwisely. 
An \emph{ideal arc} in $\bSigma$ is the isotopy class of an immersed arc in $\Sigma$ with endpoints in $\bM$ having no self-intersections except for its endpoints, and not contractible in $\Sigma^\ast$. 
An \emph{ideal triangulation} is a triangulation $\tri$ of $\Sigma$ whose set of $0$-cells (vertices) coincides with $\bM$, $1$-cells (edges) being ideal arcs. 
In this paper, we always consider an ideal triangulation without \emph{self-folded triangles} where two of its sides are identified. 
The conditions (S1)--(S3) ensure the existence of such an ideal triangulation. See, for instance, \cite[Lemma 2.13]{FST}. 
For an ideal triangulation $\tri$, denote the set of edges (resp. interior edges, triangles) of $\tri$ by $e(\tri)$ (resp. $e_{\interior}(\tri)$, $t(\tri)$). Since the boundary intervals belong to any ideal triangulation, $e(\tri)=e_{\interior}(\tri) \sqcup \bB$. By a computation on the Euler characteristics, we get
\begin{align*}
    &|e(\tri)|=-3\chi(\Sigma^*)+2|\bM_\partial|, \quad |e_{\interior}(\tri)|=-3\chi(\Sigma^*)+|\bM_\partial|, \\
    &|t(\tri)|=-2\chi(\Sigma^*)+|\bM_\partial|.
\end{align*}

\subsection{Spaces of \texorpdfstring{$\fsp_4/\fso_5$-}{sp(4)/so(5)-}diagrams}\label{subsec:diagram-definition}

We consider a \emph{1,3,4-valent graph}, which is a (possibly disconnected) graph whose vertices have valency either one, three or four. It is allowed to have a loop component ({\it i.e.}, a connected component without vertices). An
\emph{$\fsp_4$-coloring} of such a 1,3,4-valent graph is an assignment of a color on each edge, called \emph{type~1} and \emph{type~2}, such that
\begin{itemize}
    \item any 3-valent vertex is a joint of a single type~1 edge and two type~2 edges;
    \item any 4-valent vertex is a joint of four type~1 edges. 
\end{itemize}
In figures, we show a type~1 (resp. type~2) edge by a single (resp. double) line. 
When we show a portion of graph with an arbitrary color, it is shown in black. A $4$-valent vertex is called a \emph{crossroad}, and denoted by the symbol $\begin{tikzpicture}[scale=.6,baseline=-0.6ex]
		\draw[dashed, fill=white] (0,0) circle [radius=0.7];
		\draw[webline] (45:0.7) -- (-135:0.7);
		\draw[webline] (-45:0.7) -- (135:0.7);
		\draw[fill=pink,thick] (0,0) circle(3pt);
\end{tikzpicture}$\ .
So the local pictures of $\fsp_4$-colored 1,3,4-valent graphs are
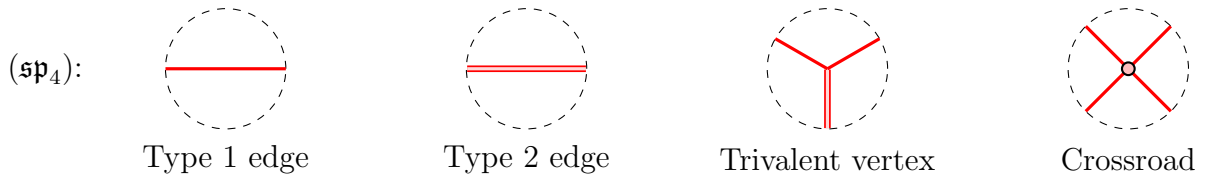
\begin{align*}
\begin{tikzpicture}[scale=0.8]
\node at (-3,0) {$(\fsp_4)$:};
\draw[dashed] (0,0) circle(1cm);
\draw[webline] (-1,0) -- (1,0);
\node at (0,-1.5) {Type~1 edge};
\begin{scope}[xshift=5cm]
\draw[dashed] (0,0) circle(1cm);
\draw[wline] (-1,0) -- (1,0);
\node at (0,-1.5) {Type~2 edge};
\end{scope}
\begin{scope}[xshift=10cm]
\draw[dashed] (0,0) circle(1cm);
{\color{red} \triv{-90:1}{30:1}{150:1};}
\node at (0,-1.5) {Trivalent vertex};
\end{scope}
\begin{scope}[xshift=15cm]
\draw[dashed] (0,0) circle(1cm);
\draw[webline] (45:1) -- (-135:1);
\draw[webline] (-45:1) -- (135:1);
\draw[fill=pink,thick] (0,0) circle(3pt);
\node at (0,-1.5) {Crossroad};
\end{scope}
\end{tikzpicture}\ .
\end{align*}
A type~2 edge whose endpoints are both incident to trivalent vertices is called a \emph{rung}. 

As a dual notion, an \emph{$\fso_5$-coloring} is defined by exchanging the type~1 and type~2. Their local pictures are:
\begin{align*}
\begin{tikzpicture}[scale=0.8]
\node at (-3,0) {($\fso_5$):};
\draw[dashed] (0,0) circle(1cm);
\draw[wlined] (-1,0) -- (1,0);
\node at (0,-1.5) {Type~1 edge};
\begin{scope}[xshift=5cm]
\draw[dashed] (0,0) circle(1cm);
\draw[weblined] (-1,0) -- (1,0);
\node at (0,-1.5) {Type~2 edge};
\end{scope}
\begin{scope}[xshift=10cm]
\draw[dashed] (0,0) circle(1cm);
{\color{red} \trivd{-90:1}{30:1}{150:1};}
\node at (0,-1.5) {Trivalent vertex};
\end{scope}
\begin{scope}[xshift=15cm]
\draw[dashed] (0,0) circle(1cm);
\draw[wlined] (45:1) -- (-135:1);
\draw[wlined] (-45:1) -- (135:1);
\draw[fill=green!30,thick] (0,0) circle(3pt);
\node at (0,-1.5) {Crossroad};
\end{scope}
\end{tikzpicture}
\end{align*}

\begin{dfn}Let $\mathfrak{g}$ be either $\fsp_4$ or $\fso_5$. 
\begin{enumerate}
    \item 
    \begin{itemize}
        \item A \emph{bounded $\mathfrak{g}$-diagram} on a marked surface $\bSigma$ is an immersed $\mathfrak{g}$-colored 1,3,4-valent graph on $\Sigma$ such that each 1-valent vertex lie in $\partial \Sigma^\ast$, and the other part is embedded into $\interior \Sigma^\ast$. A \emph{bounded crossroad $\mathfrak{g}$-diagram} $D$ is a bounded $\mathfrak{g}$-diagram with no rungs.
        \item For a bounded crossroad $\mathfrak{g}$-diagram $D$, an \emph{elliptic interior face} of $D$ is a $0,1,2,3$-gon of $\Sigma^\ast\setminus D$ containing no boundary intervals. An \emph{elliptic bordered face} is a face shown in \cref{fig:elliptic}.
        \item A bounded crossroad $\mathfrak{g}$-diagram $D$ is said to be \emph{non-elliptic} if it has none of the elliptic bordered and elliptic interior faces. 
        If moreover $D$ has no \emph{bordered $H$-faces} in \cref{fig:H-face}, it is said to be \emph{reduced}.
        \item Let $\diag_{\mathfrak{g},\bSigma}$ denote the set of all bounded $\mathfrak{g}$-diagrams on $\bSigma$. We will denote the subset of reduced bounded $\mathfrak{g}$-diagrams by $\Bdiag_{\mathfrak{g},\bSigma}$. We regard the \emph{empty diagram $\varnothing$} as an element of $\Bdiag_{\mathfrak{g},\bSigma}$.
    \end{itemize}
    \item An \emph{unbounded $\mathfrak{g}$-diagram} on $\bSigma$ is an immersed $\mathfrak{g}$-colored 1,3,4-valent graph $W$ on $\Sigma$ such that each 1-valent vertex lies in $\bM$, and the other part is embedded into $\interior \Sigma^\ast$. 
    Let $\diag^\infty_{\mathfrak{g},\bSigma}$ denote the set of all unbounded $\mathfrak{g}$-diagrams on $\bSigma$, and $\Bdiag_{\mathfrak{g},\bSigma}^\infty$ the subset of reduced unbounded crossroad $\mathfrak{g}$-diagrams. 
\end{enumerate}
\end{dfn}

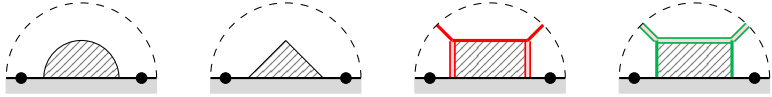
\begin{figure}[ht]
    \centering
\begin{tikzpicture}[scale=.1]
    \coordinate (N) at (0,10);
    \coordinate (S) at (0,-10);
    \coordinate (W) at (-10,0);
    \coordinate (E) at (10,0);
    \coordinate (NE) at (45:10);
    \coordinate (NW) at (135:10);
    \coordinate (SE) at (-45:10);
    \coordinate (SW) at (-135:10);
    \coordinate (C) at (0,0);
    \coordinate (CN) at (0,5);
    \coordinate (CS) at (0,-5);
    \coordinate (CW) at (-5,0);
    \coordinate (CE) at (5,0);
    \begin{scope}
        \clip (E) arc (0:180:10);
        \fill[pattern=north east lines, pattern color=gray] (CE) arc (0:180:5) -- cycle;
        \draw (CE) arc (0:180:5);
    \end{scope}
    \draw[dashed] (E) arc (0:180:10);
    \bline{W}{E}{2}
    \draw[fill=black] ($(W)+(2,0)$) circle (20pt);
    \draw[fill=black] ($(E)+(-2,0)$) circle (20pt);
\end{tikzpicture}
\hspace{1em}
\begin{tikzpicture}[scale=.1]
    \coordinate (N) at (0,10);
    \coordinate (S) at (0,-10);
    \coordinate (W) at (-10,0);
    \coordinate (E) at (10,0);
    \coordinate (NE) at (45:10);
    \coordinate (NW) at (135:10);
    \coordinate (SE) at (-45:10);
    \coordinate (SW) at (-135:10);
    \coordinate (C) at (0,0);
    \coordinate (CN) at (0,5);
    \coordinate (CS) at (0,-5);
    \coordinate (CW) at (-5,0);
    \coordinate (CE) at (5,0);
    \begin{scope}
        \clip (E) arc (0:180:10);
        \fill[pattern=north east lines, pattern color=gray] (CW) -- (CN) -- (CE) -- cycle;
        \draw (CW) -- (CN);
        \draw (CE) -- (CN);
    \end{scope}
    \draw[dashed] (E) arc (0:180:10);
    \bline{W}{E}{2}
    \draw[fill=black] ($(W)+(2,0)$) circle (20pt);
    \draw[fill=black] ($(E)+(-2,0)$) circle (20pt);
\end{tikzpicture}
\hspace{1em}
\begin{tikzpicture}[scale=.1]
    \coordinate (N) at (0,10);
    \coordinate (S) at (0,-10);
    \coordinate (W) at (-10,0);
    \coordinate (E) at (10,0);
    \coordinate (NE) at (45:10);
    \coordinate (NW) at (135:10);
    \coordinate (SE) at (-45:10);
    \coordinate (SW) at (-135:10);
    \coordinate (C) at (0,0);
    \coordinate (CN) at (0,5);
    \coordinate (CS) at (0,-5);
    \coordinate (CW) at (-5,0);
    \coordinate (CE) at (5,0);
    \begin{scope}
        \clip (E) arc (0:180:10);
        \fill[pattern=north east lines, pattern color=gray] (CW)--($(CW)+(0,5)$)--($(CE)+(0,5)$)--(CE)--cycle;
        \draw[wline] (CW)--($(CW)+(0,5)$);
        \draw[wline] (CE)--($(CE)+(0,5)$);
        \draw[webline] ($(CW)+(0,5)$)-- ($(CE)+(0,5)$);
        \draw[webline] ($(CW)+(0,5)$) -- +(135:5);
        \draw[webline] ($(CE)+(0,5)$) -- +(45:5);
    \end{scope}
    \draw[dashed] (E) arc (0:180:10);
    \bline{W}{E}{2}
    \draw[fill=black] ($(W)+(2,0)$) circle (20pt);
    \draw[fill=black] ($(E)+(-2,0)$) circle (20pt);
\end{tikzpicture}
\hspace{1em}
\begin{tikzpicture}[scale=.1]
    \coordinate (N) at (0,10);
    \coordinate (S) at (0,-10);
    \coordinate (W) at (-10,0);
    \coordinate (E) at (10,0);
    \coordinate (NE) at (45:10);
    \coordinate (NW) at (135:10);
    \coordinate (SE) at (-45:10);
    \coordinate (SW) at (-135:10);
    \coordinate (C) at (0,0);
    \coordinate (CN) at (0,5);
    \coordinate (CS) at (0,-5);
    \coordinate (CW) at (-5,0);
    \coordinate (CE) at (5,0);
    \begin{scope}
        \clip (E) arc (0:180:10);
        \fill[pattern=north east lines, pattern color=gray] (CW)--($(CW)+(0,5)$)--($(CE)+(0,5)$)--(CE)--cycle;
        \draw[wlined] ($(CW)+(0,5)$)-- ($(CE)+(0,5)$);
        \draw[wlined] ($(CW)+(0,5)$) -- +(135:5);
        \draw[wlined] ($(CE)+(0,5)$) -- +(45:5);
        \draw[weblined] (CW)--($(CW)+(0,5)$);
        \draw[weblined] (CE)--($(CE)+(0,5)$);
    \end{scope}
    \draw[dashed] (E) arc (0:180:10);
    \bline{W}{E}{2}
    \draw[fill=black] ($(W)+(2,0)$) circle (20pt);
    \draw[fill=black] ($(E)+(-2,0)$) circle (20pt);
\end{tikzpicture}
    \caption{The shaded faces are elliptic faces adjacent to a boundary interval. Trivalent vertices or crossroads are arranged at the corners of the faces.}
    \label{fig:elliptic}
\end{figure}

\begin{figure}[ht]
    \centering
\begin{tikzpicture}[scale=.1]
    \coordinate (N) at (0,10);
    \coordinate (S) at (0,-10);
    \coordinate (W) at (-10,0);
    \coordinate (E) at (10,0);
    \coordinate (NE) at (45:10);
    \coordinate (NW) at (135:10);
    \coordinate (SE) at (-45:10);
    \coordinate (SW) at (-135:10);
    \coordinate (C) at (0,0);
    \coordinate (CN) at (0,5);
    \coordinate (CS) at (0,-5);
    \coordinate (CW) at (-5,0);
    \coordinate (CE) at (5,0);
    \begin{scope}
        \clip (E) arc (0:180:10);
        \fill[pattern=north east lines, pattern color=gray] (CW)--($(CW)+(0,5)$)--($(CE)+(0,5)$)--(CE)--cycle;
        \draw[webline] (CW)--($(CW)+(0,5)$);
        \draw[wline] (CE)--($(CE)+(0,5)$);
        \draw[webline] ($(CW)+(0,5)$)-- ($(CE)+(0,5)$);
        \draw[wline] ($(CW)+(0,5)$) -- +(135:5);
        \draw[webline] ($(CE)+(0,5)$) -- +(45:5);
    \end{scope}
    \draw[dashed] (E) arc (0:180:10);
    \bline{W}{E}{2}
    \draw[fill=black] ($(W)+(2,0)$) circle (20pt);
    \draw[fill=black] ($(E)+(-2,0)$) circle (20pt);
\end{tikzpicture}
\hspace{1em}
\begin{tikzpicture}[scale=.1]
    \coordinate (N) at (0,10);
    \coordinate (S) at (0,-10);
    \coordinate (W) at (-10,0);
    \coordinate (E) at (10,0);
    \coordinate (NE) at (45:10);
    \coordinate (NW) at (135:10);
    \coordinate (SE) at (-45:10);
    \coordinate (SW) at (-135:10);
    \coordinate (C) at (0,0);
    \coordinate (CN) at (0,5);
    \coordinate (CS) at (0,-5);
    \coordinate (CW) at (-5,0);
    \coordinate (CE) at (5,0);
    \begin{scope}
        \clip (E) arc (0:180:10);
        \fill[pattern=north east lines, pattern color=gray] (CW)--($(CW)+(0,5)$)--($(CE)+(0,5)$)--(CE)--cycle;
        \draw[wline] (CW)--($(CW)+(0,5)$);
        \draw[webline] (CE)--($(CE)+(0,5)$);
        \draw[webline] ($(CW)+(0,5)$)-- ($(CE)+(0,5)$);
        \draw[webline] ($(CW)+(0,5)$) -- +(135:5);
        \draw[wline] ($(CE)+(0,5)$) -- +(45:5);
    \end{scope}
    \draw[dashed] (E) arc (0:180:10);
    \bline{W}{E}{2}
    \draw[fill=black] ($(W)+(2,0)$) circle (20pt);
    \draw[fill=black] ($(E)+(-2,0)$) circle (20pt);
\end{tikzpicture}
\hspace{1em}
\begin{tikzpicture}[scale=.1]
    \coordinate (N) at (0,10);
    \coordinate (S) at (0,-10);
    \coordinate (W) at (-10,0);
    \coordinate (E) at (10,0);
    \coordinate (NE) at (45:10);
    \coordinate (NW) at (135:10);
    \coordinate (SE) at (-45:10);
    \coordinate (SW) at (-135:10);
    \coordinate (C) at (0,0);
    \coordinate (CN) at (0,5);
    \coordinate (CS) at (0,-5);
    \coordinate (CW) at (-5,0);
    \coordinate (CE) at (5,0);
    \begin{scope}
        \clip (E) arc (0:180:10);
        \fill[pattern=north east lines, pattern color=gray] (CW)--($(CW)+(0,5)$)--($(CE)+(0,5)$)--(CE)--cycle;
        \draw[wline] (CW)--($(CW)+(0,5)$);
        \draw[webline] (CE)--($(CE)+(0,10)$);
        \draw[webline] ($(CW)+(0,5)$)-- ($(CE)+(5,5)$);
        \draw[webline] ($(CW)+(0,5)$) -- +(135:5);
	\draw[fill=pink,thick] ($(CE)+(0,5)$) circle(20pt);
    \end{scope}
    \draw[dashed] (E) arc (0:180:10);
    \bline{W}{E}{2}
    \draw[fill=black] ($(W)+(2,0)$) circle (20pt);
    \draw[fill=black] ($(E)+(-2,0)$) circle (20pt);
\end{tikzpicture}
\hspace{1em}
\begin{tikzpicture}[scale=.1, xscale=-1]
    \coordinate (N) at (0,10);
    \coordinate (S) at (0,-10);
    \coordinate (W) at (-10,0);
    \coordinate (E) at (10,0);
    \coordinate (NE) at (45:10);
    \coordinate (NW) at (135:10);
    \coordinate (SE) at (-45:10);
    \coordinate (SW) at (-135:10);
    \coordinate (C) at (0,0);
    \coordinate (CN) at (0,5);
    \coordinate (CS) at (0,-5);
    \coordinate (CW) at (-5,0);
    \coordinate (CE) at (5,0);
    \begin{scope}
        \clip (E) arc (0:180:10);
        \fill[pattern=north east lines, pattern color=gray] (CW)--($(CW)+(0,5)$)--($(CE)+(0,5)$)--(CE)--cycle;
        \draw[wline] (CW)--($(CW)+(0,5)$);
        \draw[webline] (CE)--($(CE)+(0,10)$);
        \draw[webline] ($(CW)+(0,5)$)-- ($(CE)+(5,5)$);
        \draw[webline] ($(CW)+(0,5)$) -- +(135:5);
	\draw[fill=pink,thick] ($(CE)+(0,5)$) circle(20pt);
    \end{scope}
    \draw[dashed] (E) arc (0:180:10);
    \bline{W}{E}{2}
    \draw[fill=black] ($(W)+(2,0)$) circle (20pt);
    \draw[fill=black] ($(E)+(-2,0)$) circle (20pt);
\end{tikzpicture}
    \caption{These shaded faces are the list of bordered $H$-faces for a boundary $\fsp_4$-diagram. The bordered $H$-faces for $\fso_5$ are defined similarly.}
    \label{fig:H-face}
\end{figure}
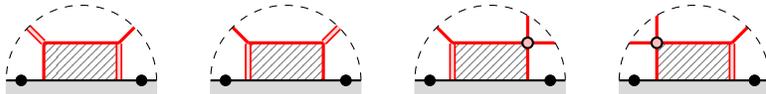

Observe that a bounded $\mathfrak{g}$-diagram $W$ is away from punctures (therefore the term `bounded'). 

\begin{dfn}\label{def:peripheral}
\begin{enumerate}
    \item A loop component in a bounded $\mathfrak{g}$-diagram is said to be \emph{peripheral} if it is contractible after removing a single puncture.
    \item An arc component in a bounded $\mathfrak{g}$-diagram is called a \emph{corner arc} if it connects to consecutive boundary intervals and cuts out a contractible region.  
\end{enumerate}
These components are collectively called the \emph{peripheral components}. 
\end{dfn}

While we are going to study bounded $\fsp_4$-diagrams as underlying data of bounded $\fsp_4$-laminations, the unbounded $\fso_5$-diagrams will be used as `test diagrams' for which we consider intersection coordinates of $\fsp_4$-laminations. 

\begin{rem}
The correspondence of the terminologies with related literature for the $\fsl_3$-case is shown in \cref{tab:dictionary}. 
\end{rem}

\begin{table}[ht]
    \centering
    \begin{tabular}{|r:ccc:l|}\hline
        Frohman--Sikora \cite{FS22} & elliptic face& $\leftrightarrow$& elliptic interior face& Our paper \\
        & reduced& $\leftrightarrow$& reduced& \\ \hline
        Douglas--Sun \cite{DS20I} & elliptic face& $\leftrightarrow$& elliptic interior face& Our paper\\ 
        & taut& $\leftrightarrow$& without elliptic bordered faces& \\
        & essential& $\leftrightarrow$& non-elliptic& \\
        & rungless& $\leftrightarrow$& reduced& \\ \hline
        Ishibashi--Kano \cite{IK22} & elliptic face& $\leftrightarrow$& elliptic interior or bordered faces& Our paper \\
        & reduced& $\leftrightarrow$& reduced& \\ \hline
        Shen--Sun--Weng \cite{SSW25} & $\A$-web & $\leftrightarrow$& bounded diagram & Our paper \\
         & $\X$-web & $\leftrightarrow$& unbounded diagram & \\
         \hline
    \end{tabular}
    \caption{The correspondence of terminologies}
    \label{tab:dictionary}
\end{table}

\paragraph{\textbf{Elementary moves.}} 
The following moves can be applied for bounded $\fsp_4$-diagrams on $\bSigma$:
\begin{enumerate}
\item[(D1)] \emph{Loop parallel-move} (a.~k.~a.~\emph{flip move} \cite{FS22} or \emph{global parallel move} \cite{DS20I}):  
\begin{align}\label{eq:loop_parallel-move}
\begin{tikzpicture}[scale=1]
\draw[webline] (0.7,0.75) [partial ellipse=90:-90:0.2cm and 0.75cm];
\draw[webline,dashed] (0.7,0.75) [partial ellipse=-90:-270:0.2cm and 0.75cm];
\draw[wline] (1.3,0.75) [partial ellipse=90:-90:0.2cm and 0.75cm];
\draw[wline] (1.3,0.75) [partial ellipse=-90:-270:0.2cm and 0.75cm];
\draw[line width=5pt,white,dashed] (1.3,0.75) [partial ellipse=-90:-270:0.2cm and 0.75cm];
\draw(0,0) -- (2,0);
\draw(0,1.5) -- (2,1.5);
\draw[<->] (2.5,0.75) -- (3.5,0.75);
\begin{scope}[xshift=4cm]
\draw[wline] (0.7,0.75) [partial ellipse=90:-90:0.2cm and 0.75cm];
\draw[wline] (0.7,0.75) [partial ellipse=-90:-270:0.2cm and 0.75cm];
\draw[line width=5pt,white,dashed] (0.7,0.75) [partial ellipse=-90:-270:0.2cm and 0.75cm];
\draw[webline] (1.3,0.75) [partial ellipse=90:-90:0.2cm and 0.75cm];
\draw[webline,dashed] (1.3,0.75) [partial ellipse=-90:-270:0.2cm and 0.75cm];
\draw(0,0) -- (2,0);
\draw(0,1.5) -- (2,1.5);
\end{scope}
\end{tikzpicture}
\end{align}

\item[(D2)] \emph{Arc parallel-move}:
\begin{align}\label{eq:arc_parallel-move}
    \mbox{
        \tikz[baseline=-.6ex, scale=.1]{
            \draw[dashed, fill=white] (-5,-7) rectangle (5,7);
            \coordinate (N1) at (2,7);
            \coordinate (S1) at (2,-7);
            \coordinate (N2) at (-2,7);
            \coordinate (S2) at (-2,-7);
            \draw[webline] (S1) -- (N1);
            \draw[wline] (S2) -- (N2);
            \bline{-5,-7}{5,-7}{2}
            \bline{-5,7}{5,7}{-2}
        }
    }
    \mbox{
        \tikz[baseline=-.6ex, scale=.1]{
        \draw[<->] (-4,0) -- (4,0);
        }
    }
    \mbox{
        \tikz[baseline=-.6ex, scale=.1]{
            \draw[dashed, fill=white] (-5,-7) rectangle (5,7);
            \coordinate (N1) at (2,7);
            \coordinate (S1) at (2,-7);
            \coordinate (N2) at (-2,7);
            \coordinate (S2) at (-2,-7);
            \draw[wline] (S1) -- (N1);
            \draw[webline] (S2) -- (N2);
            \bline{-5,-7}{5,-7}{2}
            \bline{-5,7}{5,7}{-2}
        }
    }
\end{align}

\item[(D3)] \emph{Boundary H-move}:
\begin{align}\label{eq:boundary_H-move}
\begin{tikzpicture}
\draw[webline] (0.4,0) -- (0.4,0.4);
\draw[wline] (-0.4,0) -- (-0.4,0.4);
\draw[wline] (0.4,0.4) -- (0.4,0.917);
\draw[webline] (-0.4,0.4) -- (-0.4,0.917);
\draw[webline] (0.4,0.4) -- (-0.4,0.4);
\draw[dashed] (1,0) arc (0:180:1cm);
\bline{-1,0}{1,0}{0.2}
\draw[<->] (1.5,.5) -- (2.5,.5);
\begin{scope}[xshift=4cm]
\draw[wline] (0.4,0) -- (0.4,0.917);
\draw[webline] (-0.4,0) -- (-0.4,0.917);
\draw[dashed] (1,0) arc (0:180:1cm);
\bline{-1,0}{1,0}{0.2}
\end{scope}
\end{tikzpicture}\ , \qquad
\begin{tikzpicture}
\draw[wline] (0.4,0) -- (0.4,0.4);
\draw[webline] (-0.4,0) -- (-0.4,0.4);
\draw[webline] (0.4,0.4) -- (0.4,0.917);
\draw[wline] (-0.4,0.4) -- (-0.4,0.917);
\draw[webline] (0.4,0.4) -- (-0.4,0.4);
\draw[dashed] (1,0) arc (0:180:1cm);
\bline{-1,0}{1,0}{0.2}
\draw[<->] (1.5,.5) -- (2.5,.5);
\begin{scope}[xshift=4cm]
\draw[webline] (0.4,0) -- (0.4,0.917);
\draw[wline] (-0.4,0) -- (-0.4,0.917);
\draw[dashed] (1,0) arc (0:180:1cm);
\bline{-1,0}{1,0}{0.2}
\end{scope}
\end{tikzpicture}
\end{align}

\item[(D4)] \emph{Ladder resolutions} (and their inverses): 
\begin{align}\label{eq:ladder-resolution}
\begin{tikzpicture}[scale=0.8]
\draw[dashed] (0,0) circle(1cm);
\draw[webline] (45:1) -- (0,0.5);
\draw[webline] (135:1) -- (0,0.5);
\draw[webline] (-45:1) -- (0,-0.5);
\draw[webline] (-135:1) -- (0,-0.5);
\draw[wline] (0,0.5) -- (0,-0.5);
\begin{scope}[xshift=4cm]
\draw[dashed] (0,0) circle(1cm);
\draw[webline] (45:1) -- (-135:1);
\draw[webline] (-45:1) -- (135:1);
\draw[fill=pink,thick] (0,0) circle(3pt);
\end{scope}
\draw[thick,->] (2.5,0) -- (1.5,0);
\draw[thick,->] (5.5,0) -- (6.5,0);
\begin{scope}[xshift=8cm]
\draw[dashed] (0,0) circle(1cm);
\draw[webline] (45:1) -- (0.5,0);
\draw[webline] (-45:1) -- (0.5,0);
\draw[webline] (135:1) -- (-0.5,0);
\draw[webline] (-135:1) -- (-0.5,0);
\draw[wline] (0.5,0) -- (-0.5,0);
\end{scope}
\end{tikzpicture}
\end{align}
\end{enumerate}

\begin{rem}\label{rem:arc-parallel}
    The boundary $H$-move implies the arc-parallel move. However, we need the arc parallel-move when we consider an equivalence relation among reduced $\fsp_4$-diagrams.
\end{rem}
\begin{rem}
In the $\fsp_4$-skein algebra, crossroads are defined by the relation
\begin{align*}
    \mbox{
        \ \tikz[baseline=-.6ex, scale=.08, rotate=90]{
            \draw[dashed, fill=white] (0,0) circle [radius=7];
            \draw[webline] (45:7) -- (-135:7);
            \draw[webline] (135:7) -- (-45:7);
            \draw[fill=pink, thick] (0,0) circle [radius=30pt];
        }
    \ }
    :=
    \mbox{
        \ \tikz[baseline=-.6ex, scale=.08]{
            \draw[dashed, fill=white] (0,0) circle [radius=7];
            \draw[webline] (45:7) -- (90:3);
            \draw[webline] (135:7) -- (90:3);
            \draw[webline] (225:7) -- (-90:3);
            \draw[webline] (315:7) -- (-90:3);
            \draw[wline] (90:3) -- (-90:3);
            }
    \ }
    -\frac{1}{[2]}
    \mbox{
        \ \tikz[baseline=-.6ex, scale=.08, rotate=90]{
            \draw[dashed, fill=white] (0,0) circle [radius=7];
            \draw[webline] (-45:7) to[out=north west, in=south] (3,0) to[out=north, in=south west] (45:7);
            \draw[webline] (-135:7) to[out=north east, in=south] (-3,0) to[out=north, in=south east] (135:7);
        }
    \ }
    =
    \mbox{
        \ \tikz[baseline=-.6ex, scale=.08, rotate=90]{
            \draw[dashed, fill=white] (0,0) circle [radius=7];
            \draw[webline] (45:7) -- (90:3);
            \draw[webline] (135:7) -- (90:3);
            \draw[webline] (225:7) -- (-90:3);
            \draw[webline] (315:7) -- (-90:3);
            \draw[wline] (90:3) -- (-90:3);
            }
    \ }
    -\frac{1}{[2]}
    \mbox{
        \ \tikz[baseline=-.6ex, scale=.08]{
            \draw[dashed, fill=white] (0,0) circle [radius=7];
            \draw[webline] (-45:7) to[out=north west, in=south] (3,0) to[out=north, in=south west] (45:7);
            \draw[webline] (-135:7) to[out=north east, in=south] (-3,0) to[out=north, in=south east] (135:7);
        }
    \ }.
\end{align*}
We consider \eqref{eq:ladder-resolution} to be an equivalence relation, as the $\fsp_4$-laminations should parametrize ``highest terms'' of $\fsp_4$-webs in some sense; see \cref{lem:equality-in-gr}.
\end{rem}

\begin{dfn}\label{def:transverse}
Let $S$ be a finite collection of ideal arcs that are mutually disjoint except for the endpoints. 
\begin{itemize}
    \item An $\fsp_4$-diagram $D$ is said to be \emph{$S$-transverse} if $D \cap \gamma$ is finite and consisting only of transverse double points for any $\gamma \in S$. 
    \item Two $S$-transverse $\fsp_4$-diagrams $D_1,D_2$ are said to be \emph{ladder-equivalent} rel.~to $S$ and denoted by $D_1\approx_{S} D_2$ if $D_1$ is related to $D_2$ by an isotopy rel.~to $S$ and a sequence of elementary moves (D1)--(D4) outside $S$. When $S=\emptyset$, we simply denote $\approx_{\emptyset}$ by $\approx$. 
\end{itemize}
Let $\Blad_{\fsp_4,\bSigma}$ be the set of ladder-equivalence classes of reduced bounded $\fsp_4$-diagrams on $\bSigma$. When $S=\emptyset$, observe that the above definitions work as well for unbounded diagrams. So we define $\Blad_{\fso_5,\bSigma}^\infty$ in a similar way. 
\end{dfn}

\begin{rem}
    We consider the empty diagram $\varnothing$ is $S$-transverse for any $S$, and it defines the unique elements in $\Blad_{\fsp_4,\bSigma}$ or $\Blad_{\fso_5,\bSigma}^{\infty}$.
\end{rem}

Here, the general equivalence relation $\approx_S$ is introduced for a later discussion on minimal position with respect to an ideal triangulation.


The following is a consequence of confluence theory for non-elliptic $\fsp_4$-diagrams in \cite[Lemma~2.13]{IYsp4}:

\begin{lem}\label{lem:equivalence-generator}
$\Blad_{\fsp_4,\bSigma}$ is identified with the set $\Bdiag_{\fsp_4,\Sigma}$ modulo the equivalence relation generated by (D1), (D2), (D4); see also \cref{rem:arc-parallel}.
\end{lem}

Let us introduce a useful description of certain $\fsp_4$-diagrams.
\begin{dfn}\label{dfn:box-notation}
    For any positive integer $n$ and $m$, an \emph{$\fsp_4$-tetrapod} for the $\fsp_4$-diagram is defined by
    \begin{align*}
    \mbox{
        \tikz[baseline=-.6ex, scale=.1]{
            \coordinate (N) at (0,10);
            \coordinate (S) at (0,-10);
            \coordinate (W) at (-10,0);
            \coordinate (E) at (10,0);
            \coordinate (NE) at (45:10);
            \coordinate (NW) at (135:10);
            \coordinate (SE) at (-45:10);
            \coordinate (SW) at (-135:10);
            \coordinate (C) at (0,0);
            \coordinate (CN) at (0,5);
            \coordinate (CS) at (0,-5);
            \coordinate (CW) at (-5,0);
            \coordinate (CE) at (5,0);
            \begin{scope}
                \clip (C) circle [radius=10cm];
                \draw[webline] (W) -- (E);
                \draw[webline] (S) -- (N);
                \node at ($(C)+(1,0)$) [xscale=.5]{$\cdots$};
                \node at (C) {\sqbox};
            \end{scope}
            \draw[dashed] (C) circle [radius=10cm];
            \node at (W) [left]{\scriptsize $1$};
            \node at (E) [right]{\scriptsize $1$};
            \node at (N) [above]{\scriptsize $n$};
            \node at (S) [below]{\scriptsize $n$};
        }
    }
    \coloneqq
    \mbox{
        \tikz[baseline=-.6ex, scale=.1]{
            \coordinate (N) at (0,10);
            \coordinate (S) at (0,-10);
            \coordinate (W) at (-10,0);
            \coordinate (E) at (10,0);
            \coordinate (NE) at (45:10);
            \coordinate (NW) at (135:10);
            \coordinate (SE) at (-45:10);
            \coordinate (SW) at (-135:10);
            \coordinate (C) at (0,0);
            \coordinate (CN) at (0,5);
            \coordinate (CS) at (0,-5);
            \coordinate (CW) at (-5,0);
            \coordinate (CE) at (5,0);
            \begin{scope}
                \clip (C) circle [radius=10cm];
                \draw[webline] (W) -- (E);
                \draw[webline] (-5,10) -- (-5,-10);
                \draw[webline] (-3,10) -- (-3,-10);
                \draw[webline] (5,10) -- (5,-10);
                \node at ($(CN)+(1,0)$){$\cdots$};
                \node at ($(CS)+(1,0)$){$\cdots$};
                \draw[fill=pink, thick] (-5,0) circle [radius=20pt];
                \draw[fill=pink, thick] (-3,0) circle [radius=20pt];
                \draw[fill=pink, thick] (5,0) circle [radius=20pt];
            \end{scope}
            \draw[dashed] (C) circle [radius=10cm];
            \node at ($(N)+(0,5)$) {\scriptsize $n$ edges};
            \node at ($(N)+(0,2)$) [rotate=-90, yscale=3]{\scriptsize $\{$};
        }
    },\quad
    \mbox{
        \tikz[baseline=-.6ex, scale=.1]{
            \coordinate (N) at (0,10);
            \coordinate (S) at (0,-10);
            \coordinate (W) at (-10,0);
            \coordinate (E) at (10,0);
            \coordinate (NE) at (45:10);
            \coordinate (NW) at (135:10);
            \coordinate (SE) at (-45:10);
            \coordinate (SW) at (-135:10);
            \coordinate (C) at (0,0);
            \coordinate (CN) at (0,5);
            \coordinate (CS) at (0,-5);
            \coordinate (CW) at (-5,0);
            \coordinate (CE) at (5,0);
            \begin{scope}
                \clip (C) circle [radius=10cm];
                \draw[webline] (W) -- (E);
                \draw[webline] (S) -- (N);
                \node at ($(C)+(1,0)$) [xscale=.5]{$\cdots$};
                \node at (C) {\sqbox};
            \end{scope}
            \draw[dashed] (C) circle [radius=10cm];
            \node at (W) [left]{\scriptsize $m$};
            \node at (E) [right]{\scriptsize $m$};
            \node at (N) [above]{\scriptsize $n$};
            \node at (S) [below]{\scriptsize $n$};
        }
    }
    \coloneqq
    \mbox{
        \tikz[baseline=-.6ex, scale=.1]{
            \coordinate (N) at (0,10);
            \coordinate (S) at (0,-10);
            \coordinate (W) at (-10,0);
            \coordinate (E) at (10,0);
            \coordinate (NE) at (45:10);
            \coordinate (NW) at (135:10);
            \coordinate (SE) at (-45:10);
            \coordinate (SW) at (-135:10);
            \coordinate (C) at (0,0);
            \coordinate (CN) at (0,5);
            \coordinate (CS) at (0,-5);
            \coordinate (CW) at (-5,0);
            \coordinate (CE) at (5,0);
            \begin{scope}
                \clip (C) circle [radius=10cm];
                \draw[webline] (S) -- (N);
                \draw[webline] ($(W)+(0,-4)$) -- ($(E)+(0,-4)$);
                \draw[webline] ($(W)+(0,4)$) -- ($(E)+(0,4)$);
                \node at ($(C)+(0,4)$) [yscale=.8]{\sqbox};
                \node at ($(C)+(0,-4)$) [yscale=.8]{\sqbox};
            \end{scope}
            \draw[dashed] (C) circle [radius=10cm];
            \node at (N) [above]{\scriptsize $n$};
            \node at (S) [below]{\scriptsize $n$};
            \node at ($(W)+(0,4)$) [left]{\scriptsize $1$};
            \node at ($(W)+(0,-4)$) [left]{\scriptsize $m-1$};
            \node at ($(E)+(0,4)$) [right]{\scriptsize $1$};
            \node at ($(E)+(0,-4)$) [right]{\scriptsize $m-1$};
        }
    },
    \end{align*}
    where an edge labeled by a positive integer $n$ means its \emph{$n$-cabling}.
    For any positive integer $n$, an \emph{$\fsp_4$-tripod of degree $n$} is defined by
    \begin{align*}
    \mbox
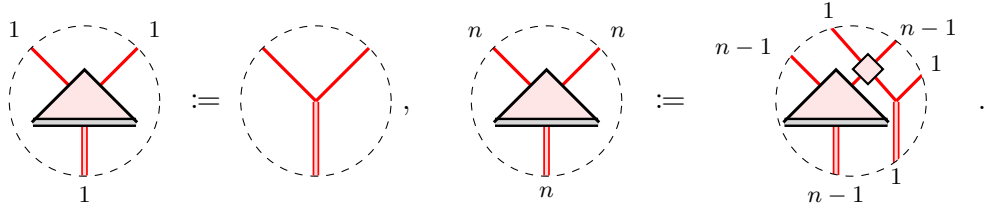
{
        \tikz[baseline=-.6ex, scale=.1]{
            \coordinate (N) at (0,10);
            \coordinate (S) at (0,-10);
            \coordinate (W) at (-10,0);
            \coordinate (E) at (10,0);
            \coordinate (NE) at (45:10);
            \coordinate (NW) at (135:10);
            \coordinate (SE) at (-45:10);
            \coordinate (SW) at (-135:10);
            \coordinate (C) at (0,0);
            \coordinate (CN) at (0,5);
            \coordinate (CS) at (0,-5);
            \coordinate (CW) at (-5,0);
            \coordinate (CE) at (5,0);
            \begin{scope}
                \clip (C) circle [radius=10cm];
                \draw[wline] (S) -- (C);
                \draw[webline] (NE) -- (C);
                \draw[webline] (NW) -- (C);
                \node at (C) [scale=1.4]{\tribox};
            \end{scope}
            \draw[dashed] (C) circle [radius=10cm];
            \node at (S) [below]{\scriptsize $1$};
            \node at (NE) [above right]{\scriptsize $1$};
            \node at (NW) [above left]{\scriptsize $1$};
        }
    }
    \coloneqq
    \mbox{
        \tikz[baseline=-.6ex, scale=.1]{
            \coordinate (N) at (0,10);
            \coordinate (S) at (0,-10);
            \coordinate (W) at (-10,0);
            \coordinate (E) at (10,0);
            \coordinate (NE) at (45:10);
            \coordinate (NW) at (135:10);
            \coordinate (SE) at (-45:10);
            \coordinate (SW) at (-135:10);
            \coordinate (C) at (0,0);
            \coordinate (CN) at (0,5);
            \coordinate (CS) at (0,-5);
            \coordinate (CW) at (-5,0);
            \coordinate (CE) at (5,0);
            \begin{scope}
                \clip (C) circle [radius=10cm];
                \draw[wline] (S) -- (C);
                \draw[webline] (NE) -- (C);
                \draw[webline] (NW) -- (C);
            \end{scope}
            \draw[dashed] (C) circle [radius=10cm];
        }
    }, \quad
    \mbox{
        \tikz[baseline=-.6ex, scale=.1]{
            \coordinate (N) at (0,10);
            \coordinate (S) at (0,-10);
            \coordinate (W) at (-10,0);
            \coordinate (E) at (10,0);
            \coordinate (NE) at (45:10);
            \coordinate (NW) at (135:10);
            \coordinate (SE) at (-45:10);
            \coordinate (SW) at (-135:10);
            \coordinate (C) at (0,0);
            \coordinate (CN) at (0,5);
            \coordinate (CS) at (0,-5);
            \coordinate (CW) at (-5,0);
            \coordinate (CE) at (5,0);
            \begin{scope}
                \clip (C) circle [radius=10cm];
                \draw[wline] (S) -- (C);
                \draw[webline] (NE) -- (C);
                \draw[webline] (NW) -- (C);
                \node at (C) [scale=1.4]{\tribox};
            \end{scope}
            \draw[dashed] (C) circle [radius=10cm];
            \node at (S) [below]{\scriptsize $n$};
            \node at (NE) [above right]{\scriptsize $n$};
            \node at (NW) [above left]{\scriptsize $n$};
        }
    }
    \coloneqq
    \mbox{
        \tikz[baseline=-.6ex, scale=.1]{
            \coordinate (N) at (0,10);
            \coordinate (S) at (0,-10);
            \coordinate (W) at (-10,0);
            \coordinate (E) at (10,0);
            \coordinate (NE) at (45:10);
            \coordinate (NW) at (135:10);
            \coordinate (SE) at (-45:10);
            \coordinate (SW) at (-135:10);
            \coordinate (C) at (0,0);
            \coordinate (CN) at (0,5);
            \coordinate (CS) at (0,-5);
            \coordinate (CW) at (-5,0);
            \coordinate (CE) at (5,0);
            \begin{scope}
                \clip (C) circle [radius=10cm];
                \draw[wline] ($(S)-(2,0)$) -- ($(C)-(2,0)$);
                \draw[webline, shorten >= -.3cm] ($(C)-(2,0)$) -- ($(NE)-(2,0)$);
                \draw[webline] ($(NW)-(2,0)$) -- ($(C)-(2,0)$);
                \node at ($(C)-(2,0)$) [scale=1.4]{\tribox};
                \draw[wline] ($(S)-(-6,0)$) -- ($(C)-(-6,0)$);
                \draw[webline] ($(C)-(-6,0)$) -- ($(NE)-(-6,0)$);
                \draw[webline, shorten <= -1cm] ($(C)!.6!(NE)-(2,0)$) -- ($(C)-(-6,0)$);
            \end{scope}
            \draw[dashed] (C) circle [radius=10cm];
            \draw[thick, fill=red!10] ($(C)!.6!(NE)-(2,2)$)--($(C)!.6!(NE)-(0,0)$)--($(C)!.6!(NE)-(2,-2)$)--($(C)!.6!(NE)-(4,0)$)--cycle;
            \node at ($(S)-(2,0)$) [below]{\scriptsize $n-1$};
            \node at ($(NE)-(2,0)$) [above right]{\scriptsize $n-1$};
            \node at ($(NW)-(2,0)$) [left]{\scriptsize $n-1$};
            \node at ($(S)+(6,0)$) {\scriptsize $1$};
            \node at ($(NE)+(2,-2)$) [right]{\scriptsize $1$};
            \node at ($(N)+(-3,2)$) {\scriptsize $1$};
        }
    }.
    \end{align*}
\end{dfn}
    For example, it is easily seen that the ladder equivalence splits $\fsp_4$-tetrapod into two $\fsp_4$-tripods:
    \begin{align*}
    \mbox
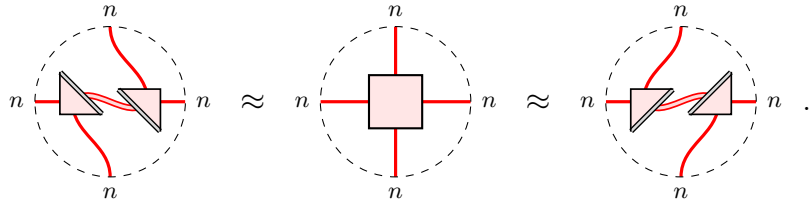
{
        \tikz[baseline=-.6ex, scale=.1, xscale=-1]{
            \coordinate (N) at (0,10);
            \coordinate (S) at (0,-10);
            \coordinate (W) at (-10,0);
            \coordinate (E) at (10,0);
            \coordinate (NE) at (45:10);
            \coordinate (NW) at (135:10);
            \coordinate (SE) at (-45:10);
            \coordinate (SW) at (-135:10);
            \coordinate (C) at (0,0);
            \coordinate (CN) at (0,5);
            \coordinate (CS) at (0,-5);
            \coordinate (CW) at (-5,0);
            \coordinate (CE) at (5,0);
            \begin{scope}
                \clip (C) circle [radius=10cm];
                \draw[webline] (N) to[out=south, in=north] (CW);
                \draw[webline] (S) to[out=north, in=south] (CE);
                \draw[webline] (W) -- ($(C)+(-3,0)$);
                \draw[webline] (E) -- ($(C)+(3,0)$);
                \draw[wline] (CW) to[out=south east, in=north west] (CE);
                \node at (CW) [xscale=-1, scale=.8, rotate=45]{\tribox};
                \node at (CE) [xscale=-1, scale=.8, rotate=-135]{\tribox};
            \end{scope}
            \draw[dashed] (C) circle [radius=10cm];
            \node at (N) [above]{\scriptsize $n$};
            \node at (S) [below]{\scriptsize $n$};
            \node at (W) [right]{\scriptsize $n$};
            \node at (E) [left]{\scriptsize $n$};
        }
    }
    \approx
    \mbox{
        \tikz[baseline=-.6ex, scale=.1]{
            \coordinate (N) at (0,10);
            \coordinate (S) at (0,-10);
            \coordinate (W) at (-10,0);
            \coordinate (E) at (10,0);
            \coordinate (NE) at (45:10);
            \coordinate (NW) at (135:10);
            \coordinate (SE) at (-45:10);
            \coordinate (SW) at (-135:10);
            \coordinate (C) at (0,0);
            \coordinate (CN) at (0,5);
            \coordinate (CS) at (0,-5);
            \coordinate (CW) at (-5,0);
            \coordinate (CE) at (5,0);
            \begin{scope}
                \clip (C) circle [radius=10cm];
                \draw[webline] (W) -- (E);
                \draw[webline] (S) -- (N);
                \node at ($(C)+(1,0)$) [xscale=.5]{$\cdots$};
                \node at (C) {\sqbox};
            \end{scope}
            \draw[dashed] (C) circle [radius=10cm];
            \node at (W) [left]{\scriptsize $n$};
            \node at (E) [right]{\scriptsize $n$};
            \node at (N) [above]{\scriptsize $n$};
            \node at (S) [below]{\scriptsize $n$};
        }
    }
    \approx
    \mbox{
        \tikz[baseline=-.6ex, scale=.1]{
            \coordinate (N) at (0,10);
            \coordinate (S) at (0,-10);
            \coordinate (W) at (-10,0);
            \coordinate (E) at (10,0);
            \coordinate (NE) at (45:10);
            \coordinate (NW) at (135:10);
            \coordinate (SE) at (-45:10);
            \coordinate (SW) at (-135:10);
            \coordinate (C) at (0,0);
            \coordinate (CN) at (0,5);
            \coordinate (CS) at (0,-5);
            \coordinate (CW) at (-5,0);
            \coordinate (CE) at (5,0);
            \begin{scope}
                \clip (C) circle [radius=10cm];
                \draw[webline] (N) to[out=south, in=north] (CW);
                \draw[webline] (S) to[out=north, in=south] (CE);
                \draw[webline] (W) -- ($(C)+(-3,0)$);
                \draw[webline] (E) -- ($(C)+(3,0)$);
                \draw[wline] (CW) to[out=south east, in=north west] (CE);
                \node at (CW) [scale=.8, rotate=45]{\tribox};
                \node at (CE) [scale=.8, rotate=-135]{\tribox};
            \end{scope}
            \draw[dashed] (C) circle [radius=10cm];
            \node at (N) [above]{\scriptsize $n$};
            \node at (S) [below]{\scriptsize $n$};
            \node at (W) [left]{\scriptsize $n$};
            \node at (E) [right]{\scriptsize $n$};
        }
    }.
    \end{align*}



\subsection{Non-elliptic \texorpdfstring{$\fsp_4$}{sp(4)}-diagrams in biangles and triangles}
Kuperberg discussed non-elliptic crossroad $\fsp_4$-diagrams with no crossings in a biangle and obtained a lemma below by computing the ``angular defect'' of $\fsp_4$-diagrams. We will give a proof of it in \cref{sec:minimal}.
\begin{lem}[{\cite[Lemma~6.5]{Kuperberg}}]\label{lem:biangle_blad}
    Let us assume that there exists a crossroad $\fsp_4$-diagram $D$ in a biangle such that one of the boundary intervals has $n_1$ type~$1$ edges and $m_1$ type~$2$ edges and another $n_2$ type~$1$ edges and $m_2$ type~$2$ edges.
    \begin{itemize}
        \item If $(n_1,m_1)\neq (n_2,m_2)$, then $D$ has at least one elliptic face.
        \item if $D$ is non-elliptic and $(n_1,m_1)=(n_2,m_2)$, then $D$ can be deformed to parallel strands by a sequence of ladder-resolutions and boundary $H$-moves.
    \end{itemize}
\end{lem}

For a given arrangement of type~$1$ and type~$2$ edges on the boundary intervals of a biangle, we can easily construct a non-elliptic $\fsp_4$-diagram by a procedure in \cref{fig:ladder-web}. We sometimes call such a non-elliptic $\fsp_4$-diagram a flat braid diagram by identifying (B) and (D) in \cref{fig:ladder-web}.

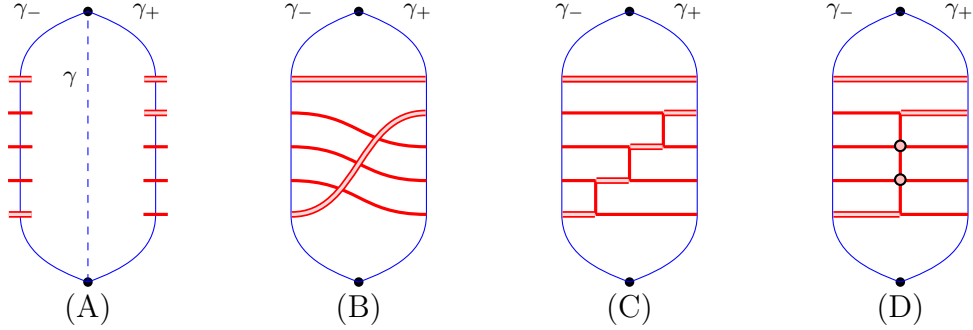
\begin{figure}[ht]
\centering
\begin{tikzpicture}[scale=0.9]
\begin{scope}[rotate=90]
\fill(2,0) circle(2pt);
\fill(-2,0) circle(2pt);
\draw[blue](-2,0) to[out=70,in=180] (-1,1) -- (1,1) to[out=0,in=110] (2,0);
\draw[blue](-2,0) to[out=-70,in=180] (-1,-1) -- (1,-1) to[out=0,in=-110] (2,0);
\draw[blue, dashed](-2,0) -- (2,0);
\foreach \i in {-1,-0.5,0}
\draw[webline] (\i,-1.18)--++(90:0.36);
\foreach \i in {0.5,1}
\draw[wline] (\i,-1+0.18)--++(-90:0.36);
\foreach \i in {-1,1}
\draw[wline] (\i,1.18)--++(-90:0.36);
\foreach \i in {-0.5,0,0.5}
\draw[webline] (\i,1-0.18)--++(90:0.36);
\node at (2,.5) [left]{\scriptsize $\gamma_{-}$};
\node at (2,-.5) [right]{\scriptsize $\gamma_{+}$};
\node at (1,0) [left]{\scriptsize $\gamma$};
\node at (-2,0) [below]{(A)};
\end{scope}

\begin{scope}[xshift=4cm,rotate=90]
\fill(2,0) circle(2pt);
\fill(-2,0) circle(2pt);
\draw[blue](-2,0) to[out=70,in=180] (-1,1) -- (1,1) to[out=0,in=110] (2,0);
\draw[blue](-2,0) to[out=-70,in=180] (-1,-1) -- (1,-1) to[out=0,in=-110] (2,0);
\foreach \i in {-1,-0.5,0}
\draw[webline] (\i,-1) ..controls (\i,0) and (\i+0.5,0).. (\i+0.5,1);
\draw[wline] (-1,1) ..controls (-1,0) and (0.5,0).. (0.5,-1);
\draw[wline] (1,1) -- (1,-1);
\node at (2,.5) [left]{\scriptsize $\gamma_{-}$};
\node at (2,-.5) [right]{\scriptsize $\gamma_{+}$};
\node at (-2,0) [below]{(B)};
\end{scope}

\begin{scope}[xshift=8cm,rotate=90]
\fill(2,0) circle(2pt);
\fill(-2,0) circle(2pt);
\draw[blue](-2,0) to[out=70,in=180] (-1,1) -- (1,1) to[out=0,in=110] (2,0);
\draw[blue](-2,0) to[out=-70,in=180] (-1,-1) -- (1,-1) to[out=0,in=-110] (2,0);
\foreach \i in {-1,-0.5,0}
{
\draw[webline] (\i,-\i-0.5) --++(0.5,0);
}
\draw[wline] (0.5,-1) -- (0.5,-0.5);
\draw[webline] (0.5,1) -- (0.5,-0.5);
\draw[webline] (0,-1) -- (0,-0.5);
\draw[wline] (0,0) -- (0,-0.5);
\draw[webline] (0,0) -- (0,1);
\draw[wline] (1,1) -- (1,-1);
\draw[webline] (-0.5,-1) -- (-0.5,0);
\draw[wline] (-0.5,0) -- (-0.5,0.5);
\draw[webline] (-0.5,0.5) -- (-0.5,1);
\draw[webline] (-1,-1) -- (-1,0.5);
\draw[wline] (-1,0.5) -- (-1,1);
\draw[wline] (1,1) -- (1,-1);
\node at (2,.5) [left]{\scriptsize $\gamma_{-}$};
\node at (2,-.5) [right]{\scriptsize $\gamma_{+}$};
\node at (-2,0) [below]{(C)};
\end{scope}

\begin{scope}[xshift=12cm,rotate=90]
\fill(2,0) circle(2pt);
\fill(-2,0) circle(2pt);
\draw[blue](-2,0) to[out=70,in=180] (-1,1) -- (1,1) to[out=0,in=110] (2,0);
\draw[blue](-2,0) to[out=-70,in=180] (-1,-1) -- (1,-1) to[out=0,in=-110] (2,0);
\draw[webline] (0.5,0) -- (-1,0);
\draw[wline] (0.5,-1) -- (0.5,0);
\draw[webline] (0.5,1) -- (0.5,0);
\draw[webline] (0,1) -- (0,-1);
\draw[webline] (-0.5,1) -- (-0.5,-1);
\draw[webline] (-1,-1) -- (-1,0);
\draw[wline] (-1,0) -- (-1,1);
\draw[wline] (1,1) -- (1,-1);
\node at (0,0) {$\crossroad$};
\node at (-.5,0) {$\crossroad$};
\node at (2,.5) [left]{\scriptsize $\gamma_{-}$};
\node at (2,-.5) [right]{\scriptsize $\gamma_{+}$};
\node at (-2,0) [below]{(D)};
\end{scope}

 \end{tikzpicture}
    \caption{Construction of the non-elliptic crossroad $\fsp_4$-diagram.
    (A) Arrangement of type $1$ and type $2$ edges on $\gamma_{-}$ and $\gamma_{+}$ so that the number of type $1$ (resp.~type $2$) edges on $\gamma_{-}$ coincides with that on $\gamma_{+}$. 
    (B) The flat braid diagram 
    with minimal intersection which represents the matching between left and right edges.
    (C) The associated $\fsp_4$-diagram with rungs.
    (D) The associated non-elliptic crossroad $\fsp_4$-diagram corresponding the flat braid diagram.
    }
    \label{fig:ladder-web}
\end{figure}

The first and third authors classified any non-elliptic crossroad $\fsp_4$-diagram in a triangle by using Kuperberg's argument. It implies the following classification.

\begin{lem}[{\cite[Proposition~2.24]{IYsp4}}]\label{lem:triangle_blad}
    When $\bSigma=T$ is a triangle, any ladder equivalence class of non-elliptic $\fsp_4$-diagrams has a unique reduced representative, up to arc parallel-moves, such as in \cref{fig:pyramid} with non-negative integers $k_i,l_i,n_1,n_2$ ($i=1,2,3$).
\end{lem}

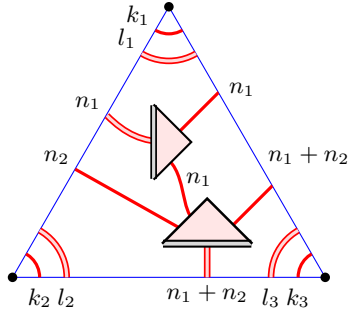
\begin{figure}[ht]
\begin{tikzpicture}[scale=.12]
    \coordinate (N) at (90:20);
    \coordinate (S) at (-90:20);
    \coordinate (W) at (180:20);
    \coordinate (E) at (0:20);
    \coordinate (NE) at (45:20);
    \coordinate (NW) at (135:20);
    \coordinate (SE) at (-45:20);
    \coordinate (SW) at (-135:20);
    \coordinate (C) at (0,0);
    \coordinate (CN) at ($(C)!.5!(N)$);
    \coordinate (CS) at ($(C)!.5!(S)$);
    \coordinate (CW) at ($(C)!.5!(W)$);
    \coordinate (CE) at ($(C)!.5!(E)$);
    \coordinate (P1) at (90:20);
    \coordinate (P2) at (210:20);
    \coordinate (P3) at (-30:20);
    \draw[blue] (P1) -- (P2) -- (P3)--cycle;
    \foreach \i in {-30,210,90} \fill(\i:20) circle(15pt);
    \begin{scope}
        \clip (P1) -- (P2) -- (P3)--cycle;
        \draw[webline] ($(C)!.4!(SE)$) --+ (45:10);
        \draw[webline] ($(C)!.3!(SE)$) to[out=west, in=south east] ($(C)!.3!(CN)$);
        \draw[webline] ($(C)!.3!(SE)+(0,-2)$) -- ($(90:20)!.6!(210:20)$);
        \draw[wline] ($(C)!.3!(SE)$) -- +(0,-10);
        \draw[webline] ($(C)!.5!(CN)$) --+ (45:10);
        \draw[wline] ($(C)!.5!(CN)$) to[out=west, in=south east] ($(90:20)!.4!(210:20)$);
        \node at ($(C)!.3!(SE)$) [scale=1.2]{$\tribox$};
        \node at ($(C)!.5!(CN)$) [rotate=-90]{$\tribox$};
        \foreach \i in {1,2,3}{
            \draw[webline] (P\i) circle [radius=3cm];
            \draw[wline] (P\i) circle [radius=6cm];
        }
    \end{scope}
    \node at ($(P1)!.4!(P2)$) [above left=-.1cm]{\scriptsize $n_1$};
    \node at ($(P1)!.6!(P2)$) [above left=-.1cm]{\scriptsize $n_2$};
    \node at ($(C)!.3!(SE)$) [below=.65cm]{\scriptsize $n_1+n_2$};
    \node at ($(P1)!.6!(P3)$) [above right=-.1cm]{\scriptsize $n_1+n_2$};
    \node at ($(P1)!.35!(P3)$) [above right=-.1cm]{\scriptsize $n_1$};
    \node at (C) [yshift=.1cm, xshift=.4cm]{\scriptsize $n_1$};
    \node at ($(P1)+(-120:3)$) [above left=-.1cm]{\scriptsize $k_1$};
    \node at ($(P1)+(-120:6)$) [above left=-.1cm]{\scriptsize $l_1$};
    \node at ($(P2)+(0:3)$) [below]{\scriptsize $k_2$};
    \node at ($(P2)+(0:6)$) [below]{\scriptsize $l_2$};
    \node at ($(P3)+(0:-3)$) [below]{\scriptsize $k_3$};
    \node at ($(P3)+(0:-6)$) [below]{\scriptsize $l_3$};
\end{tikzpicture}
    \caption{A reduced $\fsp_4$-diagram in a triangle consists of two $\fsp_4$-tripods and arbitrary number of corner arcs. We also have its reflection/rotations. Any reduced $\fsp_4$-diagram is uniquely represented by this diagram up to flips of corner arcs.} 
    \label{fig:pyramid}
\end{figure}

\subsection{Bounded \texorpdfstring{$\fsp_4$}{sp(4)}-laminations}\label{sec:lamination}

\begin{dfn}[Bounded $\fsp_4$-laminations]\label{def:lamination}
A \emph{rational bounded $\fsp_4$-lamination} on $\bSigma$ is a reduced ladder-equivalence class $L \in \Blad_{\fsp_4,\bSigma}$ of $\fsp_4$-diagrams on $\bSigma$ equipped with a rational number (called a \emph{measure}) on each component such that they are positive except for peripheral components. 
It is considered modulo the following ``cabling'' operations: 
\begin{enumerate}
    \item[(L1)] Combine a pair of isotopic loops with the same color with measures $u$ and $v$ into a single loop with the measure $u+v$. Similarly combine a pair of isotopic arcs with the same color into a single one by adding their measures. 
    \item[(L2)] For a positive integer $n \in \bZ_{>0}$ and a rational number $u \in \bQ$, replace a component with measure $nu$ with its $n$-cabling with measure $u$, which locally looks like
    
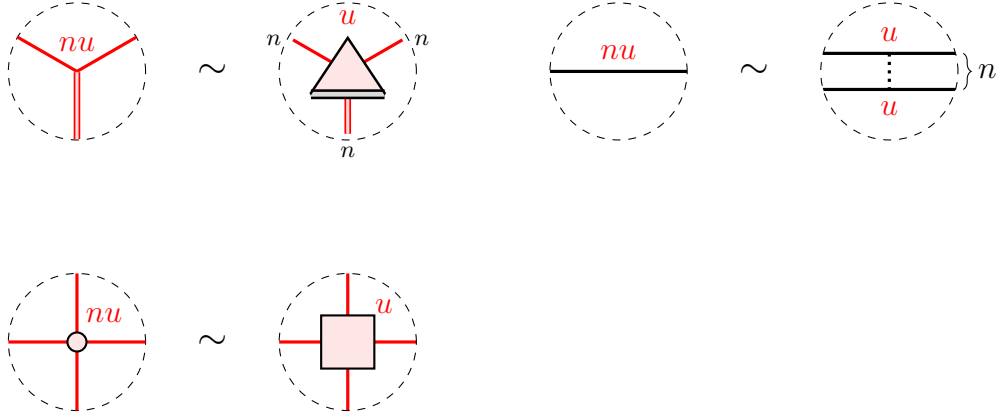
\begin{align*}
        \begin{tikzpicture}[scale=1.2]
        \draw[dashed] (0,0) circle(0.76cm);
        \foreach \i in {30,150}
        \draw[webline] (0,0) -- (\i:0.76);
        \draw[wline] (0,0) -- (270:0.76);
        \node[red] at (0,0.3) {$nu$};
        \node at (1.5,0) {\scalebox{1.2}{$\sim$}}; 
        \begin{scope}[xshift=3cm]
        \draw[webline] ($(-30:0.4)!1/2!(90:0.4)$) --++(30:0.5) node[right,black]{\scriptsize $n$};
        \draw[webline] ($(90:0.4)!1/2!(210:0.4)$) --++(150:0.5) node[left,black]{\scriptsize $n$};
        \draw[wline] ($(210:0.4)!1/2!(-30:0.4)$) --++(-90:0.5) node[below,black]{\scriptsize $n$};
        \node at (0,0) [yscale=1.5]{$\tribox$};
        \node[red] at (0,0.6) {$u$};
        \draw[dashed] (0,0) circle(0.76cm);
        \end{scope}
        \begin{scope}[xshift=6cm]
        \draw[dashed] (0,0) circle(0.76cm);
        \node at (1.5,0) {\scalebox{1.2}{$\sim$}};
        \clip (0,0) circle(0.76cm);
        \draw[very thick] (-1,0) --node[midway,above,red]{$nu$} (1,0);
        \end{scope}
        \begin{scope}[xshift=9cm]
        \draw[dashed] (0,0) circle(0.76cm);
        \draw[decorate,decoration={brace,amplitude=3pt,raise=2pt}] (0.76,0.2) --node[midway,right=0.3em]{$n$} (0.76,-0.2);
        \clip (0,0) circle(0.76cm);
        \draw[very thick] (-1,0.2) --node[above,red]{$u$} (1,0.2);
        \draw[very thick,dotted] (0,0.2) -- (0,-0.2);
        \draw[very thick] (-1,-0.2) --node[below,red]{$u$} (1,-0.2);
        \end{scope}
        \begin{scope}[yshift=-3cm]
        \draw[dashed] (0,0) circle(0.76cm);
        \foreach \i in {0,90,180,270}
        \draw[webline] (0,0) -- (\i:0.76);
        \draw[thick,fill=red!10] (0,0) circle(3pt);
        \node[red] at (0.3,0.3) {$nu$};
        \node at (1.5,0) {\scalebox{1.2}{$\sim$}};         
        \end{scope}
        \begin{scope}[xshift=3cm,yshift=-3cm]
        \draw[webline] (-0.76,0) -- (0.76,0);
        \draw[webline] (0,-0.76) -- (0,0.76);
        \node at (0,0) {\sqbox};
        \node[red] at (0.4,0.4) {$u$};
        \draw[dashed] (0,0) circle(0.76cm);
        \end{scope}
        \end{tikzpicture}
    \end{align*}
    Here the weight $u$ is positive unless the local picture belongs to a peripheral component. 
    For a loop or arc component, this operation is just the $n-1$ times applications of the operation (L1). One can also verify that the cabling operation is associative in the sense that the $n$-cabling followed by the $m$-cabling agrees with the $nm$-cabling. 
\end{enumerate}
\end{dfn}
Let $\cL^a(\bSigma,\bQ)$ denote the set of equivalence classes of the rational bounded $\fsp_4$-laminations. We have a natural $\bQ_{>0}$-action on $\cL^a(\bSigma,\bQ)$ that simultaneously rescales the measures. 
A rational bounded $\fsp_4$-lamination is said to be \emph{integral} if all of the measures are integers. The subset of integral bounded $\fsp_4$-laminations is denoted by $\cL^a(\bSigma,\bZ) \subset \cL^a(\bSigma,\bQ)$.

We have the following structures on the space $\cL^a(\bSigma,\bQ)$.

\subsubsection{Potential functions}

\begin{dfn}[potential functions]
For a marked point $m \in \bM$ and $s \in \{1,2\}$, the \emph{potential function of type $s$ at $m$} is the function 
\begin{align*}
    \pot_{m}^s: \cL^a(\bSigma,\bQ) \to \bQ
\end{align*}
that counts the total measure of the peripheral components of type $s$ around $m$. Let $\cL^a(\bSigma,\bZ) \subset \cL^a(\bSigma,\bQ)$ denote the subspace of integral $\fsp_4$-laminations. 
\end{dfn}

\begin{lem}\label{lem:lamination_Blad}
We have a canonical injective map 
\begin{align}\label{eq:lamination_Blad}
    \Blad_{\fsp_4,\bSigma} \to \cL^a(\bSigma,\bZ),
\end{align}
whose image is characterized by the \emph{potential condition} $\mathsf{w}_{m}^s(L) \geq 0$ for any marked point $m \in \bM$ and $s\in \{1,2\}$.
\end{lem}
Henceforth we will regard $\Blad_{\fsp_4,\bSigma}$ as a subset of $\cL^a(\bSigma,\bZ)$ via the map \eqref{eq:lamination_Blad}. 

\begin{proof}
An element of $\Blad_{\fsp_4,\bSigma}$ gives an integral $\fsp_4$-lamination, where all the components have measure $1$. Conversely, given any integral $\fsp_4$-lamination, we can represent it by a diagram with measure $1$ by applying the operations (L1) and (L2). The ladder-equivalence class of such a diagram is unique. 
The potential condition ensures that there are no components of negative measures. Thus it gives an element of $\Blad_{\fsp_4,\bSigma}$.
\end{proof}

Fix an ideal triangulation $\tri$ of $\bSigma$. 
For a marked point $m$ and a triangle $T \in t(\tri)$ having $m$ as a vertex, define the `local' potential  
\begin{align*}
    \pot_{m;T}^s: \cL^a(\bSigma,\bQ) \to \bQ
\end{align*}
to be the function that counts the total measure of the components of type $s$ whose restrictions to $T$ give corner arcs around $m$.

\begin{lem}\label{lem:potential_local}
The potential function $\pot_{m}^s$ is given in terms of local potentials by
\begin{align*}
    \pot_{m}^s = \min \{\pot^s_{m;T}\mid \text{$T \in t(\tri)$ have $m$ as a vertex} \}.
\end{align*}
\end{lem}

\begin{proof}
Clearly, the peripheral components around $m$ equally contribute to the both sides. Consider a non-peripheral component $c$. It does not contribute to the left-hand side. There must be at least one triangle $T$ incident to $m$ such that $c\cap T$ is not a corner arc around $m$, hence $\pot^s_{m;T}(c)=0$. Since $c$ has a positive measure, the other terms are non-negative. Hence the contribution to the right-hand side is also zero.  
\end{proof}

\subsubsection{Peripheral actions}\label{subsec:peripheral_action}
For each $m \in \bM$ and $s=1,2$, consider the $\bQ$-action on $\cL^a(\bSigma,\bQ)$, where $u \in \bQ$ adds a peripheral component of type $s$ with measure $u$ around $m$ (either a peripheral loop or a corner arc).
They form an action $\alpha_\bM: (\bQ^2)^{\bM} \times \cL^a(\bSigma,\bQ) \to \cL^a(\bSigma,\bQ)$. The following is clear from the definitions:

\begin{lem}\label{lem:action_potential}
For $u=(u_m^s)\in (\bQ^2)^{\bM}$ and $L \in \cL^a(\bSigma,\bQ)$, we have
\begin{align*}
    \pot_m^s (\alpha_\bM(u,L)) = \pot_m^s(L) + u_m^s.
\end{align*}
\end{lem}

\begin{rem}\label{rem:Cartan_action}
In order to identify the action $\alpha_\bM$ with the tropical analogue of the action 
\begin{align}\label{eq:Cartan_action}
    H^\bM \times \A_{Sp_4,\bSigma} \to \A_{Sp_4,\bSigma}
\end{align}
of the Cartan subgroup $H \subset Sp_4$, let us embed the coroot lattice $\mathsf{Q}^\vee= \bZ \alpha_1^\vee \oplus \bZ \alpha_2^\vee$ of $\fsp_4$ into $\bQ^2$ as $\alpha_1^\vee=(-2,2)$ $\alpha_2^\vee=(2,-1)$. Then the restricted action $\alpha_\bM: (\mathsf{Q}^\vee)^{\bM} \times \times \cL^a(\bSigma,\bQ) \to \cL^a(\bSigma,\bQ)$ preserves the subset $\cL^a(\bSigma,\bZ)_\congr$ (see \cref{def:congruent} and \cref{rem:potential_vector}), and gives the integral tropical analogue of \eqref{eq:Cartan_action}. See \cite{IK25} for a similar statement for the $\fsl_3$-case. 
\end{rem}

\subsection{Clasped \texorpdfstring{$\fsp_4$}{sp(4)}-skein algebra}\label{sec:sp-skein}
We briefly introduce the clasped $\fsp_4$-skein algebra. See \cite{IYsp4} for its basic properties. 
The clasped $\fso_5$-skein algebra is defined as a dual of it. See \cref{sec:skein} for details.

Let $\bSigma$ be a marked surface equipped with the set of marked points $\bM=\bP\sqcup\bM_\partial$.
\begin{dfn}
    A \emph{tangled $\fsp_4$-diagram} $W$ is an immersion of $\fsp_4$-colored $1,3,4$-valent graph $G$ into $\Sigma$ satisfying the following.
    \begin{itemize}
        \item $W(v)\in\bM_{\partial}$ for every $1$-valent vertices $p$ of $G$ and $W(x)\in\Sigma^{\ast}$ if $x\in G$ is not a $1$-valent vertices.
        \item Each intersection point $p$ in $\Sigma^{\ast}$ is a transverse double point and $W^{-1}(p)$ lies in interior of edges of $G$.
        \item $\mathrm{Arc}_{p}(G)$ is assigned a map $\phi_p\colon \mathrm{Arc}_{p}(G)\to \bR$ such that $\phi_p$ is injective if $p$ is a transverse double point in $\Sigma^{\ast}$, where
        \begin{itemize}
            \item $\mathrm{Arc}_{p}(G)$ be the set of connected components of $W^{-1}(N_p\cap W(G))$ for each intersection point $p$ of $W(G)$ and some small neighborhood $N_p$ of $\Sigma$ at $p$, and
            \item two maps $\phi_p$ and $\psi_p$ from $\mathrm{Arc}_{p}(G)$ to $\bR$ are equivalent if $\phi_p=f\circ\psi_p$ holds for some orientation preserving self-homeomorphism $f\colon \bR\to \bR$.
        \end{itemize} 
    \end{itemize}
    We call an intersection point $p$ in $\Sigma^{\ast}$ (resp. $\bM_{\partial}$) with the assignment $\phi_p$ an \emph{internal crossing} (resp. \emph{boundary crossing}). 
    The assignment defines over/under information of arcs: \emph{$\alpha$ passes over $\beta$} if $\phi_p(\alpha)>\phi_p(\beta)$ for $\alpha,\beta\in\mathrm{Arc}_{p}(G)$.
    Particularly, the assignment $\phi_p$ is called the \emph{elevation} for $p\in\bM_{\partial}$, and $\alpha$ and $\beta$ in $\mathrm{Arc}_{p}(G)$ has a \emph{simultaneous crossing} if $\phi_p(\alpha)=\phi_p(\beta)$. 
    A crossing is described as 
    $\mbox{
        \ 
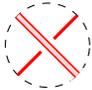
\tikz[baseline=-.6ex, scale=.08]{
            \draw[dashed, fill=white] (0,0) circle [radius=7];
            \draw[webline] (225:7) -- (45:7);
            \draw[overarc] (-45:7) -- (135:7);
        }
    \ }$,
    $\mbox{
        \ \tikz[baseline=-.6ex, scale=.08]{
            \draw[dashed, fill=white] (0,0) circle [radius=7];
            \draw[wline] (225:7) -- (45:7);
            \draw[owarc] (-45:7) -- (135:7);
        }
    \ }$,
    $\mbox{
        \ \tikz[baseline=-.6ex, scale=.08]{
            \draw[dashed, fill=white] (0,0) circle [radius=7];
            \draw[wline] (225:7) -- (45:7);
            \draw[overarc] (-45:7) -- (135:7);
        }
    \ }$, or
    $\mbox{
        \ \tikz[baseline=-.6ex, scale=.08]{
            \draw[dashed, fill=white] (0,0) circle [radius=7];
            \draw[webline] (225:7) -- (45:7);
            \draw[owarc] (-45:7) -- (135:7);
        }
    \ }$
    and the elevation is indicated by the distance from $p$ as \cref{fig:elevation}.
\end{dfn}

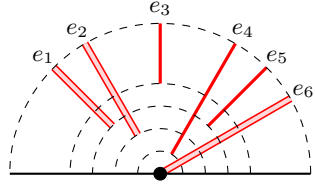
\begin{figure}
    \begin{tikzpicture}[scale=.2]
        \coordinate (P) at (0,0);
        \draw[webline, shorten <= .3cm] (P) -- (60:10);
        \draw[uwarc, shorten <= .6cm] (P) -- (120:10);
        \draw[webline, shorten <= 1.2cm] (P) -- (90:10);
        \draw[webline, shorten <= 0.9cm] (P) -- (45:10);
        \draw[uwarc, shorten <= 0.9cm] (P) -- (135:10);
        \draw[uwarc] (P) -- (30:10);
        \draw[dashed] (10,0) arc (0:180:10cm);
        \draw[dashed] (6,0) arc (0:180:6cm);
        \draw[dashed] (4.5,0) arc (0:180:4.5cm);
        \draw[dashed] (3,0) arc (0:180:3cm);
        \draw[dashed] (1.5,0) arc (0:180:1.5cm);
        \draw[fill=black] (P) circle [radius=12pt];
        \draw[thick] (-10,0) -- (10,0);
        \node at (30:11) {\scriptsize $e_6$};
        \node at (45:11) {\scriptsize $e_5$};
        \node at (60:11) {\scriptsize $e_4$};
        \node at (90:11) {\scriptsize $e_3$};
        \node at (120:11) {\scriptsize $e_2$};
        \node at (135:11) {\scriptsize $e_1$};
    \end{tikzpicture}
    \caption{$\mathrm{Arc}_p(G)=\{e_1,e_2,\dots,e_6\}$ is the set of half-arcs incident to $p$ for $p\in\bM_{\partial}$. The elevation $\phi_p\colon \mathrm{Arc}_{p}(G)\to\bR$ satisfies $\phi_p(e_3)<\phi_p(e_1)=\phi_p(e_5)<\phi_p(e_2)<\phi_p(e_4)<\phi_p(e_6)$. $e_1$ and $e_5$ has a simultaneous crossing.}
    \label{fig:elevation}
\end{figure}

\begin{dfn}\label{def:skein-rel}
    Let $[n]\coloneqq (q^{n}-q^{-n})/(q-q^{-1})$ for any non-negative integer $n$.
    The \emph{clasped $\fsp_4$-skein algebra $\Skein{\fsp_4}{\bSigma}^q$} of the marked surface $\bSigma$ is $\bZ[q^{\pm 1/2},1/[2]]$-linear combination of isotopy classes of tangled $\fsp_4$-diagrams with the \emph{$\fsp_4$-skein relation};
    \begin{align*}
        &\Oweb{uarc}=-\frac{[2][6]}{[3]}\noweb,\quad \Oweb{uwarc}=\frac{[5][6]}{[2][3]}\noweb\\
        &\monoweb{uwarc}{uarc}=0,\quad \biweb{uwarc}{uarc}{uarc}{uwarc}=-[2]\Iweb{uarc},\quad \triweb=0\\
        &\Xcross=\myHweb[rotate=90]{uarc}{uarc}{uwarc}{uarc}{uarc}-\frac{1}{[2]}\IIweb[rotate=90]{uarc}{uarc}=\myHweb{uarc}{uarc}{uwarc}{uarc}{uarc}-\frac{1}{[2]}\IIweb{uarc}{uarc}\\
        &\Xpos{oarc}{oarc}=q^{}\IIweb{uarc}{uarc}+q^{-1}\IIweb[rotate=90]{uarc}{uarc}+\Xcross\\
        &\Xpos{oarc}{uwarc}=q^{}\myHweb{uarc}{uwarc}{uarc}{uwarc}{uarc}+q^{-1}\myHweb[rotate=90]{uwarc}{uarc}{uarc}{uarc}{uwarc},\quad 
        \Xneg{owarc}{uarc}=q^{}\myHweb[rotate=90]{uwarc}{uarc}{uarc}{uarc}{uwarc}+q^{-1}\myHweb{uarc}{uwarc}{uarc}{uwarc}{uarc}\\
        &\Xpos{owarc}{owarc}=q^{2}\IIweb{uwarc}{uwarc}+q^{-2}\IIweb[rotate=90]{uwarc}{uwarc}+\XJcross
    \end{align*}
    and the \emph{clasped $\fsp_4$-skein relation};
    \begin{align*}
        q^{-\frac{1}{2}}\mbox{
            \tikz[baseline=-.6ex, scale=.08, yshift=-4cm]{
                \coordinate (P) at (0,0);
                \draw[webline, shorten <=.2cm] ($(P)$) -- (135:10);
                \draw[webline] (P) -- (45:10);
                \draw[dashed] (10,0) arc (0:180:10cm);
                \draw[very thick] (-10,0) -- (10,0);
                \draw[fill=black] (P) circle [radius=20pt];
            }
        }
        &=
        \mbox{
            \tikz[baseline=-.6ex, scale=.08, yshift=-4cm]{
                \coordinate (P) at (0,0);
                \draw[webline] (P) -- (45:10);
                \draw[webline] (P) -- (135:10);
                \draw[dashed] (10,0) arc (0:180:10cm);
                \draw[very thick] (-10,0) -- (10,0);
                \draw[fill=black] (P) circle [radius=20pt];
            }
        }
        =q^{\frac{1}{2}}
        \mbox{
            \tikz[baseline=-.6ex, scale=.08, yshift=-4cm]{
                \coordinate (P) at (0,0);
                \draw[webline, shorten <=.2cm] (P) -- (45:10);
                \draw[webline] (P) -- (135:10);
                \draw[dashed] (10,0) arc (0:180:10cm);
                \draw[very thick] (-10,0) -- (10,0);
                \draw[fill=black] (P) circle [radius=20pt];
            }
        },\\
        q^{-1}\mbox{
            \tikz[baseline=-.6ex, scale=.08, yshift=-4cm]{
                \coordinate (P) at (0,0);
                \draw[uwarc, shorten <=.2cm] (P) -- (135:10);
                \draw[uwarc] (P) -- (45:10);
                \draw[dashed] (10,0) arc (0:180:10cm);
                \draw[very thick] (-10,0) -- (10,0);
                \draw[fill=black] (P) circle [radius=20pt];
            }
        }
        &=
        \mbox{
            \tikz[baseline=-.6ex, scale=.08, yshift=-4cm]{
                \coordinate (P) at (0,0);
                \draw[uwarc] (P) -- (135:10);
                \draw[uwarc] (P) -- (45:10);
                \draw[dashed] (10,0) arc (0:180:10cm);
                \draw[very thick] (-10,0) -- (10,0);
                \draw[fill=black] (P) circle [radius=20pt];
            }
        }
        =q
        \mbox{
            \tikz[baseline=-.6ex, scale=.08, yshift=-4cm]{
                \coordinate (P) at (0,0);
                \draw[uwarc, shorten <=.2cm] (P) -- (45:10);
                \draw[uwarc] (P) -- (135:10);
                \draw[dashed] (10,0) arc (0:180:10cm);
                \draw[very thick] (-10,0) -- (10,0);
                \draw[fill=black] (P) circle [radius=20pt];
            }
        },\\
        q^{-\frac{1}{2}}
        \mbox{
            \tikz[baseline=-.6ex, scale=.08, yshift=-4cm]{
                \coordinate (P) at (0,0);
                \draw[uwarc, shorten <=.2cm] ($(P)$) -- (135:10);
                \draw[webline] (P) -- (45:10);
                \draw[dashed] (10,0) arc (0:180:10cm);
                \draw[very thick] (-10,0) -- (10,0);
                \draw[fill=black] (P) circle [radius=20pt];
            }
        }
        &=
        \mbox{
            \tikz[baseline=-.6ex, scale=.08, yshift=-4cm]{
                \coordinate (P) at (0,0);
                \draw[uwarc] ($(P)$) -- (135:10);
                \draw[webline] (P) -- (45:10);
                \draw[dashed] (10,0) arc (0:180:10cm);
                \draw[very thick] (-10,0) -- (10,0);
                \draw[fill=black] (P) circle [radius=20pt];
            }
        }
        =q^{\frac{1}{2}}
        \mbox{
            \tikz[baseline=-.6ex, scale=.08, yshift=-4cm]{
                \coordinate (P) at (0,0);
                \draw[webline, shorten <=.2cm] (P) -- (45:10);
                \draw[uwarc] (P) -- (135:10);
                \draw[dashed] (10,0) arc (0:180:10cm);
                \draw[very thick] (-10,0) -- (10,0);
                \draw[fill=black] (P) circle [radius=20pt];
            }
        },\\
        q^{-\frac{1}{2}}
        \mbox{
            \tikz[baseline=-.6ex, scale=.08, yshift=-4cm]{
                \coordinate (P) at (0,0);
                \draw[webline, shorten <=.2cm] ($(P)$) -- (135:10);
                \draw[uwarc] (P) -- (45:10);
                \draw[dashed] (10,0) arc (0:180:10cm);
                \draw[very thick] (-10,0) -- (10,0);
                \draw[fill=black] (P) circle [radius=20pt];
            }
        }
        &=
        \mbox{
            \tikz[baseline=-.6ex, scale=.08, yshift=-4cm]{
                \coordinate (P) at (0,0);
                \draw[webline, shorten <=.2cm] ($(P)$) -- (135:10);
                \draw[uwarc] (P) -- (45:10);
                \draw[dashed] (10,0) arc (0:180:10cm);
                \draw[very thick] (-10,0) -- (10,0);
                \draw[fill=black] (P) circle [radius=20pt];
            }
        }
        =q^{\frac{1}{2}}
        \mbox{
            \tikz[baseline=-.6ex, scale=.08, yshift=-4cm]{
                \coordinate (P) at (0,0);
                \draw[uwarc, shorten <=.2cm] (P) -- (45:10);
                \draw[webline] (P) -- (135:10);
                \draw[dashed] (10,0) arc (0:180:10cm);
                \draw[very thick] (-10,0) -- (10,0);
                \draw[fill=black] (P) circle [radius=20pt];
            }
        },\\
        \mbox{
            \tikz[baseline=-.6ex, scale=.08, yshift=-4cm]{
                \coordinate (P) at (0,0);
                \coordinate (R) at (45:7) {};
                \coordinate (L) at (135:7) {};
                \draw[webline] (L) -- (R);
                \draw[uwarc] (P) -- (L);
                \draw[uwarc] (R) -- (45:10);
                \draw[webline] (L) -- (135:10);
                \draw[webline] (P) -- (R);
                \draw[dashed] (10,0) arc (0:180:10cm);
                \draw[very thick] (-10,0) -- (10,0);
                \draw[fill=black] (P) circle [radius=20pt];
            }
        }
        &=
        \mbox{
            \tikz[baseline=-.6ex, scale=.08, yshift=-4cm]{
                \coordinate (P) at (0,0) {};
                \draw[webline] (P) -- (135:10);
                \draw[uwarc] (P) -- (45:10);
                \draw[dashed] (10,0) arc (0:180:10cm);
                \draw[very thick] (-10,0) -- (10,0);
                \draw[fill=black] (P) circle [radius=20pt];
            }
        },\quad
        \mbox{
            \tikz[baseline=-.6ex, scale=.08, yshift=-4cm]{
                \coordinate (P) at (0,0);
                \coordinate (R) at (45:7) {};
                \coordinate (L) at (135:7) {};
                \draw[webline] (L) -- (R);
                \draw[webline] (P) -- (L);
                \draw[webline] (R) -- (45:10);
                \draw[uwarc] (L) -- (135:10);
                \draw[uwarc] (P) -- (R);
                \draw[dashed] (10,0) arc (0:180:10cm);
                \draw[very thick] (-10,0) -- (10,0);
                \draw[fill=black] (P) circle [radius=20pt];
            }
        }
        =
        \mbox{
            \tikz[baseline=-.6ex, scale=.08, yshift=-4cm]{
                \coordinate (P) at (0,0) {};
                \draw[uwarc] (P) -- (135:10);
                \draw[webline] (P) -- (45:10);
                \draw[dashed] (10,0) arc (0:180:10cm);
                \draw[very thick] (-10,0) -- (10,0);
                \draw[fill=black] (P) circle [radius=20pt];
            }
        },\\
        \mbox{
            \tikz[baseline=-.6ex, scale=.08, yshift=-4cm]{
                \coordinate (P) at (0,0);
                \coordinate (R) at (45:7) {};
                \coordinate (L) at (135:7) {};
                \draw[uwarc] (L) -- (R);
                \draw[webline] (P) -- (L);
                \draw[webline] (R) -- (45:10);
                \draw[webline] (L) -- (135:10);
                \draw[webline] (P) -- (R);
                \draw[dashed] (10,0) arc (0:180:10cm);
                \draw[very thick] (-10,0) -- (10,0);
                \draw[fill=black] (P) circle [radius=20pt];
            }
        }
        &=\frac{1}{[2]}
        \mbox{
            \tikz[baseline=-.6ex, scale=.08, yshift=-4cm]{
                \coordinate (P) at (0,0) {};
                \draw[webline] (P) -- (135:10);
                \draw[webline] (P) -- (45:10);
                \draw[dashed] (10,0) arc (0:180:10cm);
                \draw[very thick] (-10,0) -- (10,0);
                \draw[fill=black] (P) circle [radius=20pt];
            }
        },\quad
        \mbox{
            \tikz[baseline=-.6ex, scale=.08, yshift=-4cm]{
                \coordinate (P) at (0,0);
                \coordinate (R) at (45:7) {};
                \coordinate (L) at (135:7) {};
                \draw[webline] (L) -- (R);
                \draw[uwarc] (P) -- (L);
                \draw[webline] (R) -- (45:10);
                \draw[webline] (L) -- (135:10);
                \draw[uwarc] (P) -- (R);
                \draw[dashed] (10,0) arc (0:180:10cm);
                \draw[very thick] (-10,0) -- (10,0);
                \draw[fill=black] (P) circle [radius=20pt];
            }
        }
        =0,\\
        \mbox{
            \tikz[baseline=-.6ex, scale=.08, yshift=-4cm]{
                \coordinate (P) at (0,0);
                \coordinate (C) at (90:7) {};
                \draw[webline] (P) to[out=north west, in=west] (C);
                \draw[webline] (P) to[out=north east, in=east] (C);
                \draw[uwarc] (C) -- (90:10);
                \draw[dashed] (10,0) arc (0:180:10cm);
                \draw[very thick] (-10,0) -- (10,0);
                \draw[fill=black] (P) circle [radius=20pt];
            }
        }
        &=0,\quad
        \mbox{
            \tikz[baseline=-.6ex, scale=.08, yshift=-4cm]{
                \coordinate (P) at (0,0);
                \coordinate (C) at (90:7) {};
                \draw[uwarc] (P) to[out=north west, in=west] (C);
                \draw[webline] (P) to[out=north east, in=east] (C);
                \draw[webline] (C) -- (90:10);
                \draw[dashed] (10,0) arc (0:180:10cm);
                \draw[very thick] (-10,0) -- (10,0);
                \draw[fill=black] (P) circle [radius=20pt];
            }
        }
        =0,\quad
        \mbox{
            \tikz[baseline=-.6ex, scale=.08, yshift=-4cm]{
                \coordinate (P) at (0,0);
                \coordinate (C) at (90:7) {};
                \draw[webline] (P) to[out=north west, in=west] (C);
                \draw[uwarc] (P) to[out=north east, in=east] (C);
                \draw[webline] (C) -- (90:10);
                \draw[dashed] (10,0) arc (0:180:10cm);
                \draw[very thick] (-10,0) -- (10,0);
                \draw[fill=black] (P) circle [radius=20pt];
            }
        }
        =0,\\
        \mbox{
            \tikz[baseline=-.6ex, scale=.08, yshift=-4cm]{
                \coordinate (P) at (0,0);
                \coordinate (C) at (90:7);
                \draw[webline] (P) to[out=north west, in=west] (C);
                \draw[webline] (P) to[out=north east, in=east] (C);
                \draw[dashed] (10,0) arc (0:180:10cm);
                \draw[very thick] (-10,0) -- (10,0);
                \draw[fill=black] (P) circle [radius=20pt];
            }
        }
        &=0,\quad
        \mbox{
            \tikz[baseline=-.6ex, scale=.08, yshift=-4cm]{
                \coordinate (P) at (0,0) {};
                \coordinate (C) at (90:7);
                \draw[uwarc] (P) to[out=north west, in=west] (C);
                \draw[uwarc] (P) to[out=north east, in=east] (C);
                \draw[dashed] (10,0) arc (0:180:10cm);
                \draw[very thick] (-10,0) -- (10,0);
                \draw[fill=black] (P) circle [radius=20pt];
            }
        }
        =0.
\end{align*}
    The multiplication of $\Skein{\fsp_4}{\bSigma}^q$ is defined by the superposition of their diagrams. We call a tangled $\fsp_4$-diagram in $\Skein{\fsp_4}{\bSigma}^q$ an \emph{$\fsp_4$-web}.
\end{dfn}

The above $\fsp_4$-skein relation implies the invariance of $\fsp_4$-webs under the Reidemeister moves. See \cite[Lemma~2.5]{IYsp4} for details.

We will discuss the graded version of the clasped $\fsp_4$-skein algebra in \cref{sec:reconstruction}.

\subsection{Clasped \texorpdfstring{$\fso_5$}{so(5)}-skein algebra}\label{sec:skein}
For the reader's convenience, we present here the defining relations of the \emph{clasped $\fso_5$-skein algebra} $\mathscr{S}_{\fso_5,\bSigma}^v$, which is the dual version of the skein algebra studied in \cite{IYsp4}. It is actually isomorphic to $\mathscr{S}_{\fsp_4,\bSigma}^v$, while only the diagram presentation is different. 
Let $\cR_v:=\bZ[v^{\pm 1/2},1/[2]]$ be the coefficient ring, where $[2]=[2]_v:=(v^2-v^{-2})/(v-v^{-1})$. 

The $\cR$-algebra $\mathscr{S}_{\fso_5,\bSigma}^v$ is spanned by tangled $\fso_5$-graph diagrams on $\bSigma$ defined similarly to the tangled $\fsp_4$-graph diagrams \cite[Definition 2.1]{IYsp4} but with $\fso_5$-colorings, modulo the $\fso_5$-skein relations obtained by interchanging the type~1/type~2 edges in \cite[Definitions 2.2 and 2.3]{IYsp4}. In particular, the set $\mathsf{Web}^\infty_{\fso_5,\bSigma}$ of unbounded $\fso_5$-webs is regarded as a subset of $\mathscr{S}_{\fso_5,\bSigma}^v$. By dualizing the argument in \cite[Section 2.2]{IYsp4}, we obtain an $\cR_v$-basis $\mathsf{BWeb}_{\fso_5,\bSigma}$ consisting of flat crossroad $\fso_5$-webs without elliptic faces, where the $\fso_5$-crossroad is introduced as
\begin{align}\label{eq:so5_crossroad}
    \mathord{
        \ \tikz[baseline=-.6ex, scale=.08, rotate=90]{
            \draw[dashed, fill=white] (0,0) circle [radius=7];
            \draw[wlined] (45:7) -- (-135:7);
            \draw[wlined] (135:7) -- (-45:7);
            \draw[fill=mygreen!50, thick] (0,0) circle [radius=30pt];
        }
    \ }
    :=
    \mathord{
        \ \tikz[baseline=-.6ex, scale=.08]{
            \draw[dashed, fill=white] (0,0) circle [radius=7];
            \draw[wlined] (45:7) -- (90:3);
            \draw[wlined] (135:7) -- (90:3);
            \draw[wlined] (225:7) -- (-90:3);
            \draw[wlined] (315:7) -- (-90:3);
            \draw[weblined] (90:3) -- (-90:3);
            }
    \ }
    -\frac{1}{[2]}
    \mathord{
        \ \tikz[baseline=-.6ex, scale=.08, rotate=90]{
            \draw[dashed, fill=white] (0,0) circle [radius=7];
            \draw[wlined] (-45:7) to[out=north west, in=south] (3,0) to[out=north, in=south west] (45:7);
            \draw[wlined] (-135:7) to[out=north east, in=south] (-3,0) to[out=north, in=south east] (135:7);
        }
    \ }
    =
    \mathord{
        \ \tikz[baseline=-.6ex, scale=.08, rotate=90]{
            \draw[dashed, fill=white] (0,0) circle [radius=7];
            \draw[wlined] (45:7) -- (90:3);
            \draw[wlined] (135:7) -- (90:3);
            \draw[wlined] (225:7) -- (-90:3);
            \draw[wlined] (315:7) -- (-90:3);
            \draw[weblined] (90:3) -- (-90:3);
            }
    \ }
    -\frac{1}{[2]}
    \mathord{
        \ \tikz[baseline=-.6ex, scale=.08]{
            \draw[dashed, fill=white] (0,0) circle [radius=7];
            \draw[wlined] (-45:7) to[out=north west, in=south] (3,0) to[out=north, in=south west] (45:7);
            \draw[wlined] (-135:7) to[out=north east, in=south] (-3,0) to[out=north, in=south east] (135:7);
        }
    \ }.
\end{align}
Then the product of two unbounded $\fso_5$-webs can be written as a (non-unique) $\cR$-linear combination of unbounded $\fso_5$-webs by first expanding in $\mathsf{BWeb}_{\fso_5,\bSigma}$ and then deleting the resulting $\fso_5$-crossroads using \eqref{eq:so5_crossroad}.

\subsection{Flat realization}\label{sec:bangle}
We discuss a relation between a basis of the $\fsp_4$-skein algebra and the bounded $\fsp_4$-laminations.
\begin{dfn}[endpoint-shift]\label{dfn:shift}
For any $D\in\diag_{\fsp_4,\bSigma}$, we define the corresponding $\fsp_4$-web $\diagshift(D)\in\Skein{\mathfrak{sp}_4}{\bSigma}^{q}$ as follows:
For each boundary interval $I\in\bB$, slide univalent vertices of $D$ on $I$ in the positive direction until they hit a special point, and collect them into a simultaneous crossing at the special point.
Then we get a map $\diagshift\colon\diag_{\fsp_4,\bSigma}\to\Skein{\mathfrak{sp}_4}{\bSigma}^{q}$, $D \mapsto \diagshift(D)$. 
We denote the image $\diagshift(\Bdiag_{\fsp_4,\bSigma})$ by $\Bweb_{\fsp_4,\bSigma}$, and call its element a \emph{basis web}.
\end{dfn}

\begin{thm}[{\cite[Theorem~2.15]{IYsp4}}]\label{thm:basis-web}
    $\Bweb_{\fsp_4,\bSigma}$ is an $\mathcal{R}$-basis of $\Skein{\mathfrak{sp}_4}{\bSigma}^{q}$.
\end{thm}

As mentioned in \cref{sec:lamination}, any element in $\Blad_{\fsp_4,\bSigma}$ has a unique reduced representative in $\Bdiag_{\fsp_4,\bSigma}$ up to loop and arc parallel-moves (\cref{eq:loop_parallel-move,eq:arc_parallel-move}).
Moreover, for a reduced bounded $\fsp_4$-diagram $D\in\Bdiag_{\fsp_4,\bSigma}$, $\diagshift(D)\in\Skein{\fsp_4}{\bSigma}^{q}$ is invariant under the loop and arc parallel-moves of $D$ by the clasped $\fsp_4$-skein relations and the Reidemeister moves. 
Hence, $\diagshift\colon\Bdiag_{\fsp_4,\bSigma}\to\Bweb_{\fsp_4,\bSigma}$ descends to $\ladshift\colon\Blad_{\fsp_4,\bSigma}\to\Bweb_{\fsp_4,\bSigma}$.
\begin{dfn}[flat realization]\label{dfn:bang-realization}
We denote the above bijection between bounded $\fsp_4$-laminations and a basis of $\fsp_4$-skein algebra by
\begin{align*}
\mathrm{Flat}\colon\Blad_{\fsp_4,\bSigma}\xrightarrow{\ladshift} \Bweb_{\fsp_4,\bSigma}
\end{align*}
and call it the \emph{flat realization} of the integral $\fsp_4$-laminations with potential condition; see also \cref{lem:lamination_Blad}.
\end{dfn}
\begin{rem}\label{rem:positive-basis}
The construction of linear bases for (quantum) cluster algebras is an important aspect of the study of the positivity conjecture. D.~Thurston~\cite{Thu14} proposed three types of bases for the $\mathfrak{sl}_2$-skein algebra of a surface, called the bangle basis, the band basis, and the bracelet basis. 
These bases are introduced into the cluster algebras of the surface in the $\mathfrak{sl}_2$-case: see \cite{Qin21} for a concise review. Mandel and Qin~\cite{MQ} showed that the bracelet basis coincides with the theta basis of Gross-Hacking-Keel-Kontsevich~\cite{GHKK}, and thus satisfies the strong positivity. 
In the $\fsp_4$-case, our $\mathrm{Flat}$ is seemingly an $\fsp_4$-analogue of the bangle basis. 
We still do not known an $\fsp_4$-analogue of the bracelet basis, but believe that it should be realized in a similar way by incorporating an appropriate Chebyshev polynomial.
\end{rem}

\section{Intersection pairing between \texorpdfstring{$\fsp_4$-}{sp(4)-} and \texorpdfstring{$\fso_5$-}{so(5)-}webs}\label{sec:pairing}

\subsection{Intersection pairing}

Recall $\Blad_{\fsp_4,\bSigma}$ and $\Blad_{\fso_5,\bSigma}^\infty$ in Definitions \ref{def:transverse}.
We are going to define the \emph{intersection pairing}
\begin{align*}
    \bi_\bSigma: \Blad_{\fsp_4,\bSigma} \times \Blad_{\fso_5,\bSigma}^\infty \to \frac 1 2 \bZ_{\geq 0}.
\end{align*}


\begin{dfn}
Let $[W] \in \Blad_{\fsp_4,\bSigma}$ and $[V] \in \Blad^\infty_{\fso_5,\bSigma}$. Take representatives $W$, $V$ of $[W]$, $[V]$, respectively, so that they have no crossroads ({\em i.e.} only have rungs) and intersect transversely to each other. For each transverse intersection point $p \in W \cap V$, we define the \emph{intersection index} $\varepsilon_p(W,V) \in \{1/2,1\}$ as follows:
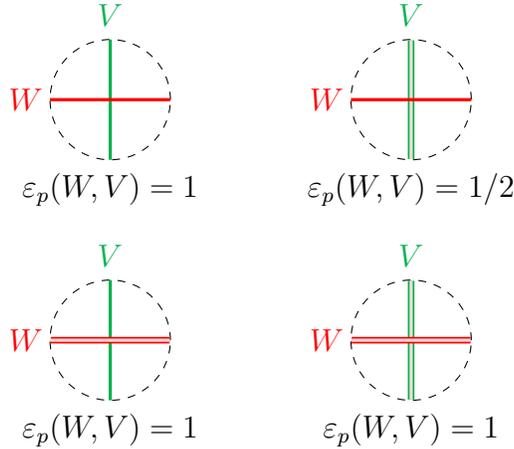
\begin{align*}
\begin{tikzpicture}[scale=0.8]\draw[dashed] (0,0) circle(1cm);
\draw[weblined] (0,-1) -- (0,1);
\draw[webline] (-1,0) -- (1,0);
\node at (0,-1.5) {$\varepsilon_p(W,V)=1$};
\node[red] at (-1.4,0) {$W$};
\node[mygreen] at (0,1.4) {$V$};
\begin{scope}[xshift=5cm]
\draw[dashed] (0,0) circle(1cm);
\draw[wlined] (0,-1) -- (0,1);
\draw[webline] (-1,0) -- (1,0);
\node at (0,-1.5) {$\varepsilon_p(W,V)=1/2$};
\node[red] at (-1.4,0) {$W$};
\node[mygreen] at (0,1.4) {$V$};
\end{scope}
\begin{scope}[yshift=-4cm]
\draw[dashed] (0,0) circle(1cm);
\draw[weblined] (0,-1) -- (0,1);
\draw[wline] (-1,0) -- (1,0);
\node at (0,-1.5) {$\varepsilon_p(W,V)=1$};
\node[red] at (-1.4,0) {$W$};
\node[mygreen] at (0,1.4) {$V$};
\end{scope}
\begin{scope}[xshift=5cm,yshift=-4cm]
\draw[dashed] (0,0) circle(1cm);
\draw[wlined] (0,-1) -- (0,1);
\draw[wline] (-1,0) -- (1,0);
\node at (0,-1.5) {$\varepsilon_p(W,V)=1$};
\node[red] at (-1.4,0) {$W$};
\node[mygreen] at (0,1.4) {$V$};
\end{scope}
\end{tikzpicture}
\end{align*}
Then we define the \emph{intersection number} of the ladder-equivalence classes $[W]$ and $[V]$ to be
\begin{align}
    \bi_\bSigma([W],[V]):= \min \sum_{p \in W \cap V} \varepsilon_p(W,V) \in \frac{1}{2}\bZ_{\geq 0},
\end{align}
where the minimum is taken over the transverse representatives $(W,V)$ of $([W],[V])$ with no crossroads. The resulting map
\begin{align*}
    \bi_\bSigma: \Blad_{\fsp_4,\bSigma} \times \Blad^\infty_{\fso_5,\bSigma} \to \frac 1 2 \bZ_{\geq 0},
\end{align*}
is called the \emph{intersection pairng}. 
\end{dfn}

\begin{rem}
    For any bounded $\fsp_4$-diagram $V$ and unbounded $\fso_5$-diagram $W$, any local move that reduces the geometric intersection number between $V$ and $W$ also strictly decreases $\sum_{p\in W\cap V}\varepsilon_{p}(W,V)$, except for the following two cases:
    \begin{align*}
        \mbox
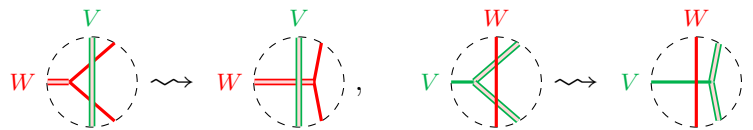
{
        \tikz[baseline=-.6ex, scale=.1]{
            \draw[dashed, fill=white] (0,0) circle [radius=6];
            \draw[webline] (60:6) -- (-3,0);
            \draw[webline] (-60:6) -- (-3,0);
            \draw[wline] (180:6) -- (-3,0);
            \draw[wlined] (90:6) -- (-90:6);
            \node at (90:6) [above, mygreen]{\scriptsize $V$};
            \node at (180:6) [left, red]{\scriptsize $W$};
        }
        }
        \scalebox{1.5}[1]{$\rightsquigarrow$}
        \mbox{
        \tikz[baseline=-.6ex, scale=.1]{
            \draw[dashed, fill=white] (0,0) circle [radius=6];
            \draw[webline] (60:6) -- (2,0);
            \draw[webline] (-60:6) -- (2,0);
            \draw[wline] (180:6) -- (2,0);
            \draw[wlined] (90:6) -- (-90:6);
            \node at (90:6) [above, mygreen]{\scriptsize $V$};
            \node at (180:6) [left, red]{\scriptsize $W$};
        }
        },\quad         
        \mbox{
        \tikz[baseline=-.6ex, scale=.1]{
            \draw[dashed, fill=white] (0,0) circle [radius=6];
            \draw[wlined] (60:6) -- (-3,0);
            \draw[wlined] (-60:6) -- (-3,0);
            \draw[weblined] (180:6) -- (-3,0);
            \draw[webline] (90:6) -- (-90:6);
            \node at (90:6) [above, red]{\scriptsize $W$};
            \node at (180:6) [left, mygreen]{\scriptsize $V$};
        }
        }
        \scalebox{1.5}[1]{$\rightsquigarrow$}
        \mbox{
        \tikz[baseline=-.6ex, scale=.1]{
            \draw[dashed, fill=white] (0,0) circle [radius=6];
            \draw[wlined] (60:6) -- (2,0);
            \draw[wlined] (-60:6) -- (2,0);
            \draw[weblined] (180:6) -- (2,0);
            \draw[webline] (90:6) -- (-90:6);
            \node at (90:6) [above, red]{\scriptsize $W$};
            \node at (180:6) [left, mygreen]{\scriptsize $V$};
        }
        }.
    \end{align*}
    These local moves preserve the value of $\sum_{p\in W\cap V}\varepsilon_{p}(W,V)$.
\end{rem}

\begin{rem}
Recall that the Cartan matrix $C(\fso_5)=(C_{st})_{s,t=1,2}$ of type $B_2$ is given by
\begin{align*}
    C(\fso_5) = \begin{pmatrix}
    2 & -1 \\
    -2 & 2
    \end{pmatrix},
\end{align*}
whose inverse matrix $C(\fso_5)^{-1}=(C^{st})_{s,t=1,2}$ is 
\begin{align*}
    C(\fso_5)^{-1} = \begin{pmatrix}
    1 & 1/2 \\
    1 & 1
    \end{pmatrix}.
\end{align*}
Observe that $\varepsilon_p(W,V)=C^{st}$ when $W$ and $V$ is colored by $s$ and $t$ at $p$, respectively. It is designed so that it is compatible with the highest degree of the coordinate expression of trace functions on the moduli space $\P_{PSO_5,\bSigma}$.
\end{rem}

\subsection{Intersection coordinate associated with a web cluster}
The notions of \emph{elementary $\fso_5$-webs} and \emph{$\fso_5$-web clusters} are defined similarly to \cite[Definitions 2.16 and 2.22]{IYsp4}. 
Elementary webs are represented by unbounded $\fso_5$-diagrams on $\bSigma$, and hence determine a ladder-equivalence class in $\Blad_{\fso_5,\Sigma}^\infty$. By abuse of terminology, we also call the latter class an elemenatary $\fso_5$-web, though it is not an element of the skein algebra. Similarly for the subset determined by an $\fso_5$-web cluster.

Typical web clusters are associated with decorated triangulations, which we describe in a detail. 
We refer the reader to \cref{sec:cluster_notation} for a detail on the tropical cluster structure. 
Recall that a decorated triangulation is a triple $\bD=(\tri,m_\tri,\bs_\tri)$, where
\begin{itemize}
    \item $\tri$ is an ideal triangulation of $\bSigma$;
    \item $m_\tri:t(\tri) \to \bM$ is a choice of a vertex of each triangle;
    \item $\bs_\tri:t(\tri) \to \{+,-\}$ is a choice of a sign at each triangle. 
\end{itemize}
Given a decorated triangulation $\bD$, we draw the quiver $Q_{m_\tri(T),\bs_\tri(T)}$ on each triangle $T \in t(\tri)$ as shown in the left of \cref{fig:labeling_triangle_case}, and amalgamate them to obtain a quiver $Q^{\bD}$ drawn on $\Sigma$. These quivers $Q^{\bD}$ (or the corresponding exchange matrices $\ve^{\bD}$) are mutation-equivalent to each other (\cref{thm:classical_mutation_equivalence}), and hence define a canonical mutation class $\sfs(\fsp_4,\bSigma)$ of exchange matrices. The associated $\fso_5$-web cluster $\cC_\bD=\{V_i\}$ is defined to be the collection of webs shown in the right of \cref{fig:labeling_triangle_case} over all triangles. 

\begin{rem}
Recall \cref{rem:reduced_word}. 
\begin{enumerate}
    \item The assignment of $\fso_5$-web clusters is compatible with that of $\fsp_4$-web clusters in \cite[Figure 5.1]{IYsp4} via the identification explained in \cref{rem:reduced_word}. 
    \item We expect that the cluster variables correspond to tree-type elementary webs \cite[Conjecture 4]{IYsp4} in the case of $\bM_\circ=\emptyset$. We do not know how to assign an $\fso_5$-web cluster to seeds in $\sfs(\fso_5,\bSigma)$, and $\fso_5$-web clusters are possibly more general than the seeds in  $\sfs(\fso_5,\bSigma)$.
\end{enumerate}
\end{rem}



\begin{figure}[ht]
\begin{tikzpicture}[scale=0.9]
\foreach \i in {90,210,330}
    {
    \markedpt{\i:3};
    \draw[blue] (\i:3) -- (\i+120:3);
    }
\quiverplusC{90:3}{210:3}{330:3};
\node at (0,-3) {$(\tri,m,+)$};
{\color{mygreen}
\uniarrow{x122}{x121}{dashed,shorten >=4pt, shorten <=2pt,bend left=20}
\uniarrow{x311}{x312}{dashed,shorten >=2pt, shorten <=4pt,bend right=20}
}
\uniarrow{x232}{x231}{myblue,dashed,shorten >=2pt, shorten <=4pt,bend left=20}
\draw (x121) node[above left]{$4$};
\draw (x122) node[above left]{$3$};
\draw (x231) node[below]{$5$};
\draw (x232) node[below]{$6$};
\draw (x311) node[above right]{$7$};
\draw (x312) node[above right]{$8$};
\draw (G1) node[below right]{$1$};
\draw (G2) node[above=0.2em]{$2$};

{\begin{scope}[xshift=8cm,yshift=0.5cm]
\begin{scope}
\draw[weblined] (0,0) -- (210:1);
	\draw[wlined] (0,0) -- (90:1);
	\draw[wlined] (0,0) -- (-30:1);
	\foreach \i in {90,210,330}
	{
	\markedpt{\i:1};\draw[blue] (\i:1) -- (\i+120:1);
	}
    \node at (0,-1) {$V_1$};
\end{scope}
\begin{scope}[yshift=2.5cm]
	\draw[weblined] (90:0.5) -- (90:1);
	\draw[weblined] (210:0.5) -- (210:1);
	\draw[wlined] (90:0.5) -- (210:0.5);
	\draw[wlined] (90:0.5) -- (-30:1); 
	\draw[wlined] (210:0.5) -- (-30:1);
	\foreach \i in {90,210,330}
	{
	\markedpt{\i:1};\draw[blue] (\i:1) -- (\i+120:1);
	}
	\node at (0,-1) {$V_2$};
\end{scope}
	
\begin{scope}[xshift=-2.5cm]
	\draw[wlined] (210:1) to[bend right=15] (90:1);
	\foreach \i in {90,210,330}
	{
	\markedpt{\i:1};\draw[blue] (\i:1) -- (\i+120:1);
	}
    \node at (0,-1) {$V_3$};
\end{scope}
	
\begin{scope}[xshift=-2.5cm,yshift=2.5cm]
	\draw[weblined] (210:1) to[bend right=15] (90:1);
	\foreach \i in {90,210,330}
	{
	\markedpt{\i:1};\draw[blue] (\i:1) -- (\i+120:1);
	}
    \node at (0,-1) {$V_4$};
\end{scope}
	
\begin{scope}[xshift=2.5cm]
	\draw[wlined] (90:1) to[bend right=15] (-30:1);
	\foreach \i in {90,210,330}
	{
	\markedpt{\i:1};\draw[blue] (\i:1) -- (\i+120:1);
	}
    \node at (0,-1) {$V_7$};
\end{scope}
	
\begin{scope}[xshift=2.5cm,yshift=2.5cm]
	\draw[weblined] (90:1) to[bend right=15] (-30:1);
	\foreach \i in {90,210,330}
	{
	\markedpt{\i:1};\draw[blue] (\i:1) -- (\i+120:1);
	}
    \node at (0,-1) {$V_8$};
\end{scope}
	
\begin{scope}[xshift=1.25cm,yshift=-2.5cm]
	\draw[weblined] (-30:1) to[bend right=15] (210:1);
	\foreach \i in {90,210,330}
	{
	\markedpt{\i:1};\draw[blue] (\i:1) -- (\i+120:1);
	}
    \node at (0,-1) {$V_6$};
\end{scope}
	
\begin{scope}[xshift=-1.25cm,yshift=-2.5cm]
	\draw[wlined] (-30:1) to[bend right=15] (210:1);
	\foreach \i in {90,210,330}
	{
	\markedpt{\i:1};\draw[blue] (\i:1) -- (\i+120:1);
	}
    \node at (0,-1) {$V_5$};
	\end{scope}
\end{scope}}
{\begin{scope}[yshift=-8cm]
\foreach \i in {90,210,330}
    {
    \markedpt{\i:3};
    \draw[blue] (\i:3) -- (\i+120:3);
    }
\quiverminusC{90:3}{210:3}{330:3};
\node at (0,-3) {$(\tri,m,-)$};
{\color{mygreen}
\uniarrow{x122}{x121}{dashed,shorten >=4pt, shorten <=2pt,bend left=20}
\uniarrow{x311}{x312}{dashed,shorten >=2pt, shorten <=4pt,bend right=20}
}
\uniarrow{x232}{x231}{myblue,dashed,shorten >=2pt, shorten <=4pt,bend left=20}
\draw (x121) node[above left]{$4$};
\draw (x122) node[above left]{$3$};
\draw (x231) node[below]{$5$};
\draw (x232) node[below]{$6$};
\draw (x311) node[above right]{$7$};
\draw (x312) node[above right]{$8$};
\draw (G1) node[below left]{$1$};
\draw (G2) node[above=0.2em]{$2$};

{\begin{scope}[xshift=8cm,yshift=0.5cm]
\begin{scope}
\draw[weblined] (0,0) -- (-30:1);
	\draw[wlined] (0,0) -- (90:1);
	\draw[wlined] (0,0) -- (210:1);
	\foreach \i in {90,210,330}
	{
	\markedpt{\i:1};\draw[blue] (\i:1) -- (\i+120:1);
	}
    \node at (0,-1) {$V_1$};
\end{scope}
\begin{scope}[yshift=2.5cm]
	\draw[weblined] (90:0.5) -- (90:1);
	\draw[weblined] (-30:0.5) -- (-30:1);
	\draw[wlined] (90:0.5) -- (-30:0.5);
	\draw[wlined] (90:0.5) -- (210:1); 
	\draw[wlined] (-30:0.5) -- (210:1);
	\foreach \i in {90,210,330}
	{
	\markedpt{\i:1};\draw[blue] (\i:1) -- (\i+120:1);
	}
	\node at (0,-1) {$V_2$};
\end{scope}
	
\begin{scope}[xshift=-2.5cm]
	\draw[wlined] (210:1) to[bend right=15] (90:1);
	\foreach \i in {90,210,330}
	{
	\markedpt{\i:1};\draw[blue] (\i:1) -- (\i+120:1);
	}
    \node at (0,-1) {$V_3$};
\end{scope}
	
\begin{scope}[xshift=-2.5cm,yshift=2.5cm]
	\draw[weblined] (210:1) to[bend right=15] (90:1);
	\foreach \i in {90,210,330}
	{
	\markedpt{\i:1};\draw[blue] (\i:1) -- (\i+120:1);
	}
    \node at (0,-1) {$V_4$};
\end{scope}
	
\begin{scope}[xshift=2.5cm]
	\draw[wlined] (90:1) to[bend right=15] (-30:1);
	\foreach \i in {90,210,330}
	{
	\markedpt{\i:1};\draw[blue] (\i:1) -- (\i+120:1);
	}
    \node at (0,-1) {$V_7$};
\end{scope}
	
\begin{scope}[xshift=2.5cm,yshift=2.5cm]
	\draw[weblined] (90:1) to[bend right=15] (-30:1);
	\foreach \i in {90,210,330}
	{
	\markedpt{\i:1};\draw[blue] (\i:1) -- (\i+120:1);
	}
    \node at (0,-1) {$V_8$};
\end{scope}
	
\begin{scope}[xshift=1.25cm,yshift=-2.5cm]
	\draw[weblined] (-30:1) to[bend right=15] (210:1);
	\foreach \i in {90,210,330}
	{
	\markedpt{\i:1};\draw[blue] (\i:1) -- (\i+120:1);
	}
    \node at (0,-1) {$V_6$};
\end{scope}
	
\begin{scope}[xshift=-1.25cm,yshift=-2.5cm]
	\draw[wlined] (-30:1) to[bend right=15] (210:1);
	\foreach \i in {90,210,330}
	{
	\markedpt{\i:1};\draw[blue] (\i:1) -- (\i+120:1);
	}
    \node at (0,-1) {$V_5$};
	\end{scope}
\end{scope}}
\end{scope}}
\end{tikzpicture}
    \caption{Assginment of $\fso_5$-elementary webs corresponding to the decorated triangulations.}
    \label{fig:labeling_triangle_case}
\end{figure}
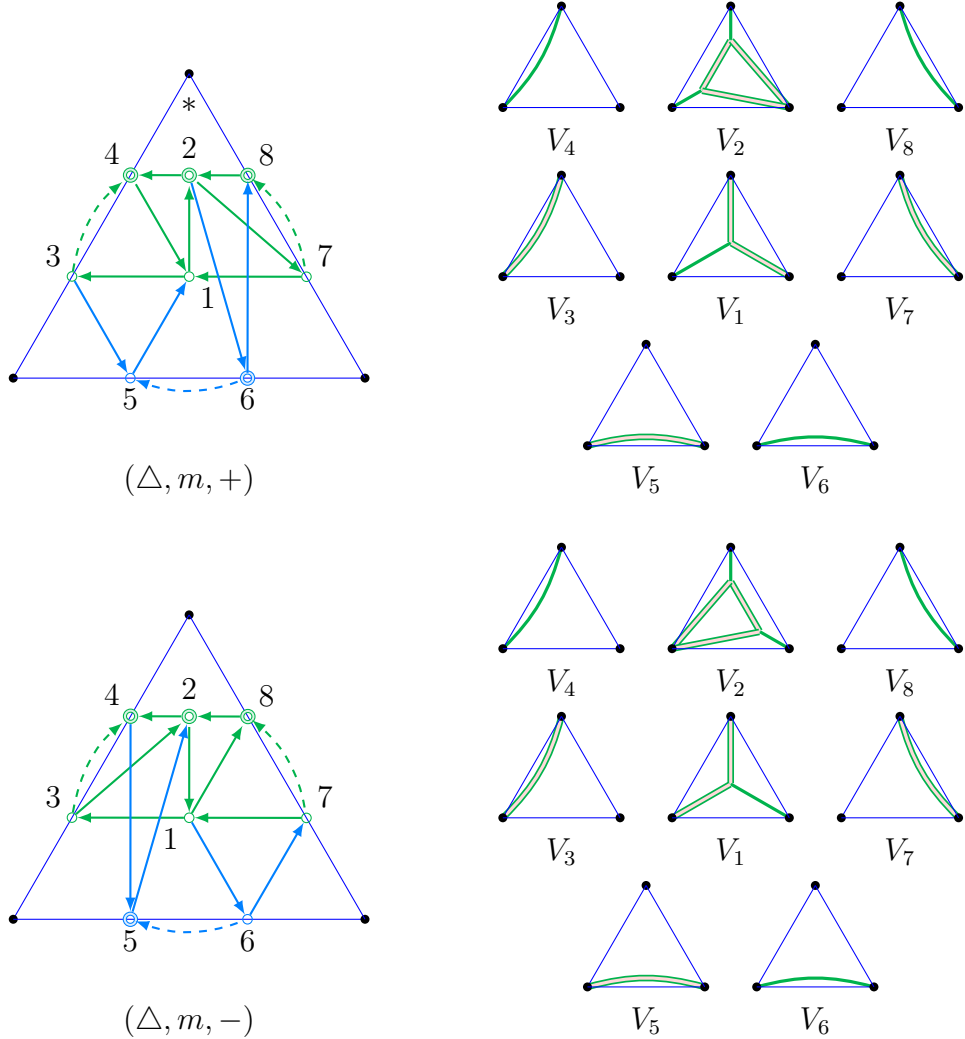

Given an $\fso_5$-web cluster $\cC=\{V_i\}_{i \in I}$ in the skein algebra $\mathscr{S}^q_{\fso_5,\bSigma}$, define the associated \emph{intersection coordinate system} to be
\begin{align}\label{eq:coord_web_cluster}
    \bsfa^\cC=(\sfa_i^\cC)_{i \in I} :=(\bi_\bSigma(-,V_i))_{i \in I}: \Blad_{\fsp_4,\bSigma} \to \bigg(\frac 1 2 \bZ_{\geq 0}\bigg)^{I}.
\end{align}
Write $\bsfa^{\bD}:=\bsfa^{\cC_{\bD}}$ for the web cluster $\cC_\bD$ associated with a decorated triangulation $\bD$.

Let 
\begin{align}\label{eq:coord_vect}
    \vect_{\bSigma}^{\cC}:= \bsfa^{\cC}(\Blad_{\fsp_4,\bSigma}) \subset \bigg(\frac 1 2 \bZ_{\geq 0}\bigg)^I
\end{align}
denote the image under the coordinate map. It is an intersection of a cone in $\bR^I$ with $\bigg(\frac 1 2 \bZ_{\geq 0}\bigg)^I$.



\subsection{Extension to the space of bounded \texorpdfstring{$\fsp_4$}{sp(4)}-laminations}
Let us upgrade the intersection pairing to a map
\begin{align}\label{eq:pairing_lamination}
    \bi_\bSigma: \cL^a(\bSigma,\bQ) \times \Blad_{\fso_5,\Sigma}^\infty \to \bQ.
\end{align}
For this, let $L \in \cL^a(\bSigma,\bZ)$ be a bounded integral $\fsp_4$-lamination, which is a weighted union $\{([W_\alpha],u_\alpha)\}_\alpha$ of components $[W_\alpha] \in \Blad_{\fsp_4,\bSigma}$ with measures $u_\alpha \in \bQ$. 
Then for any $[V] \in \Blad_{\fso_5,\Sigma}^\infty$, define
\begin{align*}
    \bi_\bSigma(L,[V]) := \sum_\alpha u_\alpha \bi_\bSigma([W_\alpha],[V]).
\end{align*}
One can easily verify that $\bi_\bSigma(L,[V])$ is invariant under the cabling operations (L1) and (L2) and hence well-defined. Moreover, it is equivariant under the rescaling $\bQ_{>0}$-action: $\bi_\bSigma(u\cdot L,[V])=u\cdot \bi_\bSigma(L,[V])$. 

\begin{dfn}[Intersection coordinates]\label{def:intersection_coordinate}
We extend the intersection coordinate system \eqref{eq:coord_web_cluster} associated with an $\fso_5$-web cluster $\cC=\{V_i\}_{i \in I}$ to be a coordinate system
\begin{align*}
    \bsfa^\cC=(\sfa_i^\cC)_{i \in I}=(\bi_\bSigma(-,V_i))_{i \in I}: \cL^a
    (\bSigma,\bQ) \to \bQ^{I}.
\end{align*}

\end{dfn}



The pair $(\ve^{\bD},\bsfa^{\bD})$ with $\bsfa^{\bD}:=\bsfa^{\cC_\bD}$ is regarded as a tropical seed on $X=\cL^a(\bSigma,\bQ)$ (see \cref{sec:cluster_notation}). 
Here is the main theorem of this paper:

\begin{thm}[Proof in \cref{subsec:reconstruction}]\label{thm:bijection}
For any decorated triangulation $\bD=(\tri,m_\tri,\bs_\tri)$ of $\bSigma$, the coordinate system 
\begin{align*}
    \bsfa^{\bD}: \cL^a
    (\bSigma,\bQ) \xrightarrow{\sim} \bQ^{I(\bD)}
\end{align*}
gives a bijection. 
\end{thm}


Let us prepare some useful notions. 
\begin{dfn}
Fix a web cluster $\cC=\{V_i\}_{i \in I}$. 
We say that the coordinate vector $v \in \vect_\bSigma^\cC \subset (\frac 1 2 \bZ_{\geq 0})^I$ is \emph{$\cC$-irreducible} if it admits no non-trivial decomposition into a sum $v=v_1+v_2$ with $v_1,v_2 \in \vect_\bSigma^\cC$. The set $\mathscr{H}_{\fsp_4,\bSigma}^{\cC} \subset \vect_\bSigma^\cC$ of irreducible vectors is called the \emph{Hilbert basis}. 
\end{dfn}

\begin{rem}
The set $\vect_{\bSigma}^{\cC} \subset (\frac 1 2 \bZ_{\geq 0})^I$ turns out to be the lattice points of a pointed convex polyhedral cone, and we expect the same for general $\cC$. It follows that $\mathscr{H}_{\fsp_4,\bSigma}^{\cC}$ is a unique non-empty finite set. See \cite{DS20II} for a detail on the Hilbert basis and its usage in the $\fsl_3$-case. 
\end{rem}

\begin{dfn}\label{def:congruent}
An integral bounded $\fsp_4$-lamination $L \in \cL^a(\bSigma,\bZ)$ is said to be \emph{congruent} if $\bsfa^{\bD}(L) \in \bZ^{I(\bD)}$ for any decorated triangulation $\bD$. Let $\cL^a(\bSigma,\bZ)_{\congr} \subset \cL^a(\bSigma,\bQ)$ denote the subset of congruent laminations. 
\end{dfn}
As a direct consequence of \cref{thm:bijection}, we have a bijection
\begin{align*}
    \bsfa^{\bD}: \cL^a
    (\bSigma,\bZ)_\congr \xrightarrow{\sim} \bZ^{I(\bD)}.
\end{align*}


\input{4_coordinates}

\subsection{Triangle case}\label{subsec:triangle_case}

Let us consider a triangle $\bSigma=T$ together with one of the six decorated triangulation $\bD$ in \cref{fig:exch_triangle}. 
Let $\bsfa^{\bD}=(\sfa_i^{\bD})_{i \in I(\bD)}$ be the corresponding tuple of intersection coordinates.

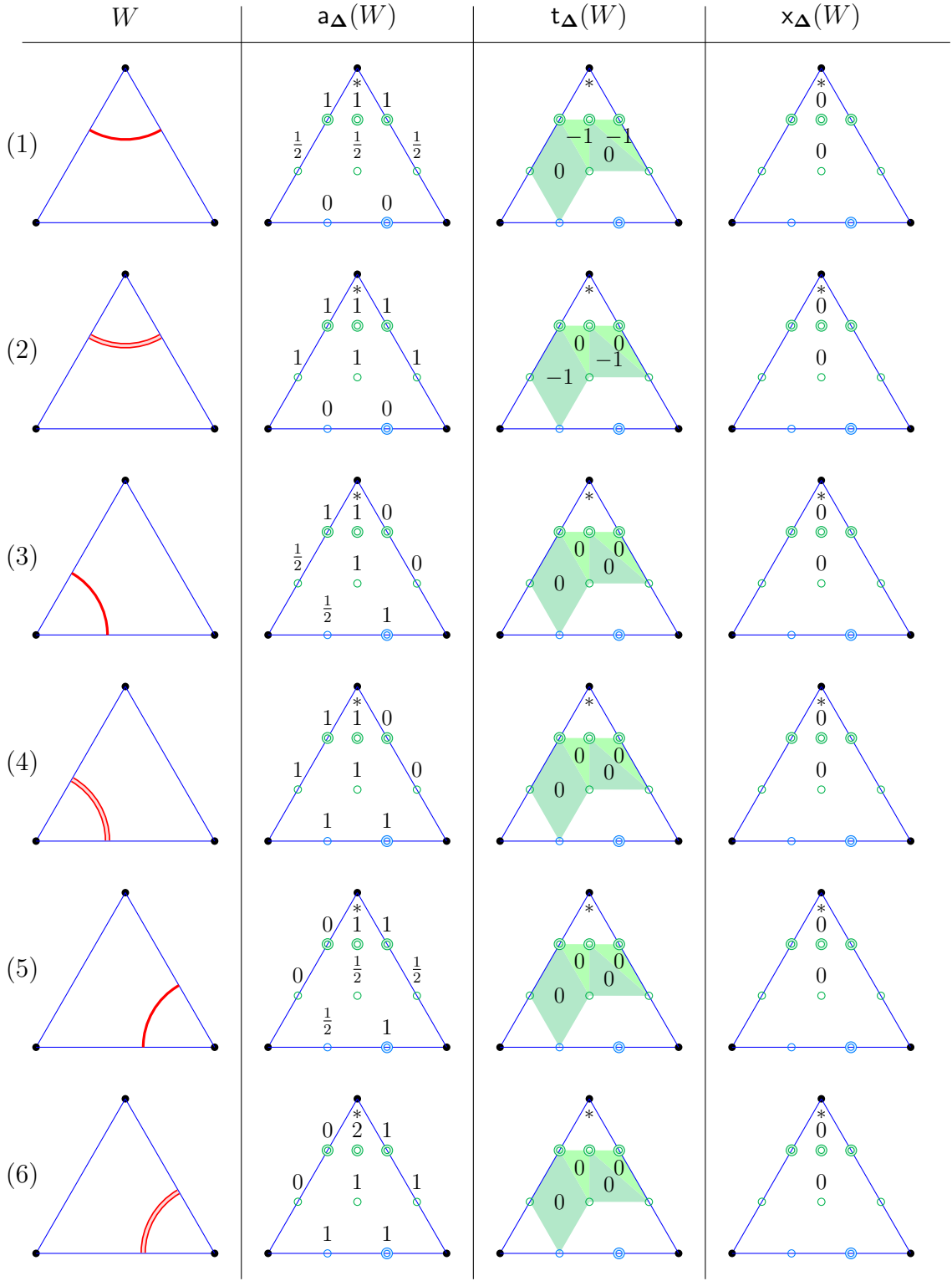
\begin{figure}[ht]
    \centering
\begin{tikzpicture}[scale=0.82]
\node at (-2,0.5) {(1)};
\node at (0,3) {$W$};
\node at (4.5,3) {$\mathsf{a}_{\bD}(W)$};
\node at (9,3) {$\sft_{\bD}(W)$};
\node at (13.5,3) {$\mathsf{x}_{\bD}(W)$};
\draw(2.25,3.2) -- (2.25,-21.5);
\draw(4.5+2.25,3.2) -- (4.5+2.25,-21.5);
\draw(9+2.25,3.2) -- (9+2.25,-21.5);
\draw(-2,2.5) -- (16,2.5);
\foreach \i in {90,210,330}
    {
    \markedpt{\i:2};
    \draw[blue] (\i:2) -- (\i+120:2);
    }
\draw[webline] ($(90:2)!0.4!(210:2)$) arc(-120:-60:0.8*1.732);
{\begin{scope}[xshift=4.5cm]
\foreach \i in {90,210,330}
    {
    \markedpt{\i:2};
    \draw[blue] (\i:2) -- (\i+120:2);
    }
\node[scale=0.75] at (90:1.7) {$\ast$};
\quiververtexC{90:2}{210:2}{-30:2}
\node[above=0.2em,scale=0.85] at (x121) {$1$};
\node[above=0.2em,scale=0.85] at (x122) {$\frac{1}{2}$};
\node[above=0.2em,scale=0.85] at (x312) {$1$};
\node[above=0.2em,scale=0.85] at (x311) {$\frac{1}{2}$};
\node[above=0.2em,scale=0.85] at (G2) {$1$};
\node[above=0.2em,scale=0.85] at (G1) {$\frac{1}{2}$};
\node[above=0.2em,scale=0.85] at (x231) {$0$};
\node[above=0.2em,scale=0.85] at (x232) {$0$};
\end{scope}}
{\begin{scope}[xshift=9cm]
\foreach \i in {90,210,330}
    {
    \markedpt{\i:2};
    \draw[blue] (\i:2) -- (\i+120:2);
    }
\node[scale=0.75] at (90:1.7) {$\ast$};
\quiververtexC{90:2}{210:2}{-30:2}
\begin{pgfonlayer}{bg} 
\fill[rbA!30] (x122) -- (x231) -- (G1) -- (x121) --cycle;
\fill[rbB!30] (x121) -- (G1) -- (G2) --cycle;
\fill[rbA!30] (G1) -- (G2) -- (x311) --cycle;
\fill[rbB!30] (G2) -- (x312) -- (x311) --cycle;
\end{pgfonlayer}
\node[scale=0.85] at ($(G1)!0.5!(x122)$) {$0$};
\CoG{G1}{G2}{x311};
\node[scale=0.85] at (G) {$0$};
\CoG{G1}{G2}{x121};
\node[scale=0.85] at (G) {$-1$};
\CoG{x311}{x312}{G2};
\node[scale=0.85] at (G) {$-1$};
\end{scope}}
{\begin{scope}[xshift=13.5cm]
\foreach \i in {90,210,330}
    {
    \markedpt{\i:2};
    \draw[blue] (\i:2) -- (\i+120:2);
    }
\node[scale=0.75] at (90:1.7) {$\ast$};
\quiververtexC{90:2}{210:2}{-30:2}
\node[above=0.2em,scale=0.85] at (G2) {$0$};
\node[above=0.2em,scale=0.85] at (G1) {$0$};
\end{scope}}

\begin{scope}[yshift=-4cm]
\node at (-2,0.5) {(2)};
\foreach \i in {90,210,330}
    {
    \markedpt{\i:2};
    \draw[blue] (\i:2) -- (\i+120:2);
    }
\draw[wline] ($(90:2)!0.4!(210:2)$) arc(-120:-60:0.8*1.732);
{\begin{scope}[xshift=4.5cm]
\foreach \i in {90,210,330}
    {
    \markedpt{\i:2};
    \draw[blue] (\i:2) -- (\i+120:2);
    }
\node[scale=0.75] at (90:1.7) {$\ast$};
\quiververtexC{90:2}{210:2}{-30:2}
\node[above=0.2em,scale=0.85] at (x121) {$1$};
\node[above=0.2em,scale=0.85] at (x122) {$1$};
\node[above=0.2em,scale=0.85] at (x312) {$1$};
\node[above=0.2em,scale=0.85] at (x311) {$1$};
\node[above=0.2em,scale=0.85] at (G2) {$1$};
\node[above=0.2em,scale=0.85] at (G1) {$1$};
\node[above=0.2em,scale=0.85] at (x231) {$0$};
\node[above=0.2em,scale=0.85] at (x232) {$0$};
\end{scope}}
{\begin{scope}[xshift=9cm]
\foreach \i in {90,210,330}
    {
    \markedpt{\i:2};
    \draw[blue] (\i:2) -- (\i+120:2);
    }
\node[scale=0.75] at (90:1.7) {$\ast$};
\quiververtexC{90:2}{210:2}{-30:2}
\begin{pgfonlayer}{bg} 
\fill[rbA!30] (x122) -- (x231) -- (G1) -- (x121) --cycle;
\fill[rbB!30] (x121) -- (G1) -- (G2) --cycle;
\fill[rbA!30] (G1) -- (G2) -- (x311) --cycle;
\fill[rbB!30] (G2) -- (x312) -- (x311) --cycle;
\end{pgfonlayer}
\node[scale=0.85] at ($(G1)!0.5!(x122)$) {$-1$};
\CoG{G1}{G2}{x311};
\node[scale=0.85] at (G) {$-1$};
\CoG{G1}{G2}{x121};
\node[scale=0.85] at (G) {$0$};
\CoG{x311}{x312}{G2};
\node[scale=0.85] at (G) {$0$};
\end{scope}}
{\begin{scope}[xshift=13.5cm]
\foreach \i in {90,210,330}
    {
    \markedpt{\i:2};
    \draw[blue] (\i:2) -- (\i+120:2);
    }
\node[scale=0.75] at (90:1.7) {$\ast$};
\quiververtexC{90:2}{210:2}{-30:2}
\node[above=0.2em,scale=0.85] at (G2) {$0$};
\node[above=0.2em,scale=0.85] at (G1) {$0$};
\end{scope}}
\end{scope}

\begin{scope}[yshift=-8cm]
\node at (-2,0.5) {(3)};
\foreach \i in {90,210,330}
    {
    \markedpt{\i:2};
    \draw[blue] (\i:2) -- (\i+120:2);
    }
\draw[webline] ($(210:2)!0.4!(-30:2)$) arc(0:60:0.8*1.732);
{\begin{scope}[xshift=4.5cm]
\foreach \i in {90,210,330}
    {
    \markedpt{\i:2};
    \draw[blue] (\i:2) -- (\i+120:2);
    }
\node[scale=0.75] at (90:1.7) {$\ast$};
\quiververtexC{90:2}{210:2}{-30:2}
\node[above=0.2em,scale=0.85] at (x121) {$1$};
\node[above=0.2em,scale=0.85] at (x122) {$\frac{1}{2}$};
\node[above=0.2em,scale=0.85] at (x312) {$0$};
\node[above=0.2em,scale=0.85] at (x311) {$0$};
\node[above=0.2em,scale=0.85] at (G2) {$1$};
\node[above=0.2em,scale=0.85] at (G1) {$1$};
\node[above=0.2em,scale=0.85] at (x231) {$\frac{1}{2}$};
\node[above=0.2em,scale=0.85] at (x232) {$1$};
\end{scope}}
{\begin{scope}[xshift=9cm]
\foreach \i in {90,210,330}
    {
    \markedpt{\i:2};
    \draw[blue] (\i:2) -- (\i+120:2);
    }
\node[scale=0.75] at (90:1.7) {$\ast$};
\quiververtexC{90:2}{210:2}{-30:2}
\begin{pgfonlayer}{bg} 
\fill[rbA!30] (x122) -- (x231) -- (G1) -- (x121) --cycle;
\fill[rbB!30] (x121) -- (G1) -- (G2) --cycle;
\fill[rbA!30] (G1) -- (G2) -- (x311) --cycle;
\fill[rbB!30] (G2) -- (x312) -- (x311) --cycle;
\end{pgfonlayer}
\node[scale=0.85] at ($(G1)!0.5!(x122)$) {$0$};
\CoG{G1}{G2}{x311};
\node[scale=0.85] at (G) {$0$};
\CoG{G1}{G2}{x121};
\node[scale=0.85] at (G) {$0$};
\CoG{x311}{x312}{G2};
\node[scale=0.85] at (G) {$0$};
\end{scope}}
{\begin{scope}[xshift=13.5cm]
\foreach \i in {90,210,330}
    {
    \markedpt{\i:2};
    \draw[blue] (\i:2) -- (\i+120:2);
    }
\node[scale=0.75] at (90:1.7) {$\ast$};
\quiververtexC{90:2}{210:2}{-30:2}
\node[above=0.2em,scale=0.85] at (G2) {$0$};
\node[above=0.2em,scale=0.85] at (G1) {$0$};
\end{scope}}
\end{scope}

\begin{scope}[yshift=-12cm]
\node at (-2,0.5) {(4)};
\foreach \i in {90,210,330}
    {
    \markedpt{\i:2};
    \draw[blue] (\i:2) -- (\i+120:2);
    }
\draw[wline] ($(210:2)!0.4!(-30:2)$) arc(0:60:0.8*1.732);
{\begin{scope}[xshift=4.5cm]
\foreach \i in {90,210,330}
    {
    \markedpt{\i:2};
    \draw[blue] (\i:2) -- (\i+120:2);
    }
\node[scale=0.75] at (90:1.7) {$\ast$};
\quiververtexC{90:2}{210:2}{-30:2}
\node[above=0.2em,scale=0.85] at (x121) {$1$};
\node[above=0.2em,scale=0.85] at (x122) {$1$};
\node[above=0.2em,scale=0.85] at (x312) {$0$};
\node[above=0.2em,scale=0.85] at (x311) {$0$};
\node[above=0.2em,scale=0.85] at (G2) {$1$};
\node[above=0.2em,scale=0.85] at (G1) {$1$};
\node[above=0.2em,scale=0.85] at (x231) {$1$};
\node[above=0.2em,scale=0.85] at (x232) {$1$};
\end{scope}}
{\begin{scope}[xshift=9cm]
\foreach \i in {90,210,330}
    {
    \markedpt{\i:2};
    \draw[blue] (\i:2) -- (\i+120:2);
    }
\node[scale=0.75] at (90:1.7) {$\ast$};
\quiververtexC{90:2}{210:2}{-30:2}
\begin{pgfonlayer}{bg} 
\fill[rbA!30] (x122) -- (x231) -- (G1) -- (x121) --cycle;
\fill[rbB!30] (x121) -- (G1) -- (G2) --cycle;
\fill[rbA!30] (G1) -- (G2) -- (x311) --cycle;
\fill[rbB!30] (G2) -- (x312) -- (x311) --cycle;
\end{pgfonlayer}
\node[scale=0.85] at ($(G1)!0.5!(x122)$) {$0$};
\CoG{G1}{G2}{x311};
\node[scale=0.85] at (G) {$0$};
\CoG{G1}{G2}{x121};
\node[scale=0.85] at (G) {$0$};
\CoG{x311}{x312}{G2};
\node[scale=0.85] at (G) {$0$};
\end{scope}}
{\begin{scope}[xshift=13.5cm]
\foreach \i in {90,210,330}
    {
    \markedpt{\i:2};
    \draw[blue] (\i:2) -- (\i+120:2);
    }
\node[scale=0.75] at (90:1.7) {$\ast$};
\quiververtexC{90:2}{210:2}{-30:2}
\node[above=0.2em,scale=0.85] at (G2) {$0$};
\node[above=0.2em,scale=0.85] at (G1) {$0$};
\end{scope}}
\end{scope}

\begin{scope}[yshift=-16cm]
\node at (-2,0.5) {(5)};
\foreach \i in {90,210,330}
    {
    \markedpt{\i:2};
    \draw[blue] (\i:2) -- (\i+120:2);
    }
\draw[webline] ($(-30:2)!0.4!(90:2)$) arc(120:180:0.8*1.732);
{\begin{scope}[xshift=4.5cm]
\foreach \i in {90,210,330}
    {
    \markedpt{\i:2};
    \draw[blue] (\i:2) -- (\i+120:2);
    }
\node[scale=0.75] at (90:1.7) {$\ast$};
\quiververtexC{90:2}{210:2}{-30:2}
\node[above=0.2em,scale=0.85] at (x121) {$0$};
\node[above=0.2em,scale=0.85] at (x122) {$0$};
\node[above=0.2em,scale=0.85] at (x312) {$1$};
\node[above=0.2em,scale=0.85] at (x311) {$\frac{1}{2}$};
\node[above=0.2em,scale=0.85] at (G2) {$1$};
\node[above=0.2em,scale=0.85] at (G1) {$\frac{1}{2}$};
\node[above=0.2em,scale=0.85] at (x231) {$\frac{1}{2}$};
\node[above=0.2em,scale=0.85] at (x232) {$1$};
\end{scope}}
{\begin{scope}[xshift=9cm]
\foreach \i in {90,210,330}
    {
    \markedpt{\i:2};
    \draw[blue] (\i:2) -- (\i+120:2);
    }
\node[scale=0.75] at (90:1.7) {$\ast$};
\quiververtexC{90:2}{210:2}{-30:2}
\begin{pgfonlayer}{bg} 
\fill[rbA!30] (x122) -- (x231) -- (G1) -- (x121) --cycle;
\fill[rbB!30] (x121) -- (G1) -- (G2) --cycle;
\fill[rbA!30] (G1) -- (G2) -- (x311) --cycle;
\fill[rbB!30] (G2) -- (x312) -- (x311) --cycle;
\end{pgfonlayer}
\node[scale=0.85] at ($(G1)!0.5!(x122)$) {$0$};
\CoG{G1}{G2}{x311};
\node[scale=0.85] at (G) {$0$};
\CoG{G1}{G2}{x121};
\node[scale=0.85] at (G) {$0$};
\CoG{x311}{x312}{G2};
\node[scale=0.85] at (G) {$0$};
\end{scope}}
{\begin{scope}[xshift=13.5cm]
\foreach \i in {90,210,330}
    {
    \markedpt{\i:2};
    \draw[blue] (\i:2) -- (\i+120:2);
    }
\node[scale=0.75] at (90:1.7) {$\ast$};
\quiververtexC{90:2}{210:2}{-30:2}
\node[above=0.2em,scale=0.85] at (G2) {$0$};
\node[above=0.2em,scale=0.85] at (G1) {$0$};
\end{scope}}
\end{scope}

\begin{scope}[yshift=-20cm]
\node at (-2,0.5) {(6)};
\foreach \i in {90,210,330}
    {
    \markedpt{\i:2};
    \draw[blue] (\i:2) -- (\i+120:2);
    }
\draw[wline] ($(-30:2)!0.4!(90:2)$) arc(120:180:0.8*1.732);
{\begin{scope}[xshift=4.5cm]
\foreach \i in {90,210,330}
    {
    \markedpt{\i:2};
    \draw[blue] (\i:2) -- (\i+120:2);
    }
\node[scale=0.75] at (90:1.7) {$\ast$};
\quiververtexC{90:2}{210:2}{-30:2}
\node[above=0.2em,scale=0.85] at (x121) {$0$};
\node[above=0.2em,scale=0.85] at (x122) {$0$};
\node[above=0.2em,scale=0.85] at (x312) {$1$};
\node[above=0.2em,scale=0.85] at (x311) {$1$};
\node[above=0.2em,scale=0.85] at (G2) {$2$};
\node[above=0.2em,scale=0.85] at (G1) {$1$};
\node[above=0.2em,scale=0.85] at (x231) {$1$};
\node[above=0.2em,scale=0.85] at (x232) {$1$};
\end{scope}}
{\begin{scope}[xshift=9cm]
\foreach \i in {90,210,330}
    {
    \markedpt{\i:2};
    \draw[blue] (\i:2) -- (\i+120:2);
    }
\node[scale=0.75] at (90:1.7) {$\ast$};
\quiververtexC{90:2}{210:2}{-30:2}
\begin{pgfonlayer}{bg} 
\fill[rbA!30] (x122) -- (x231) -- (G1) -- (x121) --cycle;
\fill[rbB!30] (x121) -- (G1) -- (G2) --cycle;
\fill[rbA!30] (G1) -- (G2) -- (x311) --cycle;
\fill[rbB!30] (G2) -- (x312) -- (x311) --cycle;
\end{pgfonlayer}
\node[scale=0.85] at ($(G1)!0.5!(x122)$) {$0$};
\CoG{G1}{G2}{x311};
\node[scale=0.85] at (G) {$0$};
\CoG{G1}{G2}{x121};
\node[scale=0.85] at (G) {$0$};
\CoG{x311}{x312}{G2};
\node[scale=0.85] at (G) {$0$};
\end{scope}}
{\begin{scope}[xshift=13.5cm]
\foreach \i in {90,210,330}
    {
    \markedpt{\i:2};
    \draw[blue] (\i:2) -- (\i+120:2);
    }
\node[scale=0.75] at (90:1.7) {$\ast$};
\quiververtexC{90:2}{210:2}{-30:2}
\node[above=0.2em,scale=0.85] at (G2) {$0$};
\node[above=0.2em,scale=0.85] at (G1) {$0$};
\end{scope}}
\end{scope}
\end{tikzpicture}
    \caption{Tropical coordinates of corner $\fsp_4$-webs on a triangle.}
    \label{fig:coordinate_corners}
\end{figure}

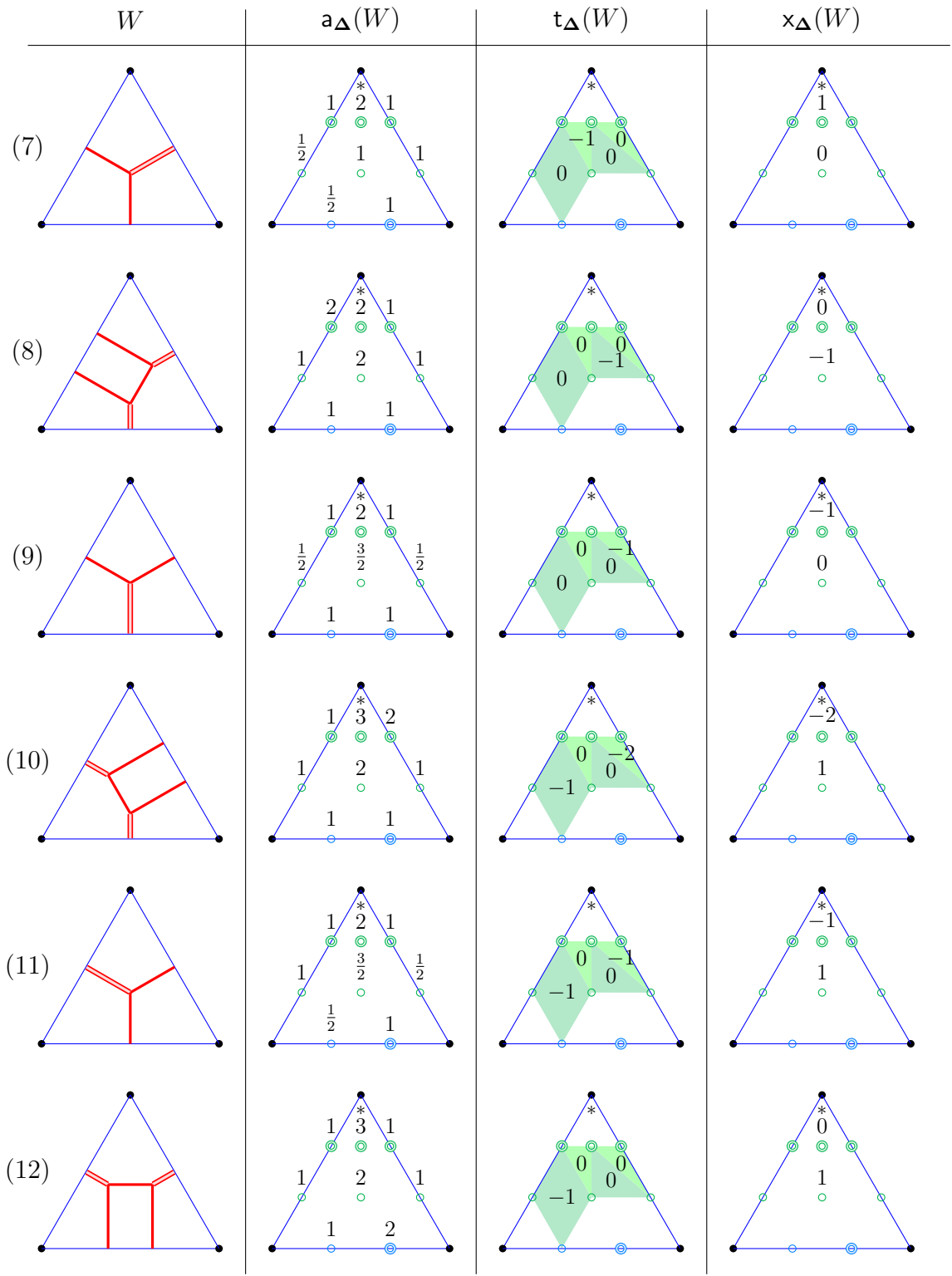
\begin{figure}[ht]
    \centering
\begin{tikzpicture}[scale=0.82]
\node at (-2,0.5) {(7)};
\node at (0,3) {$W$};
\node at (4.5,3) {$\mathsf{a}_{\bD}(W)$};
\node at (9,3) {$\sft_{\bD}(W)$};
\node at (13.5,3) {$\mathsf{x}_{\bD}(W)$};
\draw(2.25,3.2) -- (2.25,-21.5);
\draw(4.5+2.25,3.2) -- (4.5+2.25,-21.5);
\draw(9+2.25,3.2) -- (9+2.25,-21.5);
\draw(-2,2.5) -- (16,2.5);
\foreach \i in {90,210,330}
    {
    \markedpt{\i:2};
    \draw[blue] (\i:2) -- (\i+120:2);
    }
\draw[wline] (0,0) -- ($(90:2)!0.5!(-30:2)$);
\draw[webline] (0,0) -- ($(210:2)!0.5!(90:2)$);
\draw[webline] (0,0) -- ($(210:2)!0.5!(-30:2)$);
{\begin{scope}[xshift=4.5cm]
\foreach \i in {90,210,330}
    {
    \markedpt{\i:2};
    \draw[blue] (\i:2) -- (\i+120:2);
    }
\node[scale=0.75] at (90:1.7) {$\ast$};
\quiververtexC{90:2}{210:2}{-30:2}
\node[above=0.2em,scale=0.85] at (x121) {$1$};
\node[above=0.2em,scale=0.85] at (x122) {$\frac{1}{2}$};
\node[above=0.2em,scale=0.85] at (x312) {$1$};
\node[above=0.2em,scale=0.85] at (x311) {$1$};
\node[above=0.2em,scale=0.85] at (G2) {$2$};
\node[above=0.2em,scale=0.85] at (G1) {$1$};
\node[above=0.2em,scale=0.85] at (x231) {$\frac{1}{2}$};
\node[above=0.2em,scale=0.85] at (x232) {$1$};
\end{scope}}
{\begin{scope}[xshift=9cm]
\foreach \i in {90,210,330}
    {
    \markedpt{\i:2};
    \draw[blue] (\i:2) -- (\i+120:2);
    }
\node[scale=0.75] at (90:1.7) {$\ast$};
\quiververtexC{90:2}{210:2}{-30:2}
\begin{pgfonlayer}{bg} 
\fill[rbA!30] (x122) -- (x231) -- (G1) -- (x121) --cycle;
\fill[rbB!30] (x121) -- (G1) -- (G2) --cycle;
\fill[rbA!30] (G1) -- (G2) -- (x311) --cycle;
\fill[rbB!30] (G2) -- (x312) -- (x311) --cycle;
\end{pgfonlayer}
\node[scale=0.85] at ($(G1)!0.5!(x122)$) {$0$};
\CoG{G1}{G2}{x311};
\node[scale=0.85] at (G) {$0$};
\CoG{G1}{G2}{x121};
\node[scale=0.85] at (G) {$-1$};
\CoG{x311}{x312}{G2};
\node[scale=0.85] at (G) {$0$};
\end{scope}}
{\begin{scope}[xshift=13.5cm]
\foreach \i in {90,210,330}
    {
    \markedpt{\i:2};
    \draw[blue] (\i:2) -- (\i+120:2);
    }
\node[scale=0.75] at (90:1.7) {$\ast$};
\quiververtexC{90:2}{210:2}{-30:2}
\node[above=0.2em,scale=0.85] at (G2) {$1$};
\node[above=0.2em,scale=0.85] at (G1) {$0$};
\end{scope}}

\begin{scope}[yshift=-4cm]
\node at (-2,0.5) {(8)};
\foreach \i in {90,210,330}
    {
    \markedpt{\i:2};
    \draw[blue] (\i:2) -- (\i+120:2);
    }
\draw($(210:2)!0.5!(-30:2)$)++(0,0.5) coordinate(H1);
\draw($(90:2)!0.5!(-30:2)$)++(-150:0.5) coordinate(H2);
\draw[webline] (H1) -- (H2);
\draw[wline] (H1) -- ($(210:2)!0.5!(-30:2)$);
\draw[wline] (H2) -- ($(90:2)!0.5!(-30:2)$);
\draw[webline] (H1) --++(150:1.25);
\draw[webline] (H2) --++(150:1.25);
{\begin{scope}[xshift=4.5cm]
\foreach \i in {90,210,330}
    {
    \markedpt{\i:2};
    \draw[blue] (\i:2) -- (\i+120:2);
    }
\node[scale=0.75] at (90:1.7) {$\ast$};
\quiververtexC{90:2}{210:2}{-30:2}
\node[above=0.2em,scale=0.85] at (x121) {$2$};
\node[above=0.2em,scale=0.85] at (x122) {$1$};
\node[above=0.2em,scale=0.85] at (x312) {$1$};
\node[above=0.2em,scale=0.85] at (x311) {$1$};
\node[above=0.2em,scale=0.85] at (G2) {$2$};
\node[above=0.2em,scale=0.85] at (G1) {$2$};
\node[above=0.2em,scale=0.85] at (x231) {$1$};
\node[above=0.2em,scale=0.85] at (x232) {$1$};
\end{scope}}
{\begin{scope}[xshift=9cm]
\foreach \i in {90,210,330}
    {
    \markedpt{\i:2};
    \draw[blue] (\i:2) -- (\i+120:2);
    }
\node[scale=0.75] at (90:1.7) {$\ast$};
\quiververtexC{90:2}{210:2}{-30:2}
\begin{pgfonlayer}{bg} 
\fill[rbA!30] (x122) -- (x231) -- (G1) -- (x121) --cycle;
\fill[rbB!30] (x121) -- (G1) -- (G2) --cycle;
\fill[rbA!30] (G1) -- (G2) -- (x311) --cycle;
\fill[rbB!30] (G2) -- (x312) -- (x311) --cycle;
\end{pgfonlayer}
\node[scale=0.85] at ($(G1)!0.5!(x122)$) {$0$};
\CoG{G1}{G2}{x311};
\node[scale=0.85] at (G) {$-1$};
\CoG{G1}{G2}{x121};
\node[scale=0.85] at (G) {$0$};
\CoG{x311}{x312}{G2};
\node[scale=0.85] at (G) {$0$};
\end{scope}}
{\begin{scope}[xshift=13.5cm]
\foreach \i in {90,210,330}
    {
    \markedpt{\i:2};
    \draw[blue] (\i:2) -- (\i+120:2);
    }
\node[scale=0.75] at (90:1.7) {$\ast$};
\quiververtexC{90:2}{210:2}{-30:2}
\node[above=0.2em,scale=0.85] at (G2) {$0$};
\node[above=0.2em,scale=0.85] at (G1) {$-1$};
\end{scope}}
\end{scope}

\begin{scope}[yshift=-8cm]
\node at (-2,0.5) {(9)};
\foreach \i in {90,210,330}
    {
    \markedpt{\i:2};
    \draw[blue] (\i:2) -- (\i+120:2);
    }
\draw[webline] (0,0) -- ($(90:2)!0.5!(-30:2)$);
\draw[webline] (0,0) -- ($(210:2)!0.5!(90:2)$);
\draw[wline] (0,0) -- ($(210:2)!0.5!(-30:2)$);
{\begin{scope}[xshift=4.5cm]
\foreach \i in {90,210,330}
    {
    \markedpt{\i:2};
    \draw[blue] (\i:2) -- (\i+120:2);
    }
\node[scale=0.75] at (90:1.7) {$\ast$};
\quiververtexC{90:2}{210:2}{-30:2}
\node[above=0.2em,scale=0.85] at (x121) {$1$};
\node[above=0.2em,scale=0.85] at (x122) {$\frac{1}{2}$};
\node[above=0.2em,scale=0.85] at (x312) {$1$};
\node[above=0.2em,scale=0.85] at (x311) {$\frac{1}{2}$};
\node[above=0.2em,scale=0.85] at (G2) {$2$};
\node[above=0.2em,scale=0.85] at (G1) {$\frac{3}{2}$};
\node[above=0.2em,scale=0.85] at (x231) {$1$};
\node[above=0.2em,scale=0.85] at (x232) {$1$};
\end{scope}}
{\begin{scope}[xshift=9cm]
\foreach \i in {90,210,330}
    {
    \markedpt{\i:2};
    \draw[blue] (\i:2) -- (\i+120:2);
    }
\node[scale=0.75] at (90:1.7) {$\ast$};
\quiververtexC{90:2}{210:2}{-30:2}
\begin{pgfonlayer}{bg} 
\fill[rbA!30] (x122) -- (x231) -- (G1) -- (x121) --cycle;
\fill[rbB!30] (x121) -- (G1) -- (G2) --cycle;
\fill[rbA!30] (G1) -- (G2) -- (x311) --cycle;
\fill[rbB!30] (G2) -- (x312) -- (x311) --cycle;
\end{pgfonlayer}
\node[scale=0.85] at ($(G1)!0.5!(x122)$) {$0$};
\CoG{G1}{G2}{x311};
\node[scale=0.85] at (G) {$0$};
\CoG{G1}{G2}{x121};
\node[scale=0.85] at (G) {$0$};
\CoG{x311}{x312}{G2};
\node[scale=0.85] at (G) {$-1$};
\end{scope}}
{\begin{scope}[xshift=13.5cm]
\foreach \i in {90,210,330}
    {
    \markedpt{\i:2};
    \draw[blue] (\i:2) -- (\i+120:2);
    }
\node[scale=0.75] at (90:1.7) {$\ast$};
\quiververtexC{90:2}{210:2}{-30:2}
\node[above=0.2em,scale=0.85] at (G2) {$-1$};
\node[above=0.2em,scale=0.85] at (G1) {$0$};
\end{scope}}
\end{scope}

\begin{scope}[yshift=-12cm]
\node at (-2,0.5) {(10)};
\foreach \i in {90,210,330}
    {
    \markedpt{\i:2};
    \draw[blue] (\i:2) -- (\i+120:2);
    }
\draw($(210:2)!0.5!(-30:2)$)++(0,0.5) coordinate(H1);
\draw($(90:2)!0.5!(210:2)$)++(-30:0.5) coordinate(H2);
\draw[webline] (H1) -- (H2);
\draw[wline] (H1) -- ($(210:2)!0.5!(-30:2)$);
\draw[wline] (H2) -- ($(90:2)!0.5!(210:2)$);
\draw[webline] (H1) --++(30:1.25);
\draw[webline] (H2) --++(30:1.25);
{\begin{scope}[xshift=4.5cm]
\foreach \i in {90,210,330}
    {
    \markedpt{\i:2};
    \draw[blue] (\i:2) -- (\i+120:2);
    }
\node[scale=0.75] at (90:1.7) {$\ast$};
\quiververtexC{90:2}{210:2}{-30:2}
\node[above=0.2em,scale=0.85] at (x121) {$1$};
\node[above=0.2em,scale=0.85] at (x122) {$1$};
\node[above=0.2em,scale=0.85] at (x312) {$2$};
\node[above=0.2em,scale=0.85] at (x311) {$1$};
\node[above=0.2em,scale=0.85] at (G2) {$3$};
\node[above=0.2em,scale=0.85] at (G1) {$2$};
\node[above=0.2em,scale=0.85] at (x231) {$1$};
\node[above=0.2em,scale=0.85] at (x232) {$1$};
\end{scope}}
{\begin{scope}[xshift=9cm]
\foreach \i in {90,210,330}
    {
    \markedpt{\i:2};
    \draw[blue] (\i:2) -- (\i+120:2);
    }
\node[scale=0.75] at (90:1.7) {$\ast$};
\quiververtexC{90:2}{210:2}{-30:2}
\begin{pgfonlayer}{bg} 
\fill[rbA!30] (x122) -- (x231) -- (G1) -- (x121) --cycle;
\fill[rbB!30] (x121) -- (G1) -- (G2) --cycle;
\fill[rbA!30] (G1) -- (G2) -- (x311) --cycle;
\fill[rbB!30] (G2) -- (x312) -- (x311) --cycle;
\end{pgfonlayer}
\node[scale=0.85] at ($(G1)!0.5!(x122)$) {$-1$};
\CoG{G1}{G2}{x311};
\node[scale=0.85] at (G) {$0$};
\CoG{G1}{G2}{x121};
\node[scale=0.85] at (G) {$0$};
\CoG{x311}{x312}{G2};
\node[scale=0.85] at (G) {$-2$};
\end{scope}}
{\begin{scope}[xshift=13.5cm]
\foreach \i in {90,210,330}
    {
    \markedpt{\i:2};
    \draw[blue] (\i:2) -- (\i+120:2);
    }
\node[scale=0.75] at (90:1.7) {$\ast$};
\quiververtexC{90:2}{210:2}{-30:2}
\node[above=0.2em,scale=0.85] at (G2) {$-2$};
\node[above=0.2em,scale=0.85] at (G1) {$1$};
\end{scope}}
\end{scope}

\begin{scope}[yshift=-16cm]
\node at (-2,0.5) {(11)};
\foreach \i in {90,210,330}
    {
    \markedpt{\i:2};
    \draw[blue] (\i:2) -- (\i+120:2);
    }
\draw[webline] (0,0) -- ($(90:2)!0.5!(-30:2)$);
\draw[wline] (0,0) -- ($(210:2)!0.5!(90:2)$);
\draw[webline] (0,0) -- ($(210:2)!0.5!(-30:2)$);
{\begin{scope}[xshift=4.5cm]
\foreach \i in {90,210,330}
    {
    \markedpt{\i:2};
    \draw[blue] (\i:2) -- (\i+120:2);
    }
\node[scale=0.75] at (90:1.7) {$\ast$};
\quiververtexC{90:2}{210:2}{-30:2}
\node[above=0.2em,scale=0.85] at (x121) {$1$};
\node[above=0.2em,scale=0.85] at (x122) {$1$};
\node[above=0.2em,scale=0.85] at (x312) {$1$};
\node[above=0.2em,scale=0.85] at (x311) {$\frac{1}{2}$};
\node[above=0.2em,scale=0.85] at (G2) {$2$};
\node[above=0.2em,scale=0.85] at (G1) {$\frac{3}{2}$};
\node[above=0.2em,scale=0.85] at (x231) {$\frac{1}{2}$};
\node[above=0.2em,scale=0.85] at (x232) {$1$};
\end{scope}}
{\begin{scope}[xshift=9cm]
\foreach \i in {90,210,330}
    {
    \markedpt{\i:2};
    \draw[blue] (\i:2) -- (\i+120:2);
    }
\node[scale=0.75] at (90:1.7) {$\ast$};
\quiververtexC{90:2}{210:2}{-30:2}
\begin{pgfonlayer}{bg} 
\fill[rbA!30] (x122) -- (x231) -- (G1) -- (x121) --cycle;
\fill[rbB!30] (x121) -- (G1) -- (G2) --cycle;
\fill[rbA!30] (G1) -- (G2) -- (x311) --cycle;
\fill[rbB!30] (G2) -- (x312) -- (x311) --cycle;
\end{pgfonlayer}
\node[scale=0.85] at ($(G1)!0.5!(x122)$) {$-1$};
\CoG{G1}{G2}{x311};
\node[scale=0.85] at (G) {$0$};
\CoG{G1}{G2}{x121};
\node[scale=0.85] at (G) {$0$};
\CoG{x311}{x312}{G2};
\node[scale=0.85] at (G) {$-1$};
\end{scope}}
{\begin{scope}[xshift=13.5cm]
\foreach \i in {90,210,330}
    {
    \markedpt{\i:2};
    \draw[blue] (\i:2) -- (\i+120:2);
    }
\node[scale=0.75] at (90:1.7) {$\ast$};
\quiververtexC{90:2}{210:2}{-30:2}
\node[above=0.2em,scale=0.85] at (G2) {$-1$};
\node[above=0.2em,scale=0.85] at (G1) {$1$};
\end{scope}}
\end{scope}

\begin{scope}[yshift=-20cm]
\node at (-2,0.5) {(12)};
\foreach \i in {90,210,330}
    {
    \markedpt{\i:2};
    \draw[blue] (\i:2) -- (\i+120:2);
    }
\draw($(210:2)!0.5!(90:2)$)++(-30:0.5) coordinate(H1);
\draw($(90:2)!0.5!(-30:2)$)++(-150:0.5) coordinate(H2);
\draw[webline] (H1) -- (H2);
\draw[wline] (H1) -- ($(210:2)!0.5!(90:2)$);
\draw[wline] (H2) -- ($(-30:2)!0.5!(90:2)$);
\draw[webline] (H1) --++(-90:1.25);
\draw[webline] (H2) --++(-90:1.25);
{\begin{scope}[xshift=4.5cm]
\foreach \i in {90,210,330}
    {
    \markedpt{\i:2};
    \draw[blue] (\i:2) -- (\i+120:2);
    }
\node[scale=0.75] at (90:1.7) {$\ast$};
\quiververtexC{90:2}{210:2}{-30:2}
\node[above=0.2em,scale=0.85] at (x121) {$1$};
\node[above=0.2em,scale=0.85] at (x122) {$1$};
\node[above=0.2em,scale=0.85] at (x312) {$1$};
\node[above=0.2em,scale=0.85] at (x311) {$1$};
\node[above=0.2em,scale=0.85] at (G2) {$3$};
\node[above=0.2em,scale=0.85] at (G1) {$2$};
\node[above=0.2em,scale=0.85] at (x231) {$1$};
\node[above=0.2em,scale=0.85] at (x232) {$2$};
\end{scope}}
{\begin{scope}[xshift=9cm]
\foreach \i in {90,210,330}
    {
    \markedpt{\i:2};
    \draw[blue] (\i:2) -- (\i+120:2);
    }
\node[scale=0.75] at (90:1.7) {$\ast$};
\quiververtexC{90:2}{210:2}{-30:2}
\begin{pgfonlayer}{bg} 
\fill[rbA!30] (x122) -- (x231) -- (G1) -- (x121) --cycle;
\fill[rbB!30] (x121) -- (G1) -- (G2) --cycle;
\fill[rbA!30] (G1) -- (G2) -- (x311) --cycle;
\fill[rbB!30] (G2) -- (x312) -- (x311) --cycle;
\end{pgfonlayer}
\node[scale=0.85] at ($(G1)!0.5!(x122)$) {$-1$};
\CoG{G1}{G2}{x311};
\node[scale=0.85] at (G) {$0$};
\CoG{G1}{G2}{x121};
\node[scale=0.85] at (G) {$0$};
\CoG{x311}{x312}{G2};
\node[scale=0.85] at (G) {$0$};
\end{scope}}
{\begin{scope}[xshift=13.5cm]
\foreach \i in {90,210,330}
    {
    \markedpt{\i:2};
    \draw[blue] (\i:2) -- (\i+120:2);
    }
\node[scale=0.75] at (90:1.7) {$\ast$};
\quiververtexC{90:2}{210:2}{-30:2}
\node[above=0.2em,scale=0.85] at (G2) {$0$};
\node[above=0.2em,scale=0.85] at (G1) {$1$};
\end{scope}}
\end{scope}
\end{tikzpicture}
    \caption{Tropical coordinates of some $\fsp_4$-webs on a triangle.}
    \label{fig:coordinates_elementary}
\end{figure}

\begin{thm}\label{lem:triangle_case}
Let $\bSigma=T$ be a triangle. Then:
\begin{enumerate}
    \item All the tropical seeds $(\ve^{\bD},\bsfa^{\bD})$ are mutation-equivalent to each other. 
    \item The coordinate system $\bsfa^{\bD}: \cL^a(T,\bQ) \to \bQ^8$ gives a bijection for any decorated triangulation $\bD$.     
\end{enumerate}
In particular, $\cL^a(T,\bQ)$ is identified with the tropical cluster variety $\A_{\fsp_4,T}(\bQ^\sfT)$. 
\end{thm}
Recall from \cref{fig:exch_triangle} that we have six decorated triangulations of $T$, according to the three choices of a marked point and two choices of a sign. 
Let $m,m',m''$ be the three marked points of $T$ in this counter-clockwise order, where $m$ is shown as the top vertex of $T$ in figures. Let $\rho$ be the cyclic rotation of $T$ such that $m \mapsto m' \mapsto m''$, and $r$ the reflection of $T$ that fixes $m$ and interchanges the other two marked points. Then $\rho$ and $r$ generate a dihedral symmetry group of order $6$. 

\subsubsection{Hilbert basis}
The following consideration is useful for the proof of \cref{lem:triangle_case}. 
In \cref{fig:coordinate_corners,fig:coordinates_elementary}, the coordinates of 12 typical diagrams $W_1,\dots,W_{12}$ on $T$ associated with $\bD=(\tri,m,+)$ are shown. 
The coordinates for $(\tri,m',+)$ and $(\tri,m'',+)$ are obtained by the rotation symmetry $\rho$ in the obvious way.  The coordinates for the $-$ sign are obtained by the reflection symmetry $r$, as follows:
\begin{lem}\label{lem:reflection_coord}
Then for any $L \in \Blad_{\fsp_4,T}$, we have the following:
\begin{enumerate}
    \item $\sfa_i^{(\tri,m,-)}(L)=\sfa_i^{(\tri,m,+)}(L)$ for $i=3,\dots,8$;
    \item $\sfa_i^{(\tri,m,-)}(L)=\sfa_i^{(\tri,m,+)}(r(L))$ for $i=1,2$.
\end{enumerate}
Here we use the labeling of coordinates in the top of \cref{fig:labeling_triangle_case}. The other coordinates for $m',m''$ are related by the corresponding reflections $r':=\rho r \rho^{-1},r'':=\rho^{-1} r \rho$. 
\end{lem}

\begin{proof}
The first assertion is clear, since the $\fso_5$-elementary webs $V_3,\dots,V_8$ are shared by $(\tri,m,+)$ and $(\tri,m,-)$. For the second assertion, observe from that the $\fso_5$-elementary webs $V^-_1,V^-_2$ associated to $(\tri,m,-)$ are the images of those $V^+_1,V^+_2$ associated to $(\tri,m,+)$ under the reflection $r$. Moreover, the intersection pairing is invariant under $r$. Therefore, 
\begin{align*}
    \sfa_i^{(\tri,m,-)}(L) = \bi_{\bSigma}(L,V^-_i) = \bi_{\bSigma}(L,r(V^+_i))) = \bi_{\bSigma}(r(L),V^+_i) = \sfa_i^{(\tri,m,+)}(r(L)) 
\end{align*}
as desired.
\end{proof}

\begin{lem}\label{lem:Hilbert_basis}
The coordinate vectors $\bsfa^\bD(W_\lambda)$ for $\lambda=1,\dots,12$ form the Hilbert basis $\mathscr{H}_{\fsp_4,T}^{\bD}:=\mathscr{H}_{\fsp_4,T}^{\cC_\bD}$ for any decorated triangulation $\bD$.
\end{lem}

\begin{proof}
We take $\bD=(\tri,m,+)$, and the other cases follow by symmetry. 
Recall that any ladder-class in $\Blad_{\fsp_4,T}$  has a representative as shown in \cref{fig:pyramid}, up to the dihedral symmetry. It is easily verified that its coordinate vector is 
\begin{align}\label{eq:Hilbert_expansion}
    \begin{aligned}
        k_1\bsfa^{\bD}(W_{1})+ l_1 \bsfa^{\bD}(W_{2})+ k_2 \bsfa^{\bD}(W_{3})+ l_2\bsfa^{\bD}(W_{4}) + &k_3 \bsfa^{\bD}(W_{5}) + l_3 \bsfa^{\bD}(W_{6}) \\
    +&n_1 \bsfa^{\bD}(W_{10})+n_2\bsfa^{\bD}(W_{9}).
    \end{aligned}
\end{align}
Its reflected/rotated versions have similar coordinate vectors, where the latter two terms are replaced with other two from \cref{fig:coordinates_elementary}. Therefore the assertion follows. 
\end{proof}
The structure of the Hilbert basis is best viewed in the plane of tropical $\X$-coordinates $(\sfx_1^{\bD},\sfx_2^{\bD})$. \cref{fig:cone_Hilbert} shows the vectors $v_\lambda:=(\sfx_1^{\bD}(W_{\lambda}),\sfx_2^{\bD}(W_{\lambda}))$ for $\lambda=7,\dots,12$. They divide the coordinate plane into six chambers. Any element in $\Blad_{\fsp_4,T}$ belongs to one of them, corresponding to its $3\times 2$ dihedral symmetry. 

\begin{figure}[ht]
    \centering
\begin{tikzpicture}
\draw[->](-3,0) -- (3,0) node[right]{$\sfx_1$};
\draw[->](0,-3) -- (0,3) node[above]{$\sfx_2$};
\draw (0,0) -- (3,-3);
\draw (0,0) -- (1.5,-3);
\begin{scope}[>=latex]
    \draw[->,thick] (0,0) -- (0,1) node[right]{$v_7$};
    \draw[->,thick] (0,0) -- (-1,0) node[above left]{$v_8$};
    \draw[->,thick] (0,0) -- (0,-1) node[left]{$v_9$};
    \draw[->,thick] (0,0) -- (1,-2) node[left]{$v_{10}$};
    \draw[->,thick] (0,0) -- (1,-1) node[above right]{$v_{11}$};
    \draw[->,thick] (0,0) -- (1,0) node[above]{$v_{12}$};
\end{scope}
\end{tikzpicture}
    \caption{The tropical $\X$-coordinates of the Hilbert basis elements $W_{7},\dots,W_{12}$. }
    \label{fig:cone_Hilbert}
\end{figure}
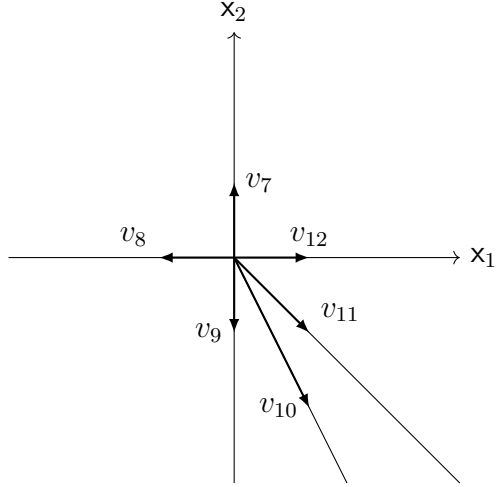

\subsubsection{Mutation-equivalence}

The following is easily verified by inspection into \cref{fig:coordinate_corners,fig:coordinates_elementary}:
\begin{lem}\label{lem:triangle_sign_coherence}
The typical diagrams on $T$ are sign-coherent (\cref{def:sign_coherence}) with each other for the two unfrozen coordinates $\sfx_i^{\bD}$ in any decorated triangulation $\bD$ of $T$. 
\end{lem}

\begin{proof}[Proof of \cref{lem:triangle_case} (1)]
We take $\bD=(\tri,m,+)$ for the initial coordinates. We label the associated coordinates $\bsfa^{\bD}=(\sfa_i)_{i=1}^8$ as shown in the top-left of \cref{fig:labeling_triangle_case}. 
By symmetry, we only need to verify the two mutation formulae from $\bD$:
\begin{align*}
    \mu_1: \quad &\sfa_1^{(\tri,m',-)}(L) + \sfa_1(L)= \max\{ \sfa_2+ \sfa_3, \sfa_4 + \sfa_5 + \sfa_7\}(L), \\
    \mu_2: \quad &\sfa_2^{(\tri,m'',-)}(L) + \sfa_2(L)= \max\{ \sfa_4 + \sfa_6 + 2\sfa_7 , 2\sfa_1 + \sfa_8\}(L),
\end{align*}
where we use the labeling in \cref{fig:labeling_triangle_case}. By \cref{lem:reflection_coord}, these equations are equivalent to
\begin{align}\label{eq:mutation_check}
    \begin{aligned}
    \mu_1: \quad &\sfa_1(r\rho^{-1}(L)) + \sfa_1(L)= \max\{ \sfa_2+ \sfa_3, \sfa_4 + \sfa_5 + \sfa_7\}(L), \\
    \mu_2: \quad &\sfa_2(r\rho(L)) + \sfa_2(L)= \max\{ \sfa_4 + \sfa_6 + 2\sfa_7 , 2\sfa_1 + \sfa_8\}(L).
    \end{aligned}
\end{align}

By \cref{lem:Hilbert_basis}, the coordinate vector of any rational bounded $\fsp_4$-lamination $L$ is of the form \eqref{eq:Hilbert_expansion}, where we now allow the coefficients to be rational, and those in the first row to be negative. 
Observe that the negative coordinates do not affect on the sign-coherence of $L$ with $W_\lambda$, since the $\sfx$-coordinates of $W_\lambda$ with $1 \leq \lambda \leq 6$ are zero. 
Then by \cref{lem:triangle_sign_coherence}, \cref{lem:sign-coherence} (and \cref{rem:sign_coherence}), it suffices to check that the mutation formulae \eqref{eq:mutation_check} holds for each $W_\lambda$. They are easily verified from \cref{fig:coordinate_corners,fig:coordinates_elementary}. 
\end{proof}

\subsubsection{Reconstruction from coordinates}
By \cref{lem:triangle_case} (1), it suffices to show that one of the six coordinate systems, say for $\bD=(\tri,m,+)$, is a bijection. We label the associated coordinates $\bsfa^{\bD}=(\sfa_i)_{i=1}^8$ as shown in the top-left of \cref{fig:labeling_triangle_case}. 

Let us also consider $\bD'=(\tri,m',+)$ and $\bD''=(\tri,m'',+)$, and the associated coordinate systems $\sfa^{\bD'}=(\sfa'_i)_{i=1}^8$, $\sfa_{\bD''}=(\sfa''_i)_{i=1}^8$, together with the same labeling rule for their respective different directions. For example, $\sfa_3=\sfa'_7=\sfa''_5$, $\sfa_4=\sfa'_8=\sfa''_6$ for the frozen variables. \cref{lem:triangle_case} (1) tells us that the coordinates $(\sfa'_1,\sfa'_2,\sfa''_1,\sfa''_2)$ are PL functions of $(\sfa_i)_{i=1}^8$. 
We also write $\pot_s:=\pot_{m,s}$, $\pot'_s:=\pot_{m',s}$, $\pot''_s:=\pot_{m'',s}$ for $s=1,2$.  

Recall the tropical $\X$-coordinates
\begin{align}
    \begin{aligned}
    &\sfx_2 = \sfa_4+\sfa_6+2\sfa_7-2\sfa_1-\sfa_8, \\
    &\sfx_1 = \sfa_2+\sfa_3-\sfa_4-\sfa_5-\sfa_7. \label{eq:trop_X_1}
    \end{aligned}
\end{align}
It is also useful to introduce the following tropical frozen $\X$-coordinates (cf. \cite[Proposition 13.4]{GS19}):

\begin{align}
    \begin{aligned}
    &\sfx_3 := \sfa_4 + \sfa_5 - \sfa_1 - \sfa_3, \\
    &\sfx_4 := 2\sfa_1 -\sfa_2 -\sfa_4, \\
    &\sfx_5 := \sfa_1 -\sfa_3 -\sfa_5, \\
    &\sfx_6 := 2\sfa_5 + \sfa_8 - \sfa_2 -\sfa_6, \\
    &\sfx_7 := \sfa_1 + \sfa_8 - \sfa_2 -\sfa_7, \\
    &\sfx_8 := \sfa_2 - \sfa_6 - \sfa_8. \label{eq:trop_X_2}
    \end{aligned}
\end{align}

\begin{lem}\label{lem:p-bijective}
The linear map $p: \bQ^8 \to \bQ^8$ given by $(\sfa_i)_{i=1}^8 \mapsto (\sfx_i)_{i=1}^8$ is bijective. 
\end{lem}

\begin{proof}
The matrix presentation of the map $p$ is given by
\begin{align*}
p = 
\begin{pmatrix}
0 & 1 & 1 & -1 & -1 & 0 & -1 & 0 \\
-2 & 0 & 0 & 1 & 2 & -1 & 0 & 1 \\
-1 & 0 & 1 & 0 & 0 & 0 & 1 & 0 \\
2 & -1 & -2 & 1 & 0 & 0 & 0 & 0 \\
1 & -1 & 0 & 0 & 1 & 0 & 0 & 0 \\
0 & 1 & 0 & 0 & -2 & 1 & 0 & -1 \\
1 & 0 & -1 & 0 & 0 & 0 & 1 & -1 \\
0 & -1 & 0 & 0 & 0 & 1 & 0 & 1
\end{pmatrix}.
\end{align*}
It is invertible over $\bQ$, whose inverse matrix is 
\begin{align*}
p^{-1} = 
-\begin{pmatrix}
2 & \frac{3}{2} & 1 & \frac{1}{2} & 1 & \frac{1}{2} & 1 & 1 \\
2 & 2 & 1 & 1 & 2 & 1 & 1 & 1 \\
1 & \frac{1}{2} & 1 & \frac{1}{2} & 0 & 0 & 1 & \frac{1}{2} \\
2 & 1 & 1 & 1 & 0 & 0 & 1 & 1 \\
1 & \frac{1}{2} & 1 & \frac{1}{2} & 1 & \frac{1}{2} & 0 & 0 \\
1 & 1 & 1 & 1 & 1 & 1 & 0 & 0 \\
1 & 1 & 0 & 0 & 1 & \frac{1}{2} & 1 & \frac{1}{2} \\
1 & 1 & 0 & 0 & 1 & 1 & 1 & 1
\end{pmatrix}.
\end{align*}    
\end{proof}

\begin{lem}\label{lem:potential_coordinates}
We have
\begin{align*}
    \pot_{2} =& \min\{ \sfa_1+\sfa_3-\sfa_4-\sfa_5,\ \sfa_1+\sfa_7-\sfa_2\}, \\
    \pot_{1} =& \min\{ \sfa_2+\sfa_4-2\sfa_1,\ \sfa_2+\sfa_8-\sfa_6-2\sfa_7\}, \\
    \pot'_{2} =& \sfa_3+\sfa_5-\sfa_1, \\
    \pot'_{1} =& \min\{ \sfa_2+\sfa_6-2\sfa_5-\sfa_8,\ 2\sfa_1+\sfa_2 -\sfa_4-2\sfa_5-2\sfa_7, \\
    &\qquad\qquad 2\sfa_1-\sfa_3-\sfa_5-\sfa_7,\ 2\sfa_1+\sfa_4-\sfa_2-2\sfa_3\}, \\
    \pot''_{2} =& \min\{ \sfa_2+\sfa_7-\sfa_1-\sfa_8,\ \sfa_1+\sfa_2-\sfa_4-\sfa_6-\sfa_7,\ \sfa_1+\sfa_5-\sfa_3-\sfa_6\}, \\
    \pot''_{1} =& \sfa_6+\sfa_8-\sfa_2
\end{align*}
as functions on $\cL^a(T,\bQ)$. 
\end{lem}


\begin{proof} 
In terms of the tropical $\X$-coordinates, 
the asserted equations are rewritten as 
\begin{align}
    \begin{aligned}
    \pot_{2} =& -\sfx_3 + \min\{ 0,-\sfx_1\}, \\
    \pot_{1} =& -\sfx_4 + \min\{ 0, -\sfx_2 \}, \\
    \pot'_{2} =& -\sfx_5, \\
    \pot'_{1} =& -\sfx_6 + \min\{ 0, -\sfx_2, -(\sfx_1+\sfx_2), -(2\sfx_1+\sfx_2)\}, \\
    \pot''_{2} =& -\sfx_7 + \min\{0, -\sfx_2, - (\sfx_1+\sfx_2)\}, \\
    \pot''_{1} =& -\sfx_8.   \label{eq:potential_X} 
    \end{aligned} 
\end{align}
In particular, they are linear on each of the six chambers in \cref{fig:cone_Hilbert}. Therefore it suffices to check the equations for each element in the Hilbert basis. Then they are easily verified from \cref{fig:coordinate_corners,fig:coordinates_elementary}.
\end{proof}

\begin{prop}\label{prop:a-coord_shape}
The PL map $\bQ^8 \to \bQ^8$, $(\sfa_i)_{i=1}^8 \mapsto (\sfx_1,\sfx_2,\pot_1,\pot_2,\pot'_1,\pot'_2,\pot''_1,\pot''_2)$ is bijective. Here the $\pot$-functions are regarded as PL functions of $(\sfa_i)$ by the expressions in \cref{lem:potential_coordinates}. 
\end{prop}

\begin{proof}
Let $(\sfx_1,\sfx_2,\pot_1,\pot_2,\pot'_1,\pot'_2,\pot''_1,\pot''_2) \in \bQ^8$ be given. We first specify one of the chambers in \cref{fig:cone_Hilbert} which $(\sfx_1,\sfx_2)$ belongs to. Then by \eqref{eq:potential_X}, we know that the $\pot$-functions are linear functions of $(\sfx_i)_{i=1}^8$. From these expressions, we can read off the values of $(\sfx_i)_{i=1}^8$. Then by \cref{lem:p-bijective}, we get the tropical $\A$-coordinates. 
\end{proof}

\begin{proof}[Proof of \cref{lem:triangle_case} (2)]
Given a coordinate tuple $(\sfa_i)_{i=1}^8 \in \bQ^8$, we can bijectively transform it into the coordinate tuple $(\sfx_1,\sfx_2,\pot_1,\pot_2,\pot'_1,\pot'_2,\pot''_1,\pot''_2) \in \bQ^8$ by \cref{prop:a-coord_shape}. By \cref{lem:potential_coordinates}, the six $\pot$-functions uniquely determine the parameters $(k_i,l_i)_{i=1}^3$ in \cref{fig:pyramid}. Moreover, the two parameters $(\sfx_1,\sfx_2)$ uniquely determine the direction of the honeycomb and its parameters $(n_1,n_2)$. For example, if $(\sfx_1,\sfx_2)$ lies in the chamber bounded by the rays of $v_9$ and $v_{10}$, then $\sfx_1 \geq$ and $2\sfx_1+\sfx_2 \leq 0$ and the corresponding diagram is in the direction of \cref{fig:pyramid}, and the explicit relations are
\begin{align*}
    \sfx_1 &= n_1, \quad \sfx_2 = -n_2-2n_1, \\
    \pot_1 &= k_1, \quad \pot_2 = l_1, \quad \pot'_1 = k_2, \quad \pot'_2 = l_2, \quad \pot''_1 = k_3, \quad \pot''_2 = l_3.
\end{align*}
The other cases are simiar. 
Thus the assertion is proved. 
\end{proof}

\subsubsection{Integral points}
For half-integers $x,y \in \frac 1 2\bZ$, let us write $x \equiv y$ if $x-y \in \bZ$. 
\begin{thm}\label{lem:congruence}
In terms of the coordinate systems associated with $\bD=(\tri,m,+)$, the subset $\cL^a(T,\bZ) \subset \cL^a(T,\bQ)$ of integral laminations is characterized either by 
\begin{description}
    \item[Integrality condition] $\sfx_i \in \bZ$ for $i=1,\dots,8$, or
    \item[Congruence condition] $\sfa_1 \equiv \sfa_5$, $\sfa_2 \equiv \sfa_4 \equiv \sfa_6 \equiv \sfa_8 \equiv 0$, $\sfa_1+\sfa_3+\sfa_7 \equiv 0$.
\end{description}
In particular, the subset $\Blad_{\fsp_4,T} \subset \cL^a(T,\bZ)$ is characterized by these conditions together with
\begin{description}
    \item[Potential condition] $\pot_{s},\pot'_s,\pot''_s\geq 0$ for $s=1,2$.
\end{description} 
\end{thm}

\begin{proof}
Recall that any integral $\fsp_4$-lamination has a representative as shown in \cref{fig:pyramid}. Moreover, the parameters $(n_1,n_2,k_1,l_1,k_2,l_2,k_3,l_3) \in \bZ_{\geq 0}^2 \times \bZ^4$ corresponds to $(\sfx_1,\sfx_2,\pot_1,\pot_2,\pot'_1,\pot'_2,\pot''_1,\pot''_2) \in \bZ^6$, the latter lying in one of the six chambers in \cref{fig:cone_Hilbert}. The relation \eqref{eq:potential_X} shows that the parameters  $(\sfx_1,\sfx_2,\pot_1,\pot_2,\pot'_1,\pot'_2,\pot''_1,\pot''_2)$ are integral if and only if $(\sfx_i)_{i=1}^8$ is integral. Therefore we get the first characterization. 

Using \eqref{eq:trop_X_1} and \eqref{eq:trop_X_2}, 
it is straightforward to verify the equivalence between the integrality and the congruence relations. 
\end{proof}

\begin{rem}
The integrality and congruence conditions are also equivalent to the condition that each term in the potentials $\pot_s,\pot_s',\pot_s''$ is an integer. 
\end{rem}

\begin{dfn}\label{def:Blad_vect}
\begin{enumerate}
    \item Let $\vect_{T,\bD} \subset (\frac 1 2 \bZ)^8$ denote the convex subset of half-integral vectors satisfying the three conditions in \cref{lem:congruence}, so that we have the coordinate bijection
\begin{align}
    \bsfa^{\bD}: \Blad_{\fsp_4,T} \xrightarrow{\sim} \vect_{T,\bD}. 
\end{align}
    \item Let $\widetilde{\vect}_{T,\bD} \subset (\frac 1 2 \bZ)^8$ denote the linear subset defined only by the Integrality condition and the Congruence condition, so that we have the coordinate bijection $\cL^a(T,\bZ) \xrightarrow{\sim} \widetilde{\vect}_{T,\bD}$. 
\end{enumerate}
\end{dfn}

\begin{cor}
We have the following diagram:
\begin{equation*}
\begin{tikzcd}
    & \cL^a(T,\bQ) \ar[r,"\bsfa^\bD"] & \bQ^8 \\
\Blad_{\fsp_4,T} \ar[r,phantom,"\subset"] & \cL^a(T,\bZ) \ar[u,phantom,"\rotatebox{90}{$\subset$}"] \ar[r] &   \widetilde{\vect}_{T,\bD} \ar[u,phantom,"\rotatebox{90}{$\subset$}"] & \vect_{T,\bD} \ar[l,phantom,"\supset"]\\
     & \cL^a(T,\bZ)_\congr \ar[u,phantom,"\rotatebox{90}{$\subset$}"] \ar[r] & \bZ^8. \ar[u,phantom,"\rotatebox{90}{$\subset$}"]
\end{tikzcd}
\end{equation*}
Here, the horizontal arrows are bijections. In particular, $\cL^a(T,\bZ)_\congr$ is identified with the tropical cluster variety $\A_{\fsp_4,T}(\bZ^\sfT)$.
\end{cor}

\begin{rem}\label{rem:potential_vector}
Let
\begin{align*}
    \mathsf{v}_{m,1} :=& \textstyle(\frac 1 2, 1, \frac 1 2, 1, 0, 0, \frac 1 2, 1), \\
    \mathsf{v}_{m,2} :=& \textstyle(1,1,1,1,0,0,1,1), \\
    \mathsf{v}_{m',1} :=& \textstyle(1, 1, \frac 1 2, 1, \frac 1 2, 1, 0, 0), \\
    \mathsf{v}_{m',2} :=& \textstyle(1,1,1,1,1,1,0,0), \\
    \mathsf{v}_{m'',1} :=& \textstyle(\frac 1 2, 1, 0, 0, \frac 1 2, 1, \frac 1 2, 1), \\
    \mathsf{v}_{m'',2} :=& \textstyle(1,2,0,0,1,1,1,1)
\end{align*}
be the coordinate vectors of the corner arcs listed in \cref{fig:coordinate_corners}. Then by definition of the peripheral action $\alpha_\bM$ (\cref{subsec:peripheral_action}), we have 
\begin{align*}
    \bsfa^{\bD}(\alpha_\bM(u,L))) = \bsfa^{\bD}(L) +  \sum_{s=1,2} (u_{m}^s \mathsf{v}_{m}^s + u_{m'}^s \mathsf{v}_{m'}^s + u_{m''}^s \mathsf{v}_{m''}^s)
\end{align*}
for $u=(u_m^s,u_{m'}^s,u_{m''}^s) \in (\bQ^2)^3$ and $L \in \cL^a(T,\bQ)$. Observe that for a linear combination coming from the coroot lattice $\mathsf{Q}^\vee$ (\cref{rem:Cartan_action}), for example $-2\mathsf{v}_m^1 + 2\mathsf{v}_m^2$ and $2\mathsf{v}_m^1 -\mathsf{v}_m^2$, are integral vectors.
\end{rem}






\section{Reconsruction via graded skein algebras}\label{sec:reconstruction}

\subsection{Filtered and graded modules}
We note some facts on a filtered $\cR$-module and its associated graded $\cR$-module over a commutative ring $\cR$.
Let $M=\bigcup_{a\in\bZ_{\geq 0}}\cF_{a}M$ be a filtered $\cR$-algebra and $\cG M=\bigoplus_{a\in\bZ_{\geq 0}}\cG_{a}M$ the associated graded algebra of $M$, where $\cG_{a}M\coloneqq\cF_{a}M/\cF_{a-1}M$ and $\cF_{-1}M=\{0\}$.
We call the projection $\pi_{a}\colon\cF_{a}M\to\cG_{a}M$ the \emph{$a$-th symbol map} and define a map $\pi\colon M\to\cG M$ by $\pi(u)=\pi_{a}(u)$ for $a\in\cF_{a}M\setminus\cF_{a-1}M$.
We remark that $\pi$ is not an additive group homomorphism on $M$.
\begin{lem}\label{lem:fil-gr-isom}
    If $M$ is a free $\cR$-module with basis $X=\bigsqcup_{a\in\bZ_{\geq 0}} X_a$ such that $\bigsqcup_{a=0}^{n} X_a$ gives a basis of $\cF_{n}M$ for each $n \in \bZ_{\geq 0}$, then there exists a canonical $\cR$-module isomorphism $\eta^{X}\colon\cG M \to M$ defined by $\eta^{X}(\pi(x))=x$ for any $x\in X$ and $a\in\bZ_{\geq 0}$.
\end{lem}
\begin{proof}
    There exists an $\cR$-module isomorphism $\cG_{a}M=\cF_{a}M/\cF_{a-1}M\cong\mathrm{span}_{\cR}X_a$.
    We define a section $\eta^{X}_{a}\colon\cG_{a}M \to \cF_{a}M$ through this isomorphism as $\eta_{a}^{X}(\pi_{a}(x))=x$ for any $x\in X_{a}$.
    Then $\eta^{X}(v)=\eta^{X}_{a}(v)$ for $v\in\cG_{a}M$ gives an $\cR$-module isomorphism.
\end{proof}
We remark that this isomorphism depends on a choice of the basis $X$.
It is easy to see that any lift of a basis of $\cG M$ becomes a basis of $M$.
\begin{lem}\label{lem:lifting-basis}
    For any $a\in\bZ_{\geq 0}$, let $Y_{a}$ be a basis of $\cG_{a}M$ and $\eta_{a}\colon\cG_{a}M\to\cF_{a}M$ a section of $\pi_{a}$.
    Then, $\eta_{a}(Y_{a})$ is a basis of $\cF_{a}M$.
\end{lem}

\subsection{Graded skein algebras} 
Fix a finite set $S$ of mutually disjoint ideal arcs. We call an element of $S$ a \emph{cut path} to distinguish it from an $\fsp_4$-diagram, and draw it by a blue arc in pictures.
We sometimes identify $S$ with a cellular decomposition of $\Sigma$ whose vertices are $\bM$ and edges are $S\cup\bB$.
For an $S$-transverse $\mathfrak{sp}_4$-diagram $D$ (\cref{def:transverse}), let $D\intersec_i S$ denote the set of intersections of type~$i$ edges of $D$ with $\cup_{\gamma\in S}\gamma$ for $i=1,2$. Let $D\intersec S:=(D\intersec_1 S) \cup (D\intersec_2 S)$ be the set of all intersections.

\begin{dfn}[$S$-degree]\label{dfn:intersection}
\ 
\begin{enumerate}
    \item For a non-empty $S$-transverse $\mathfrak{sp}_4$-diagram $D\in \diag_{\fsp_4,\bSigma}$, we define its \emph{$S$-degree} of $D$ by 
    \begin{align*}
        \widetilde{\deg}_S(D)\coloneqq 2\#D\intersec_1 S+3\#D\intersec_2 S\in\bZ_{\geq 0}.
    \end{align*}
    We define $\widetilde{\deg}_S(\varnothing)=0$ for the empty diagram $\varnothing$.
    \item An $S$-transverse $\mathfrak{sp}_4$-diagram $D\in \diag_{\fsp_4,\bSigma}$ is said to be \emph{$S$-minimal} if it realizes the minimum of the $S$-degree in its ladder-equivalence class.
    \item For $L\in\diag_{\fsp_4,\bSigma}/\sim$, we define its \emph{$S$-degree $\deg_S(L)$ of $L$} to be the $S$-degree of any $S$-minimal representative of it. 
\end{enumerate}
\end{dfn}


\begin{rem}
    We can also consider a $\bZ_{\geq 0}^2$-valued $S$-degree defined by $\widetilde{\mathbf{deg}}_S(D)\coloneqq (\#D\intersec_1 S+2\#D\intersec_2 S,\ \#D\intersec S)\in\bZ_{\geq 0}^2$. All the arguments involving the $S$-degree also work for this $\bZ_{\geq 0}^2$-valued degree with the lexicographic order of $\bZ_{\geq 0}^2$.
\end{rem}

The $S$-degree on $\Blad_{\fsp_4,\bSigma}$ introduces a partition $\Blad_{\fsp_4,\bSigma}=\bigsqcup_{a\in\bZ_{\geq 0}}(\Blad_{\fsp_4,\bSigma})_{a}$ of $\Blad_{\fsp_4,\bSigma}$, where $(\Blad_{\fsp_4,\bSigma})_{a}\coloneqq\{L\in \Blad_{\fsp_4,\bSigma}\mid \deg_{S}(L)=a\}$.
We also import the $S$-degree on $\Blad_{\fsp_4,\bSigma}$ into $\Bweb_{\fsp_4,\bSigma}$ by $\deg_S(\sigma_\bM(L))\coloneqq\deg_S(L)$ (recall \cref{dfn:shift}), and its partition $\Bweb_{\fsp_4,\bSigma}=\bigsqcup_{a\in\bZ_{\geq 0}}(\Bweb_{\fsp_4,\bSigma})_{a}$

\begin{dfn}[filtered/graded $\fsp_4$-skein algebras]\label{dfn:filtration}
Let $\bSigma$ is a marked surface, and $\cR_q\coloneqq\bZ[q^{\pm 1/2},1/[2]]$.
\begin{enumerate}
    \item Define a $\bZ_{\geq 0}$-filtration $\{\mathcal{F}_{a}^{S}\Skein{\mathfrak{sp}_4}{\bSigma}^{q}\}_{a \in\bZ_{\geq 0}}$ on $\Skein{\mathfrak{sp}_4}{\bSigma}^{q}$ by 
    \begin{align*}
        \mathcal{F}_{a}^{S}\Skein{\mathfrak{sp}_4}{\bSigma}^{q}\coloneqq\mathrm{span}_{\mathcal{R}_q}\bigsqcup_{i=0}^{a}(\Bweb_{\fsp_4,\bSigma})_{i}.
    \end{align*}
    It is easy to see that the filtration is compatible with the multiplication of the $\mathfrak{sp}_4$-skein algebra and that $
    \bigcup_{a}\mathcal{F}_{a}^{S}\Skein{\mathfrak{sp}_4}{\bSigma}^{q}=\Skein{\mathfrak{sp}_4}{\bSigma}^{q}$, so that $\Skein{\mathfrak{sp}_4}{\bSigma}^{q}$ becomes a filtered algebra over $\cR_q$. We call it the \emph{filtered $\mathfrak{sp}_4$-skein algebra}.  
    \item The \emph{associated graded $\mathfrak{sp}_4$-skein algbra} is defined to be 
    \begin{align*}
        \mathcal{G}^{S}\Skein{\mathfrak{sp}_4}{\bSigma}^{q}:=\bigoplus_{a}\mathcal{G}_{a}^{S}\Skein{\mathfrak{sp}_4}{\bSigma}^{q},
    \end{align*}
    where $\mathcal{G}_{a}^{S}\Skein{\mathfrak{sp}_4}{\bSigma}^{q}:=\mathcal{F}_{a}^{S}\Skein{\mathfrak{sp}_4}{\bSigma}^{q}/\mathcal{F}_{a-1}^{S}\Skein{\mathfrak{sp}_4}{\bSigma}^{q}$. 
    For $W,W' \in \Skein{\mathfrak{sp}_4}{\bSigma}^{q}$, we write $W\greq W'$ if $W$ is equal to $W'$ in $\mathcal{G}^{S}\Skein{\mathfrak{sp}_4}{\bSigma}^{q}$.
\end{enumerate}
\end{dfn}

\begin{rem}
    \begin{enumerate}
        \item For a filtration $\mathcal{F}^{\tri}$ associated with an ideal triangulation $S:=\tri$, we have $\mathcal{F}_{0}^{\tri}\Skein{\mathfrak{sp}_4}{\bSigma}^{q}=\mathcal{R}_q\varnothing$ since $L\in\Blad_{\fsp_4,\bSigma}$ with $\deg_{\tri}(L)=0$ means that $L$ is contained in a single triangle, which implies $L$ is the empty diagram $\varnothing$.
        \item We use $W\greq W'$ when $W$ is equal to $W'$ in $\mathcal{G}^{S}\Skein{\mathfrak{sp}_4}{\bSigma}^{q}$ to distinguish equality in $\Skein{\mathfrak{sp}_4}{\bSigma}^{q}$ clearly.
    \end{enumerate} 
\end{rem}

\begin{lem}\label{lem:equality-in-gr}
    The following equation holds in the graded $\mathfrak{sp}_4$-skein algebra $\mathcal{G}^{S}\Skein{\mathfrak{sp}_4}{\bSigma}^{q}$:    
    \begin{align*}
        \mbox{
            \tikz[baseline=-.6ex, scale=.1]{
                \draw[dashed, fill=white] (0,0) circle [radius=6];
                \draw[webline, rounded corners] (45:6) -- ($(45:6)+(-4,0)$) --  (-3,0) -- (-135:6);
                \draw[webline, rounded corners] (-45:6) -- ($(-45:6)+(-4,0)$) -- (-3,0) -- (135:6);
                \draw[fill=pink, thick] (-3,0) circle [radius=20pt];
                \draw[blue] (90:6) -- (-90:6);
            }
        }
        &\greq
        \mbox{
            \tikz[baseline=-.6ex, scale=.1]{
                \draw[dashed, fill=white] (0,0) circle [radius=6];
                \draw[webline] (45:6) -- (3,0);
                \draw[webline] (-45:6) -- (3,0);
                \draw[webline] (135:6) -- (-3,0);
                \draw[webline] (-135:6) -- (-3,0);
                \draw[wline] (-3,0) -- (3,0);
                \draw[blue] (90:6) -- (-90:6);
            }
        },&
        q^{-1}
        \mbox{
        \tikz[baseline=-.6ex, scale=.1]{
            \draw[dashed, fill=white] (0,0) circle [radius=6];
            \coordinate (U) at ($(150:6)!.2!(30:6)$);
            \coordinate (D) at ($(-150:6)!.2!(-30:6)$);
            \coordinate (U2) at ($(150:6)!.6!(30:6)$);
            \coordinate (D2) at ($(-150:6)!.6!(-30:6)$);
            \draw[oarc, rounded corners] (-150:6) -- (U2) -- (30:6);
            \draw[oarc, rounded corners] (150:6) -- (D2) -- (-30:6);
            \draw[blue] (60:6) -- (-60:6);
        }
        }
        &\greq
        \mbox{
        \tikz[baseline=-.6ex, scale=.1]{
            \draw[dashed, fill=white] (0,0) circle [radius=6];
            \coordinate (U) at ($(150:6)!.2!(30:6)$);
            \coordinate (D) at ($(-150:6)!.2!(-30:6)$);
            \coordinate (U2) at ($(150:6)!.6!(30:6)$);
            \coordinate (D2) at ($(-150:6)!.6!(-30:6)$);
            \draw[webline] (150:6) -- (30:6);
            \draw[webline] (-150:6) -- (-30:6);
            \draw[blue] (60:6) -- (-60:6);
        }
        },&
        q^{-2}
        \mbox{
        \tikz[baseline=-.6ex, scale=.1]{
            \draw[dashed, fill=white] (0,0) circle [radius=6];
            \coordinate (U) at ($(150:6)!.2!(30:6)$);
            \coordinate (D) at ($(-150:6)!.2!(-30:6)$);
            \coordinate (U2) at ($(150:6)!.6!(30:6)$);
            \coordinate (D2) at ($(-150:6)!.6!(-30:6)$);
            \draw[owarc, rounded corners] (-150:6) -- (U2) -- (30:6);
            \draw[owarc, rounded corners] (150:6) -- (D2) -- (-30:6);
            \draw[blue] (60:6) -- (-60:6);
        }
        }
        &\greq
        \mbox{
        \tikz[baseline=-.6ex, scale=.1]{
            \draw[dashed, fill=white] (0,0) circle [radius=6];
            \coordinate (U) at ($(150:6)!.2!(30:6)$);
            \coordinate (D) at ($(-150:6)!.2!(-30:6)$);
            \coordinate (U2) at ($(150:6)!.6!(30:6)$);
            \coordinate (D2) at ($(-150:6)!.6!(-30:6)$);
            \draw[wline] (150:6) -- (30:6);
            \draw[wline] (-150:6) -- (-30:6);
            \draw[blue] (60:6) -- (-60:6);
        }
        },&\\
        q^{-1}
        \mbox{
        \tikz[baseline=-.6ex, scale=.1]{
            \draw[dashed, fill=white] (0,0) circle [radius=6];
            \coordinate (U) at ($(150:6)!.2!(30:6)$);
            \coordinate (D) at ($(-150:6)!.2!(-30:6)$);
            \coordinate (U2) at ($(150:6)!.6!(30:6)$);
            \coordinate (D2) at ($(-150:6)!.6!(-30:6)$);
            \draw[owarc, rounded corners] (-150:6) -- (U2) -- (30:6);
            \draw[oarc, rounded corners] (150:6) -- (D2) -- (-30:6);
            \draw[blue] (60:6) -- (-60:6);
        }
        }
        &\greq
        \mbox{
        \tikz[baseline=-.6ex, scale=.1]{
            \draw[dashed, fill=white] (0,0) circle [radius=6];
            \coordinate (U) at ($(150:6)!.2!(30:6)$);
            \coordinate (D) at ($(-150:6)!.2!(-30:6)$);
            \coordinate (U2) at ($(150:6)!.6!(30:6)$);
            \coordinate (D2) at ($(-150:6)!.6!(-30:6)$);
            \draw[webline] (150:6) -- (U2);
            \draw[wline] (U2) -- (30:6);
            \draw[wline] (-150:6) -- (D2);
            \draw[webline] (D2) -- (-30:6);
            \draw[webline] (U2) -- (D2);
            \draw[blue] (60:6) -- (-60:6);
        }
        },&
        q^{-1}
        \mbox{
        \tikz[baseline=-.6ex, scale=.1]{
            \draw[dashed, fill=white] (0,0) circle [radius=6];
            \coordinate (U) at ($(150:6)!.2!(30:6)$);
            \coordinate (D) at ($(-150:6)!.2!(-30:6)$);
            \coordinate (U2) at ($(150:6)!.6!(30:6)$);
            \coordinate (D2) at ($(-150:6)!.6!(-30:6)$);
            \draw[oarc, rounded corners] (-150:6) -- (U2) -- (30:6);
            \draw[owarc, rounded corners] (150:6) -- (D2) -- (-30:6);
            \draw[blue] (60:6) -- (-60:6);
        }
        }
        &\greq
        \mbox{
        \tikz[baseline=-.6ex, scale=.1]{
            \draw[dashed, fill=white] (0,0) circle [radius=6];
            \coordinate (U) at ($(150:6)!.2!(30:6)$);
            \coordinate (D) at ($(-150:6)!.2!(-30:6)$);
            \coordinate (U2) at ($(150:6)!.6!(30:6)$);
            \coordinate (D2) at ($(-150:6)!.6!(-30:6)$);
            \draw[wline] (150:6) -- (U2);
            \draw[webline] (U2) -- (30:6);
            \draw[webline] (-150:6) -- (D2);
            \draw[wline] (D2) -- (-30:6);
            \draw[webline] (U2) -- (D2);
            \draw[blue] (60:6) -- (-60:6);
        }
        },&
        \mbox{
        \tikz[baseline=-.6ex, scale=.1]{
            \draw[dashed, fill=white] (0,0) circle [radius=6];
            \coordinate (LU) at ($(150:6)!.3!(30:6)$);
            \coordinate (LD) at ($(-150:6)!.3!(-30:6)$);
            \coordinate (RU) at ($(150:6)!.7!(30:6)$);
            \coordinate (RD) at ($(-150:6)!.7!(-30:6)$);
            \coordinate (L) at (-180:6);
            \coordinate (R) at (0:6);
            \draw[webline] (L) -- (LU) -- (30:6);
            \draw[webline] (-150:6) -- (LD);
            \draw[wline] (LD) -- (RD);
            \draw[webline] (150:6) -- (LU) -- (LD);
            \draw[webline] (RD) -- (R);
            \draw[webline] (RD) -- (-30:6);
            \draw[blue] (90:6) -- (-90:6);
            \draw[fill=pink, thick] (LU) circle [radius=20pt];
        }
        } 
        &\greq
        \mbox{
        \tikz[baseline=-.6ex, scale=.1, yscale=-1, xscale=-1]{
            \draw[dashed, fill=white] (0,0) circle [radius=6];
            \coordinate (LU) at ($(150:6)!.3!(30:6)$);
            \coordinate (LD) at ($(-150:6)!.3!(-30:6)$);
            \coordinate (RU) at ($(150:6)!.7!(30:6)$);
            \coordinate (RD) at ($(-150:6)!.7!(-30:6)$);
            \coordinate (L) at (-180:6);
            \coordinate (R) at (0:6);
            \draw[webline] (L) -- (LU) -- (30:6);
            \draw[webline] (-150:6) -- (LD);
            \draw[wline] (LD) -- (RD);
            \draw[webline] (150:6) -- (LU) -- (LD);
            \draw[webline] (RD) -- (R);
            \draw[webline] (RD) -- (-30:6);
            \draw[blue] (90:6) -- (-90:6);
            \draw[fill=pink, thick] (LU) circle [radius=20pt];
        }
        },\\
        \mbox{
        \tikz[baseline=-.6ex, scale=.1]{
            \draw[dashed, fill=white] (0,0) circle [radius=6];
            \coordinate (U) at ($(150:6)!.2!(30:6)$);
            \coordinate (D) at ($(-150:6)!.2!(-30:6)$);
            \coordinate (U2) at ($(150:6)!.6!(30:6)$);
            \coordinate (D2) at ($(-150:6)!.6!(-30:6)$);
            \draw[webline] (150:6) -- (U);
            \draw[wline] (U) -- (U2);
            \draw[webline] (U2) -- (30:6);
            \draw[wline] (-150:6) -- (D);
            \draw[webline] (D) -- (D2);
            \draw[wline] (D2) -- (-30:6);
            \draw[webline] (U) -- (D);
            \draw[webline] (U2) -- (D2);
            \draw[blue] (60:6) -- (-60:6);
        }
        }
        &\greq
        \mbox{
        \tikz[baseline=-.6ex, scale=.1]{
            \draw[dashed, fill=white] (0,0) circle [radius=6];
            \coordinate (U) at ($(150:6)!.2!(30:6)$);
            \coordinate (D) at ($(-150:6)!.2!(-30:6)$);
            \coordinate (U2) at ($(150:6)!.6!(30:6)$);
            \coordinate (D2) at ($(-150:6)!.6!(-30:6)$);
            \draw[webline] (150:6) -- (30:6);
            \draw[wline] (-150:6) -- (-30:6);
            \draw[blue] (60:6) -- (-60:6);
        }
        },&
        \mbox{
        \tikz[baseline=-.6ex, scale=.1]{
            \draw[dashed, fill=white] (0,0) circle [radius=6];
            \coordinate (U) at ($(150:6)!.2!(30:6)$);
            \coordinate (D) at ($(-150:6)!.2!(-30:6)$);
            \coordinate (U2) at ($(150:6)!.6!(30:6)$);
            \coordinate (D2) at ($(-150:6)!.6!(-30:6)$);
            \draw[wline] (150:6) -- (U);
            \draw[webline] (U) -- (U2);
            \draw[wline] (U2) -- (30:6);
            \draw[webline] (-150:6) -- (D);
            \draw[wline] (D) -- (D2);
            \draw[webline] (D2) -- (-30:6);
            \draw[webline] (U) -- (D);
            \draw[webline] (U2) -- (D2);
            \draw[blue] (60:6) -- (-60:6);
        }
        }
        &\greq
        \mbox{
        \tikz[baseline=-.6ex, scale=.1]{
            \draw[dashed, fill=white] (0,0) circle [radius=6];
            \coordinate (U) at ($(150:6)!.2!(30:6)$);
            \coordinate (D) at ($(-150:6)!.2!(-30:6)$);
            \coordinate (U2) at ($(150:6)!.6!(30:6)$);
            \coordinate (D2) at ($(-150:6)!.6!(-30:6)$);
            \draw[wline] (150:6) -- (30:6);
            \draw[webline] (-150:6) -- (-30:6);
            \draw[blue] (60:6) -- (-60:6);
        }
        }.&&
    \end{align*}
\end{lem}

\begin{proof}
    These equations are derived from the $\fsp_4$-skein relation in $\Skein{\mathfrak{sp}_4}{\bSigma}^{q}$.
    For example, we have
    \begin{align*}
        \mbox{
        \tikz[baseline=-.6ex, scale=.1]{
            \draw[dashed, fill=white] (0,0) circle [radius=6];
            \draw[webline] (45:6) -- (-135:6);
            \draw[webline] (135:6) -- (-45:6);
            \draw[fill=pink, thick] (0,0) circle [radius=20pt];
            \draw[blue] (60:6) -- (-60:6);
        }
        } 
        =
        \mbox{
        \tikz[baseline=-.6ex, scale=.1]{
            \draw[dashed, fill=white] (0,0) circle [radius=6];
            \draw[webline] (45:6) -- (3,0);
            \draw[webline] (-45:6) -- (3,0);
            \draw[webline] (135:6) -- (-3,0);
            \draw[webline] (-135:6) -- (-3,0);
            \draw[wline] (-3,0) -- (3,0);
            \draw[blue] (90:6) -- (-90:6);
        }
        } 
        -\frac{1}{[2]}
        \mbox{
        \tikz[baseline=-.6ex, scale=.1]{
            \draw[dashed, fill=white] (0,0) circle [radius=6];
            \draw[webline, rounded corners] (45:6) -- (3,0) -- (-45:6);
            \draw[webline, rounded corners] (135:6) -- (-3,0) -- (-135:6);
            \draw[blue] (90:6) -- (-90:6);
        }
        } 
    \end{align*}
    and the second term in the right-hand side vanishes in the graded quotient, since it is smaller than the first term in $S$-degree.
    Equations in the second line come from the first line, isotopy, and the Reidemeister move.
\end{proof}

\begin{dfn}\label{dfn:flat-braid}
We can define a flat braid diagram across a cut path (blue line) in $\mathcal{G}^{S}\Skein{\mathfrak{sp}_4}{\bSigma}^{q}$ with double points, interpreted as follows.
\begin{align*}
        \mbox{
        \tikz[baseline=-.6ex, scale=.08]{
            \draw[dashed, fill=white] (0,0) circle [radius=6];
            \coordinate (U) at ($(150:6)!.2!(30:6)$);
            \coordinate (D) at ($(-150:6)!.2!(-30:6)$);
            \coordinate (U2) at ($(150:6)!.6!(30:6)$);
            \coordinate (D2) at ($(-150:6)!.6!(-30:6)$);
            \draw[webline, rounded corners] (-150:6) -- (U2) -- (30:6);
            \draw[webline, rounded corners] (150:6) -- (D2) -- (-30:6);
            \draw[blue] (60:6) -- (-60:6);
        }
        }
        \coloneqq
        q^{-1}
        \mbox{
        \tikz[baseline=-.6ex, scale=.08]{
            \draw[dashed, fill=white] (0,0) circle [radius=6];
            \coordinate (U) at ($(150:6)!.2!(30:6)$);
            \coordinate (D) at ($(-150:6)!.2!(-30:6)$);
            \coordinate (U2) at ($(150:6)!.6!(30:6)$);
            \coordinate (D2) at ($(-150:6)!.6!(-30:6)$);
            \draw[oarc, rounded corners] (-150:6) -- (U2) -- (30:6);
            \draw[oarc, rounded corners] (150:6) -- (D2) -- (-30:6);
            \draw[blue] (60:6) -- (-60:6);
        }
        },
        \mbox{
        \tikz[baseline=-.6ex, scale=.08]{
            \draw[dashed, fill=white] (0,0) circle [radius=6];
            \coordinate (U) at ($(150:6)!.2!(30:6)$);
            \coordinate (D) at ($(-150:6)!.2!(-30:6)$);
            \coordinate (U2) at ($(150:6)!.6!(30:6)$);
            \coordinate (D2) at ($(-150:6)!.6!(-30:6)$);
            \draw[wline, rounded corners] (-150:6) -- (U2) -- (30:6);
            \draw[wline, rounded corners] (150:6) -- (D2) -- (-30:6);
            \draw[blue] (60:6) -- (-60:6);
        }
        }
        \coloneqq
        q^{-2}
        \mbox{
        \tikz[baseline=-.6ex, scale=.08]{
            \draw[dashed, fill=white] (0,0) circle [radius=6];
            \coordinate (U) at ($(150:6)!.2!(30:6)$);
            \coordinate (D) at ($(-150:6)!.2!(-30:6)$);
            \coordinate (U2) at ($(150:6)!.6!(30:6)$);
            \coordinate (D2) at ($(-150:6)!.6!(-30:6)$);
            \draw[owarc, rounded corners] (-150:6) -- (U2) -- (30:6);
            \draw[owarc, rounded corners] (150:6) -- (D2) -- (-30:6);
            \draw[blue] (60:6) -- (-60:6);
        }
        },
        \mbox{
        \tikz[baseline=-.6ex, scale=.08]{
            \draw[dashed, fill=white] (0,0) circle [radius=6];
            \coordinate (U) at ($(150:6)!.2!(30:6)$);
            \coordinate (D) at ($(-150:6)!.2!(-30:6)$);
            \coordinate (U2) at ($(150:6)!.6!(30:6)$);
            \coordinate (D2) at ($(-150:6)!.6!(-30:6)$);
            \draw[webline, rounded corners] (150:6) -- (D2) -- (-30:6);
            \draw[wline, rounded corners] (-150:6) -- (U2) -- (30:6);
            \draw[blue] (60:6) -- (-60:6);
        }
        }
        \coloneqq
        q^{-1}
        \mbox{
        \tikz[baseline=-.6ex, scale=.08]{
            \draw[dashed, fill=white] (0,0) circle [radius=6];
            \coordinate (U) at ($(150:6)!.2!(30:6)$);
            \coordinate (D) at ($(-150:6)!.2!(-30:6)$);
            \coordinate (U2) at ($(150:6)!.6!(30:6)$);
            \coordinate (D2) at ($(-150:6)!.6!(-30:6)$);
            \draw[owarc, rounded corners] (-150:6) -- (U2) -- (30:6);
            \draw[oarc, rounded corners] (150:6) -- (D2) -- (-30:6);
            \draw[blue] (60:6) -- (-60:6);
        }
        },
        \mbox{
        \tikz[baseline=-.6ex, scale=.08]{
            \draw[dashed, fill=white] (0,0) circle [radius=6];
            \coordinate (U) at ($(150:6)!.2!(30:6)$);
            \coordinate (D) at ($(-150:6)!.2!(-30:6)$);
            \coordinate (U2) at ($(150:6)!.6!(30:6)$);
            \coordinate (D2) at ($(-150:6)!.6!(-30:6)$);
            \draw[webline, rounded corners] (-150:6) -- (U2) -- (30:6);
            \draw[wline, rounded corners] (150:6) -- (D2) -- (-30:6);
            \draw[blue] (60:6) -- (-60:6);
        }
        }
        \coloneqq
        q^{-1}
        \mbox{
        \tikz[baseline=-.6ex, scale=.08]{
            \draw[dashed, fill=white] (0,0) circle [radius=6];
            \coordinate (U) at ($(150:6)!.2!(30:6)$);
            \coordinate (D) at ($(-150:6)!.2!(-30:6)$);
            \coordinate (U2) at ($(150:6)!.6!(30:6)$);
            \coordinate (D2) at ($(-150:6)!.6!(-30:6)$);
            \draw[oarc, rounded corners] (-150:6) -- (U2) -- (30:6);
            \draw[owarc, rounded corners] (150:6) -- (D2) -- (-30:6);
            \draw[blue] (60:6) -- (-60:6);
        }
        }.
\end{align*}
\end{dfn}
    We see that the diagrams (B), (C), and (D) in \cref{fig:ladder-web} define same elements in $\mathcal{G}^{S}\Skein{\mathfrak{sp}_4}{\bSigma}^{q}$ if $\gamma_{+}$ or $\gamma_{-}$ belongs to $S$, by applying the interpretation above.
\begin{lem}\label{rem:flat-braid}
    For any colors $c_1,c_2,c_3\in\{\text{type 1},\text{type 2}\}$, the following relations hold in $\mathcal{G}^{S}\Skein{\mathfrak{sp}_4}{\bSigma}^{q}$:
    \begin{align*}
        \mbox{
        \tikz[baseline=-.6ex, scale=.1]{
            \draw[dashed, fill=white] (0,0) circle [radius=6];
            \coordinate (U) at ($(150:6)!.2!(30:6)$);
            \coordinate (D) at ($(-150:6)!.2!(-30:6)$);
            \coordinate (U2) at ($(150:6)!.6!(30:6)$);
            \coordinate (D2) at ($(-150:6)!.6!(-30:6)$);
            \draw[webline, black, rounded corners] (150:6) -- (U) -- ($(D)!.5!(D2)$) -- (U2)-- (30:6);
            \draw[webline, black, rounded corners] (-150:6) -- (D) -- ($(U)!.5!(U2)$) -- (D2) -- (-30:6);
            \draw[blue] (60:6) -- (-60:6);
            \node at (150:6) [left]{\scriptsize $c_1$};
            \node at (-150:6) [left]{\scriptsize $c_2$};
        }
        } 
        \greq
        \mbox{
        \tikz[baseline=-.6ex, scale=.1]{
            \draw[dashed, fill=white] (0,0) circle [radius=6];
            \coordinate (U) at ($(150:6)!.2!(30:6)$);
            \coordinate (D) at ($(-150:6)!.2!(-30:6)$);
            \coordinate (U2) at ($(150:6)!.6!(30:6)$);
            \coordinate (D2) at ($(-150:6)!.6!(-30:6)$);
            \draw[webline, black] (150:6) -- (30:6);
            \draw[webline, black] (-150:6) -- (-30:6);
            \draw[blue] (60:6) -- (-60:6);
            \node at (150:6) [left]{\scriptsize $c_1$};
            \node at (-150:6) [left]{\scriptsize $c_2$};
        }
        },\quad
        \mbox{
        \tikz[baseline=-.6ex, scale=.1]{
            \draw[dashed, fill=white] (0,0) circle [radius=6];
            \coordinate (U1) at (150:4);
            \coordinate (U2) at (90:4);
            \coordinate (U3) at (30:4);
            \coordinate (D1) at (-150:4);
            \coordinate (D2) at (-90:4);
            \coordinate (D3) at (-30:4);
            \draw[webline, black, rounded corners] (150:6) -- (-30:6);
            \draw[webline, black, rounded corners] (180:6) -- (D1) -- (D2)-- (D3) -- (0:6);
            \draw[webline, black, rounded corners] (-150:6) -- (30:6);
            \draw[blue] (60:6) -- (-60:6);
            \node at (150:6) [left]{\scriptsize $c_1$};
            \node at (180:6) [left]{\scriptsize $c_2$};
            \node at (-150:6) [left]{\scriptsize $c_3$};
        }
        } 
        \greq
        \mbox{
        \tikz[baseline=-.6ex, scale=.1]{
            \draw[dashed, fill=white] (0,0) circle [radius=6];
            \coordinate (U1) at (150:4);
            \coordinate (U2) at (90:4);
            \coordinate (U3) at (30:4);
            \coordinate (D1) at (-150:4);
            \coordinate (D2) at (-90:4);
            \coordinate (D3) at (-30:4);
            \draw[webline, black, rounded corners] (150:6) -- (-30:6);
            \draw[webline, black, rounded corners] (180:6) -- (U1) -- (U2) -- (U3) -- (0:6);
            \draw[webline, black, rounded corners] (-150:6) -- (30:6);
            \draw[blue] (60:6) -- (-60:6);
            \node at (150:6) [left]{\scriptsize $c_1$};
            \node at (180:6) [left]{\scriptsize $c_2$};
            \node at (-150:6) [left]{\scriptsize $c_3$};
        }
        }
    \end{align*}
    where double points are as defined in \cref{dfn:flat-braid}.
    We call these two deformations the \emph{flat Reidemeister {II} and {III} moves}.
\end{lem}
\begin{proof}
    Apply relations in \cref{lem:equality-in-gr}.
\end{proof}

\begin{rem}
    For $W\in\Skein{\mathfrak{sp}_4}{\bSigma}^{q}$, we define $\deg_S(W)\coloneqq\min\{a\in\bZ_{\geq 0}\mid W\in\mathcal{F}_{a}^{S}\Skein{\mathfrak{sp}_4}{\bSigma}^{q}\}$.
    For an $S$-transverse representative $D\in \diag_{\fsp_4,\bSigma}$ possibly with elliptic faces, $\deg_S(\diagshift(D))\leq\widetilde{\deg}_S(D)$ since it is a monotonically increasing function of the number of intersection points, and any $\mathfrak{sp}_4$-skein relation does not increase the number of intersections with $S$. 
\end{rem}

Let $\pi_{S}\colon \Skein{\mathfrak{sp}_4}{\bSigma}^{q}\to\mathcal{G}^{S}\Skein{\mathfrak{sp}_4}{\bSigma}^{q}$ be the symbol map defined by $\pi_{S}(W)\coloneqq W+\cF^{S}_{a-1}\Skein{\mathfrak{sp}_4}{\bSigma}^{q}$ for $a=\deg_S(W)$.


\begin{lem}\label{lem:gr-basis}
    The image $\pi_{S}(\Bweb_{\fsp_4,\bSigma})$ forms a basis of $\mathcal{G}^{S}\Skein{\mathfrak{sp}_4}{\bSigma}^{q}$.
\end{lem}
\begin{proof}
    By \cref{lem:fil-gr-isom}, there exists a canonical $\cR_q$-module isomorphism $\eta^{\Bweb_{\fsp_4,\bSigma}}\colon\cG^{S}\Skein{\mathfrak{sp}_4}{\bSigma}^{q}\to\Skein{\mathfrak{sp}_4}{\bSigma}^{q}$ such that $\eta^{\Bweb_{\fsp_4,\bSigma}}(\pi_{S}(W))=W$ for $W\in\Bweb_{\fsp_4,\bSigma}$. Thus $\pi_{S}(\Bweb_{\fsp_4,\bSigma})$ forms a basis of $\mathcal{G}^{S}\Skein{\mathfrak{sp}_4}{\bSigma}^{q}$.
\end{proof}


The situation is summarized in the following diagram:
\begin{equation*}
\begin{tikzcd}
    \mathscr{W}_{\fsp_4,\bSigma} =\Blad_{\fsp_4,\bSigma} \ar[r,"\sigma_\bM"] \ar[d] & \Bweb_{\fsp_4,\bSigma} \ar[d] \ar[r,phantom,"\subset"] & \Skein{\mathfrak{sp}_4}{\bSigma}^{q} \ar[d,"\pi_{S}"]\\
    \pi_{S}(\Blad_{\fsp_4,\bSigma})\ar[r] & \pi_{S}(\Bweb_{\fsp_4,\bSigma}) \ar[r,phantom,"\subset"] & \mathcal{G}^S\Skein{\mathfrak{sp}_4}{\bSigma}^{q}
\end{tikzcd}
\end{equation*}

For a cut path $\gamma$ between special points $p$ and $q$ in $\bM$, we fix two cut paths $\gamma_{\pm}$ between $p$ and $q$ bounding a biangle $B_{\gamma}$ which is divided by $\gamma$ in half.

\begin{prop}\label{lem:gr-bigon}
    Let $\gamma$ be a cut path, and $D$ a $\{\gamma_{\pm}\}$-minimal crossroad $\mathfrak{sp}_4$-diagram on $\Sigma$.
    For any ladder resolution $D'$ of $D$ at all crossroads in $B_\gamma$, we have $\diagshift(D)\greq \diagshift(D')$ in $\mathcal{G}^{\{\gamma\}}\Skein{\mathfrak{sp}_4}{\bSigma}^{q}$. 
\end{prop}
\begin{proof}
    We will prove this claim by induction on the number $n$ of crossroads in $B_{\gamma}$.
    For $n=1$, $D$ can be described below (or its horizontal reflection):
    \begin{align*}
        D=
        \mbox{
        \tikz[baseline=-.6ex, scale=.5]{
        \draw[blue](0,-3) -- (0,3);
        \draw[blue] (3,-3) -- (3,3);
        \draw[webline] (0,1) -- (1,1) -- (2,1);
        \draw[wline] (2,1) -- (3,1);
        \draw[webline] (0,0) -- (1,0) -- (2,0) -- (3,0);
        \draw[webline] (0,-1) -- (1,-1) -- (2,-1) -- (3,-1);
        \draw[wline] (0,-1) -- (2,-1);
        \draw[webline] (2,1) -- (2,0);
        \draw[webline] (2,0) -- (2,-1);
        \draw[webline, black] (0,1.5) -- (3,1.5);
        \draw[webline, black] (0,2.5) -- (3,2.5);
        \draw[webline, black] (0,-1.5) -- (3,-1.5);
        \draw[webline, black] (0,-2.5) -- (3,-2.5);
        \node at (1.5,2) [rotate=90]{\scriptsize $\cdots$};
        \node at (1.5,-2) [rotate=90]{\scriptsize $\cdots$};
        \node at (0,3) [above]{\scriptsize $\gamma_{-}$};
        \node at (3,3) [above]{\scriptsize $\gamma_{+}$};
        \node at (2,0) {$\crossroad$};
        }
        }
    \end{align*}
    where the black horizontal lines represent type~1 or type~2 arcs.
    Then, by the skein relation in the third line of \cref{def:skein-rel}, $\diagshift(D)=\diagshift(L)-\frac{1}{[2]}\diagshift(E)=\diagshift(L')-\frac{1}{[2]}\diagshift(E')$ where
    \begin{align*}
        L=
        \mbox{
        \tikz[baseline=-.6ex, scale=.5]{
        \draw[blue](0,-3) -- (0,3);
        \draw[blue] (3,-3) -- (3,3);
        \draw[webline] (0,1) -- (1,1) -- (2,1);
        \draw[wline] (2,1) -- (3,1);
        \draw[webline] (0,0) -- (1,0) -- (2,0) -- (3,0);
        \draw[wline] (1,0) -- (2,0);
        \draw[webline] (0,-1) -- (1,-1) -- (2,-1) -- (3,-1);
        \draw[wline] (0,-1) -- (1,-1);
        \draw[webline] (2,1) -- (2,0);
        \draw[webline] (1,0) -- (1,-1);
        \draw[webline, black] (0,1.5) -- (3,1.5);
        \draw[webline, black] (0,2.5) -- (3,2.5);
        \draw[webline, black] (0,-1.5) -- (3,-1.5);
        \draw[webline, black] (0,-2.5) -- (3,-2.5);
        \node at (1.5,2) [rotate=90]{\scriptsize $\cdots$};
        \node at (1.5,-2) [rotate=90]{\scriptsize $\cdots$};
        \node at (0,3) [above]{\scriptsize $\gamma_{-}$};
        \node at (3,3) [above]{\scriptsize $\gamma_{+}$};
        }
        },\quad
        E=
        \mbox{
        \tikz[baseline=-.6ex, scale=.5]{
        \draw[blue](0,-3) -- (0,3);
        \draw[blue] (3,-3) -- (3,3);
        \draw[webline] (0,1) -- (1,1) -- (2,1);
        \draw[wline] (2,1) -- (3,1);
        \draw[webline] (0,0) -- (1,0);
        \draw[webline] (2,0) -- (3,0);
        \draw[webline] (0,-1) -- (1,-1) -- (2,-1) -- (3,-1);
        \draw[wline] (0,-1) -- (1,-1);
        \draw[webline] (2,1) -- (2,0);
        \draw[webline] (1,0) -- (1,-1);
        \draw[webline, black] (0,1.5) -- (3,1.5);
        \draw[webline, black] (0,2.5) -- (3,2.5);
        \draw[webline, black] (0,-1.5) -- (3,-1.5);
        \draw[webline, black] (0,-2.5) -- (3,-2.5);
        \node at (1.5,2) [rotate=90]{\scriptsize $\cdots$};
        \node at (1.5,-2) [rotate=90]{\scriptsize $\cdots$};
        \draw[blue] (1.8,3) -- (1.8,-3);
        \node at (0,3) [above]{\scriptsize $\gamma_{-}$};
        \node at (3,3) [above]{\scriptsize $\gamma_{+}$};
        }
        },\quad
        L'=
        \mbox{
        \tikz[baseline=-.6ex, scale=.5]{
        \draw[blue](0,-3) -- (0,3);
        \draw[blue] (3,-3) -- (3,3);
        \draw[webline] (0,1) -- (1,1) -- (2,1);
        \draw[wline] (2,1) -- (3,1);
        \draw[webline] (0,0) -- (1.5,.5);
        \draw[webline] (1.5,-.5) -- (3,0);
        \draw[webline] (0,-1) -- (1,-1) -- (2,-1) -- (3,-1);
        \draw[wline] (0,-1) -- (1,-1);
        \draw[webline] (2,1) -- (1.5,.5);
        \draw[webline] (1.5,-.5) -- (1,-1);
        \draw[wline] (1.5,.5) -- (1.5,-.5);
        \draw[webline, black] (0,1.5) -- (3,1.5);
        \draw[webline, black] (0,2.5) -- (3,2.5);
        \draw[webline, black] (0,-1.5) -- (3,-1.5);
        \draw[webline, black] (0,-2.5) -- (3,-2.5);
        \node at (1.5,2) [rotate=90]{\scriptsize $\cdots$};
        \node at (1.5,-2) [rotate=90]{\scriptsize $\cdots$};
        \node at (0,3) [above]{\scriptsize $\gamma_{-}$};
        \node at (3,3) [above]{\scriptsize $\gamma_{+}$};
        }
        },\quad
        E'=
        \mbox{
        \tikz[baseline=-.6ex, scale=.5]{
        \draw[blue](0,-3) -- (0,3);
        \draw[blue] (3,-3) -- (3,3);
        \draw[webline] (0,1) -- (1,1) -- (2,1);
        \draw[wline] (2,1) -- (3,1);
        \draw[webline] (0,0) -- (1.5,.5);
        \draw[webline] (1.5,-.5) -- (3,0);
        \draw[webline] (0,-1) -- (1,-1) -- (2,-1) -- (3,-1);
        \draw[wline] (0,-1) -- (1,-1);
        \draw[webline] (2,1) -- (1.5,.5);
        \draw[webline] (1.5,-.5) -- (1,-1);
        \draw[webline, black] (0,1.5) -- (3,1.5);
        \draw[webline, black] (0,2.5) -- (3,2.5);
        \draw[webline, black] (0,-1.5) -- (3,-1.5);
        \draw[webline, black] (0,-2.5) -- (3,-2.5);
        \node at (1.5,2) [rotate=90]{\scriptsize $\cdots$};
        \node at (1.5,-2) [rotate=90]{\scriptsize $\cdots$};
        \draw[blue,rounded corners] (2.5,3) -- (2,0) -- (1,0) -- (0.5,-3);
        \node at (0,3) [above]{\scriptsize $\gamma_{-}$};
        \node at (3,3) [above]{\scriptsize $\gamma_{+}$};
        }
        }.
    \end{align*}
    The middle blue vertical lines in $E$ and $E'$ realize smaller degree. It concludes $\diagshift(D)\greq\diagshift(L)\greq\diagshift(L')$.
    Assume the assertion holds for $n-1$.
    Let us consider $D_{n-1}$ with $n-1$ crossroads in $B_{\gamma}$. Then, for any ladder resolution $L_{n-1}$ of $D_{n-1}$, we can obtain a ladder resolution $L_{n}$ with $n$ rungs by attaching an $H$-web to the right side of $L_{n-1}$. $L_n$ has a new horizontal rung on the right. In a similar way to the above calculation, one can show that $\diagshift(L_{n})\greq\diagshift(L_{n}')$ where $L_{n}'$ is an $\fsp_4$-diagram obtained by replacing the rightmost rung of $L_{n}$ with one in the another direction (the vertical rung). By induction assumption, $n-1$ rungs in $L_{n-1}$ can be freely replaced in $\mathcal{G}^{\{\gamma\}}\Skein{\mathfrak{sp}_4}{\bSigma}^{q}$. Hence, the assertion holds for $n$.
\end{proof}

\begin{prop}\label{lem:gr-triangle}
    Let $S=\{\gamma_1,\gamma_2,\gamma_3\}$ be a set of cut paths bounding an ideal triangle $T$, and $D$ an $S$-transverse bounded crossroad $\mathfrak{sp}_4$-diagram such that $D\cap T$ is reduced. 
    Then, $\diagshift(D)\greq \diagshift(D')$ holds for any ladder-resolution $D'$ at crossroads in $T$.
\end{prop}

This proposition immediately follows from the lemmas below.

\begin{lem}\label{lem:tripod-clasp}
    For any positive integer $n,k,l$ with $k+l=n-2$ and $n\geq 2$,
    \begin{align*}
        \mbox{
            \tikz[baseline=-.6ex, scale=.1]{
                \coordinate (N) at (0,10);
                \coordinate (S) at (0,-10);
                \coordinate (W) at (-10,0);
                \coordinate (E) at (10,0);
                \coordinate (NE) at (45:10);
                \coordinate (NW) at (135:10);
                \coordinate (SE) at (-45:10);
                \coordinate (SW) at (-135:10);
                \coordinate (C) at (0,0);
                \coordinate (CN) at (0,5);
                \coordinate (CS) at (0,-5);
                \coordinate (CW) at (-5,0);
                \coordinate (CE) at (5,0);
                \begin{scope}
                    \clip (C) circle [radius=10cm];
                    \draw[wline] (-90:10) -- (C);
                    \draw[webline] (150:10) -- (C);
                    \draw[webline, xshift=1cm, rounded corners, rotate=-45, yshift=2cm] (-1,0) -- (-1,4) -- (1,4) -- (1,0);
                    \draw[webline, rotate=-45, xshift=1cm, yshift=3cm] (-3,0) -- (-3,10);
                    \draw[webline, rotate=-45, xshift=1cm, yshift=3cm] (3,0) -- (3,10);
                    \node at (C) [scale=1.4]{$\tribox$};
                    \draw[blue] ($(150:8)+(60:10)$) -- ($(150:8)+(60:-10)$);
                    \draw[blue] (SW) -- (SE);
                \end{scope}
                \node at (SW) [above left=10pt]{\scriptsize $\beta$};
                \node at (SW) [left]{\scriptsize $\alpha$};
                \draw[dashed] (C) circle [radius=10cm];
                \node at (C) {\scriptsize $n$};
                \node at (-90:12) {\scriptsize $n$};
                \node at (60:12) {\scriptsize $l$};
                \node at (25:12) {\scriptsize $k$};
                \node at (150:12) {\scriptsize $n$};
            }
        },
        \mbox{
            \tikz[baseline=-.6ex, scale=.1]{
                \coordinate (N) at (0,10);
                \coordinate (S) at (0,-10);
                \coordinate (W) at (-10,0);
                \coordinate (E) at (10,0);
                \coordinate (NE) at (45:10);
                \coordinate (NW) at (135:10);
                \coordinate (SE) at (-45:10);
                \coordinate (SW) at (-135:10);
                \coordinate (C) at (0,0);
                \coordinate (CN) at (0,5);
                \coordinate (CS) at (0,-5);
                \coordinate (CW) at (-5,0);
                \coordinate (CE) at (5,0);
                \begin{scope}
                    \clip (C) circle [radius=10cm];
                    \draw[wline] (-90:10) -- (C);
                    \draw[webline] (150:10) -- (C);
                    \draw[webline, xshift=1cm, rounded corners, rotate=-45, yshift=2cm] (-1,0) -- (-1,4) -- (1,4) -- (1,0);
                    \draw[wline, xshift=1cm, rotate=-45, yshift=2cm] (0,4) -- (0,10);
                    \draw[webline, rotate=-45, xshift=1cm, yshift=3cm] (-3,0) -- (-3,10);
                    \draw[webline, rotate=-45, xshift=1cm, yshift=3cm] (3,0) -- (3,10);
                    \node at (C) [scale=1.4]{$\tribox$};
                    \draw[blue] ($(150:8)+(60:10)$) -- ($(150:8)+(60:-10)$);
                    \draw[blue] (SW) -- (SE);
                \end{scope}
                \node at (SW) [above left=10pt]{\scriptsize $\beta$};
                \node at (SW) [left]{\scriptsize $\alpha$};
                \draw[dashed] (C) circle [radius=10cm];
                \node at (C) {\scriptsize $n$};
                \node at (-90:12) {\scriptsize $n$};
                \node at (60:12) {\scriptsize $l$};
                \node at (25:12) {\scriptsize $k$};
                \node at (150:12) {\scriptsize $n$};
            }
        }
        \in\cF_{d}^{\{\alpha,\beta\}}\Skein{\fsp_4}{\bSigma}^{q},
    \end{align*}
    where $d=\widetilde{\deg}_{\alpha}(D)+\widetilde{\deg}_{\beta}(D)-1$ for the above diagram $D$.
\end{lem}
\begin{proof}
    We prove the assertion by induction on $n$. For $n=1$, it follows immediately from the skein relation below:
    \begin{align*}
    &\mbox{
        \tikz[baseline=-.6ex, scale=.1]{
            \draw[dashed, fill=white] (0,0) circle [radius=7];
            \coordinate (NW) at (135:7);
            \coordinate (SW) at (-135:7);
            \coordinate (L) at (-2,0);
            \coordinate (R) at (4,0);
            \draw[webline] (SW) to[out=north east, in=south west] (L) to[out=north east, in=north] (R) to[out=south, in=south east] (L) to[out=north west, in=south east] (NW);
            \draw[wline] (R) -- (7,0);
            \draw[fill=pink] (L) circle [radius=20pt];
        }
    }
    =-[2]
    \mbox{
        \tikz[baseline=-.6ex, scale=.1]{
            \draw[dashed, fill=white] (0,0) circle [radius=7];
            \coordinate (NW) at (135:7);
            \coordinate (SW) at (-135:7);
            \coordinate (L) at (-2,0);
            \coordinate (R) at (4,0);
            \draw[webline] (SW) -- (0,0) -- (NW);
            \draw[wline] (0,0) -- (7,0);
        }
    },&
    \mbox{
        \tikz[baseline=-.6ex, scale=.1]{
            \draw[dashed, fill=white] (0,0) circle [radius=7];
            \coordinate (NW) at (135:7);
            \coordinate (NE) at (45:7);
            \coordinate (SW) at (-135:7);
            \coordinate (SE) at (-45:7);
            \coordinate (L) at (-3,0);
            \coordinate (R) at (3,0);
            \draw[webline] (SW) -- (L) -- (1,4) -- (1,-4) -- (L) -- (NW);
            \draw[wline] (1,4) -- (NE);
            \draw[wline] (1,-4) -- (SE);
            \draw[fill=pink] (L) circle [radius=20pt];
        }
    }=
    \mbox{
        \tikz[baseline=-.6ex, scale=.1]{
            \draw[dashed, fill=white] (0,0) circle [radius=7];
            \coordinate (NW) at (135:7);
            \coordinate (NE) at (45:7);
            \coordinate (SW) at (-135:7);
            \coordinate (SE) at (-45:7);
            \draw[webline] (SW) to[bend right] (NW);
            \draw[wline] (SE) to[bend left] (NE);
        }
    }.
    \end{align*}
    After applying these skein relations, the resulting diagrams have elliptic bordered faces along $\alpha$ or $\beta$ (see \cref{dfn:red-H-move}). Hence, one can make intersection points of these $\fsp_4$-diagrams smaller by the intersection reduction moves in \cref{fig:red-move}.
    Assume the assertion holds for $n-1$, and we will prove it for $n>2$. 
    By the definition of the $\fsp_4$-tripod of degree $n$, it decomposes into the $\fsp_4$-tripod of degree $n-1$ and an $\fsp_4$-tetrapod as
    \begin{align*}
    \mbox{
        \tikz[baseline=-.6ex, scale=.1]{
            \coordinate (N) at (0,10);
            \coordinate (S) at (0,-10);
            \coordinate (W) at (-10,0);
            \coordinate (E) at (10,0);
            \coordinate (NE) at (45:10);
            \coordinate (NW) at (135:10);
            \coordinate (SE) at (-45:10);
            \coordinate (SW) at (-135:10);
            \coordinate (C) at (0,0);
            \coordinate (CN) at (0,5);
            \coordinate (CS) at (0,-5);
            \coordinate (CW) at (-5,0);
            \coordinate (CE) at (5,0);
            \begin{scope}
                \clip (C) circle [radius=10cm];
                \draw[wline] ($(S)-(2,0)$) -- ($(C)-(2,0)$);
                \draw[webline, shorten >= -.3cm] ($(C)-(2,0)$) -- ($(NE)-(2,0)$);
                \draw[webline] ($(NW)-(2,0)$) -- ($(C)-(2,0)$);
                \node at ($(C)-(2,0)$) [scale=1]{\tribox};
                \draw[wline] ($(S)-(-6,0)$) -- ($(C)-(-6,0)$);
                \draw[webline] ($(C)-(-6,0)$) -- ($(NE)-(-6,0)$);
                \draw[webline, shorten <= -1cm] ($(C)!.6!(NE)-(2,0)$) -- ($(C)-(-6,0)$);
                \node at (2,4) [rotate=-45]{\tikz{\node at (0,0) [yscale=.4, xscale=.5]{\sqbox}}};
            \end{scope}
            \draw[dashed] (C) circle [radius=10cm];
            \node at ($(C)-(2,0)$) [xscale=.5, yscale=.7]{\scriptsize $n-1$};
            \node at ($(S)-(2,0)$) [below]{\scriptsize $n-1$};
            \node at ($(NE)-(2,0)$) [above right]{\scriptsize $n-1$};
            \node at ($(NW)-(2,0)$) [left]{\scriptsize $n-1$};
            \node at ($(S)+(6,0)$) {\scriptsize $1$};
            \node at ($(NE)+(2,-2)$) [right]{\scriptsize $1$};
            \node at ($(N)+(-3,2)$) {\scriptsize $1$};
        }
    }.
    \end{align*}
    Hence, we only have to show that every $\fsp_4$-diagram in the expansion of a top right elliptic face produces a new elliptic face adjacent to the $\fsp_4$-tripod of degree $n-1$. Then we can apply the induction assumption.
    The top right elliptic faces can be expanded by the following skein relations:
    \begin{align*}
    &\mbox{
        \tikz[baseline=-.6ex, scale=.1]{
            \draw[dashed, fill=white] (0,0) circle [radius=7];
            \coordinate (NW) at (150:7);
            \coordinate (NE) at (30:7);
            \coordinate (SW) at (-150:7);
            \coordinate (SE) at (-30:7);
            \coordinate (R) at (3,0);
            \coordinate (L) at (-3,0);
            \draw[webline] (SW) -- (L) to[out=north east, in=north west] (R) -- (SE);
            \draw[webline] (NW) -- (L) to[out=south east, in=south west] (R) -- (NE);
            \draw[fill=pink] (L) circle [radius=20pt];
            \draw[fill=pink] (R) circle [radius=20pt];
        }
    }
    =\frac{[4]}{[2]}
    \mbox{
        \tikz[baseline=-.6ex, scale=.1]{
            \draw[dashed, fill=white] (0,0) circle [radius=7];
            \coordinate (NW) at (150:7);
            \coordinate (NE) at (30:7);
            \coordinate (SW) at (-150:7);
            \coordinate (SE) at (-30:7);
            \draw[webline] (NW) to[bend left=60] (SW);
            \draw[webline] (NE) to[bend right=60] (SE);
        }
    }
    -[2]
    \mbox{
        \tikz[baseline=-.6ex, scale=.1]{
            \draw[dashed, fill=white] (0,0) circle [radius=7];
            \coordinate (NW) at (150:7);
            \coordinate (NE) at (30:7);
            \coordinate (SW) at (-150:7);
            \coordinate (SE) at (-30:7);
            \draw[webline] (SW) -- (NE);
            \draw[webline] (NW) -- (SE);
            \draw[fill=pink] (0,0) circle [radius=20pt];
        } 
    },&
    \mbox{
        \tikz[baseline=-.6ex, scale=.1, rotate=180]{
            \draw[dashed, fill=white] (0,0) circle [radius=7];
            \coordinate (NW) at (150:7);
            \coordinate (NE) at (30:7);
            \coordinate (SW) at (-150:7);
            \coordinate (SE) at (-30:7);
            \coordinate (R) at (3,0);
            \coordinate (L) at (-3,0);
            \draw[webline] (SW) -- (L) to[out=north east, in=north west] (R) -- (SE);
            \draw[webline] (NW) -- (L) -- (0,-4) -- (R) -- (NE);
            \draw[wline] (0,-4) -- (0,-7);
            \draw[fill=pink] (L) circle [radius=20pt];
            \draw[fill=pink] (R) circle [radius=20pt];
        }
    }=
    \mbox{
        \tikz[baseline=-.6ex, scale=.1, rotate=180]{
            \draw[dashed, fill=white] (0,0) circle [radius=7];
            \coordinate (NW) at (150:7);
            \coordinate (NE) at (30:7);
            \coordinate (SW) at (-150:7);
            \coordinate (SE) at (-30:7);
            \coordinate (R) at (3,0);
            \coordinate (L) at (-3,0);
            \draw[webline] (NW) -- (L) -- (SW);
            \draw[webline] (NE) to[bend right=60] (SE);
            \draw[wline] (L) to[out=east, in=north] (-90:7);
        }
    }
    +
    \mbox{
        \tikz[baseline=-.6ex, scale=.1, rotate=180]{
            \draw[dashed, fill=white] (0,0) circle [radius=7];
            \coordinate (NW) at (150:7);
            \coordinate (NE) at (30:7);
            \coordinate (SW) at (-150:7);
            \coordinate (SE) at (-30:7);
            \coordinate (R) at (3,0);
            \coordinate (L) at (-3,0);
            \draw[webline] (NW) to[bend left=60] (SW);
            \draw[webline] (NE) -- (R) -- (SE);
            \draw[wline] (R) to[out=west, in=north] (-90:7);
        }
    }.
    \end{align*}
    We can see that the above expansion creates elliptic faces adjacent to the top right of the $\fsp_4$-tripod labeled by $n-1$ except for $k=0$ or $l=0$. When $k=0$ or $l=0$, elliptic bordered faces along $\alpha$ or $\beta$ appears. Therefore, the statement holds for any $n\geq 2$.
\end{proof}

\begin{lem}\label{lem:tripod-ladder}
    For any positive integer $n$,
    \begin{align}
        \mbox{
            \tikz[baseline=-.6ex, scale=.1]{
                \coordinate (N) at (0,10);
                \coordinate (S) at (0,-10);
                \coordinate (W) at (-10,0);
                \coordinate (E) at (10,0);
                \coordinate (NE) at (45:10);
                \coordinate (NW) at (135:10);
                \coordinate (SE) at (-45:10);
                \coordinate (SW) at (-135:10);
                \coordinate (C) at (0,0);
                \coordinate (CN) at (0,5);
                \coordinate (CS) at (0,-5);
                \coordinate (CW) at (-5,0);
                \coordinate (CE) at (5,0);
                \begin{scope}
                    \clip (C) circle [radius=10cm];
                    \draw[wline] (-90:10) -- (C);
                    \draw[webline] (30:10) -- (C);
                    \draw[webline] (150:10) -- (C);
                    \node at (C) [scale=1.5]{$\tribox$};
                    \draw[blue] ($(150:8)+(60:10)$) -- ($(150:8)+(60:-10)$);
                    \draw[blue] (SW) -- (SE);
                \end{scope}
                \draw[dashed] (C) circle [radius=10cm];
                \node at (SW) [above left=10pt]{\scriptsize $\beta$};
                \node at (SW) [left]{\scriptsize $\alpha$};
                \node at (C) {\scriptsize $n$};
                \node at (-90:12) {\scriptsize $n$};
                \node at (30:12) {\scriptsize $n$};
                \node at (150:12) {\scriptsize $n$};
            }
        }
        \greq
        \mbox{
            \tikz[baseline=-.6ex, scale=.1]{
                \coordinate (N) at (0,10);
                \coordinate (S) at (0,-10);
                \coordinate (W) at (-10,0);
                \coordinate (E) at (10,0);
                \coordinate (NE) at (45:10);
                \coordinate (NW) at (135:10);
                \coordinate (SE) at (-45:10);
                \coordinate (SW) at (-135:10);
                \coordinate (C) at (0,0);
                \coordinate (CN) at (0,5);
                \coordinate (CS) at (0,-5);
                \coordinate (CW) at (-5,0);
                \coordinate (CE) at (5,0);
                \begin{scope}
                    \clip (C) circle [radius=10cm];
                    \draw[wline] (-90:10) -- (C);
                    \draw[webline] (30:10) -- (C);
                    \draw[webline] (150:10) -- (C);
                     \node at (C) [scale=1.5]{$\tribox$};
                    \draw[blue] ($(150:8)+(60:10)$) -- ($(150:8)+(60:-10)$);
                    \draw[blue] (SW) -- (SE);
                \end{scope}
                \draw[dashed] (C) circle [radius=10cm];
                \node at (C) {\scriptsize $L$};
                \node at (SW) [above left=10pt]{\scriptsize $\beta$};
                \node at (SW) [left]{\scriptsize $\alpha$};
                \node at (-90:12) {\scriptsize $n$};
                \node at (30:12) {\scriptsize $n$};
                \node at (150:12) {\scriptsize $n$};
            }
        }\label{eq:clasp-decomp}
    \end{align}
    holds in $\cG^{\{\alpha,\beta\}}\Skein{\fsp_4}{\bSigma}^{q}$ where $L$ is any ladder-resolution of $\fsp_4$-tripod.
\end{lem}

\begin{proof}
    By induction on $n$. For $n=2$, we can prove it by straightforward calculation using the $\fsp_4$-skein relation at a crossroad. Assume that any ladder resolution of the $\fsp_4$-tripod of degree smaller than $n$ can be replaced by the original $\fsp_4$-tripod in $\cG^{\{\alpha,\beta\}}\Skein{\fsp_4}{\bSigma}^{q}$. 
    For any ladder-resolution $L$ of the $\fsp_4$-tripod of degree $n$, we consider a decomposition of $L$ below:
    \begin{align*}
    \mbox
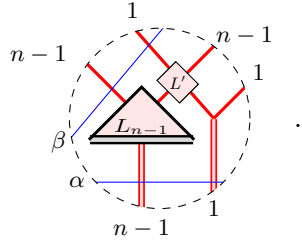
{
        \tikz[baseline=-.6ex, scale=.12]{
            \coordinate (N) at (0,10);
            \coordinate (S) at (0,-10);
            \coordinate (W) at (-10,0);
            \coordinate (E) at (10,0);
            \coordinate (NE) at (45:10);
            \coordinate (NW) at (135:10);
            \coordinate (SE) at (-45:10);
            \coordinate (SW) at (-135:10);
            \coordinate (C) at (0,0);
            \coordinate (CN) at (0,5);
            \coordinate (CS) at (0,-5);
            \coordinate (CW) at (-5,0);
            \coordinate (CE) at (5,0);
            \begin{scope}
                \clip (C) circle [radius=10cm];
                \draw[wline] ($(S)-(2,0)$) -- ($(C)-(2,0)$);
                \draw[webline, shorten >= -.3cm] ($(C)-(2,0)$) -- ($(NE)-(2,0)$);
                \draw[webline] ($(NW)-(2,0)$) -- ($(C)-(2,0)$);
                \node at ($(C)-(2,0)$) [scale=1.4]{\tribox};
                \draw[wline] ($(S)-(-6,0)$) -- ($(C)-(-6,0)$);
                \draw[webline] ($(C)-(-6,0)$) -- ($(NE)-(-6,0)$);
                \draw[webline, shorten <= -1cm] ($(C)!.6!(NE)-(2,0)$) -- ($(C)-(-6,0)$);
                \node at (2,4) [rotate=-45]{\tikz{\node at (0,0) [yscale=.5, xscale=.6]{\sqbox}}};
                    \draw[blue] ($(150:8)+(40:10)$) -- ($(150:8)+(60:-10)$);
                    \draw[blue] (SW) -- (SE);
            \end{scope}
            \draw[dashed] (C) circle [radius=10cm];
            \node at ($(C)-(2,1)$) [yscale=.8]{\scriptsize $L_{n-1}$};
            \node at (2,4) [xscale=.7, yscale=.7]{\scriptsize $L'$};
            \node at ($(S)-(2,0)$) [below]{\scriptsize $n-1$};
            \node at ($(NE)-(2,0)$) [above right]{\scriptsize $n-1$};
            \node at ($(NW)-(2,0)$) [left]{\scriptsize $n-1$};
            \node at ($(S)+(6,0)$) {\scriptsize $1$};
            \node at ($(NE)+(2,-2)$) [right]{\scriptsize $1$};
            \node at ($(N)+(-3,2)$) {\scriptsize $1$};
                \node at (SW) [above left=10pt]{\scriptsize $\beta$};
                \node at (SW) [left]{\scriptsize $\alpha$};
        }
    }.
    \end{align*}
\end{proof}
Then we can replace any rung in $L'$ with a crossroad by induction assumption. We apply the $\fsp_4$-skein relation to $n-1$ rungs in $L'$ to replace all rungs with crossroads. This expansion is described as 
    \begin{align*}
    \mbox{
        \tikz[baseline=-.6ex, scale=.12]{
            \coordinate (N) at (0,10);
            \coordinate (S) at (0,-10);
            \coordinate (W) at (-10,0);
            \coordinate (E) at (10,0);
            \coordinate (NE) at (45:10);
            \coordinate (NW) at (135:10);
            \coordinate (SE) at (-45:10);
            \coordinate (SW) at (-135:10);
            \coordinate (C) at (0,0);
            \coordinate (CN) at (0,5);
            \coordinate (CS) at (0,-5);
            \coordinate (CW) at (-5,0);
            \coordinate (CE) at (5,0);
            \begin{scope}
                \clip (C) circle [radius=10cm];
                \draw[wline] ($(S)-(2,0)$) -- ($(C)-(2,0)$);
                \draw[webline, shorten >= -.3cm] ($(C)-(2,0)$) -- ($(NE)-(2,0)$);
                \draw[webline] ($(NW)-(2,0)$) -- ($(C)-(2,0)$);
                \node at ($(C)-(2,0)$) [scale=1.4]{\tribox};
                \draw[wline] ($(S)-(-6,0)$) -- ($(C)-(-6,0)$);
                \draw[webline] ($(C)-(-6,0)$) -- ($(NE)-(-6,0)$);
                \draw[webline, shorten <= -1cm] ($(C)!.6!(NE)-(2,0)$) -- ($(C)-(-6,0)$);
                \node at (2,4) [rotate=-45]{\tikz{\node at (0,0) [yscale=.5, xscale=.6]{\sqbox}}};
                    \draw[blue] ($(150:8)+(40:10)$) -- ($(150:8)+(60:-10)$);
                    \draw[blue] (SW) -- (SE);
            \end{scope}
            \draw[dashed] (C) circle [radius=10cm];
            \node at ($(C)-(2,0)$) {\scriptsize $n-1$};
            \node at (2,4) [xscale=.7, yscale=.7]{\scriptsize $L'$};
            \node at ($(S)-(2,0)$) [below]{\scriptsize $n-1$};
            \node at ($(NE)-(2,0)$) [above right]{\scriptsize $n-1$};
            \node at ($(NW)-(2,0)$) [left]{\scriptsize $n-1$};
            \node at ($(S)+(6,0)$) {\scriptsize $1$};
            \node at ($(NE)+(2,-2)$) [right]{\scriptsize $1$};
            \node at ($(N)+(-3,2)$) {\scriptsize $1$};
                \node at (SW) [above left=10pt]{\scriptsize $\beta$};
                \node at (SW) [left]{\scriptsize $\alpha$};
        }
    }
    =
    \mbox{
        \tikz[baseline=-.6ex, scale=.12]{
            \coordinate (N) at (0,10);
            \coordinate (S) at (0,-10);
            \coordinate (W) at (-10,0);
            \coordinate (E) at (10,0);
            \coordinate (NE) at (45:10);
            \coordinate (NW) at (135:10);
            \coordinate (SE) at (-45:10);
            \coordinate (SW) at (-135:10);
            \coordinate (C) at (0,0);
            \coordinate (CN) at (0,5);
            \coordinate (CS) at (0,-5);
            \coordinate (CW) at (-5,0);
            \coordinate (CE) at (5,0);
            \begin{scope}
                \clip (C) circle [radius=10cm];
                \draw[wline] ($(S)-(2,0)$) -- ($(C)-(2,0)$);
                \draw[webline, shorten >= -.3cm] ($(C)-(2,0)$) -- ($(NE)-(2,0)$);
                \draw[webline] ($(NW)-(2,0)$) -- ($(C)-(2,0)$);
                \node at ($(C)-(2,0)$) [scale=1.4]{\tribox};
                \draw[wline] ($(S)-(-6,0)$) -- ($(C)-(-6,0)$);
                \draw[webline] ($(C)-(-6,0)$) -- ($(NE)-(-6,0)$);
                \draw[webline, shorten <= -1cm] ($(C)!.6!(NE)-(2,0)$) -- ($(C)-(-6,0)$);
                \node at (2,4) [rotate=-45]{\tikz{\node at (0,0) [yscale=.5, xscale=.6]{\sqbox}}};
                    \draw[blue] ($(150:8)+(40:10)$) -- ($(150:8)+(60:-10)$);
                    \draw[blue] (SW) -- (SE);
            \end{scope}
            \draw[dashed] (C) circle [radius=10cm];
            \node at ($(C)-(2,0)$) {\scriptsize $n-1$};
            \node at ($(S)-(2,0)$) [below]{\scriptsize $n-1$};
            \node at ($(NE)-(2,0)$) [above right]{\scriptsize $n-1$};
            \node at ($(NW)-(2,0)$) [left]{\scriptsize $n-1$};
            \node at ($(S)+(6,0)$) {\scriptsize $1$};
            \node at ($(NE)+(2,-2)$) [right]{\scriptsize $1$};
            \node at ($(N)+(-3,2)$) {\scriptsize $1$};
                \node at (SW) [above left=10pt]{\scriptsize $\beta$};
                \node at (SW) [left]{\scriptsize $\alpha$};
        }
    }
    +\sum_{X}
    \mbox{
        \tikz[baseline=-.6ex, scale=.12]{
            \coordinate (N) at (0,10);
            \coordinate (S) at (0,-10);
            \coordinate (W) at (-10,0);
            \coordinate (E) at (10,0);
            \coordinate (NE) at (45:10);
            \coordinate (NW) at (135:10);
            \coordinate (SE) at (-45:10);
            \coordinate (SW) at (-135:10);
            \coordinate (C) at (0,0);
            \coordinate (CN) at (0,5);
            \coordinate (CS) at (0,-5);
            \coordinate (CW) at (-5,0);
            \coordinate (CE) at (5,0);
            \begin{scope}
                \clip (C) circle [radius=10cm];
                \draw[wline] ($(S)-(2,0)$) -- ($(C)-(2,0)$);
                \draw[webline, shorten >= -.3cm] ($(C)-(2,0)$) -- ($(NE)-(2,0)$);
                \draw[webline] ($(NW)-(2,0)$) -- ($(C)-(2,0)$);
                \node at ($(C)-(2,0)$) [scale=1.4]{\tribox};
                \draw[wline] ($(S)-(-6,0)$) -- ($(C)-(-6,0)$);
                \draw[webline] ($(C)-(-6,0)$) -- ($(NE)-(-6,0)$);
                \draw[webline, shorten <= -1cm] ($(C)!.6!(NE)-(2,0)$) -- ($(C)-(-6,0)$);
                \node at (2,4) [rotate=-45]{\tikz{\node at (0,0) [yscale=.5, xscale=.6]{\sqbox}}};
                    \draw[blue] ($(150:8)+(40:10)$) -- ($(150:8)+(60:-10)$);
                    \draw[blue] (SW) -- (SE);
            \end{scope}
            \draw[dashed] (C) circle [radius=10cm];
            \node at ($(C)-(2,0)$) {\scriptsize $n-1$};
            \node at (2,4) [xscale=.7, yscale=.7]{\scriptsize $X$};
            \node at ($(S)-(2,0)$) [below]{\scriptsize $n-1$};
            \node at ($(NE)-(2,0)$) [above right]{\scriptsize $n-1$};
            \node at ($(NW)-(2,0)$) [left]{\scriptsize $n-1$};
            \node at ($(S)+(6,0)$) {\scriptsize $1$};
            \node at ($(NE)+(2,-2)$) [right]{\scriptsize $1$};
            \node at ($(N)+(-3,2)$) {\scriptsize $1$};
                \node at (SW) [above left=10pt]{\scriptsize $\beta$};
                \node at (SW) [left]{\scriptsize $\alpha$};
        }
    },
    \end{align*}
    where $X$ is obtained from $L'$ by replacing any $a$ ($0\geq a<n-1$) of its rungs with crossroads and removing the remaining $n-1-a$ rungs. Then, we can see that the latter terms containing $X$ have the same $\fsp_4$-diagrams in \cref{lem:tripod-clasp}. Hence, these terms vanish in the graded quotient $\cG^{\{\alpha,\beta\}}\Skein{\fsp_4}{\bSigma}^{q}$.

\begin{proof}[Proof of \cref{lem:gr-triangle}]
    Now $D\cap T$ is reduced, hence it is described by the $\fsp_4$-diagram with two $\fsp_4$-tripods as \cref{fig:pyramid}. For any $\fsp_4$-tripod of $D\cap T$, two edges are adjacent to cut paths, and we can apply \cref{lem:tripod-ladder}.  
\end{proof}

\begin{dfn}[good position]
    Let $S$ be a set of cut paths in $\bSigma$.
    \begin{itemize}
        \item We denote by $S^{\mathrm{split}}$ a cellular decomposition of $\Sigma$ obtained by replacing the set of ideal arcs $S$ by $\{\gamma_{+},\gamma_{-}\mid \gamma\in S\}$ where $\gamma_{+}$ and $\gamma_{-}$ bound a biangle $B_{\gamma}$ containing $\gamma$ as in \cref{fig:ladder-web}~{(A)}. 
        We denote by $b(S^{\mathrm{split}})$ the set of biangles $\{B_{\gamma}\mid \gamma\in S\}$, and by $t(S^{\mathrm{split}})$ the set of connected components in $\Sigma\setminus S^{\mathrm{split}}$ other than $b(S^{\mathrm{split}})$.
        \item We call $\splittri$ \emph{split triangulation} for an ideal triangulation $\tri$
        \item An $S^{\mathrm{split}}$-transverse $\fsp_4$-diagram $D\in\diag_{\fsp_4,\bSigma}$ is in a \emph{good position} with respect to $S^{\mathrm{split}}$ if $D\cap B_{\gamma}$ is non-elliptic and $D\cap \Sigma'$ is reduced for every biangles $B_{\gamma}$ and surfaces $\Sigma'$ in $t(S^{\mathrm{split}})$. 
    \end{itemize}
\end{dfn}

\begin{cor}\label{cor:ladder-invariance}
    Let $\tri$ be an ideal triangulation of $\bSigma$, and $D\in\diag_{\fsp_4,\bSigma}$ is an $\tri$-minimal crossroad $\mathfrak{sp}_4$-diagram in good position with respect to a split ideal triangulation $\splittri$. 
    Then, $\diagshift(D)\greq \diagshift(D')$ for any ladder resolution $D'$ of $D$. 
\end{cor}

\begin{cor}\label{cor:ladder-representative}
    For any ideal triangulation $\tri$, there exists a canonical map $\Blad_{\fsp_4,\bSigma}\to\mathcal{G}^{\tri}\Skein{\mathfrak{sp}_4}{\bSigma}^{q}$ which sends a $\tri$-minimal representative $D$ to $W_D$.
\end{cor}

\subsection{Reconstruction from coordinates}\label{subsec:reconstruction}
We are going to prove \cref{thm:bijection}, namely that the coordinate map $\bsfa^{\bD}$ is bijective, by providing an explicit reconstruction procedure from a given coordinate vector. 
The reconstruction procedure will be equivariant under the scaling action of $\bQ_{>0}$ and the peripheral action $\alpha_\bM$ (\cref{subsec:peripheral_action}). Observe that:
\begin{itemize}
    \item For any vector $\bsfa \in \bQ^{I(\bD)}$, there exists a sufficiently large integer $\lambda \in \bZ_{>0}$ such that $\lambda\cdot \bsfa \in \bZ^{I(\bD)}$. 
    \item Let $\mathsf{v}^{\bD}_{m,s} \in \bZ^{I(\bD)}$ denote the coordinate vector of the peripheral component around $m \in \bM$ of type $s=1,2$. By \cref{rem:potential_vector} and \cref{lem:action_potential}, we have the potential condition $\pot_{m,s}^{\bD}(\sfa+ u_{m,s}\mathsf{v}^{\bD}_{m,s})=\pot_{m,s}^{\bD}(\alpha_\bM(u,L)) \geq 0$ for sufficiently large integer tuple $u=(u_{m,s}) \in (\bZ^2_{>0})^{\bM}$. 
\end{itemize}

\subsubsection{Reconstruction (surjectivity)}
Based on the observations above, let us begin with an integral vector $\bsfa=(\sfa_i) \in \bZ^{I(\bD)}$ satisfying the potential condition $\pot_{m,s}^{\bD}(\bsfa) \geq 0$ for all $m \in \bM$ and $s=1,2$. Then for each $T\in t(\tri^{\mathrm{split}})$, it provides a vector $\bsfa_{T}=(\sfa_i) \in \vect_{T,(\tri,m_T,\bs_T)}$. Here recall \cref{def:Blad_vect} (1), and that the potential condition on $T$ is satisfied by \cref{lem:potential_local}. 

The inverse map of $\sfa_{\bD}$ is constructed in the following steps. 
\begin{description}
    \item [Step~1] For any $T\in t(\tri^{\mathrm{split}})$, there exists a unique reduced $\fsp_4$-diagram $D_T$ in $T$ up to flips of corner arcs such that $\bsfa^{(\tri,m_T,\bs_T)}(D_T) = \bsfa_T$ by \cref{lem:congruence}, \cref{lem:lamination_Blad} and \cref{lem:triangle_blad}.
    \item [Step~2] One can construct an $\fsp_4$-diagram $D$ on $\Sigma$ by gluing $D_{T}$ by flat braids in biangles. See \cref{fig:ladder-web}. The resulting diagram may contain elliptic faces.
    \item [Step~3] $D$ can be deformed into an $\fsp_4$-diagram $D_{\tri}$ with no bigons by a sequence of flat Reidemeister moves as in \cref{rem:flat-braid}. This deformation only involves flips of corner arcs at triangles in $t(\tri^{\mathrm{split}})$, hence the coordinate is invariant. One can easily check that $D_{\tri}$ has no elliptic faces and we obtain a reduced $\fsp_4$-diagram $D_{\tri}$ realizing the given coordinate $\bsfa=(\sfa_i) \in \bZ^{I(\bD)}$ by boundary $H$-moves. \Cref{cor:elliptic-minimal} concludes the $\tri$-minimality of $D_{\tri}$. 
\end{description}

In this way, we get a ladder-equivalence class $L(\bsfa) \in \Blad_{\fsp_4,\bSigma}$ represented by $D_{\tri}$. 


\begin{prop}\label{prop:surjective}
The coordinate map $\bsfa^{\bD}: \cL^a(\bSigma;\bQ) \to \bQ^{I(\bD)}$ is surjective. 
\end{prop}

\begin{proof}
Given a coordinate vector $\bsfa \in \bQ^{I(\bD)}$, 
\begin{itemize}
    \item Take a positive integer $\lambda \in \bZ_{>0}$ such that $\lambda \bsfa \in \bZ^{I(\bD)}$. 
    \item Take a positive integer $c_{m,s} \in \bZ_{>0}$ for each $m \in \bM$ and $s=1,2$ such that the vector $\bsfa':=\lambda \bsfa + \sum_{m,s} c_{m,s} \mathsf{v}_{m,s}$ satisfies the potential condition $\pot_{m,s}^{\bD}(\bsfa') \geq 0$. Here $\mathsf{v}_{m,s} \in \bZ^{I(\bD)}$ is the coordinate vector of the peripheral component around $m \in \bM$ of type $s=1,2$ (see \cref{rem:potential_vector}).
\end{itemize}
Then we define $\sfr^{\bD}(\bsfa)$ to be the rational bounded $\fsp_4$-lamination obtained by the union of $\lambda^{-1}\cdot L(\bsfa')$ and the peripheral component of type $s$ around each marked point $m$ of weight $-\lambda^{-1}c_{m,s}$. Then it satisfies $\bsfa^{\bD}(\sfr^{\bD}(\bsfa))=\bsfa$ by construction. Thus the coordinate map $\bsfa^{\bD}$ is surjective.
\end{proof}


\subsubsection{Uniqueness (injectivity)}
Let $L, L'\in\Blad_{\fsp_4,\bSigma}$ be two ladder-equivalence classes such that $\bsfa^\bD(L)=\bsfa^\bD(L')$, and $D,D'\in\Bdiag_{\fsp_4,\bSigma}$ be their reduced representatives. 
Here, let us denote by $L_{D}\in\Blad_{\fsp_4,\bSigma}$ and $W_{D}\in\Skein{\fsp_4}{\bSigma}^{q}$ the ladder-equivalence class of $D$ and $\fsp_4$-web $\diagshift(D)$ respectively.
\begin{description}
    \item [Step~A] The classes $L=L_{D}$ and $L'=L_{D'}$ have $\tri^{\mathrm{split}}$-minimal representatives $D_{\tri}$ and $D_{\tri}'$ that realize the given coordinates and located in good positions with respect to $\tri^{\mathrm{split}}$. 
    They are obtained from $D$ and $D'$ by a sequence of intersection reduction moves, according to \cref{cor:good-position}. 
    Note that $L=L_{D}=L_{D_{\tri}}$ and $L'=L_{D'}=L_{D'_{\tri}}$ hold.
    
    Furthermore, for each $T\in t(\tri^{\mathrm{split}})$, $D_{\tri}\cap T$ and $D'_{\tri}\cap T$ are reduced $\fsp_4$-diagram with the same coordinate. Hence they are related by a sequence of flips of corner arcs by \cref{lem:triangle_case}, \cref{lem:lamination_Blad}, and \cref{lem:triangle_blad}.
    \item [Step~B] $D_{\tri}$ can be deformed into $D''\in\Bdiag_{\fsp_4,\bSigma}$ by a sequence of flat Reidemeister moves~{II} (\cref{rem:flat-braid}) at corners of triangles so that $D''\cap T$ coincides with $D'_{\tri}\cap T$ at all $T\in t(\tri^{\mathrm{split}})$. See, for example, \cref{fig:bigon-birth}.
    \item [Step~C] For $\fsp_4$-webs $W_{D_{\tri}}$ and $W_{D''}$, their images under $\pi_{\tri}\colon\Skein{\fsp_4}{\bSigma}^{q}\to\mathcal{G}^{\tri}\Skein{\fsp_4}{\bSigma}^{q}$ coincide by \cref{rem:flat-braid}. Namely, $\pi_{\tri}(W_{D_{\tri}})=\pi_{\tri}(W_{D''})$ holds in $\cG^{\tri}\Skein{\fsp_4}{\bSigma}$.
    Moreover, for any $\gamma\in\tri$, the flat braid diagram $D''\cap B_{\gamma}$ in a biangle $B_{\gamma}$ of $\tri^{\mathrm{split}}$ is related to $D_{\tri}'\cap B$ by a sequence of flat Reidemeister moves in \cref{rem:flat-braid}.
    Hence $\pi_{\tri}(W_{D''})\greq\pi_{\tri}(W_{D_{\tri}'})$.
    Consequently, we obtain $\pi_{\tri}(W_{D_{\tri}})=\pi_{\tri}(W_{D_{\tri}}')$.
    \item [Step~D] $\pi_{\tri}(W_{D})=\pi_{\tri}(W_{D'})$ because $W_{D}=W_{D_{\tri}}$ and $W_{D'}=W_{D_{\tri}'}$ in $\Skein{\fsp_4}{\bSigma}^{q}$, see the construction of $D_{\tri}$ and $D_{\tri}'$ in $\textbf{Step~A}$.
    Since $D,D'\in\Bdiag_{\fsp_4,\bSigma}$ and the above equality holds, it follows from the proof of \cref{lem:gr-basis} that $W_{D}=W_{D'}$ in $\Skein{\fsp_4}{\bSigma}^{q}$.
    $W_{D}=W_{D'}$ means $L_{D}=L_{D'}$ because $\ladshift |_{\Blad_{\fsp_4,\bSigma}}\colon\Blad_{\fsp_4,\Sigma}\to\Bweb_{\fsp_4,\Sigma}$ is bijective, see below \cref{thm:basis-web}.
\end{description}
Hence, we obtain the following theorem.
\begin{thm}\label{thm:injectivity}
    Let $D$ and $D'$ be non-elliptic $\fsp_4$-diagrams.
    Let $D_{\tri}$ and $D'_{\tri}$ denote $\tri$-minimal $\mathfrak{sp}_4$-diagrams of $D$ and $D'$, respectively, that are in a good position with respect to the split triangulation $\tri^{\mathrm{split}}$.
    If $D_{\tri}\cap T$ and $D'_{\tri}\cap T$ have the same reduced $\fsp_4$-diagram except for arc parallel-moves
    for any $T\in t(\tri)$
    then $D\approx D'$. 
\end{thm}

\begin{figure}
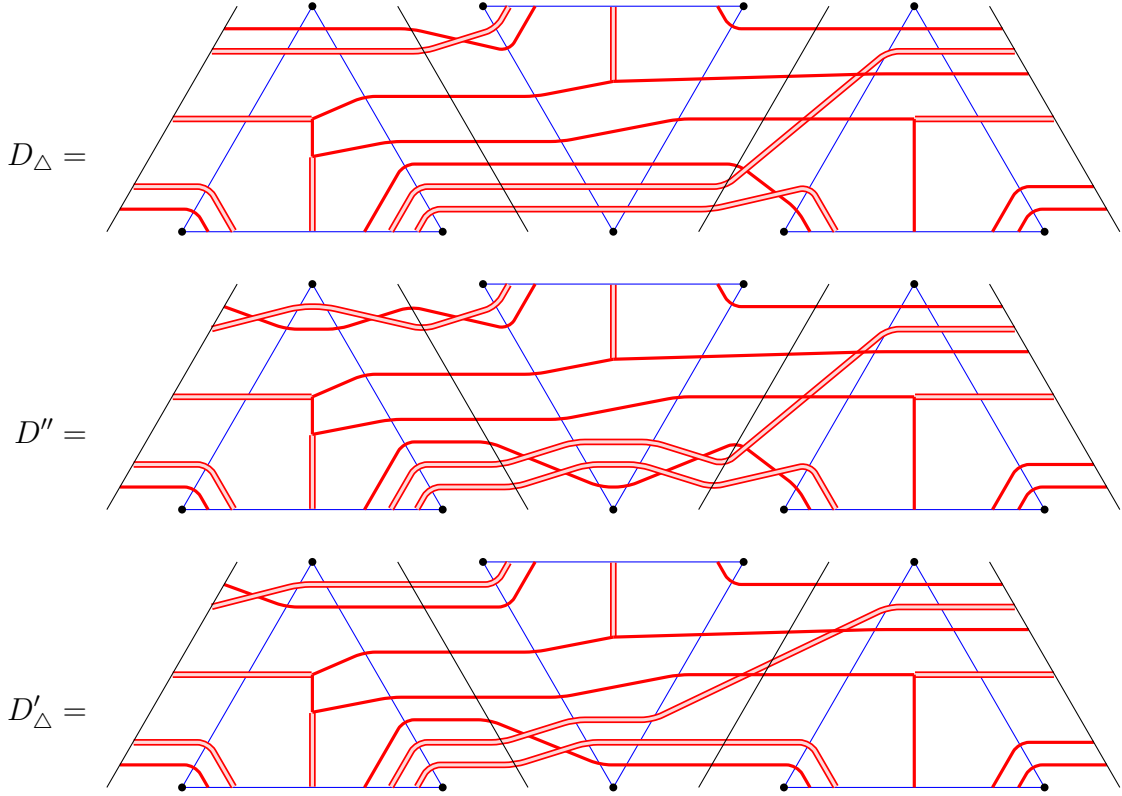

    \centering
    \begin{align*}
    D_{\tri}&=\mbox{
        \tikz[baseline=-.6ex, scale=.1]{
            \coordinate (C1) at (0,0);
            \coordinate (C2) at (40,0);
            \coordinate (C3) at (80,0);
            \coordinate (P1) at (90:20);
            \coordinate (P2) at (210:20);
            \coordinate (P3) at (-30:20);
            \coordinate (Q1) at ($(C2)+(-90:20)+(0,10)$);
            \coordinate (Q2) at ($(C2)+(30:20)+(0,10)$);
            \coordinate (Q3) at ($(C2)+(150:20)+(0,10)$);
            \coordinate (R1) at ($(C3)+(P1)$);
            \coordinate (R2) at ($(C3)+(P2)$);
            \coordinate (R3) at ($(C3)+(P3)$);
            \coordinate (K1) at ($(P1)+(-10,0)$);
            \coordinate (K2) at ($(P2)+(-10,0)$);
            \coordinate (L1) at ($(P1)!.5!(Q3)$);
            \coordinate (L2) at ($(P3)!.5!(Q1)$);
            \coordinate (M1) at ($(Q2)!.5!(R1)$);
            \coordinate (M2) at ($(Q1)!.5!(R2)$);
            \coordinate (N1) at ($(R1)+(10,0)$);
            \coordinate (N2) at ($(R3)+(10,0)$);
            \foreach \j in {1,2,...,9}
            {
            \coordinate (K1\j) at ($(K1)!.\j!(K2)$);
            \coordinate (L1\j) at ($(L1)!.\j!(L2)$);
            \coordinate (M1\j) at ($(M1)!.\j!(M2)$);
            \coordinate (N1\j) at ($(N1)!.\j!(N2)$);
            \coordinate (PW\j) at ($(P1)!.\j!(P2)$);
            \coordinate (PE\j) at ($(P1)!.\j!(P3)$);
            \coordinate (PS\j) at ($(P2)!.\j!(P3)$);
            \coordinate (QW\j) at ($(Q3)!.\j!(Q1)$);
            \coordinate (QE\j) at ($(Q2)!.\j!(Q1)$);
            \coordinate (QN\j) at ($(Q3)!.\j!(Q2)$);
            \coordinate (RW\j) at ($(R1)!.\j!(R2)$);
            \coordinate (RE\j) at ($(R1)!.\j!(R3)$);
            \coordinate (RS\j) at ($(R2)!.\j!(R3)$);
            }
            \draw[blue] (P1) -- (P2) -- (P3)--cycle;
            \draw[blue] (Q1) -- (Q2) -- (Q3)--cycle;
            \draw[blue] (R1) -- (R2) -- (R3)--cycle;
            \begin{scope}
                \clip (K1) -- (K2) -- (M2) -- (M1) -- cycle;
            \end{scope}
            \draw [webline, rounded corners] (K11) -- (PW1) -- (PE1) -- (L11) -- (QW2) -- (QN2);
            \draw [wline, rounded corners] (K12) -- (PW2) -- (PE2) -- (L12) -- (QW1) -- (QN1);
            \draw [wline, rounded corners] (K18) -- (PW8) -- (PS2);
            \draw [webline, rounded corners] (K19) -- (PW9) -- (PS1);
            \draw [webline, rounded corners] (PS7) -- (PE7) -- (QW7) -- (QE7) -- (M17) -- (RW9) -- (RS1);
            \draw [wline, rounded corners] (PS8) -- (PE8) -- (QW8) -- (QE8) -- (M18) -- (RW2) -- (RE2) -- (N12);
            \draw [wline, rounded corners] (PS9) -- (PE9) -- (QW9) -- (QE9) -- (M19) -- (RW8) -- (RS2);
            \draw [webline, rounded corners] (QN9) -- (QE1) -- (RW1) -- (RE1) -- (N11);
            \draw [webline, rounded corners] (RS8) -- (RE8) -- (N18);
            \draw [webline, rounded corners] (RS9) -- (RE9) -- (N19);
            \draw [wline] ($(C3)+(0,5)$) -- (N15);
            \draw [wline] ($(C2)+(0,10)$) -- (QN5);
            \draw [wline] ($(C1)+(0,5)$) -- (K15);
            \draw [wline] (C1) -- (PS5);
            \draw [webline, rounded corners] (N13) -- (RW3) -- ($(C2)+(0,10)$);
            \draw [webline, rounded corners] ($(C2)+(0,10)$) -- (QW4) -- (PE4) -- ($(C1)+(0,5)$);
            \draw [webline] ($(C1)+(0,5)$) -- (C1);
            \draw [webline, rounded corners] (C1) -- (PE6) -- (QW6) -- (QE5) -- ($(C3)+(0,5)$);
            \draw [webline] ($(C3)+(0,5)$) -- (RS5);
            \draw (K1)--(K2);
            \draw (L1)--(L2);
            \draw (M1)--(M2);
            \draw (N1)--(N2);
            \foreach \i in {1,2,3} \fill (P\i) circle(15pt);
            \foreach \i in {1,2,3} \fill (Q\i) circle(15pt);
            \foreach \i in {1,2,3} \fill (R\i) circle(15pt);
        }
    }\\[1em]
    D''&=\mbox{
        \tikz[baseline=-.6ex, scale=.1]{
            \coordinate (C1) at (0,0);
            \coordinate (C2) at (40,0);
            \coordinate (C3) at (80,0);
            \coordinate (P1) at (90:20);
            \coordinate (P2) at (210:20);
            \coordinate (P3) at (-30:20);
            \coordinate (Q1) at ($(C2)+(-90:20)+(0,10)$);
            \coordinate (Q2) at ($(C2)+(30:20)+(0,10)$);
            \coordinate (Q3) at ($(C2)+(150:20)+(0,10)$);
            \coordinate (R1) at ($(C3)+(P1)$);
            \coordinate (R2) at ($(C3)+(P2)$);
            \coordinate (R3) at ($(C3)+(P3)$);
            \coordinate (K1) at ($(P1)+(-10,0)$);
            \coordinate (K2) at ($(P2)+(-10,0)$);
            \coordinate (L1) at ($(P1)!.5!(Q3)$);
            \coordinate (L2) at ($(P3)!.5!(Q1)$);
            \coordinate (M1) at ($(Q2)!.5!(R1)$);
            \coordinate (M2) at ($(Q1)!.5!(R2)$);
            \coordinate (N1) at ($(R1)+(10,0)$);
            \coordinate (N2) at ($(R3)+(10,0)$);
            \foreach \j in {1,2,...,9}
            {
            \coordinate (K1\j) at ($(K1)!.\j!(K2)$);
            \coordinate (L1\j) at ($(L1)!.\j!(L2)$);
            \coordinate (M1\j) at ($(M1)!.\j!(M2)$);
            \coordinate (N1\j) at ($(N1)!.\j!(N2)$);
            \coordinate (PW\j) at ($(P1)!.\j!(P2)$);
            \coordinate (PE\j) at ($(P1)!.\j!(P3)$);
            \coordinate (PS\j) at ($(P2)!.\j!(P3)$);
            \coordinate (QW\j) at ($(Q3)!.\j!(Q1)$);
            \coordinate (QE\j) at ($(Q2)!.\j!(Q1)$);
            \coordinate (QN\j) at ($(Q3)!.\j!(Q2)$);
            \coordinate (RW\j) at ($(R1)!.\j!(R2)$);
            \coordinate (RE\j) at ($(R1)!.\j!(R3)$);
            \coordinate (RS\j) at ($(R2)!.\j!(R3)$);
            }
            \draw[blue] (P1) -- (P2) -- (P3)--cycle;
            \draw[blue] (Q1) -- (Q2) -- (Q3)--cycle;
            \draw[blue] (R1) -- (R2) -- (R3)--cycle;
            \begin{scope}
                \clip (K1) -- (K2) -- (M2) -- (M1) -- cycle;
            \end{scope}
            \draw [webline, rounded corners] (K11) -- (PW2) -- (PE2) -- (L11) -- (QW2) -- (QN2);
            \draw [wline, rounded corners] (K12) -- (PW1) -- (PE1) -- (L12) -- (QW1) -- (QN1);
            \draw [wline, rounded corners] (K18) -- (PW8) -- (PS2);
            \draw [webline, rounded corners] (K19) -- (PW9) -- (PS1);
            \draw [webline, rounded corners] (PS7) -- (PE7) -- (L17) -- (QW9) -- (QE9) -- (M17) -- (RW9) -- (RS1);
            \draw [wline, rounded corners] (PS8) -- (PE8) -- (L18) -- (QW7) -- (QE7) -- (M18) -- (RW2) -- (RE2) -- (N12);
            \draw [wline, rounded corners] (PS9) -- (PE9) -- (L19) -- (QW8) -- (QE8) -- (M19) -- (RW8) -- (RS2);
            \draw [webline, rounded corners] (QN9) -- (QE1) -- (RW1) -- (RE1) -- (N11);
            \draw [webline, rounded corners] (RS8) -- (RE8) -- (N18);
            \draw [webline, rounded corners] (RS9) -- (RE9) -- (N19);
            \draw [wline] ($(C3)+(0,5)$) -- (N15);
            \draw [wline] ($(C2)+(0,10)$) -- (QN5);
            \draw [wline] ($(C1)+(0,5)$) -- (K15);
            \draw [wline] (C1) -- (PS5);
            \draw [webline, rounded corners] (N13) -- (RW3) -- ($(C2)+(0,10)$);
            \draw [webline, rounded corners] ($(C2)+(0,10)$) -- (QW4) -- (PE4) -- ($(C1)+(0,5)$);
            \draw [webline] ($(C1)+(0,5)$) -- (C1);
            \draw [webline, rounded corners] (C1) -- (PE6) -- (QW6) -- (QE5) -- ($(C3)+(0,5)$);
            \draw [webline] ($(C3)+(0,5)$) -- (RS5);
            \draw (K1)--(K2);
            \draw (L1)--(L2);
            \draw (M1)--(M2);
            \draw (N1)--(N2);
            \foreach \i in {1,2,3} \fill (P\i) circle(15pt);
            \foreach \i in {1,2,3} \fill (Q\i) circle(15pt);
            \foreach \i in {1,2,3} \fill (R\i) circle(15pt);
        }
    }\\[1em]
    D'_{\tri}&=\mbox{
        \tikz[baseline=-.6ex, scale=.1]{
            \coordinate (C1) at (0,0);
            \coordinate (C2) at (40,0);
            \coordinate (C3) at (80,0);
            \coordinate (P1) at (90:20);
            \coordinate (P2) at (210:20);
            \coordinate (P3) at (-30:20);
            \coordinate (Q1) at ($(C2)+(-90:20)+(0,10)$);
            \coordinate (Q2) at ($(C2)+(30:20)+(0,10)$);
            \coordinate (Q3) at ($(C2)+(150:20)+(0,10)$);
            \coordinate (R1) at ($(C3)+(P1)$);
            \coordinate (R2) at ($(C3)+(P2)$);
            \coordinate (R3) at ($(C3)+(P3)$);
            \coordinate (K1) at ($(P1)+(-10,0)$);
            \coordinate (K2) at ($(P2)+(-10,0)$);
            \coordinate (L1) at ($(P1)!.5!(Q3)$);
            \coordinate (L2) at ($(P3)!.5!(Q1)$);
            \coordinate (M1) at ($(Q2)!.5!(R1)$);
            \coordinate (M2) at ($(Q1)!.5!(R2)$);
            \coordinate (N1) at ($(R1)+(10,0)$);
            \coordinate (N2) at ($(R3)+(10,0)$);
            \foreach \j in {1,2,...,9}
            {
            \coordinate (K1\j) at ($(K1)!.\j!(K2)$);
            \coordinate (L1\j) at ($(L1)!.\j!(L2)$);
            \coordinate (M1\j) at ($(M1)!.\j!(M2)$);
            \coordinate (N1\j) at ($(N1)!.\j!(N2)$);
            \coordinate (PW\j) at ($(P1)!.\j!(P2)$);
            \coordinate (PE\j) at ($(P1)!.\j!(P3)$);
            \coordinate (PS\j) at ($(P2)!.\j!(P3)$);
            \coordinate (QW\j) at ($(Q3)!.\j!(Q1)$);
            \coordinate (QE\j) at ($(Q2)!.\j!(Q1)$);
            \coordinate (QN\j) at ($(Q3)!.\j!(Q2)$);
            \coordinate (RW\j) at ($(R1)!.\j!(R2)$);
            \coordinate (RE\j) at ($(R1)!.\j!(R3)$);
            \coordinate (RS\j) at ($(R2)!.\j!(R3)$);
            }
            \draw[blue] (P1) -- (P2) -- (P3)--cycle;
            \draw[blue] (Q1) -- (Q2) -- (Q3)--cycle;
            \draw[blue] (R1) -- (R2) -- (R3)--cycle;
            \begin{scope}
                \clip (K1) -- (K2) -- (M2) -- (M1) -- cycle;
            \end{scope}
            \draw [webline, rounded corners] (K11) -- (PW2) -- (PE2) -- (L12) -- (QW2) -- (QN2);
            \draw [wline, rounded corners] (K12) -- (PW1) -- (PE1) -- (L11) -- (QW1) -- (QN1);
            \draw [wline, rounded corners] (K18) -- (PW8) -- (PS2);
            \draw [webline, rounded corners] (K19) -- (PW9) -- (PS1);
            \draw [webline, rounded corners] (PS7) -- (PE7) -- (L17) -- (QW9) -- (QE9) -- (RW9) -- (RS1);
            \draw [wline, rounded corners] (PS8) -- (PE8) -- (L18) -- (QW7) -- (QE7) -- (RW2) -- (RE2) -- (N12);
            \draw [wline, rounded corners] (PS9) -- (PE9) -- (L19) -- (QW8) -- (QE8) -- (RW8) -- (RS2);
            \draw [webline, rounded corners] (QN9) -- (QE1) -- (RW1) -- (RE1) -- (N11);
            \draw [webline, rounded corners] (RS8) -- (RE8) -- (N18);
            \draw [webline, rounded corners] (RS9) -- (RE9) -- (N19);
            \draw [wline] ($(C3)+(0,5)$) -- (N15);
            \draw [wline] ($(C2)+(0,10)$) -- (QN5);
            \draw [wline] ($(C1)+(0,5)$) -- (K15);
            \draw [wline] (C1) -- (PS5);
            \draw [webline, rounded corners] (N13) -- (RW3) -- ($(C2)+(0,10)$);
            \draw [webline, rounded corners] ($(C2)+(0,10)$) -- (QW4) -- (PE4) -- ($(C1)+(0,5)$);
            \draw [webline] ($(C1)+(0,5)$) -- (C1);
            \draw [webline, rounded corners] (C1) -- (PE6) -- (QW6) -- (QE5) -- ($(C3)+(0,5)$);
            \draw [webline] ($(C3)+(0,5)$) -- (RS5);
            \draw (K1)--(K2);
            \draw (L1)--(L2);
            \draw (M1)--(M2);
            \draw (N1)--(N2);
            \foreach \i in {1,2,3} \fill (P\i) circle(15pt);
            \foreach \i in {1,2,3} \fill (Q\i) circle(15pt);
            \foreach \i in {1,2,3} \fill (R\i) circle(15pt);
        }
    }
    \end{align*}
    \caption{$D_{\tri}$ and $D'_{\tri}$ are non-elliptic with the same coordinate. $D''$ is rearrangement of corner arcs of $D_{\tri}$ obtained by three times of the Reidemeister move~{II} in \cref{rem:flat-braid}.}
    \label{fig:bigon-birth}
\end{figure}

\begin{cor}\label{cor:injective}
The coordinate map $\sfa^{\bD}: \cL^a(\bSigma,\bQ) \xrightarrow{\sim} \bQ^{I(\bD)}$ is injective.
\end{cor}

\begin{proof}
Thanks to the equivariance under the $\bQ_{>0}$-scaling action and the peripheral action, it suffices to consider two ladder-equivalence classes $L, L'\in\Blad_{\fsp_4,\bSigma}$. Then the assertion follows from \cref{thm:injectivity}.
\end{proof}
By combining \cref{prop:surjective} and \cref{cor:injective}, we obtain \cref{thm:bijection}.

\appendix

\section{Tropical cluster structure associated with the pair \texorpdfstring{$(\fsp_4,\bSigma)$}{(sp(4),S)}}\label{sec:cluster_notation}

\subsection{Tropical seeds and their mutations}
Fix a finite index set $I$, and a subset $I_\uf \subset I$. The indices in $I_\uf$ are called \emph{unfrozen indices}, and those in $I_\f:= I \setminus I_\uf$ are called \emph{frozen indices}. Let $D=\diag(d_i \mid i \in I)$ be a positive integral diagonal matrix. We fix $\bA\in \{\bZ,\bQ,\bR\}$ and a set $X$, and consider the set $\bA^X$ of maps $X \to \bA$ equipped with the pointwise operations $(\max,+)$. 

A \emph{tropical seed} on $X$ is a pair $(\ve,\bsfa)$, where
\begin{itemize}
    \item $\ve=(\ve_{ij})_{i,j \in I}$ is a matrix with half-integral entries such that $\ve D$ is skew-symmetric, and $\ve_{ij} \in \bZ$ unless $(i,j) \in I_\f \times I_\f$;
    \item $\bsfa=(\sfa_i)_{i \in I}$ is a tuple of elements in $\bA^X$.
\end{itemize}
We call $\ve$ the \emph{exchange matrix}, and $\sfa_i$ the \emph{tropical cluster ($K_2$-)variables}. 

Given a tropical seed $(\ve,\bsfa)$ and an unfrozen index $k \in I_\uf$, the \emph{mutation at $k$} is an operation $\mu_k: (\ve,\bsfa) \to (\ve',\bsfa')$ creating a new tropical seed, where
\begin{align}
    \ve'_{ij}&=
        \begin{cases}
            -\ve_{i j} & \mbox{if $i=k$ or $j=k$,}\\
            \ve_{i j}+[\ve_{i k}]_+[\ve_{k j}]_+ - [-\ve_{i k}]+-[-\ve_{k j}]_+ & \mbox{otherwise},
        \end{cases} \label{eq:mutation_matrix}\\
    \sfa'_i &= 
        \begin{cases}
            -\sfa_k + \max\{\sum_{j \in I} [\ve_{kj}]_+\sfa_j, \sum_{j \in I} [-\ve_{kj}]_+\sfa_j\} & \mbox{if $i=k$}, \\
            \sfa_i & \mbox{otherwise}.
        \end{cases} \label{eq:mutation_tropical}
\end{align}
Here $[a]_+:=\max\{0,a\}$ for any $a \in \bA^X$.

Observe that the exchange matrices transform under the mutations themselves, independently of the tropical cluster variables. Two tropical seeds (or exchange matrices) are said to be \emph{mutation-equivalent} if they are transformed to each other by a finite composition of mutations and permutations of indices. The equivalence class $\sfs$ is usually called the \emph{mutation class} of tropical seeds/exchange matrices. 

\begin{dfn}
Let $\sfs$ be a mutation class of tropical seeds. 
If $X$ admits a structure of piecewise-linear manifold over $\bA$ with the distinguished collection of coordinate systems $\sfa=(\sfa_i): X \xrightarrow{\sim} \bA^I$ corresponding to each tropical seed in $\sfs$, we call it the \emph{tropical cluster $K_2$-variety} over $\bA$ associated with $\sfs$ and write $X=\A_\sfs(\bA^T)$.
\end{dfn}
In this definition, the only matter is the bijectivity of the coordinate tuple $\sfa=(\sfa_i): X \to \bA^I$. If it is verified for one tropical seed in $\sfs$, then the other tropical seeds inductively determine PL charts related by the coordinate transformation \eqref{eq:mutation_tropical}, which are bijective.

\bigskip
\paragraph{\textbf{Sign-coherence}}
In an application to the geometry, the main task is often the verification of the tropical mutation relation \eqref{eq:mutation_tropical} between two coordinate systems on some space $X$. The following consideration is useful for this check. 

Given a tropical seed $(\ve,\bsfa)$, let us introduce the \emph{tropical cluster Poisson variable} by
\begin{align*}
    \sfx_k:= \sum_{i \in I} \ve_{ki} \sfa_i
\end{align*}
for $k \in I_\uf$. 

\begin{lem}
The mutation rule \eqref{eq:mutation_tropical} can be rewritten as
\begin{align*}
    \sfa'_i &= 
        \begin{cases}
            -\sfa_k + \sum_{j \in I} [-\sgn(\sfx_k)\ve_{kj}]_+\sfa_j & \mbox{if $i=k$}, \\
            \sfa_i & \mbox{otherwise}.
        \end{cases} 
\end{align*}
Here $\sgn(x) \in \{+,0,-\}^X$ denotes the function assigning the sign of $x \in \bA^X$ at each point.
\end{lem}

\begin{dfn}\label{def:sign_coherence}
Two points $\xi,\xi' \in X$ are said to be \emph{sign-coherent} for $\sfx_k$ if $\sfx_k(\xi)\cdot \sfx_k(\xi') \geq 0$. 
\end{dfn}

\begin{lem}\label{lem:sign-coherence}
Let $(\ve,\bsfa)$, $(\ve',\bsfa')$ be two tropical seeds on $X$ with $\ve'=\mu_k \ve$, and $\xi_1,\xi_2 \in X$. Assume that the tropical mutation relation \eqref{eq:mutation_tropical} holds at $\xi_1$ and $\xi_2$, and $\xi_1,\xi_2$ are sign-coherent for $\sfx_k$. Then \eqref{eq:mutation_tropical} holds for any point $\xi \in X$ satisfying 
\begin{align*}
    \sfa_i(\xi)= m_1\sfa_i(\xi_1) + m_2 \sfa_i(\xi_2), \quad \sfa'_i(\xi)= m_1\sfa'_i(\xi_1) + m_2 \sfa'_i(\xi_2)
\end{align*}
for some $m_1,m_2 \in \bA_{\geq 0}$. 
\end{lem}

\begin{proof}
Observe that $\xi$ is sign-coherent with $\xi_1,\xi_2$. 
Assume that $\sfx_k(\xi_1),\sfx_k(\xi_2) \geq 0$ without loss of generality. 
For $i \neq k$, the assertion is obvious. For $i=k$, we have
\begin{align*}
    \sfa'_i(\xi) &= m_1\sfa'_i(\xi_1)+m_2\sfa'_i(\xi_2) \\
    &= m_1\big(-\sfa_k(\xi_1) + \sum_{j \in I} [-\ve_{kj}]_+\sfa_j(\xi_1)\big) + m_2 \big(-\sfa_k(\xi_2) + \sum_{j \in I} [-\ve_{kj}]_+\sfa_j(\xi_2)\big) \\
    &= -\sfa_k(\xi) + \sum_{j \in I} [-\ve_{kj}]_+\sfa_j(\xi),
\end{align*}
which coincides with the rule \eqref{eq:mutation_tropical} at $\xi$ by $\sfx_k(\xi) \geq 0$. 
\end{proof}

\begin{rem}\label{rem:sign_coherence}
The statement also holds true if $m_1,m_2$ contain negative numbers but $\xi$ is still sign-coherent with $\xi_1,\xi_2$.
\end{rem}

\bigskip
\paragraph{\textbf{Exchange matrices and weighted quivers}}
It is useful to represent an exchange matrix $\ve=(\ve_{ij})_{i,j \in I}$ by a weighted quiver $Q$. 
The weighted quiver $Q$ corresponding to $\ve$ has vertices parameterized by the set $I$, and each vertex $i \in I$ is assigned the integer weight $d_i$. The structure matrix $\sigma=(\sigma_{ij})_{i,j \in I}$ of $Q$, whose $(i,j)$-entry indicates the number of arrows from $i$ to $j$ minus the number of arrows from $j$ to $i$, is defined to be
\begin{align*}
    \sigma_{ij}:=d_i^{-1}\varepsilon_{ij}\gcd(d_i,d_j).
\end{align*}
In figures, we draw $n$ dashed arrows from $i$ to $j$ if $\sigma_{ij}=n/2$ for $n \in \bZ$, where a pair of dashed arrows is replaced with a solid arrow. In this paper, we only deal with the weighted quivers whose vertices have weights $1$ or $2$. A vertex of weight $1$ (resp. $2$) is shown by a small circle (resp. a doubled circle).
\begin{ex}
For the matrices 
$\varepsilon=\begin{pmatrix} 0 & 2 \\ -1 & 0\end{pmatrix}$
and $D=\mathrm{diag}(2,1)$, the matrix $\ve D=\begin{pmatrix} 0 & 2 \\ -2 & 0\end{pmatrix}$ is skew-symmetric. 
We have 
$\sigma=\begin{pmatrix} 0 & 1 \\ -1 & 0\end{pmatrix}$. The corresponding weighted quiver is given by
\begin{align*}
    \tikz[>=latex]{\dnode{0,0}{black} node[above]{$1$}; \draw(2,0) circle(2pt) node[above]{$2$}; \qstarrow{0,0}{2,0};}.
\end{align*}
\end{ex}

The following lemma is useful to compute the mutations in terms of the weighted quivers:
\begin{lem}[{\cite[Lemma 2.3]{IIO21}}]\label{lem:weighted_mutation}
Assume that the weights $d_i$ take only two values: $\{d_i\}_{i \in I}=\{1,d\}$ for some integer $d\geq 2$. 
Then for $\ve'=\mu_k \ve$, the corresponding mutation of
the structure matrix $\sigma' = \mu_k\sigma$ is given by
\begin{align*}
  \sigma_{ij}' = 
  \begin{cases}
    -\sigma_{ij} & i=k \text{ or } j=k, \\
    \sigma_{ij} + ([\sigma_{ik}]_+ [\sigma_{kj}]_+ - [-\sigma_{ik}]_+ [-\sigma_{kj}]_+)\alpha_{ij}^k & \text{otherwise},
  \end{cases} 
\end{align*}
where 
\begin{align*}
    \alpha_{ij}^k = \begin{cases}
    d & \mbox{if $d_i=d_j \neq d_k$},\\
    1 & \mbox{otherwise}.
    \end{cases}
\end{align*}
\end{lem}

\subsection{Mutation class of exchange matrices for the pair \texorpdfstring{$(\fsp_4,\bSigma)$}{(sp(4),S)}}
Let us review the construction of a canonical mutation class of exchange matrices associated with the pair $(\fsp_4,\bSigma)$, following \cite{GS19} and the notation in \cite{IYsp4}. It encodes the combinatorial structure of the moduli space $\A_{Sp_4,\Sigma}$. In \cite{GS19}, the associated (rational) cluster $K_2$-variables are also constructed as rational functions on $\A_{Sp_4,\Sigma}$, whose tropical analogues correspond to the tropical cluster variables $\sfa_i$.

A decorated triangulation is a triple $\bD=(\tri,m_\tri,\bs_\tri)$, where
\begin{itemize}
    \item $\tri$ is an ideal triangulation of $\Sigma$;
    \item $m_\tri:t(\tri) \to \bM$ is a choice of a vertex of each triangle;
    \item $\bs_\tri:t(\tri) \to \{+,-\}$ is a choice of a sign at each triangle. 
\end{itemize}

\begin{figure}[ht]
\begin{tikzpicture}[scale=0.9]
    \foreach \i in {90,210,330}
    {
    \markedpt{\i:3};
    \draw[blue] (\i:3) -- (\i+120:3);
    }
\quiverplusC{90:3}{210:3}{330:3};
{\color{mygreen}
\uniarrow{x122}{x121}{dashed,shorten >=4pt, shorten <=2pt,bend left=20}
\uniarrow{x311}{x312}{dashed,shorten >=2pt, shorten <=4pt,bend right=20}
}
\uniarrow{x232}{x231}{myblue,dashed,shorten >=2pt, shorten <=4pt,bend left=20}
\draw(90:3.3) node{$m$};

\begin{scope}[xshift=7cm]
    \foreach \i in {90,210,330}
    {
    \markedpt{\i:3};
    \draw[blue] (\i:3) -- (\i+120:3);
    }
\quiverminusC{90:3}{210:3}{330:3};
{\color{mygreen}
\uniarrow{x121}{x122}{dashed,shorten >=4pt, shorten <=2pt,bend right=20}
\uniarrow{x312}{x311}{dashed,shorten >=4pt, shorten <=2pt,bend left=20}
}
\uniarrow{x232}{x231}{myblue,dashed,shorten >=2pt, shorten <=4pt,bend left=20}
\draw(90:3.3) node{$m$};
\end{scope}
\end{tikzpicture}
    \caption{The quivers $Q_{m,+}$ (left) and $Q_{m,-}$ (right) placed on a triangle with a fixed special point $m$. Here the vertices on the opposite side of $m$ and the arrows incident to them are colored blue for visibility.}
    \label{fig:quiver_triangle}
\end{figure}
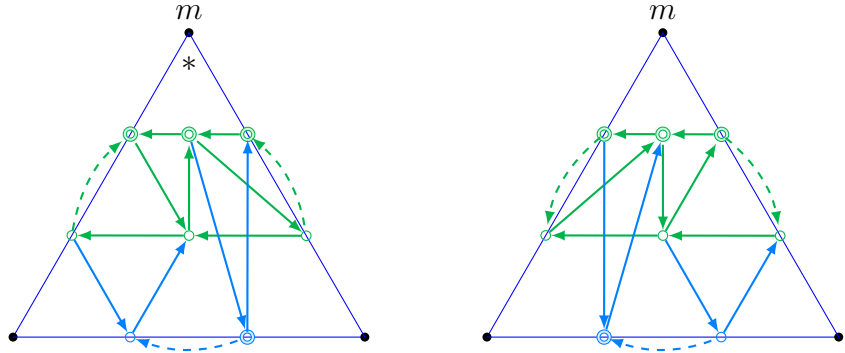

Given a decorated triangulation $\bD$, we define a weighted quiver $Q^{\bD}$ as follows. Let $Q_{m,+}$ and $Q_{m,-}$ be the weighted quivers on a triangle shown in the left and right of \cref{fig:quiver_triangle}, respectively. By convention, $Q_{m,\pm}$ are considered up to isotopy on $T$ which preserves each edge set-wisely. In particular, we are allowed to move an interior vertex inside the triangle; move and swap the two vertices on a common edge. 

For each triangle $T \in t(\tri)$, we draw the quiver $Q_{m_\tri(T),\bs_\tri(T)}$, and glue them via the \emph{amalgamation} procedure \cite{FG06} to get a weighted quiver $Q^{\bD}$ drawn on $\Sigma$. Here two vertices on a common interior edge with the same weight are identified; opposite half-arrows cancel together, and parallel half-arrows combine to give a solid arrow. 

By convention, $Q^{\bD}$ is considered up to isotopy on $\Sigma$ which preserves each boundary interval set-wisely. 
The vertex set of $Q^{\bD}$ is denoted by $I(\tri)=I_{\mathfrak{sp}_4}(\tri)$. We have $\# I(\tri) = 2\# e(\tri)+2 \# t(\tri)$. 

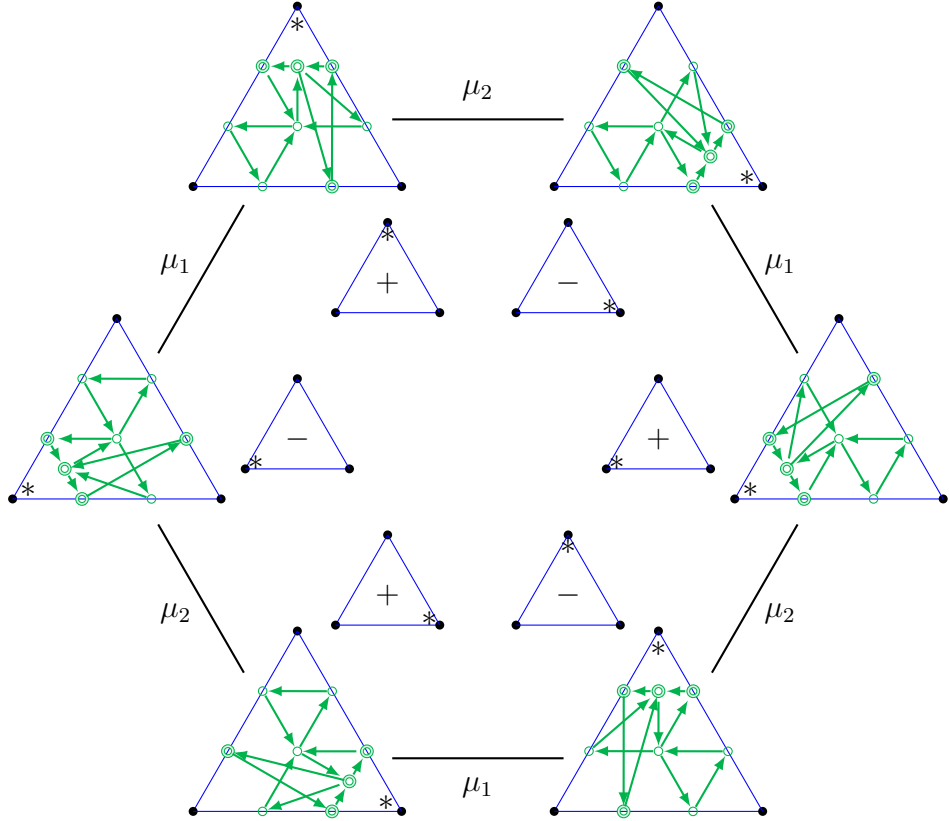
\begin{figure}[ht]
\begin{tikzpicture}[scale=0.8]
\begin{scope}[xshift=3cm]
    \foreach \i in {90,210,330}
    {
    \markedpt{\i:1};
    \draw[blue] (\i:1) -- (\i+120:1);
    }
\node at (0,0) {$+$};
\node at (210:0.8) {$\ast$};
\end{scope}

\begin{scope}[xshift=1.5cm, yshift=1.5*1.732cm]
    \foreach \i in {90,210,330}
    {
    \markedpt{\i:1};
    \draw[blue] (\i:1) -- (\i+120:1);
    }
\node at (0,0) {$-$};
\node at (330:0.8) {$\ast$};
\end{scope}

\begin{scope}[xshift=-1.5cm, yshift=1.5*1.732cm]
    \foreach \i in {90,210,330}
    {
    \markedpt{\i:1};
    \draw[blue] (\i:1) -- (\i+120:1);
    }
\node at (0,0) {$+$};
\node at (90:0.8) {$\ast$};
\end{scope}

\begin{scope}[xshift=-3cm]
    \foreach \i in {90,210,330}
    {
    \markedpt{\i:1};
    \draw[blue] (\i:1) -- (\i+120:1);
    }
\node at (0,0) {$-$};
\node at (210:0.8) {$\ast$};
\end{scope}

\begin{scope}[xshift=-1.5cm, yshift=-1.5*1.732cm]
    \foreach \i in {90,210,330}
    {
    \markedpt{\i:1};
    \draw[blue] (\i:1) -- (\i+120:1);
    }
\node at (0,0) {$+$};
\node at (330:0.8) {$\ast$};
\end{scope}

\begin{scope}[xshift=1.5cm, yshift=-1.5*1.732cm]
    \foreach \i in {90,210,330}
    {
    \markedpt{\i:1};
    \draw[blue] (\i:1) -- (\i+120:1);
    }
\node at (0,0) {$-$};
\node at (90:0.8) {$\ast$};
\end{scope}

\foreach \k in {0,60,120,180,240,300}
\draw[thick] (\k+15:5.5) -- (\k+45:5.5);
\foreach \k in {30,150,270}
\node at (\k:5.8) {$\mu_1$};
\foreach \k in {90,210,-30}
\node at (\k:5.8) {$\mu_2$};
\begin{scope}[xshift=6cm]
    \foreach \i in {90,210,330}
    {
    \markedpt{\i:2};
    \draw[blue] (\i:2) -- (\i+120:2);
    }
\quiverplus{210:2}{330:2}{90:2};
\end{scope}

\begin{scope}[xshift=3cm, yshift=3*1.732cm]
    \foreach \i in {90,210,330}
    {
    \markedpt{\i:2};
    \draw[blue] (\i:2) -- (\i+120:2);
    }
\quiverminus{330:2}{90:2}{210:2};
\end{scope}

\begin{scope}[xshift=-3cm, yshift=3*1.732cm]
    \foreach \i in {90,210,330}
    {
    \markedpt{\i:2};
    \draw[blue] (\i:2) -- (\i+120:2);
    }
\quiverplus{90:2}{210:2}{330:2};
\end{scope}

\begin{scope}[xshift=-6cm]
    \foreach \i in {90,210,330}
    {
    \markedpt{\i:2};
    \draw[blue] (\i:2) -- (\i+120:2);
    }
\quiverminus{210:2}{330:2}{90:2};
\end{scope}

\begin{scope}[xshift=-3cm, yshift=-3*1.732cm]
    \foreach \i in {90,210,330}
    {
    \markedpt{\i:2};
    \draw[blue] (\i:2) -- (\i+120:2);
    }
\quiverplus{330:2}{90:2}{210:2};
\end{scope}

\begin{scope}[xshift=3cm, yshift=-3*1.732cm]
    \foreach \i in {90,210,330}
    {
    \markedpt{\i:2};
    \draw[blue] (\i:2) -- (\i+120:2);
    }
\quiverminus{90:2}{210:2}{330:2};
\end{scope}

\end{tikzpicture}
    \caption{The exchange graph $\Exch_{\sfs(\mathfrak{sp}_4,T)}$ for a triangle $T$. Here $\mu_d$ denotes the mutation at the unique unfrozen vertex with weight $d \in \{1,2\}$.}
    \label{fig:exch_triangle}
\end{figure}

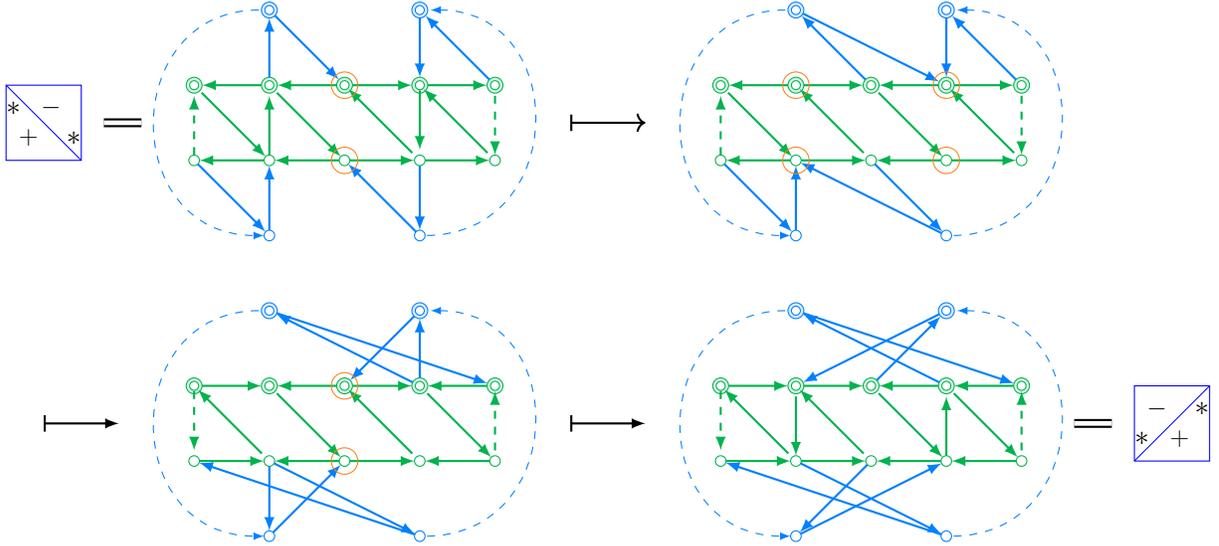
\begin{figure}[ht]
\begin{tikzpicture}
\quiversquare{0,0}{4,0}{4,3}{0,3};
\begin{scope}[>=latex]
{\color{mygreen}
\qsarrow{v21}{v20};
\qsarrow{v22}{v21};
\qsarrow{v22}{v23};
\qsarrow{v23}{v24};
\qarrow{v11}{v10};
\qarrow{v12}{v11};
\qarrow{v12}{v13};
\qarrow{v13}{v14};
\qstarrow{v20}{v11};
\qsharrow{v11}{v21};
\qstarrow{v21}{v12};
\qstarrow{v14}{v23};
\qsharrow{v23}{v13};
\qstarrow{v13}{v22};
\qshdarrow{v10}{v20};
\qstdarrow{v24}{v14};
}
{\color{myblue}
\qarrow{v10}{yl};
\qarrow{yl}{v11};
\qarrow{v13}{yr};
\qarrow{yr}{v12};
\qsarrow{v21}{zl};
\qsarrow{zl}{v22};
\qsarrow{v24}{zr};
\qsarrow{zr}{v23};
\draw[dashed,->,shorten >=2pt,shorten <=4pt] (zl) ..controls ++(-2,0) and ($(yl)+(-2,0)$).. (yl);
\draw[dashed,<-,shorten >=2pt,shorten <=4pt] (zr) ..controls ++(2,0) and ($(yr)+(2,0)$).. (yr);
}
\draw[myorange] (v12) circle(5pt);
\draw[myorange] (v22) circle(5pt);
\end{scope}
\draw[thick,|->] (5,1.5) -- (6,1.5);
\draw[blue] (-2.5,1) --++(1,0) --++(0,1)--++(-1,0) --cycle;
\draw[blue] (-2.5,2) --++(1,-1);
\node[scale=0.9] at (-2.4,1.7) {$\ast$};
\node[scale=0.9] at (-1.6,1.3) {$\ast$};
\node[scale=0.8] at (-2.2,1.3) {$+$};
\node[scale=0.8] at (-1.9,1.7) {$-$};
\draw[thick,double distance=0.15em] (-1.2,1.5) -- (-0.7,1.5);

\begin{scope}[xshift=7cm,>=latex]
\quiversquare{0,0}{4,0}{4,3}{0,3};
{\color{mygreen}
\qsarrow{v21}{v20};
\qsarrow{v21}{v22};
\qsarrow{v23}{v22};
\qsarrow{v23}{v24};
\qarrow{v11}{v10};
\qarrow{v11}{v12};
\qarrow{v13}{v12};
\qarrow{v13}{v14};
\qstarrow{v20}{v11};
\qstarrow{v12}{v21};
\qstarrow{v14}{v23};
\qstarrow{v22}{v13};
\qshdarrow{v10}{v20};
\qstdarrow{v24}{v14};
}
{\color{myblue}
\qarrow{v10}{yl};
\qarrow{yl}{v11};
\qarrow{v12}{yr};
\qarrow{yr}{v11};
\qsarrow{v22}{zl};
\qsarrow{v24}{zr};
\qsarrow{zr}{v23};
\qsarrow{zl}{v23};
\draw[dashed,->,shorten >=2pt,shorten <=4pt] (zl) ..controls ++(-2,0) and ($(yl)+(-2,0)$).. (yl);
\draw[dashed,<-,shorten >=2pt,shorten <=4pt] (zr) ..controls ++(2,0) and ($(yr)+(2,0)$).. (yr);
}
\draw[myorange] (v11) circle(5pt);
\draw[myorange] (v21) circle(5pt);
\draw[myorange] (v13) circle(5pt);
\draw[myorange] (v23) circle(5pt);
\end{scope}

\begin{scope}[yshift=-4cm,>=latex]
\quiversquare{0,0}{4,0}{4,3}{0,3};
{\color{mygreen}
\qsarrow{v20}{v21};
\qsarrow{v22}{v21};
\qsarrow{v22}{v23};
\qsarrow{v24}{v23};
\qarrow{v10}{v11};
\qarrow{v12}{v11};
\qarrow{v12}{v13};
\qarrow{v14}{v13};
\qstarrow{v11}{v20};
\qstarrow{v21}{v12};
\qstarrow{v23}{v14};
\qstarrow{v13}{v22};
\qshdarrow{v20}{v10};
\qstdarrow{v14}{v24};
}
{\color{myblue}
\qarrow{v11}{yl};
\qarrow{yl}{v12};
\qarrow{v11}{yr};
\qarrow{yr}{v10};
\qsarrow{v23}{zr};
\qsarrow{zr}{v22};
\qsarrow{v23}{zl};
\qsarrow{zl}{v24};
\draw[dashed,->,shorten >=2pt,shorten <=4pt] (zl) ..controls ++(-2,0) and ($(yl)+(-2,0)$).. (yl);
\draw[dashed,<-,shorten >=2pt,shorten <=4pt] (zr) ..controls ++(2,0) and ($(yr)+(2,0)$).. (yr);
}
\draw[myorange] (v12) circle(5pt);
\draw[myorange] (v22) circle(5pt);
\draw[thick,|->] (5,1.5) -- (6,1.5);
\draw[thick,|->] (-2,1.5) -- (-1,1.5);
\end{scope}

\begin{scope}[xshift=7cm,yshift=-4cm,>=latex]
\quiversquare{0,0}{4,0}{4,3}{0,3};
{\color{mygreen}
\qsarrow{v20}{v21};
\qsarrow{v21}{v22};
\qsarrow{v23}{v22};
\qsarrow{v24}{v23};
\qarrow{v10}{v11};
\qarrow{v11}{v12};
\qarrow{v13}{v12};
\qarrow{v14}{v13};
\qstarrow{v11}{v20};
\qstarrow{v12}{v21};
\qstarrow{v23}{v14};
\qstarrow{v22}{v13};
\qshdarrow{v20}{v10};
\qstdarrow{v14}{v24};
\qstarrow{v21}{v11};
\qsharrow{v13}{v23};
}
{\color{myblue}
\qarrow{v12}{yl};
\qarrow{yl}{v13};
\qarrow{v11}{yr};
\qarrow{yr}{v10};
\qsarrow{v22}{zr};
\qsarrow{v23}{zl};
\qsarrow{zl}{v24};
\qsarrow{zr}{v21};
\draw[dashed,->,shorten >=2pt,shorten <=4pt] (zl) ..controls ++(-2,0) and ($(yl)+(-2,0)$).. (yl);
\draw[dashed,<-,shorten >=2pt,shorten <=4pt] (zr) ..controls ++(2,0) and ($(yr)+(2,0)$).. (yr);
}
\draw[blue] (5.5,1) --++(1,0) --++(0,1)--++(-1,0) --cycle;
\draw[blue] (6.5,2) --++(-1,-1);
\node[scale=0.9] at (5.6,1.3) {$\ast$};
\node[scale=0.9] at (6.4,1.7) {$\ast$};
\node[scale=0.8] at (5.8,1.7) {$-$};
\node[scale=0.8] at (6.1,1.3) {$+$};
\draw[thick,double distance=0.15em] (5.2,1.5) -- (4.7,1.5);
\end{scope}
\end{tikzpicture}
    \caption{Mutation sequences that realize a flip of triangulation. Here we forget the boundary framing of the weighted quivers so that most parts are drawn planar. The vertices at which we perform mutations are shown in orange circles. Notice that the mutations at two vertices without arrows between them commute with each other.}
    \label{fig:flip_sequence}
\end{figure}

\begin{thm}[cf.~Goncharov--Shen {\cite[Section 12.5]{GS19}}]\label{thm:classical_mutation_equivalence}
For any two decorated triangulations $\bD$, $\bD'$ of $\Sigma$, the two weighted quivers $Q^{\bD}$, $Q^{\bD'}$ are mutation-equivalent. 
\end{thm}

\begin{proof}
Here we give an explicit mutation equivalence. Let us first consider the case $\Sigma=T$. In this case, we have six decorated triangulations. The associated weighted quivers are related as shown in \cref{fig:exch_triangle}. Thus the assertion for the triangle case is proved. We remark here that $\mu_1\mu_2$ and $\mu_2\mu_1$ amount to the rotations of the distinguished vertex, and $\mu_1\mu_2\mu_1=\mu_2\mu_1\mu_2$ amounts to the change of sign.

In the general case, a consequence of the previous paragraph is that the weighted quivers $Q^{\bD}$ and $Q^{\bD'}$ are mutation-equivalent if the underlying triangulations of $\bD$ and $\bD'$ are the same. It remains to consider the flips of ideal triangulations. Again by the previous paragraph, we can choose the decorations as shown in the left-most and right-most pictures in \cref{fig:flip_sequence}. Then it is easily verified that the flip can be realized by $2+4+2$ mutations as shown there. Since any two ideal triangulations are transformed into each other by a finite sequence of flips, the assertion is proved. 
\end{proof}

Thus the quivers $Q^{\bD}$ (or the corresponding exchange matrices $\ve^{\bD}$) define a canonical mutation class, denoted by $\sfs(\fsp_4,\bSigma)$. 

In the body of this paper, we will construct the (thought-to-be) tropical cluster variables $\bsfa^{\bD}=(\sfa_i^{\bD})_{i \in I(\tri)}$ over $\bQ$ on $X=\cL^a(\Sigma,\bQ)$. To show this is indeed the case, we need to prove that one of these coordinate systems are bijective, and they are related by (a composite of) the mutation rule \eqref{eq:mutation_tropical}.


\begin{rem}\label{rem:reduced_word}
\begin{enumerate}
    \item The sign $+$ (resp. $-$) corresponds to the reduced word $(1,2,1,2)$ (resp. $(2,1,2,1)$) of the longest element in the Weyl group of $\fsp_4$. See \cite[Section 11.2]{GS19} or \cite[Appendix C]{IO20} for a general construction of quivers associated with words. 
    \item The quiver associated with the reduced word $(1,2,1,2)$ for $\fsp_4$ is isomorphic to the one with the reduced word $(2,1,2,1)$ for $\fso_5$ via the permutation of two vertices on each edge and face. In this way, we get a bijection between the seeds in $\sfs(\fsp_4,\bSigma)$ and $\sfs(\fso_5,\bSigma)$ in a mutation-compatible way.
\end{enumerate}
\end{rem}

\section{Minimal position of \texorpdfstring{$\mathfrak{sp}_4$}{sp(4)}-diagrams}\label{sec:minimal}

In this section, we provide statements about a minimal position of the $\mathfrak{sp}_4$-diagrams on a marked surface by developing Kuperberg's argument for a marked disk in \cite{Kuperberg}. These statements are the $\fsp_4$-versions of those in \cite{FS22}, and their proofs are written in a more explicit form.
One can easily see that some arguments, especially after \cref{prop:minimal-position}, also work for $\mathfrak{sl}_3$, and we believe they also helps clarify the argument in \cite{FS22}.

Let $\bSigma=(\Sigma,\bM)$ be a marked surface possibly with punctures.
In this section, ideal arc means simple arc between marked points unless otherwise mentioned, and $\fsp_4$-diagrams are bounded.

\begin{dfn}\label{dfn:red-H-move}
    For a set $S$ of cut paths, the \emph{intersection reduction moves along $S$} are deformations of $\mathfrak{sp}_4$-diagrams (or their homotopy classes relative to $S$) in \cref{fig:red-move}. 
    An \emph{elliptic bordered face along $S$} is a face that is bounded by an ideal arc in $S$ and $\fsp_4$-diagrams, appearing on the left sides.
    The \emph{$H$-moves} along $S$ are defined by \cref{fig:H-move}.
    Similarly, a \emph{bordered $H$-face along $S$} is defined by the diagrams in \cref{fig:H-move}.
\end{dfn}

\begin{figure}[h]
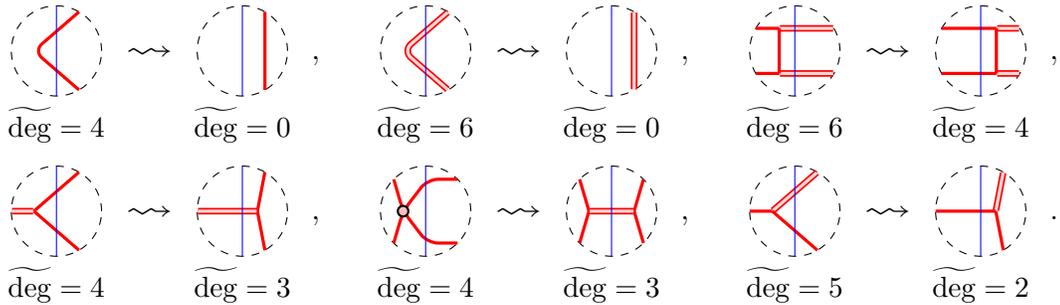

    \begin{align*}
        &\mbox{
        \tikz[baseline=-.6ex, scale=.1]{
            \draw[dashed, fill=white] (0,0) circle [radius=6];
            \draw[webline, rounded corners] (60:6) -- (-3,0) -- (-60:6);
            \draw[blue] (90:6) -- (-90:6);
            \node at (0,-6) [below]{\small $\widetilde{\deg}=4$};
        }
        }
        \scalebox{1.5}[1]{$\rightsquigarrow$}
        \mbox{
        \tikz[baseline=-.6ex, scale=.1]{
            \draw[dashed, fill=white] (0,0) circle [radius=6];
            \draw[webline] (60:6) -- (-60:6);
            \draw[blue] (90:6) -- (-90:6);
            \node at (0,-6) [below]{\small $\widetilde{\deg}=0$};
        }
        },\quad 
        \mbox{
        \tikz[baseline=-.6ex, scale=.1]{
            \draw[dashed, fill=white] (0,0) circle [radius=6];
            \draw[wline, rounded corners] (60:6) -- (-3,0) -- (-60:6);
            \draw[blue] (90:6) -- (-90:6);
            \node at (0,-6) [below]{\small $\widetilde{\deg}=6$};
        }
        }
        \scalebox{1.5}[1]{$\rightsquigarrow$}
        \mbox{
        \tikz[baseline=-.6ex, scale=.1]{
            \draw[dashed, fill=white] (0,0) circle [radius=6];
            \draw[wline] (60:6) -- (-60:6);
            \draw[blue] (90:6) -- (-90:6);
            \node at (0,-6) [below]{\small $\widetilde{\deg}=0$};
        }
        },\quad
        \mbox{
        \tikz[baseline=-.6ex, scale=.1]{
            \draw[dashed, fill=white] (0,0) circle [radius=6];
            \coordinate (U) at ($(150:6)!.3!(30:6)$);
            \coordinate (D) at ($(-150:6)!.3!(-30:6)$);
            \draw[webline] (150:6) -- (U);
            \draw[wline] (U) -- (30:6);
            \draw[webline] (-150:6) -- (D);
            \draw[wline] (D) -- (-30:6);
            \draw[webline] (U) -- (D);
            \draw[blue] (90:6) -- (-90:6);
            \node at (0,-6) [below]{\small $\widetilde{\deg}=6$};
        }
        } 
        \scalebox{1.5}[1]{$\rightsquigarrow$}
        \mbox{
        \tikz[baseline=-.6ex, scale=.1]{
            \draw[dashed, fill=white] (0,0) circle [radius=6];
            \coordinate (U) at ($(150:6)!.7!(30:6)$);
            \coordinate (D) at ($(-150:6)!.7!(-30:6)$);
            \draw[webline] (150:6) -- (U);
            \draw[wline] (U) -- (30:6);
            \draw[webline] (-150:6) -- (D);
            \draw[wline] (D) -- (-30:6);
            \draw[webline] (U) -- (D);
            \draw[blue] (90:6) -- (-90:6);
            \node at (0,-6) [below]{\small $\widetilde{\deg}=4$};
        }
        },\\
        &\mbox{
        \tikz[baseline=-.6ex, scale=.1]{
            \draw[dashed, fill=white] (0,0) circle [radius=6];
            \draw[webline] (60:6) -- (-3,0);
            \draw[webline] (-60:6) -- (-3,0);
            \draw[wline] (180:6) -- (-3,0);
            \draw[blue] (90:6) -- (-90:6);
            \node at (0,-6) [below]{\small $\widetilde{\deg}=4$};
        }
        }
        \scalebox{1.5}[1]{$\rightsquigarrow$}
        \mbox{
        \tikz[baseline=-.6ex, scale=.1]{
            \draw[dashed, fill=white] (0,0) circle [radius=6];
            \draw[webline] (60:6) -- (2,0);
            \draw[webline] (-60:6) -- (2,0);
            \draw[wline] (180:6) -- (2,0);
            \draw[blue] (90:6) -- (-90:6);
            \node at (0,-6) [below]{\small $\widetilde{\deg}=3$};
        }
        },\quad 
        \mbox{
        \tikz[baseline=-.6ex, scale=.1]{
            \draw[dashed, fill=white] (0,0) circle [radius=6];
            \draw[webline, rounded corners] (45:6) -- ($(45:6)+(-4,0)$) -- (-3,0) -- (-135:6);
            \draw[webline, rounded corners] (-45:6) -- ($(-45:6)+(-4,0)$) -- (-3,0) -- (135:6);
            \draw[fill=pink, thick] (-3,0) circle [radius=20pt];
            \draw[blue] (90:6) -- (-90:6);
            \node at (0,-6) [below]{\small $\widetilde{\deg}=4$};
        }
        } 
        \scalebox{1.5}[1]{$\rightsquigarrow$}
        \mbox{
        \tikz[baseline=-.6ex, scale=.1]{
            \draw[dashed, fill=white] (0,0) circle [radius=6];
            \draw[webline] (45:6) -- (3,0);
            \draw[webline] (-45:6) -- (3,0);
            \draw[webline] (135:6) -- (-3,0);
            \draw[webline] (-135:6) -- (-3,0);
            \draw[wline] (-3,0) -- (3,0);
            \draw[blue] (90:6) -- (-90:6);
            \node at (0,-6) [below]{\small $\widetilde{\deg}=3$};
        }
        },\quad 
        \mbox{
        \tikz[baseline=-.6ex, scale=.1]{
            \draw[dashed, fill=white] (0,0) circle [radius=6];
            \draw[wline] (60:6) -- (-3,0);
            \draw[webline] (-60:6) -- (-3,0);
            \draw[webline] (180:6) -- (-3,0);
            \draw[blue] (90:6) -- (-90:6);
            \node at (0,-6) [below]{\small $\widetilde{\deg}=5$};
        }
        }
        \scalebox{1.5}[1]{$\rightsquigarrow$}
        \mbox{
        \tikz[baseline=-.6ex, scale=.1]{
            \draw[dashed, fill=white] (0,0) circle [radius=6];
            \draw[wline] (60:6) -- (2,0);
            \draw[webline] (-60:6) -- (2,0);
            \draw[webline] (180:6) -- (2,0);
            \draw[blue] (90:6) -- (-90:6);
            \node at (0,-6) [below]{\small $\widetilde{\deg}=2$};
        }
        }.
    \end{align*}
    \caption{The intersection reduction moves (up to vertical and horizontal reflections). Here, the blue vertical line is a cut path in $S$ and $\widetilde{\deg}$ denotes the $S$-degree of the local picture. One can see that it strictly decreases by reduction moves. }
    \label{fig:red-move}
\end{figure}
\begin{figure}[h]
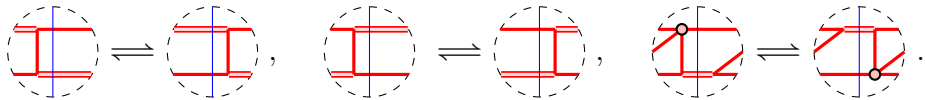

    \begin{align*}
        \mbox{
        \tikz[baseline=-.6ex, scale=.1]{
            \draw[dashed, fill=white] (0,0) circle [radius=6];
            \coordinate (U) at ($(150:6)!.3!(30:6)$);
            \coordinate (D) at ($(-150:6)!.3!(-30:6)$);
            \draw[wline] (150:6) -- (U);
            \draw[webline] (U) -- (30:6);
            \draw[webline] (-150:6) -- (D);
            \draw[wline] (D) -- (-30:6);
            \draw[webline] (U) -- (D);
            \draw[blue] (90:6) -- (-90:6);
        }
        } 
        \scalebox{1.5}[1]{$\rightleftharpoons$}
        \mbox{
        \tikz[baseline=-.6ex, scale=.1]{
            \draw[dashed, fill=white] (0,0) circle [radius=6];
            \coordinate (U) at ($(150:6)!.7!(30:6)$);
            \coordinate (D) at ($(-150:6)!.7!(-30:6)$);
            \draw[wline] (150:6) -- (U);
            \draw[webline] (U) -- (30:6);
            \draw[webline] (-150:6) -- (D);
            \draw[wline] (D) -- (-30:6);
            \draw[webline] (U) -- (D);
            \draw[blue] (90:6) -- (-90:6);
        }
        },\quad
        \mbox{
        \tikz[baseline=-.6ex, scale=.1]{
            \draw[dashed, fill=white] (0,0) circle [radius=6];
            \coordinate (U) at ($(150:6)!.3!(30:6)$);
            \coordinate (D) at ($(-150:6)!.3!(-30:6)$);
            \draw[webline] (150:6) -- (U);
            \draw[wline] (U) -- (30:6);
            \draw[wline] (-150:6) -- (D);
            \draw[webline] (D) -- (-30:6);
            \draw[webline] (U) -- (D);
            \draw[blue] (90:6) -- (-90:6);
        }
        }\ 
        \scalebox{1.5}[1]{$\rightleftharpoons$}
        \mbox{
        \tikz[baseline=-.6ex, scale=.1]{
            \draw[dashed, fill=white] (0,0) circle [radius=6];
            \coordinate (U) at ($(150:6)!.7!(30:6)$);
            \coordinate (D) at ($(-150:6)!.7!(-30:6)$);
            \draw[webline] (150:6) -- (U);
            \draw[wline] (U) -- (30:6);
            \draw[wline] (-150:6) -- (D);
            \draw[webline] (D) -- (-30:6);
            \draw[webline] (U) -- (D);
            \draw[blue] (90:6) -- (-90:6);
        }
        },\quad
        \mbox{
        \tikz[baseline=-.6ex, scale=.1]{
            \draw[dashed, fill=white] (0,0) circle [radius=6];
            \coordinate (LU) at ($(150:6)!.3!(30:6)$);
            \coordinate (LD) at ($(-150:6)!.3!(-30:6)$);
            \coordinate (RU) at ($(150:6)!.7!(30:6)$);
            \coordinate (RD) at ($(-150:6)!.7!(-30:6)$);
            \coordinate (L) at (-180:6);
            \coordinate (R) at (0:6);
            \draw[webline] (L) -- (LU) -- (30:6);
            \draw[webline] (-150:6) -- (LD);
            \draw[wline] (LD) -- (RD);
            \draw[webline] (150:6) -- (LU) -- (LD);
            \draw[webline] (RD) -- (R);
            \draw[webline] (RD) -- (-30:6);
            \draw[blue] (90:6) -- (-90:6);
            \draw[fill=pink, thick] (LU) circle [radius=20pt];
        }
        } 
        \scalebox{1.5}[1]{$\rightleftharpoons$}
        \mbox{
        \tikz[baseline=-.6ex, scale=.1, yscale=-1, xscale=-1]{
            \draw[dashed, fill=white] (0,0) circle [radius=6];
            \coordinate (LU) at ($(150:6)!.3!(30:6)$);
            \coordinate (LD) at ($(-150:6)!.3!(-30:6)$);
            \coordinate (RU) at ($(150:6)!.7!(30:6)$);
            \coordinate (RD) at ($(-150:6)!.7!(-30:6)$);
            \coordinate (L) at (-180:6);
            \coordinate (R) at (0:6);
            \draw[webline] (L) -- (LU) -- (30:6);
            \draw[webline] (-150:6) -- (LD);
            \draw[wline] (LD) -- (RD);
            \draw[webline] (150:6) -- (LU) -- (LD);
            \draw[webline] (RD) -- (R);
            \draw[webline] (RD) -- (-30:6);
            \draw[blue] (90:6) -- (-90:6);
            \draw[fill=pink, thick] (LU) circle [radius=20pt];
        }
        }.
    \end{align*}
    \caption{The $H$-moves. These moves preserve the $S$-degree. The last one is regarded as a composition of ladder resolutions and the other $H$-moves.
    }
    \label{fig:H-move}
\end{figure}


\begin{dfn}\label{def:trivial-diagram}
    Let $p,q$ be two (possibly marked) points in $\Sigma$. Let denote two non-intersecting arcs $\gamma_{-}$ and $\gamma_{+}$ between $p$ and $q$ that bound a biangle $B$.
    A \emph{trivial $\fsp_4$-diagram} in $B$ is a collection of parallel corner arcs of type $1$ or type $2$ in $B$.
\end{dfn}
\begin{lem}[{\cite[Lemma~6.5]{Kuperberg}}]\label{lem:homotopy-move}
    Let $B$ be a biangle bounded by cut paths $\gamma_{\pm}$ as in \cref{def:trivial-diagram}.
    For any $\{\gamma_{\pm}\}$-transverse $\fsp_4$-diagram $D$ in $\Sigma\setminus\{p,q\}$ whose intersection with $B$ has no elliptic interior faces, there exists a finite sequence $(D=D_{0}\leadsto D_1\leadsto\cdots\leadsto D_n=D')$ of $\{\gamma_{\pm}\}$-transverse $\fsp_4$-diagrams such that 
    \begin{itemize}
        \item $D_{i-1}\leadsto D_{i}$ (i=1,\dots,n) means either an $H$-move or an intersection reduction move along $\{\gamma_{\pm}\}$, and
        \item $D'\cap B$ is a trivial $\fsp_4$-diagram.
    \end{itemize}
    Moreover, if $D$ is $\{\gamma_{-}\}$-minimal, the above finite sequence of deformations is only performed along $\gamma_{+}$. 
\end{lem}
\begin{proof}
    This assertion is proved in a similar way to the proof of \cref{lem:triangle_blad}. Let us consider a dual graph $G$ of $D\cap B$ and assign angles, after replacing rungs in $B$ with crossroads, as shown below: 
    \[
        \mathord{
            \ 
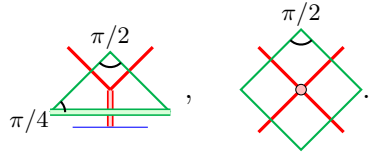
\tikz[baseline=-.6ex, scale=.1]{
                \coordinate (C) at (0,0);
                \coordinate (NW) at (135:8);
                \coordinate (NE) at (45:8);
                \coordinate (S) at (-90:5);
                \coordinate (N) at ($(90:8)+(0,-3)$);
                \coordinate (SW) at ($(-180:8)+(0,-3)$);
                \coordinate (SE) at ($(0:8)+(0,-3)$);
                \draw[wline] (C) -- (S);
                \draw[webline] (C) -- (NE);
                \draw[webline] (C) -- (NW);
                \draw[thick, mygreen] (SW) -- (N) -- (SE);
                \draw[line width=1.6pt, green!15, preaction={draw, line width=2.8pt, mygreen}] (SW) -- (SE);
                \pic [draw,thick,"\scriptsize $\pi/2$", angle eccentricity=-1, angle radius=.2cm] {angle = SW--N--SE};
                \pic [draw,thick,"\scriptsize $\pi/4$", angle eccentricity=-1.5, angle radius=.2cm] {angle = SE--SW--N};
                \draw[blue] ($(-90:5)+(-5,0)$) -- +(10,0);
        \ }
        }\ ,\quad
        \mathord{
            \ \tikz[baseline=-.6ex, scale=.1]{
                \coordinate (C) at (0,0);
                \coordinate (NW) at (135:8);
                \coordinate (NE) at (45:8);
                \coordinate (N) at (90:8);
                \coordinate (SW) at (-135:8);
                \coordinate (SE) at (-45:8);
                \coordinate (S) at (-90:8);
                \coordinate (E) at (180:8);
                \coordinate (W) at (0:8);
                \draw[webline] (C) -- (NW);
                \draw[webline] (C) -- (NE);
                \draw[webline] (C) -- (SW);
                \draw[webline] (C) -- (SE);
                \draw[thick, mygreen] (N) -- (W) -- (S) -- (E) -- cycle;
                \pic [draw,thick,"\scriptsize $\pi/2$", angle eccentricity=-1, angle radius=.2cm] {angle = E--N--W};
                \draw[fill=pink] (C) circle [radius=20pt];
        \ }
        }.
    \]
    We focus on a connected component of $G$ with at least one cycle.
    The angular defect at an interior vertex is $2\pi-\pi n_1/2$ and the angular defect at an exterior vertex is $3\pi/2-\pi (2n_1+n_2)/4$ where $n_i$ is the number of type $i$ edges incident to the vertex for $i=1,2$.
    The angular defect of an exterior vertex at a corner is at most $3\pi/4$. 
    In addition, the total angular defect must be $2\pi$ by the Gauss-Bonnet theorem.
    Furthermore, one can confirm that the summation of the angular defects of all exterior vertices is at least $\pi/2$.
    This implies that $D\cap B$ should have at least one elliptic bordered face or one bordered $H$-face.
    An intersection reduction move and an $H$-move along $\gamma_{\pm}$ exclude at least one vertex of $G$ from $B$.
    Apply this process to all connected components of $G$, then the resulting diagram is just a collection of arcs in $B$.
    It becomes the identity $\fsp_4$-diagram after applying the first two intersection reduction moves involving the elimination of returning arcs.
    For more details on the proof, see the proof in \cite[{Proposition~2.24}]{IYsp4}. A similar argument is carried out by replacing a triangle with a biangle in the proof. If $D$ is $\{\gamma_{-}\}$-minimal, there is no bordered elliptic face along $\gamma_{-}$ up to $H$-moves along $\gamma_{-}$. Hence one can deform $D$ to $D'$ by a sequence of intersection reduction moves and $H$-moves along $\{\gamma_{+}\}$.
\end{proof}

For an $\fsp_4$-diagram $D$ and an ideal arc $\gamma$, we will discuss the relationship between the minimality of $\widetilde{\deg}_{\gamma}(D)$ and faces adjacent to $\gamma$.

\begin{prop}\label{prop:minimal-position}
    Let $\gamma$ be any ideal arc from $p$ to $q$ ($p,q\in\bM$), and $D$ any $\{\gamma\}$-transverse reduced $\fsp_4$-diagram.
    If a $\{\gamma\}$-minimal $\fsp_4$-diagram $D'$ is ladder-equivalent to $D$ in $\Sigma$, then there exists a finite sequence $(D=D_0\leadsto D_1\leadsto\cdots\leadsto D_n=D')$ of $\{\gamma\}$-transverse $\fsp_4$-diagrams such that $D_{i-1}\leadsto D_{i}$ (i=1,\dots,n) is either an $H$-move along $\gamma$, an intersection reduction move along $\gamma$, a ladder equivalence $D_{i-1}\approx_{\gamma} D_{i}$, an arc parallel-move, or a loop parallel-move.
\end{prop}
\begin{proof}
    First, we remark that arc and loop parallel-moves do not change the degree of $D$, and these moves are independent of the other moves.
    Let $D'$ be any $\{\gamma\}$-minimal reduced $\fsp_4$-diagram.
    Assume that both $D$ and $D'$ have no rungs other than those intersecting $\gamma$ and $D$ is isotopic to $D'$ up to ladder resolutions.
    Then, one can construct a cut path $\psi^{\ast}\gamma$ on $D$ from $\gamma$ through a deformation $\psi$ from $D$ to $D'$ so that $D$ is $\{\psi^{\ast}\gamma\}$-minimal. See \cref{fig:minimal-arc}. Let $\alpha\coloneqq\psi^{\ast}\gamma$ for simplicity.
    Overlay $\gamma$ and $\alpha$ on the crossroad representative $X$ obtained by replacing all rungs of $D$ with crossroads. Note that $\alpha$ is homotopic to $\gamma$ relative to the endpoints $\{p,q\}$. In this diagram, $\gamma$ and $\alpha$ may have intersection points, and let us assign labels $p=p_0,p_1,\dots,p_m=q$ to these intersection points according to the direction from $p$ to $q$. Let $\gamma_{i,i+1}$ (resp. $\alpha_{i,i+1}$) denote the subarc of $\gamma$ (resp. $\alpha$) between $p_i$ and $p_{i+1}$. 
    Note that each $p_i$ may lie at a crossroad of $X$, and $p$ or $q$ must be adjacent to a biangle. However, the pair of subarcs between $p_i$ and $p_{i+1}$ may not necessarily bound a biangle if $p$ or $q$ is a puncture. If necessary, we exchange labels $p$ and $q$ to assume that $\gamma_{0,1}$ and $\alpha_{0,1}$ bound a biangle $B_{1}$.
    
    First, apply \cref{lem:homotopy-move} to the biangle $B_{1}$ after applying ladder resolutions at the crossroads on $\gamma_{0,1}$ and $\alpha_{0,1}$, and slide the resulting trivial $\fsp_4$-diagram out of $B_1$ across $p_1$.
    This deformation makes new $\fsp_4$-diagram $D^{(1)}$ such that $D^{(1)}\cap B_{1}$ is empty. See \cref{ex:minimal-argument}.
    If $p_1$ coincides with a crossroad, then we cannot apply \cref{lem:homotopy-move} to $B_{1}$. 
    In this case, the $\fsp_4$-diagram $x$ in $B_1$ can be seen as a top-left of \cref{fig:minimal-proof}.
    We take a point $p_1'$ on $\gamma$ as above, and the top left face $R$ as in \cref{fig:minimal-proof}.
    Then, one can consider the $\fsp_4$-diagram $x$ as an $\fsp_4$-diagram in a biangle $B_{1}'$ bounded by $\gamma_{p_0,p_1'}$ and $\alpha_{p_0,p_1}\ast\gamma_{p_1,p_1'}$ where the symbol $\ast$ denotes a concatenation of two arcs.
    The minimality of the ladder-resolution of $x$ implies that $R$ is a non-elliptic bordered face. Obviously, the other bordered faces along $\alpha$ are non-elliptic and $\alpha$ is minimal.
    Hence, we can apply \cref{lem:homotopy-move} to $x$ in $B_{1}'$ and obtain the trivial $\fsp_4$-diagram in $B_{1}'$ by intersection reduction moves and $H$-moves along $\gamma_{p_{0},p_{1}}$.
    Then, we can deform the $\fsp_4$-diagram by sliding the crossroad from $p_1$ to the outside of $\alpha_{p_0,p_1}$ along $\gamma$ as in the last picture in \cref{fig:minimal-proof}. This deformation corresponds to an intersection reduction move along $\gamma$, but it does not change the position with respect to $\alpha$. Consequently, we obtain a trivial $\fsp_4$-diagram in $B_1$ by intersection reduction moves $\gamma$ and $H$-moves along $\gamma$. Note that this deformation preserves the minimality of the resulting diagram along $\alpha$.
    
    Next, we apply the above argument to a new biangle $B_{2}$ bounded by $\alpha_{0,1}\ast\gamma_{1,2}$ and $\gamma_{0,1}\ast\alpha_{1,2}$. Note that $D^{(1)}\cap B_{2}$ is $\gamma_{0,1}\ast\alpha_{1,2}$-minimal. At this step, intersection reduction moves and $H$-moves are only applied along the subarc $\gamma_{1,2}$ because the sliding of the trivial $\fsp_4$-diagram eliminated elliptic faces bounded by both $\alpha_{0,1}$ and $\gamma_{1,2}$; see \cref{ex:minimal-argument}. Then we obtain an $\fsp_4$-diagram $D^{(2)}$ such that $D^{(2)}\cap (B_{1}\cup B_{2})$ is empty.
    We next consider $B_{3}$ bounded by $\alpha_{0,1}\ast\gamma_{1,2}\ast\alpha_{2,3}$ and $\gamma_{0,1}\ast\alpha_{1,2}\ast\gamma_{2,3}$ and repeat a similar argument to obtain $D^{(3)}$ with $D^{(3)}\cap (B_{1}\cup B_{2}\cup B_{3})=\emptyset$. We will repeat such deformation of $\fsp_4$-diagrams until one can not make a new biangle. Perform the same operation from the biangle whose vertex is at $q$ if necessary. As a result, we obtain a sequence $D^{(1)}, D^{(2)},\dots, D^{(m)}$. By the construction of these deformations, each deformation between $D^{(i)}$ and $D^{(i+1)}$ is constructed from a sequence of intersection reduction moves along $\gamma$ and $H$-moves along $\gamma$.
\end{proof}

\begin{figure}
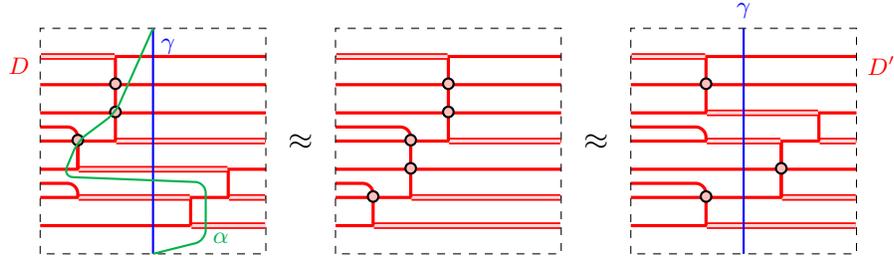

    \begin{align*}
        \mbox{
            \tikz[baseline=-.6ex, scale=.1, yshift=-15cm]{
                \foreach \i in {0,1,...,8}{
                \coordinate (P\i) at (0,30*\i/8);
                \coordinate (Q\i) at (30,30*\i/8);
                \coordinate (R\i) at (0,30*\i/8+30/16);
                \coordinate (S\i) at (30,30*\i/8+30/16);
                }
                \draw[webline] (P1) -- ($(P1)!4/6!(Q1)$);
                \draw[wline] ($(P1)!4/6!(Q1)$) -- (Q1);
                \draw[webline] ($(P1)!4/6!(Q1)$) -- ($(P2)!4/6!(Q2)$);
                \draw[webline] (P2) -- (Q2);
                \draw[wline] ($(P2)!1/6!(Q2)$) -- ($(P2)!4/6!(Q2)$);
                \draw[wline] ($(P2)!5/6!(Q2)$) -- (Q2);
                \draw[webline, rounded corners] (R2) -- ($(R2)!1/6!(S2)$) -- ($(P2)!1/6!(Q2)$);
                \draw[webline] ($(P2)!5/6!(Q2)$) -- ($(P3)!5/6!(Q3)$);
                \draw[webline] (P3) -- (Q3);
                \draw[wline] ($(P3)!1/6!(Q3)$) -- ($(P3)!5/6!(Q3)$);
                \draw[webline] ($(P3)!1/6!(Q3)$) -- ($(P4)!1/6!(Q4)$);
                \draw[webline, rounded corners] (R4) -- ($(R4)!1/6!(S4)$) -- ($(P3)!1/6!(Q3)$);
                \draw[webline] (P4) -- (Q4);
                \draw[wline] ($(P4)!2/6!(Q4)$) -- ($(P4)!6/6!(Q4)$);
                \draw[webline] ($(P4)!2/6!(Q4)$) -- ($(P7)!2/6!(Q7)$);
                \draw[webline] (P5) -- (Q5);
                \draw[webline] (P6) -- (Q6);
                \draw[webline] (P7) -- (Q7);
                \draw[wline] ($(P7)!0/6!(Q7)$) -- ($(P7)!2/6!(Q7)$);
                \draw[blue, thick] (15,0) -- (15,30);
                \draw[dashed] (0,0) rectangle (30,30);
                \node at (17,30) [below, blue]{\scriptsize $\gamma$};
                \node at (24,0) [above, mygreen]{\scriptsize $\alpha$};
                \node at (0,25) [red, left]{\scriptsize $D$};
                \node at ($(P4)!1/6!(Q4)$){\crossroad};
                \node at ($(P5)!2/6!(Q5)$){\crossroad};
                \node at ($(P6)!2/6!(Q6)$){\crossroad};
                \draw[mygreen, thick, rounded corners] (15,30) -- ($(P5)!2/6!(Q5)$) -- ($(P4)!1/6!(Q4)$) -- ($(P3)!1/6!(Q3)+(-2,-1)$) -- ($(P2)!4/6!(Q2)+(2,2)$) -- ($(P1)!4/6!(Q1)+(2,-2)$) -- (15,0);
            }
        }
        \approx
        \mbox{
            \tikz[baseline=-.6ex, scale=.1, yshift=-15cm]{
                \foreach \i in {0,1,...,8}{
                \coordinate (P\i) at (0,30*\i/8);
                \coordinate (Q\i) at (30,30*\i/8);
                \coordinate (R\i) at (0,30*\i/8+30/16);
                \coordinate (S\i) at (30,30*\i/8+30/16);
                }
                \draw[webline] (P1) -- ($(P1)!4/6!(Q1)$);
                \draw[wline] ($(P1)!1/6!(Q1)$) -- (Q1);
                \draw[webline] ($(P1)!1/6!(Q1)$) -- ($(P2)!1/6!(Q2)$);
                \draw[webline] (P2) -- (Q2);
                \draw[wline] ($(P2)!2/6!(Q2)$) -- (Q2);
                \draw[webline, rounded corners] (R2) -- ($(R2)!1/6!(S2)$) -- ($(P2)!1/6!(Q2)$);
                \draw[webline] (P3) -- (Q3);
                \draw[webline, rounded corners] (R4) -- ($(R4)!2/6!(S4)$) -- ($(P2)!2/6!(Q2)$);
                \draw[webline] (P4) -- (Q4);
                \draw[wline] ($(P4)!3/6!(Q4)$) -- ($(P4)!6/6!(Q4)$);
                \draw[webline] ($(P4)!3/6!(Q4)$) -- ($(P7)!3/6!(Q7)$);
                \draw[webline] (P5) -- (Q5);
                \draw[webline] (P6) -- (Q6);
                \draw[webline] (P7) -- (Q7);
                \draw[wline] ($(P7)!0/6!(Q7)$) -- ($(P7)!3/6!(Q7)$);
                %
                \draw[dashed] (0,0) rectangle (30,30);
                \node at ($(P2)!1/6!(Q2)$){\crossroad};
                \node at ($(P3)!2/6!(Q3)$){\crossroad};
                \node at ($(P4)!2/6!(Q4)$){\crossroad};
                \node at ($(P5)!3/6!(Q5)$){\crossroad};
                \node at ($(P6)!3/6!(Q6)$){\crossroad};
            }
        }
        \approx
        \mbox{
            \tikz[baseline=-.6ex, scale=.1, yshift=-15cm]{
                \foreach \i in {0,1,...,8}{
                \coordinate (P\i) at (0,30*\i/8);
                \coordinate (Q\i) at (30,30*\i/8);
                \coordinate (R\i) at (0,30*\i/8+30/16);
                \coordinate (S\i) at (30,30*\i/8+30/16);
                }
                \draw[webline] (P1) -- ($(P1)!4/6!(Q1)$);
                \draw[wline] ($(P1)!2/6!(Q1)$) -- (Q1);
                \draw[webline, rounded corners] (R2) -- ($(R2)!2/6!(S2)$) -- ($(P1)!2/6!(Q1)$);
                \draw[webline] (P2) -- (Q2);
                \draw[wline] ($(P2)!4/6!(Q2)$) -- (Q2);
                \draw[webline] ($(P2)!4/6!(Q2)$) -- ($(P4)!4/6!(Q4)$);
                \draw[webline] ($(P2)!4/6!(Q2)$) -- ($(P4)!4/6!(Q4)$);
                \draw[webline] (P3) -- (Q3);
                \draw[webline] (P4) -- (Q4);
                \draw[wline] ($(P4)!2/6!(Q4)$) -- ($(P4)!4/6!(Q4)$);
                \draw[wline] ($(P4)!5/6!(Q4)$) -- (Q4);
                \draw[webline, rounded corners] (R4) -- ($(R4)!2/6!(S4)$) -- ($(P4)!2/6!(Q4)$);
                \draw[webline] ($(P4)!5/6!(Q4)$) -- ($(P5)!5/6!(Q5)$);
                \draw[webline] (P5) -- (Q5);
                \draw[wline] ($(P5)!2/6!(Q5)$) -- ($(P5)!5/6!(Q5)$);
                \draw[webline] ($(P5)!2/6!(Q5)$) -- ($(P7)!2/6!(Q7)$);
                \draw[webline] (P6) -- (Q6);
                \draw[webline] (P7) -- (Q7);
                \draw[wline] (P7) -- ($(P7)!2/6!(Q7)$);
                \draw[blue, thick] (15,0) -- (15,30);
                \draw[dashed] (0,0) rectangle (30,30);
                \node at (15,30) [above, blue]{\scriptsize $\gamma$};
                \node at (30,25) [red, right]{\scriptsize $D'$};
                \node at ($(P2)!2/6!(Q2)$){\crossroad};
                \node at ($(P3)!4/6!(Q3)$){\crossroad};
                \node at ($(P6)!2/6!(Q6)$){\crossroad};
            }
        }
    \end{align*}
    \caption{$\{\gamma\}$-minimal $\fsp_4$-diagrams $D$ (left) and $D'$ (right) are ladder-equivalent to the middle $\fsp_4$-diagram. Deformation $\psi$ from $D$ to $D'$ sends $\alpha\coloneqq\psi^{\ast}\gamma$ to $\gamma$.}
    \label{fig:minimal-arc}
\end{figure}

\begin{figure}[h]
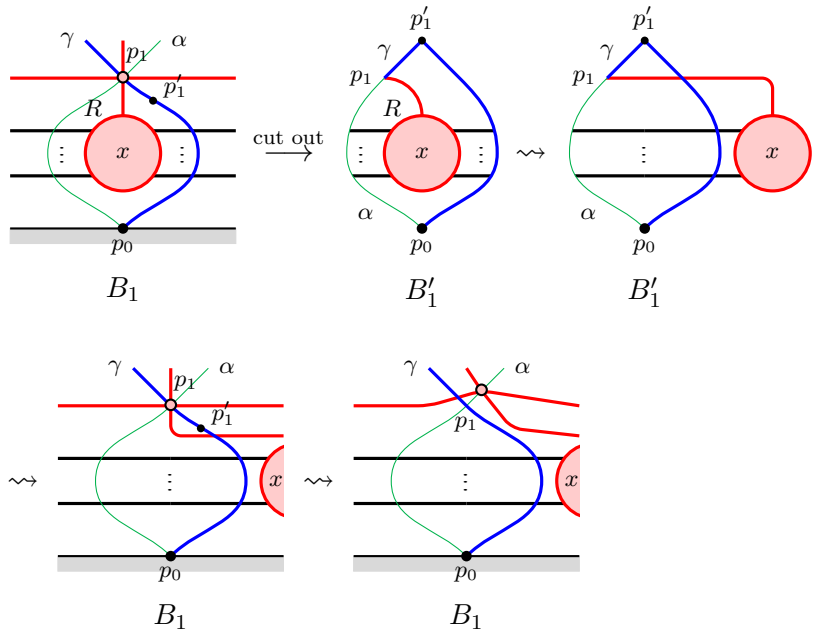

    \begin{align*}
        &\mbox{
            \tikz[baseline=-.6ex, scale=.1, yshift=-10cm]{
                \draw[webline] (10,25) -- (10,15);
                \draw[webline] (-5,20) -- (25,20);
                \draw[webline, black] (10,13) -- (-5,13);
                \draw[webline, black] (10,7) -- (-5,7);
                \draw[webline, black] (10,13) -- (25,13);
                \draw[webline, black] (10,7) -- (25,7);
                \fill[red!20!white] (10,10) circle [radius=5]; 
                \draw[webline] (10,10) circle [radius=5]; 
                \draw[mygreen] (10,0) to[out=north west, in=south] (0,10) to[out=north, in=south west] (10,20) -- (15,25);
                \draw[blue, very thick] (10,0) to[out=north east, in=south] (20,10) to[out=north, in=south east] (10,20) -- (5,25);
                \bline{-5,0}{25,0}{2}
                \fill[black] (10,0) circle [radius=20pt];
                \node at (2,10) [rotate=90, xscale=.5]{$\cdots$};
                \node at (18,10) [rotate=90, xscale=.5]{$\cdots$};
                \node at (10,0) [below]{\scriptsize $p_0$};
                \node at (12,23) {\scriptsize $p_1$};
                \node at (15,25) [right]{\scriptsize $\alpha$};
                \node at (5,25) [left]{\scriptsize $\gamma$};
                \node at (10,10) {\scriptsize $x$};
                \node at (10,20) {$\crossroad$};
                \node at (6,16) {\scriptsize $R$};
                \fill[black] (14,17) circle [radius=15pt];
                \node at (14,16) [above right]{\scriptsize $p_1'$};
                \node at (10,-5) [below]{\small $B_{1}$};
            }
        }
        \mathop{\longrightarrow}^{\text{cut out}}
        \mbox{
            \tikz[baseline=-.6ex, scale=.1, yshift=-10cm]{
                \draw[webline] (10,15) to[out=north, in=east] (5,20);
                \begin{scope}
                    \clip (10,0) to[out=north west, in=south] (0,10) to[out=north, in=south west]  (5,20)-- (10,25) -- (15,20) to[out=south east, in=north] (20,10) to[out=south, in=north east] (10,0);
                    \draw[webline, black] (10,13) -- (-5,13);
                    \draw[webline, black] (10,7) -- (-5,7);
                    \draw[webline, black] (10,13) -- (25,13);
                    \draw[webline, black] (10,7) -- (25,7);
                \end{scope}
                \fill[red!20!white] (10,10) circle [radius=5]; 
                \draw[webline] (10,10) circle [radius=5]; 
                \draw[mygreen] (10,0) to[out=north west, in=south] (0,10) to[out=north, in=south west]  (5,20);
                \draw[blue, very thick] (10,0) to[out=north east, in=south] (20,10) to[out=north, in=south east] (15,20) -- (10,25) -- (5,20);
                \fill[black] (10,0) circle [radius=20pt];
                \node at (2,10) [rotate=90, xscale=.5]{$\cdots$};
                \node at (18,10) [rotate=90, xscale=.5]{$\cdots$};
                \node at (10,0) [below]{\scriptsize $p_0$};
                \node at (5,20) [left]{\scriptsize $p_1$};
                \node at (5,2) [left]{\scriptsize $\alpha$};
                \node at (5,21) [above]{\scriptsize $\gamma$};
                \node at (10,10) {\scriptsize $x$};
                \node at (6,16) {\scriptsize $R$};
                \fill[black] (10,25) circle [radius=15pt];
                \node at (10,25) [above]{\scriptsize $p_1'$};
                \node at (10,-5) [below]{\small $B_{1}'$};
            }
        }
        \mathop{\leadsto}
        \mbox{
            \tikz[baseline=-.6ex, scale=.1, yshift=-10cm]{
                \begin{scope}
                    \clip (10,0) to[out=north west, in=south] (0,10) to[out=north, in=south west]  (5,20)-- (10,25) -- (34,25) -- (34,0);
                    \draw[webline, black] (10,13) -- (-5,13);
                    \draw[webline, black] (10,7) -- (-5,7);
                    \draw[webline, black] (10,13) -- (28,13);
                    \draw[webline, black] (10,7) -- (28,7);
                    \draw[webline, rounded corners] (-5,20) -- (27,20) -- (27,10);
                \end{scope}
                \fill[red!20!white, xshift=17cm] (10,10) circle [radius=5]; 
                \draw[webline, xshift=17cm] (10,10) circle [radius=5]; 
                \draw[mygreen] (10,0) to[out=north west, in=south] (0,10) to[out=north, in=south west]  (5,20);
                \draw[blue, very thick] (10,0) to[out=north east, in=south] (20,10) to[out=north, in=south east] (15,20) -- (10,25) -- (5,20);
                \fill[black] (10,0) circle [radius=20pt];
                \node at (10,10) [rotate=90, xscale=.5]{$\cdots$};
                \node at (10,0) [below]{\scriptsize $p_0$};
                \node at (5,20) [left]{\scriptsize $p_1$};
                \node at (5,2) [left]{\scriptsize $\alpha$};
                \node at (5,21) [above]{\scriptsize $\gamma$};
                \node at (27,10) {\scriptsize $x$};
                \fill[black] (10,25) circle [radius=15pt];
                \node at (10,25) [above]{\scriptsize $p_1'$};
                \node at (10,-5) [below]{\small $B_{1}'$};
            }
        }\\
        &\leadsto
        \mbox{
            \tikz[baseline=-.6ex, scale=.1, yshift=-10cm]{
                \clip (-5,-10) rectangle (25,30);
                \draw[webline, rounded corners] (10,25) -- (10,16) -- (25,16);
                \draw[webline] (-5,20) -- (25,20);
                \draw[webline, black] (10,13) -- (-5,13);
                \draw[webline, black] (10,7) -- (-5,7);
                \draw[webline, black] (10,13) -- (25,13);
                \draw[webline, black] (10,7) -- (25,7);
                \fill[red!20!white, xshift=17cm] (10,10) circle [radius=5]; 
                \draw[webline, xshift=17cm] (10,10) circle [radius=5]; 
                \draw[mygreen] (10,0) to[out=north west, in=south] (0,10) to[out=north, in=south west] (10,20) -- (15,25);
                \draw[blue, very thick] (10,0) to[out=north east, in=south] (20,10) to[out=north, in=south east] (10,20) -- (5,25);
                \bline{-5,0}{25,0}{2}
                \fill[black] (10,0) circle [radius=20pt];
                \node at (10,10) [rotate=90, xscale=.5]{$\cdots$};
                \node at (10,0) [below]{\scriptsize $p_0$};
                \node at (12,23) {\scriptsize $p_1$};
                \node at (15,25) [right]{\scriptsize $\alpha$};
                \node at (5,25) [left]{\scriptsize $\gamma$};
                \node at (24,10) {\scriptsize $x$};
                \node at (10,20) {$\crossroad$};
                \fill[black] (14,17) circle [radius=15pt];
                \node at (14,16) [above right]{\scriptsize $p_1'$};
                \node at (10,-5) [below]{\small $B_{1}$};
            }
        }
        \leadsto
        \mbox{
            \tikz[baseline=-.6ex, scale=.1, yshift=-10cm]{
                \clip (-5,-10) rectangle (25,30);
                \draw[webline, rounded corners] (10,25) -- ($(10,20)+(2,2)$) -- ($(10,20)+(6,-3)$) -- (25,16);
                \draw[webline, rounded corners] (-5,20) -- (5,20) -- ($(10,20)+(2,2)$) -- (25,20);
                \draw[webline, black] (10,13) -- (-5,13);
                \draw[webline, black] (10,7) -- (-5,7);
                \draw[webline, black] (10,13) -- (25,13);
                \draw[webline, black] (10,7) -- (25,7);
                \fill[red!20!white, xshift=17cm] (10,10) circle [radius=5]; 
                \draw[webline, xshift=17cm] (10,10) circle [radius=5]; 
                \draw[mygreen] (10,0) to[out=north west, in=south] (0,10) to[out=north, in=south west] (10,20) -- (15,25);
                \draw[blue, very thick] (10,0) to[out=north east, in=south] (20,10) to[out=north, in=south east] (10,20) -- (5,25);
                \bline{-5,0}{25,0}{2}
                \fill[black] (10,0) circle [radius=20pt];
                \node at (10,10) [rotate=90, xscale=.5]{$\cdots$};
                \node at (10,0) [below]{\scriptsize $p_0$};
                \node at (10,20) [below]{\scriptsize $p_1$};
                \node at (15,25) [right]{\scriptsize $\alpha$};
                \node at (5,25) [left]{\scriptsize $\gamma$};
                \node at (24,10) {\scriptsize $x$};
                \node at ($(10,20)+(2,2)$) {$\crossroad$};
                \node at (10,-5) [below]{\small $B_{1}$};
            }
        }
    \end{align*}
    \caption{The $\fsp_4$-diagram $x$ in $B_{1}'$ is obtained by cut it out along the biangle $B_{1}$. Black edges are type~$1$ or type~$2$ edges. The intersection reduction moves and $H$-moves along the right edge of $B_{1}'$ sweep $x$ out of $B_{1}'$.}
    \label{fig:minimal-proof}
\end{figure}

\begin{ex}\label{ex:minimal-argument}
    We describe an example of the minimal position argument below. The green arc (resp.~blue thick arc) corresponds to $\alpha$ (resp.~$\gamma$). Deformations ``lad.'' and ``red'' mean the ladder-resolution and the intersection reduction moves, respectively. The deformation ``iso.'' in the second line is sliding the trivial $\fsp_4$-diagram out of a small bigangle. We remark that all deformations are performed along $\gamma$.
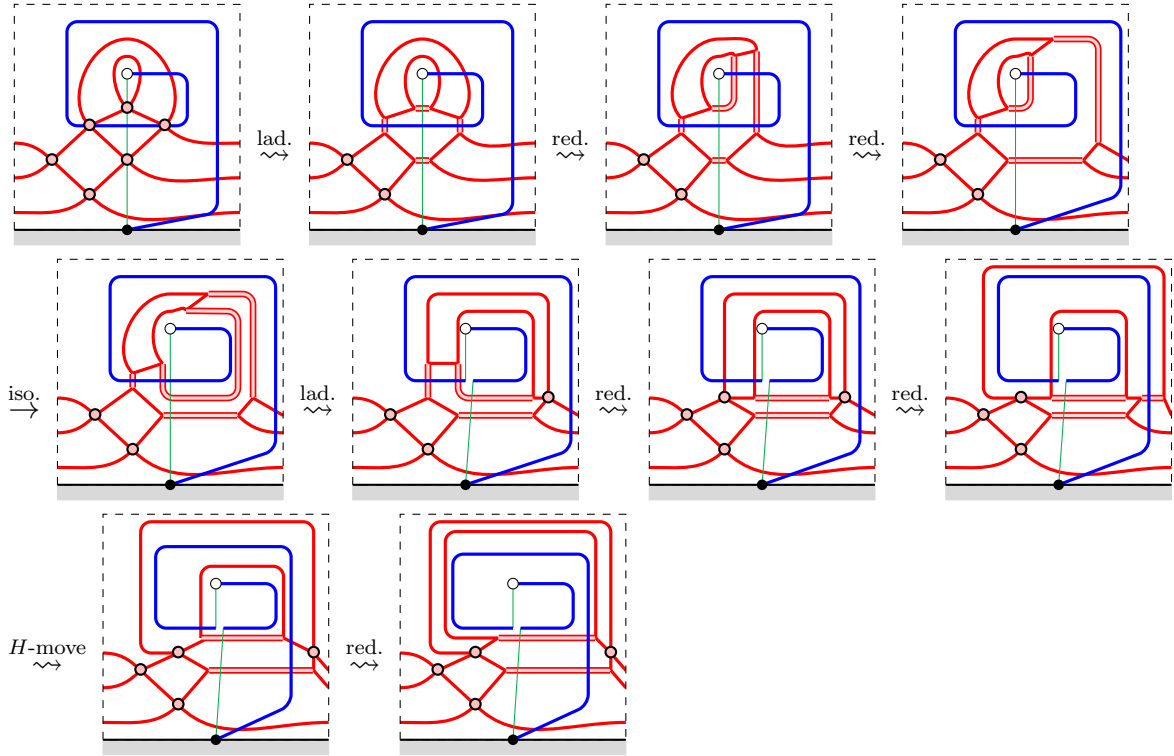
    \begin{align*}
        &\mbox{
            \tikz[baseline=-.6ex, scale=.1, yshift=-10cm]{
                \foreach \i in {0,1,...,13}{
                    \coordinate (C\i) at (0,30*\i/13);
                    \coordinate (L\i) at (-5,30*\i/13);
                    \coordinate (LL\i) at (-10,30*\i/13);
                    \coordinate (LLL\i) at (-15,30*\i/13);
                    \coordinate (R\i) at (5,30*\i/13);
                    \coordinate (RR\i) at (10,30*\i/13);
                    \coordinate (RRR\i) at (15,30*\i/13);
                }
                \begin{scope}
                \clip (-15,0) rectangle (15,30);
                    \draw[webline] (C7) to[out=north east, in=east] (C10) to[out=west, in=north west] (C7);
                    \draw[webline] (R6) to[out=north east, in=east] (C11) to[out=west, in=north west] (L6);
                    \draw[webline] (C7) -- (R6);
                    \draw[webline] (C7) -- (L6);
                    \draw[webline] (R6) -- (C4);
                    \draw[webline] (L6) -- (C4);
                    \draw[webline] (L6) -- (LL4);
                    \draw[webline] (C4) -- (L2);
                    \draw[webline] (LL4) -- (L2);
                    \draw[webline] (R6) to[out=south east, in=west] (RRR5);
                    \draw[webline] (LL4) to[out=north west, in=east] (LLL5);
                    \draw[webline] (LL4) to[out=south west, in=east] (LLL3);
                    \draw[webline] (C4) to[out=south east, in=west] (RRR3);
                    \draw[webline] (L2) to[out=south west, in=east] (LLL1);
                    \draw[webline] (L2) to[out=south east, in=west] (RRR1);
                    \draw[mygreen] (C0) -- (C9);
                    \draw[blue, very thick, rounded corners] (C0) -- ($(R1)+(7,0)$) -- ($(R12)+(7,0)$) -- ($(L12)+(-3,0)$) -- ($(L6)+(-3,0)$) -- ($(R6)+(3,0)$) -- ($(R9)+(3,0)$) -- (C9);
                \end{scope}
                \bline{-15,0}{15,0}{2}
                \draw[dashed] (-15,0) rectangle (15,30);
                \fill[black] (C0) circle [radius=20pt];
                \draw[black, fill=white] (C9) circle [radius=20pt];
                \node at (C7) {$\crossroad$};
                \node at (L6) {$\crossroad$};
                \node at (R6) {$\crossroad$};
                \node at (C4) {$\crossroad$};
                \node at (LL4) {$\crossroad$};
                \node at (L2) {$\crossroad$};
            }
        }
        \mathop{\leadsto}^{\text{lad.}}
        \mbox{
            \tikz[baseline=-.6ex, scale=.1, yshift=-10cm]{
                \foreach \i in {0,1,...,13}{
                    \coordinate (C\i) at (0,30*\i/13);
                    \coordinate (L\i) at (-5,30*\i/13);
                    \coordinate (LL\i) at (-10,30*\i/13);
                    \coordinate (LLL\i) at (-15,30*\i/13);
                    \coordinate (R\i) at (5,30*\i/13);
                    \coordinate (RR\i) at (10,30*\i/13);
                    \coordinate (RRR\i) at (15,30*\i/13);
                }
                \begin{scope}
                \clip (-15,0) rectangle (15,30);
                    \draw[webline] ($(C7)+(1,0)$) to[out=north east, in=east] (C10) to[out=west, in=north west] ($(C7)+(-1,0)$);
                    \draw[webline] ($(R6)+(0,1)$) to[out=north east, in=east] (C11) to[out=west, in=north west] ($(L6)+(0,1)$);
                    \draw[webline] ($(C7)+(1,0)$) -- ($(R6)+(0,1)$);
                    \draw[webline] ($(C7)+(-1,0)$) -- ($(L6)+(0,1)$);
                    \draw[webline] ($(R6)+(0,-1)$) -- ($(C4)+(1,0)$);
                    \draw[webline] ($(L6)+(0,-1)$) -- ($(C4)+(-1,0)$);
                    \draw[webline] ($(L6)+(0,-1)$) -- (LL4);
                    \draw[webline] ($(C4)+(-1,0)$) -- (L2);
                    \draw[webline] (LL4) -- (L2);
                    \draw[wline] ($(C7)+(-1,0)$) -- ($(C7)+(1,0)$);
                    \draw[wline] ($(C4)+(-1,0)$) -- ($(C4)+(1,0)$);
                    \draw[wline] ($(L6)+(0,-1)$) -- ($(L6)+(0,1)$);
                    \draw[wline] ($(R6)+(0,-1)$) -- ($(R6)+(0,1)$);
                    \draw[webline] ($(R6)+(0,-1)$) to[out=south east, in=west] (RRR5);
                    \draw[webline] (LL4) to[out=north west, in=east] (LLL5);
                    \draw[webline] (LL4) to[out=south west, in=east] (LLL3);
                    \draw[webline] ($(C4)+(1,0)$) to[out=south east, in=west] (RRR3);
                    \draw[webline] (L2) to[out=south west, in=east] (LLL1);
                    \draw[webline] (L2) to[out=south east, in=west] (RRR1);
                    \draw[mygreen] (C0) -- (C9);
                    \draw[blue, very thick, rounded corners] (C0) -- ($(R1)+(7,0)$) -- ($(R12)+(7,0)$) -- ($(L12)+(-3,0)$) -- ($(L6)+(-3,0)$) -- ($(R6)+(3,0)$) -- ($(R9)+(3,0)$) -- (C9);
                \end{scope}
                \bline{-15,0}{15,0}{2}
                \draw[dashed] (-15,0) rectangle (15,30);
                \fill[black] (C0) circle [radius=20pt];
                \draw[black, fill=white] (C9) circle [radius=20pt];
                \node at (LL4) {$\crossroad$};
                \node at (L2) {$\crossroad$};
            }
        }
        \mathop{\leadsto}^{\text{red.}}
        \mbox{
            \tikz[baseline=-.6ex, scale=.1, yshift=-10cm]{
                \foreach \i in {0,1,...,13}{
                    \coordinate (C\i) at (0,30*\i/13);
                    \coordinate (L\i) at (-5,30*\i/13);
                    \coordinate (LL\i) at (-10,30*\i/13);
                    \coordinate (LLL\i) at (-15,30*\i/13);
                    \coordinate (R\i) at (5,30*\i/13);
                    \coordinate (RR\i) at (10,30*\i/13);
                    \coordinate (RRR\i) at (15,30*\i/13);
                }
                \begin{scope}
                \clip (-15,0) rectangle (15,30);
                    \draw[webline] ($(C7)+(2,7)$) to[out=north, in=east] (C10) to[out=west, in=north west] ($(C7)+(-1,0)$);
                    \draw[webline] ($(R6)+(0,10)$) to[out=north, in=east] (C11) to[out=west, in=north west] ($(L6)+(0,1)$);
                    \draw[webline] ($(C7)+(2,7)$) -- ($(R6)+(0,10)$);
                    \draw[webline] ($(C7)+(-1,0)$) -- ($(L6)+(0,1)$);
                    \draw[webline] ($(R6)+(0,-1)$) -- ($(C4)+(1,0)$);
                    \draw[webline] ($(L6)+(0,-1)$) -- ($(C4)+(-1,0)$);
                    \draw[webline] ($(L6)+(0,-1)$) -- (LL4);
                    \draw[webline] ($(C4)+(-1,0)$) -- (L2);
                    \draw[webline] (LL4) -- (L2);
                    \draw[wline, rounded corners] ($(C7)+(-1,0)$) -- ($(C7)+(2,0)$) -- ($(C7)+(2,7)$);
                    \draw[wline] ($(C4)+(-1,0)$) -- ($(C4)+(1,0)$);
                    \draw[wline] ($(L6)+(0,-1)$) -- ($(L6)+(0,1)$);
                    \draw[wline] ($(R6)+(0,-1)$) -- ($(R6)+(0,10)$);
                    \draw[webline] ($(R6)+(0,-1)$) to[out=south east, in=west] (RRR5);
                    \draw[webline] (LL4) to[out=north west, in=east] (LLL5);
                    \draw[webline] (LL4) to[out=south west, in=east] (LLL3);
                    \draw[webline] ($(C4)+(1,0)$) to[out=south east, in=west] (RRR3);
                    \draw[webline] (L2) to[out=south west, in=east] (LLL1);
                    \draw[webline] (L2) to[out=south east, in=west] (RRR1);
                    \draw[mygreen] (C0) -- (C9);
                    \draw[blue, very thick, rounded corners] (C0) -- ($(R1)+(7,0)$) -- ($(R12)+(7,0)$) -- ($(L12)+(-3,0)$) -- ($(L6)+(-3,0)$) -- ($(R6)+(3,0)$) -- ($(R9)+(3,0)$) -- (C9);
                \end{scope}
                \bline{-15,0}{15,0}{2}
                \draw[dashed] (-15,0) rectangle (15,30);
                \fill[black] (C0) circle [radius=20pt];
                \draw[black, fill=white] (C9) circle [radius=20pt];
                \node at (LL4) {$\crossroad$};
                \node at (L2) {$\crossroad$};
            }
        }
        \mathop{\leadsto}^{\text{red.}}
        \mbox{
            \tikz[baseline=-.6ex, scale=.1, yshift=-10cm]{
                \foreach \i in {0,1,...,13}{
                    \coordinate (C\i) at (0,30*\i/13);
                    \coordinate (L\i) at (-5,30*\i/13);
                    \coordinate (LL\i) at (-10,30*\i/13);
                    \coordinate (LLL\i) at (-15,30*\i/13);
                    \coordinate (R\i) at (5,30*\i/13);
                    \coordinate (RR\i) at (10,30*\i/13);
                    \coordinate (RRR\i) at (15,30*\i/13);
                }
                \begin{scope}
                \clip (-15,0) rectangle (15,30);
                    \draw[webline] ($(C7)+(2,7)$) to[out=north, in=east] (C10) to[out=west, in=north west] ($(C7)+(-1,0)$);
                    \draw[webline] (R11) -- (C11) to[out=west, in=north west] ($(L6)+(0,1)$);
                    \draw[webline] ($(C7)+(2,7)$) -- (R11);
                    \draw[webline] ($(C7)+(-1,0)$) -- ($(L6)+(0,1)$);
                    \draw[webline] ($(R5)+(6,0)$) -- ($(RR4)+(-1,0)$);
                    \draw[webline] ($(L6)+(0,-1)$) -- ($(C4)+(-1,0)$);
                    \draw[webline] ($(L6)+(0,-1)$) -- (LL4);
                    \draw[webline] ($(C4)+(-1,0)$) -- (L2);
                    \draw[webline] (LL4) -- (L2);
                    \draw[wline, rounded corners] ($(C7)+(-1,0)$) -- ($(C7)+(2,0)$) -- ($(C7)+(2,7)$);
                    \draw[wline] ($(C4)+(-1,0)$) -- ($(RR4)+(-1,0)$);
                    \draw[wline] ($(L6)+(0,-1)$) -- ($(L6)+(0,1)$);
                    \draw[wline, rounded corners] ($(R5)+(6,0)$) -- ($(R11)+(6,0)$) -- (R11);
                    \draw[webline, rounded corners] ($(R5)+(6,0)$) -- (RRR4);
                    \draw[webline] (LL4) to[out=north west, in=east] (LLL5);
                    \draw[webline] (LL4) to[out=south west, in=east] (LLL3);
                    \draw[webline] ($(RR4)+(-1,0)$) to[out=south east, in=west] (RRR3);
                    \draw[webline] (L2) to[out=south west, in=east] (LLL1);
                    \draw[webline] (L2) to[out=south east, in=west] (RRR1);
                    \draw[mygreen] (C0) -- (C9);
                    \draw[blue, very thick, rounded corners] (C0) -- ($(RRR2)+(-1,0)$) -- ($(RRR12)+(-1,0)$) -- ($(L12)+(-3,0)$) -- ($(L6)+(-3,0)$) -- ($(R6)+(3,0)$) -- ($(R9)+(3,0)$) -- (C9);
                \end{scope}
                \bline{-15,0}{15,0}{2}
                \draw[dashed] (-15,0) rectangle (15,30);
                \fill[black] (C0) circle [radius=20pt];
                \draw[black, fill=white] (C9) circle [radius=20pt];
                \node at (LL4) {$\crossroad$};
                \node at (L2) {$\crossroad$};
            }
        }\\
        &\mathop{\rightarrow}^{\text{iso.}}
        \mbox{
            \tikz[baseline=-.6ex, scale=.1, yshift=-10cm]{
                \foreach \i in {0,1,...,13}{
                    \coordinate (C\i) at (0,30*\i/13);
                    \coordinate (L\i) at (-5,30*\i/13);
                    \coordinate (LL\i) at (-10,30*\i/13);
                    \coordinate (LLL\i) at (-15,30*\i/13);
                    \coordinate (R\i) at (5,30*\i/13);
                    \coordinate (RR\i) at (10,30*\i/13);
                    \coordinate (RRR\i) at (15,30*\i/13);
                }
                \begin{scope}
                \clip (-15,0) rectangle (15,30);
                    \draw[webline] ($(C7)+(2,7)$) to[out=north, in=east] (C10) to[out=west, in=north west] ($(C7)+(-1,0)$);
                    \draw[webline] (R11) -- (C11) to[out=west, in=north west] ($(L6)+(0,1)$);
                    \draw[webline] ($(C7)+(2,7)$) -- (R11);
                    \draw[webline] ($(C7)+(-1,0)$) -- ($(L6)+(0,1)$);
                    \draw[webline] ($(R5)+(6,0)$) -- ($(RR4)+(-1,0)$);
                    \draw[webline] ($(L6)+(0,-1)$) -- ($(C4)+(-1,0)$);
                    \draw[webline] ($(L6)+(0,-1)$) -- (LL4);
                    \draw[webline] ($(C4)+(-1,0)$) -- (L2);
                    \draw[webline] (LL4) -- (L2);
                    \draw[wline, rounded corners] ($(C7)+(-1,0)$) -- ($(C5)+(-1,0)$) -- ($(RR5)+(-1,0)$) -- ($(RR10)+(-1,0)$) -- ($(C7)+(2,7)$);
                    \draw[wline] ($(C4)+(-1,0)$) -- ($(RR4)+(-1,0)$);
                    \draw[wline] ($(L6)+(0,-1)$) -- ($(L6)+(0,1)$);
                    \draw[wline, rounded corners] ($(R5)+(6,0)$) -- ($(R11)+(6,0)$) -- (R11);
                    \draw[webline, rounded corners] ($(R5)+(6,0)$) -- (RRR4);
                    \draw[webline] (LL4) to[out=north west, in=east] (LLL5);
                    \draw[webline] (LL4) to[out=south west, in=east] (LLL3);
                    \draw[webline] ($(RR4)+(-1,0)$) to[out=south east, in=west] (RRR3);
                    \draw[webline] (L2) to[out=south west, in=east] (LLL1);
                    \draw[webline] (L2) to[out=south east, in=west] (RRR1);
                    \draw[mygreen] (C0) -- (C9);
                    \draw[blue, very thick, rounded corners] (C0) -- ($(RRR2)+(-1,0)$) -- ($(RRR12)+(-1,0)$) -- ($(L12)+(-3,0)$) -- ($(L6)+(-3,0)$) -- ($(R6)+(3,0)$) -- ($(R9)+(3,0)$) -- (C9);
                \end{scope}
                \bline{-15,0}{15,0}{2}
                \draw[dashed] (-15,0) rectangle (15,30);
                \fill[black] (C0) circle [radius=20pt];
                \draw[black, fill=white] (C9) circle [radius=20pt];
                \node at (LL4) {$\crossroad$};
                \node at (L2) {$\crossroad$};
            }
        }
        \mathop{\leadsto}^{\text{lad.}}
        \mbox{
            \tikz[baseline=-.6ex, scale=.1, yshift=-10cm]{
                \foreach \i in {0,1,...,13}{
                    \coordinate (C\i) at (0,30*\i/13);
                    \coordinate (L\i) at (-5,30*\i/13);
                    \coordinate (LL\i) at (-10,30*\i/13);
                    \coordinate (LLL\i) at (-15,30*\i/13);
                    \coordinate (R\i) at (5,30*\i/13);
                    \coordinate (RR\i) at (10,30*\i/13);
                    \coordinate (RRR\i) at (15,30*\i/13);
                }
                \begin{scope}
                \clip (-15,0) rectangle (15,30);
                    \draw[webline, rounded corners] ($(C7)+(-1,0)$) -- ($(C10)+(-1,0)$) -- ($(RR10)+(-1,0)$) -- ($(RR5)+(-1,0)$);
                    \draw[webline, rounded corners] (L7) -- (L11) -- ($(RR11)+(1,0)$) -- ($(RR5)+(1,0)$);
                    \draw[webline] ($(RR5)+(-1,0)$) -- ($(RR5)+(1,0)$);
                    \draw[webline] ($(C7)+(-1,0)$) -- (L7);
                    \draw[webline] ($(R5)+(6,0)$) -- ($(RR4)+(-1,0)$);
                    \draw[webline] (L5) -- ($(C4)+(-1,0)$);
                    \draw[webline] (L5) -- (LL4);
                    \draw[webline] ($(C4)+(-1,0)$) -- (L2);
                    \draw[webline] (LL4) -- (L2);
                    \draw[wline, rounded corners] ($(C7)+(-1,0)$) -- ($(C5)+(-1,0)$) -- ($(RR5)+(-1,0)$);
                    \draw[wline] ($(C4)+(-1,0)$) -- ($(RR4)+(-1,0)$);
                    \draw[wline] (L5) -- (L7);
                    %
                    \draw[webline, rounded corners] ($(R5)+(6,0)$) -- (RRR4);
                    \draw[webline] (LL4) to[out=north west, in=east] (LLL5);
                    \draw[webline] (LL4) to[out=south west, in=east] (LLL3);
                    \draw[webline] ($(RR4)+(-1,0)$) to[out=south east, in=west] (RRR3);
                    \draw[webline] (L2) to[out=south west, in=east] (LLL1);
                    \draw[webline] (L2) to[out=south east, in=west] (RRR1);
                    \draw[mygreen, rounded corners] (C0) -- ($(C6)+(1,0)$);
                    \draw[mygreen, rounded corners] (C6) -- (C9);
                    \draw[blue, very thick, rounded corners] ($(C6)+(1,0)$) -- ($(R6)+(3,0)$) -- ($(R9)+(3,0)$) -- (C9);
                    \draw[blue, very thick, rounded corners] (C0) -- ($(RRR2)+(-1,0)$) -- ($(RRR12)+(-1,0)$) -- ($(L12)+(-3,0)$) -- ($(L6)+(-3,0)$) -- (C6);
                \end{scope}
                \bline{-15,0}{15,0}{2}
                \draw[dashed] (-15,0) rectangle (15,30);
                \fill[black] (C0) circle [radius=20pt];
                \draw[black, fill=white] (C9) circle [radius=20pt];
                \node at ($(RR5)+(1,0)$) {$\crossroad$};
                \node at (LL4) {$\crossroad$};
                \node at (L2) {$\crossroad$};
            }
        }
        \mathop{\leadsto}^{\text{red.}}
        \mbox{
            \tikz[baseline=-.6ex, scale=.1, yshift=-10cm]{
                \foreach \i in {0,1,...,13}{
                    \coordinate (C\i) at (0,30*\i/13);
                    \coordinate (L\i) at (-5,30*\i/13);
                    \coordinate (LL\i) at (-10,30*\i/13);
                    \coordinate (LLL\i) at (-15,30*\i/13);
                    \coordinate (R\i) at (5,30*\i/13);
                    \coordinate (RR\i) at (10,30*\i/13);
                    \coordinate (RRR\i) at (15,30*\i/13);
                }
                \begin{scope}
                \clip (-15,0) rectangle (15,30);
                    \draw[webline, rounded corners] ($(C5)+(-1,0)$) -- ($(C10)+(-1,0)$) -- ($(RR10)+(-1,0)$) -- ($(RR5)+(-1,0)$);
                    \draw[webline, rounded corners] (L5) -- (L11) -- ($(RR11)+(1,0)$) -- ($(RR5)+(1,0)$);
                    \draw[webline] ($(RR5)+(-1,0)$) -- ($(RR5)+(1,0)$);
                    \draw[webline] ($(C5)+(-1,0)$) -- (L5);
                    \draw[webline] ($(R5)+(6,0)$) -- ($(RR4)+(-1,0)$);
                    \draw[webline] (L5) -- ($(C4)+(-1,0)$);
                    \draw[webline] (L5) -- (LL4);
                    \draw[webline] ($(C4)+(-1,0)$) -- (L2);
                    \draw[webline] (LL4) -- (L2);
                    \draw[wline] ($(C5)+(-1,0)$) -- ($(RR5)+(-1,0)$);
                    \draw[wline] ($(C4)+(-1,0)$) -- ($(RR4)+(-1,0)$);
                    \draw[webline, rounded corners] ($(R5)+(6,0)$) -- (RRR4);
                    \draw[webline] (LL4) to[out=north west, in=east] (LLL5);
                    \draw[webline] (LL4) to[out=south west, in=east] (LLL3);
                    \draw[webline] ($(RR4)+(-1,0)$) to[out=south east, in=west] (RRR3);
                    \draw[webline] (L2) to[out=south west, in=east] (LLL1);
                    \draw[webline] (L2) to[out=south east, in=west] (RRR1);
                    \draw[mygreen, rounded corners] (C0) -- ($(C6)+(1,0)$);
                    \draw[mygreen, rounded corners] (C6) -- (C9);
                    \draw[blue, very thick, rounded corners] ($(C6)+(1,0)$) -- ($(R6)+(3,0)$) -- ($(R9)+(3,0)$) -- (C9);
                    \draw[blue, very thick, rounded corners] (C0) -- ($(RRR2)+(-2,0)$) -- ($(RRR12)+(-2,0)$) -- ($(L12)+(-3,0)$) -- ($(L6)+(-3,0)$) -- (C6);
                \end{scope}
                \bline{-15,0}{15,0}{2}
                \draw[dashed] (-15,0) rectangle (15,30);
                \fill[black] (C0) circle [radius=20pt];
                \draw[black, fill=white] (C9) circle [radius=20pt];
                \node at (L5) {$\crossroad$};
                \node at ($(RR5)+(1,0)$) {$\crossroad$};
                \node at (LL4) {$\crossroad$};
                \node at (L2) {$\crossroad$};
            }
        }
        \mathop{\leadsto}^{\text{red.}}
        \mbox{
            \tikz[baseline=-.6ex, scale=.1, yshift=-10cm]{
                \foreach \i in {0,1,...,13}{
                    \coordinate (C\i) at (0,30*\i/13);
                    \coordinate (L\i) at (-5,30*\i/13);
                    \coordinate (LL\i) at (-10,30*\i/13);
                    \coordinate (LLL\i) at (-15,30*\i/13);
                    \coordinate (R\i) at (5,30*\i/13);
                    \coordinate (RR\i) at (10,30*\i/13);
                    \coordinate (RRR\i) at (15,30*\i/13);
                }
                \begin{scope}
                \clip (-15,0) rectangle (15,30);
                    \draw[webline, rounded corners] ($(C5)+(-1,0)$) -- ($(C10)+(-1,0)$) -- ($(RR10)+(-1,0)$) -- ($(RR5)+(-1,0)$);
                    \draw[webline, rounded corners] (L5) -- (LL5) -- ($(LL13)+(0,-1)$) -- ($(RRR13)+(-1,-1)$) -- ($(RRR5)+(-1,0)$);
                    \draw[webline] ($(RR5)+(-1,0)$) -- ($(RR5)+(1,0)$);
                    \draw[webline] ($(C5)+(-1,0)$) -- (L5);
                    \draw[webline] ($(R5)+(6,0)$) -- ($(RR4)+(-1,0)$);
                    \draw[webline] (L5) -- ($(C4)+(-1,0)$);
                    \draw[webline] (L5) -- (LL4);
                    \draw[webline] ($(C4)+(-1,0)$) -- (L2);
                    \draw[webline] (LL4) -- (L2);
                    \draw[wline] ($(C5)+(-1,0)$) -- ($(RR5)+(-1,0)$);
                    \draw[wline] ($(C4)+(-1,0)$) -- ($(RR4)+(-1,0)$);
                    \draw[wline] ($(RR5)+(1,0)$) -- ($(RRR5)+(-1,0)$);
                    \draw[webline] ($(RRR5)+(-1,0)$) -- (RRR4);
                    \draw[webline] (LL4) to[out=north west, in=east] (LLL5);
                    \draw[webline] (LL4) to[out=south west, in=east] (LLL3);
                    \draw[webline] ($(RR4)+(-1,0)$) to[out=south east, in=west] (RRR3);
                    \draw[webline] (L2) to[out=south west, in=east] (LLL1);
                    \draw[webline] (L2) to[out=south east, in=west] (RRR1);
                    \draw[mygreen, rounded corners] (C0) -- ($(C6)+(1,0)$);
                    \draw[mygreen, rounded corners] (C6) -- (C9);
                    \draw[blue, very thick, rounded corners] ($(C6)+(1,0)$) -- ($(R6)+(3,0)$) -- ($(R9)+(3,0)$) -- (C9);
                    \draw[blue, very thick, rounded corners] (C0) -- ($(RRR2)+(-3,0)$) -- ($(RRR12)+(-3,0)$) -- ($(L12)+(-3,0)$) -- ($(L6)+(-3,0)$) -- (C6);
                \end{scope}
                \bline{-15,0}{15,0}{2}
                \draw[dashed] (-15,0) rectangle (15,30);
                \fill[black] (C0) circle [radius=20pt];
                \draw[black, fill=white] (C9) circle [radius=20pt];
                \node at (L5) {$\crossroad$};
                \node at (LL4) {$\crossroad$};
                \node at (L2) {$\crossroad$};
            }
        }\\
        &\mathop{\leadsto}^{\text{$H$-move}}
        \mbox{
            \tikz[baseline=-.6ex, scale=.1, yshift=-10cm]{
                \foreach \i in {0,1,...,13}{
                    \coordinate (C\i) at (0,30*\i/13);
                    \coordinate (L\i) at (-5,30*\i/13);
                    \coordinate (LL\i) at (-10,30*\i/13);
                    \coordinate (LLL\i) at (-15,30*\i/13);
                    \coordinate (R\i) at (5,30*\i/13);
                    \coordinate (RR\i) at (10,30*\i/13);
                    \coordinate (RRR\i) at (15,30*\i/13);
                }
                \begin{scope}
                \clip (-15,0) rectangle (15,30);
                    \draw[webline, rounded corners] ($(C5)+(-2,2)$) -- ($(C10)+(-2,0)$) -- ($(RR10)+(-1,0)$) -- ($(RR5)+(-1,2)$);
                    \draw[webline, rounded corners] (L5) -- (LL5) -- ($(LL13)+(0,-1)$) -- ($(RRR13)+(-2,-1)$) -- ($(RRR5)+(-2,0)$);
                    \draw[webline] ($(RR5)+(-1,2)$) -- ($(RRR5)+(-2,0)$);
                    \draw[webline] ($(C5)+(-1,2)$) -- (L5);
                    \draw[webline] ($(RRR5)+(-2,0)$) -- ($(RRR4)+(-2,0)$);
                    \draw[webline] (L5) -- ($(C4)+(-1,0)$);
                    \draw[webline] (L5) -- (LL4);
                    \draw[webline] ($(C4)+(-1,0)$) -- (L2);
                    \draw[webline] (LL4) -- (L2);
                    \draw[wline] ($(C5)+(-2,2)$) -- ($(RR5)+(-1,2)$);
                    \draw[wline] ($(C4)+(-1,0)$) -- ($(RRR4)+(-2,0)$);
                    \draw[webline] ($(RRR5)+(-2,0)$) -- (RRR4);
                    \draw[webline] (LL4) to[out=north west, in=east] (LLL5);
                    \draw[webline] (LL4) to[out=south west, in=east] (LLL3);
                    \draw[webline] ($(RRR4)+(-2,0)$) -- (RRR3);
                    \draw[webline] (L2) to[out=south west, in=east] (LLL1);
                    \draw[webline] (L2) to[out=south east, in=west] (RRR1);
                    \draw[mygreen, rounded corners] (C0) -- ($(C6)+(1,1)$);
                    \draw[mygreen, rounded corners] ($(C6)+(0,1)$) -- (C9);
                    \draw[blue, very thick, rounded corners] ($(C6)+(1,1)$) -- ($(R6)+(3,1)$) -- ($(R9)+(3,0)$) -- (C9);
                    \draw[blue, very thick, rounded corners] (C0) -- ($(RRR2)+(-5,0)$) -- ($(RRR12)+(-5,-2)$) -- ($(L12)+(-3,-2)$) -- ($(L6)+(-3,1)$) -- ($(C6)+(0,1)$);
                \end{scope}
                \bline{-15,0}{15,0}{2}
                \draw[dashed] (-15,0) rectangle (15,30);
                \fill[black] (C0) circle [radius=20pt];
                \draw[black, fill=white] (C9) circle [radius=20pt];
                \node at (L5) {$\crossroad$};
                \node at ($(RRR5)+(-2,0)$) {$\crossroad$};
                \node at (LL4) {$\crossroad$};
                \node at (L2) {$\crossroad$};
            }
        }
        \mathop{\leadsto}^{\text{red.}}
        \mbox{
            \tikz[baseline=-.6ex, scale=.1, yshift=-10cm]{
                \foreach \i in {0,1,...,13}{
                    \coordinate (C\i) at (0,30*\i/13);
                    \coordinate (L\i) at (-5,30*\i/13);
                    \coordinate (LL\i) at (-10,30*\i/13);
                    \coordinate (LLL\i) at (-15,30*\i/13);
                    \coordinate (R\i) at (5,30*\i/13);
                    \coordinate (RR\i) at (10,30*\i/13);
                    \coordinate (RRR\i) at (15,30*\i/13);
                }
                \begin{scope}
                \clip (-15,0) rectangle (15,30);
                    \draw[webline, rounded corners] ($(C5)+(-2,2)$) -- ($(LL5)+(1,2)$) -- ($(LL12)+(1,0)$) -- ($(RR12)+(1,0)$) -- ($(RR5)+(1,2)$);
                    \draw[webline, rounded corners] (L5) -- ($(LL5)+(-1,0)$) -- ($(LL13)+(-1,-1)$) -- ($(RRR13)+(-2,-1)$) -- ($(RRR5)+(-2,0)$);
                    \draw[webline] ($(RR5)+(1,2)$) -- ($(RRR5)+(-2,0)$);
                    \draw[webline] ($(C5)+(-2,2)$) -- (L5);
                    \draw[webline] ($(RRR5)+(-2,0)$) -- ($(RRR4)+(-2,0)$);
                    \draw[webline] (L5) -- ($(C4)+(-1,0)$);
                    \draw[webline] (L5) -- (LL4);
                    \draw[webline] ($(C4)+(-1,0)$) -- (L2);
                    \draw[webline] (LL4) -- (L2);
                    \draw[wline] ($(C5)+(-2,2)$) -- ($(RR5)+(1,2)$);
                    \draw[wline] ($(C4)+(-1,0)$) -- ($(RRR4)+(-2,0)$);
                    \draw[webline] ($(RRR5)+(-2,0)$) -- (RRR4);
                    \draw[webline] (LL4) to[out=north west, in=east] (LLL5);
                    \draw[webline] (LL4) to[out=south west, in=east] (LLL3);
                    \draw[webline] ($(RRR4)+(-2,0)$) -- (RRR3);
                    \draw[webline] (L2) to[out=south west, in=east] (LLL1);
                    \draw[webline] (L2) to[out=south east, in=west] (RRR1);
                    \draw[mygreen, rounded corners] (C0) -- ($(C6)+(1,1)$);
                    \draw[mygreen, rounded corners] ($(C6)+(0,1)$) -- (C9);
                    \draw[blue, very thick, rounded corners] ($(C6)+(1,1)$) -- ($(R6)+(3,1)$) -- ($(R9)+(3,0)$) -- (C9);
                    \draw[blue, very thick, rounded corners] (C0) -- ($(RRR2)+(-5,0)$) -- ($(RRR12)+(-5,-3)$) -- ($(L12)+(-3,-3)$) -- ($(L6)+(-3,1)$) -- ($(C6)+(0,1)$);
                \end{scope}
                \bline{-15,0}{15,0}{2}
                \draw[dashed] (-15,0) rectangle (15,30);
                \fill[black] (C0) circle [radius=20pt];
                \draw[black, fill=white] (C9) circle [radius=20pt];
                \node at (L5) {$\crossroad$};
                \node at ($(RRR5)+(-2,0)$) {$\crossroad$};
                \node at (LL4) {$\crossroad$};
                \node at (L2) {$\crossroad$};
            }
        }
    \end{align*}
\end{ex}

\begin{cor}
   Let $\gamma$ be any ideal arc from $p$ to $q$ ($p,q\in\bM$), $D$ and $D'$ ladder-equivalent $\fsp_4$-diagrams in $\Sigma$. If both $D$ and $D'$ are $\{\gamma\}$-minimal and reduced, then $D$ can be deformed into $D'$ by a sequence of $H$-moves along $\gamma$, ladder-equivalence relation relative to $\gamma$, arc and loop parallel-moves.
\end{cor}
\begin{proof}
    \cref{prop:minimal-position} implies the assertion because intersection reduction moves along $\gamma$ reduce the $\{\gamma\}$-degree.
\end{proof}

\begin{prop}\label{prop:exist-S-minimal}
    Let $S$ be a finite collection of disjoint ideal arcs. Any $S$-transverse reduced $\fsp_4$-diagram $D$ can be deformed into $D'$ by a sequence of intersection reduction moves along $S$, $H$-moves along $S$, and ladder-equivalence relative to $S$ so that $D'$ minimizes $\{\gamma\}$-degree for each $\gamma\in S$ simultaneously.
\end{prop}
\begin{proof}
    We first remark that arc and loop parallel-moves do not affect the degree of $D$.
    The statement follows from the fact that any local move of $D$ along $\gamma\in S$, as described in \cref{prop:minimal-position}, does not increase the degree along $\gamma'$ for any $\gamma\neq\gamma'$. We will prove this fact.
    Let us consider the position of $\gamma'$ in the local disk $U$ corresponding to the local moves along $\gamma$.
    A replacement of $U$ with a smaller one makes $\gamma'\cap U$ simpler. For any $\gamma'\cap U$ parallel to $\gamma\cap U$, it is clear that the moves along $\gamma$ in $U$ do not increase the degree along $\gamma'$. Hence the configurations of $\gamma'$ in $U$ that should be considered are as follows:
    \begin{align*}
        &\mbox
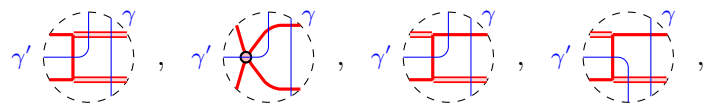
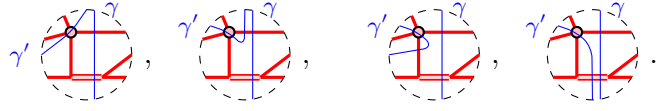
{
        \tikz[baseline=-.6ex, scale=.1]{
            \draw[dashed, fill=white] (0,0) circle [radius=6];
            \coordinate (U) at ($(150:6)!.3!(30:6)$);
            \coordinate (D) at ($(-150:6)!.3!(-30:6)$);
            \draw[webline] (150:6) -- (U);
            \draw[wline] (U) -- (30:6);
            \draw[webline] (-150:6) -- (D);
            \draw[wline] (D) -- (-30:6);
            \draw[webline] (U) -- (D);
            \draw[blue] (60:6) -- (-60:6);
            \draw[blue, rounded corners] (180:6) -- (0,0) -- (90:6);
            \node at (180:6) [left, blue]{\scriptsize $\gamma'$};
            \node at (60:6) [right, blue]{\scriptsize $\gamma$};
        }
        }, 
        \mbox{
        \tikz[baseline=-.6ex, scale=.1]{
            \draw[dashed, fill=white] (0,0) circle [radius=6];
            \draw[webline, rounded corners] (45:6) -- ($(45:6)+(-4,0)$) -- (-3,0) -- (-135:6);
            \draw[webline, rounded corners] (-45:6) -- ($(-45:6)+(-4,0)$) -- (-3,0) -- (135:6);
            \draw[fill=pink, thick] (-3,0) circle [radius=20pt];
            \draw[blue] (60:6) -- (-60:6);
            \draw[blue, rounded corners] (180:6) -- (0,0) -- (90:6);
            \node at (180:6) [left, blue]{\scriptsize $\gamma'$};
            \node at (60:6) [right, blue]{\scriptsize $\gamma$};
        }
        },
        \mbox{
        \tikz[baseline=-.6ex, scale=.1]{
            \draw[dashed, fill=white] (0,0) circle [radius=6];
            \coordinate (U) at ($(150:6)!.3!(30:6)$);
            \coordinate (D) at ($(-150:6)!.3!(-30:6)$);
            \draw[wline] (150:6) -- (U);
            \draw[webline] (U) -- (30:6);
            \draw[webline] (-150:6) -- (D);
            \draw[wline] (D) -- (-30:6);
            \draw[webline] (U) -- (D);
            \draw[blue] (60:6) -- (-60:6);
            \draw[blue, rounded corners] (180:6) -- (0,0) -- (90:6);
            \node at (180:6) [left, blue]{\scriptsize $\gamma'$};
            \node at (60:6) [right, blue]{\scriptsize $\gamma$};
        }
        },
        \mbox{
        \tikz[baseline=-.6ex, scale=.1]{
            \draw[dashed, fill=white] (0,0) circle [radius=6];
            \coordinate (U) at ($(150:6)!.3!(30:6)$);
            \coordinate (D) at ($(-150:6)!.3!(-30:6)$);
            \draw[wline] (150:6) -- (U);
            \draw[webline] (U) -- (30:6);
            \draw[webline] (-150:6) -- (D);
            \draw[wline] (D) -- (-30:6);
            \draw[webline] (U) -- (D);
            \draw[blue] (60:6) -- (-60:6);
            \draw[blue, rounded corners] (180:6) -- (0,0) -- (-90:6);
            \node at (180:6) [left, blue]{\scriptsize $\gamma'$};
            \node at (60:6) [right, blue]{\scriptsize $\gamma$};
        }
        },\\
        &\mbox{
        \tikz[baseline=-.6ex, scale=.1]{
            \draw[dashed, fill=white] (0,0) circle [radius=6];
            \coordinate (LU) at ($(150:6)!.3!(30:6)$);
            \coordinate (LD) at ($(-150:6)!.3!(-30:6)$);
            \coordinate (RU) at ($(150:6)!.7!(30:6)$);
            \coordinate (RD) at ($(-150:6)!.7!(-30:6)$);
            \coordinate (L) at (-180:6);
            \coordinate (R) at (0:6);
            \draw[webline] (160:6) -- (LU) -- (30:6);
            \draw[webline] (-150:6) -- (LD);
            \draw[wline] (LD) -- (RD);
            \draw[webline] (120:6) -- (LU) -- (LD);
            \draw[webline] (RD) -- (R);
            \draw[webline] (RD) -- (-30:6);
            \draw[blue] (80:6) -- (-80:6);
            \draw[fill=pink, thick] (LU) circle [radius=20pt];
            \draw[blue, rounded corners] (180:6) -- (LU) -- (90:6);
            \node at (180:6) [left, blue]{\scriptsize $\gamma'$};
            \node at (80:6) [right, blue]{\scriptsize $\gamma$};
        }
        },
        \mbox{
        \tikz[baseline=-.6ex, scale=.1]{
            \draw[dashed, fill=white] (0,0) circle [radius=6];
            \coordinate (LU) at ($(150:6)!.3!(30:6)$);
            \coordinate (LD) at ($(-150:6)!.3!(-30:6)$);
            \coordinate (RU) at ($(150:6)!.7!(30:6)$);
            \coordinate (RD) at ($(-150:6)!.7!(-30:6)$);
            \coordinate (L) at (-180:6);
            \coordinate (R) at (0:6);
            \draw[webline] (160:6) -- (LU) -- (30:6);
            \draw[webline] (-150:6) -- (LD);
            \draw[wline] (LD) -- (RD);
            \draw[webline] (120:6) -- (LU) -- (LD);
            \draw[webline] (RD) -- (R);
            \draw[webline] (RD) -- (-30:6);
            \draw[blue] (80:6) -- (-80:6);
            \draw[fill=pink, thick] (LU) circle [radius=20pt];
            \draw[blue, rounded corners] (140:6) -- (LU) -- ($(LU)+(2,-2)$) -- (90:6);
            \node at (140:6) [left, blue]{\scriptsize $\gamma'$};
            \node at (80:6) [right, blue]{\scriptsize $\gamma$};
        }
        },\quad
        \mbox{
        \tikz[baseline=-.6ex, scale=.1]{
            \draw[dashed, fill=white] (0,0) circle [radius=6];
            \coordinate (LU) at ($(150:6)!.3!(30:6)$);
            \coordinate (LD) at ($(-150:6)!.3!(-30:6)$);
            \coordinate (RU) at ($(150:6)!.7!(30:6)$);
            \coordinate (RD) at ($(-150:6)!.7!(-30:6)$);
            \coordinate (L) at (-180:6);
            \coordinate (R) at (0:6);
            \draw[webline] (160:6) -- (LU) -- (30:6);
            \draw[webline] (-150:6) -- (LD);
            \draw[wline] (LD) -- (RD);
            \draw[webline] (120:6) -- (LU) -- (LD);
            \draw[webline] (RD) -- (R);
            \draw[webline] (RD) -- (-30:6);
            \draw[blue] (80:6) -- (-80:6);
            \draw[fill=pink, thick] (LU) circle [radius=20pt];
            \draw[blue, rounded corners] (140:6) -- (LU) -- ($(LU)+(2,-2)$) -- (180:6);
            \node at (140:6) [left, blue]{\scriptsize $\gamma'$};
            \node at (80:6) [right, blue]{\scriptsize $\gamma$};
        }
        },
        \mbox{
        \tikz[baseline=-.6ex, scale=.1]{
            \draw[dashed, fill=white] (0,0) circle [radius=6];
            \coordinate (LU) at ($(150:6)!.3!(30:6)$);
            \coordinate (LD) at ($(-150:6)!.3!(-30:6)$);
            \coordinate (RU) at ($(150:6)!.7!(30:6)$);
            \coordinate (RD) at ($(-150:6)!.7!(-30:6)$);
            \coordinate (L) at (-180:6);
            \coordinate (R) at (0:6);
            \draw[webline] (160:6) -- (LU) -- (30:6);
            \draw[webline] (-150:6) -- (LD);
            \draw[wline] (LD) -- (RD);
            \draw[webline] (120:6) -- (LU) -- (LD);
            \draw[webline] (RD) -- (R);
            \draw[webline] (RD) -- (-30:6);
            \draw[blue] (80:6) -- (-80:6);
            \draw[fill=pink, thick] (LU) circle [radius=20pt];
            \draw[blue, rounded corners] (140:6) -- (LU) -- ($(LU)+(2,-2)$) -- (-90:6);
            \node at (140:6) [left, blue]{\scriptsize $\gamma'$};
            \node at (80:6) [right, blue]{\scriptsize $\gamma$};
        }
        }. 
    \end{align*}
    For each configuration of $\gamma'$ and the corresponding ladder resolution of $D$, one can easily construct a sequence of local moves along $\gamma'$ that does not increase the degree and realizes the original local move along $\gamma$. 
    For example,
    \begin{align*}
        &\mbox
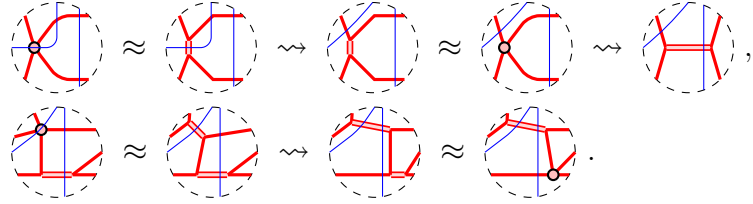
{
        \tikz[baseline=-.6ex, scale=.1]{
            \draw[dashed, fill=white] (0,0) circle [radius=6];
            \draw[webline, rounded corners] (45:6) -- ($(45:6)+(-4,0)$) -- (-3,0) -- (-135:6);
            \draw[webline, rounded corners] (-45:6) -- ($(-45:6)+(-4,0)$) -- (-3,0) -- (135:6);
            \draw[fill=pink, thick] (-3,0) circle [radius=20pt];
            \draw[blue] (60:6) -- (-60:6);
            \draw[blue, rounded corners] (180:6) -- (0,0) -- (90:6);
        }
        }
        \approx
        \mbox{
        \tikz[baseline=-.6ex, scale=.1]{
            \draw[dashed, fill=white] (0,0) circle [radius=6];
            \draw[webline] (45:6) -- ($(45:6)+(-4,0)$) -- (-3,1) -- (135:6);
            \draw[webline] (-45:6) -- ($(-45:6)+(-4,0)$) -- (-3,-1) -- (-135:6);
            \draw[wline] (-3,1) -- (-3,-1);
            \draw[blue] (60:6) -- (-60:6);
            \draw[blue, rounded corners] (180:6) -- (0,0) -- (90:6);
        }
        }
        \rightsquigarrow
        \mbox{
        \tikz[baseline=-.6ex, scale=.1]{
            \draw[dashed, fill=white] (0,0) circle [radius=6];
            \draw[webline] (45:6) -- ($(45:6)+(-4,0)$) -- (-3,1) -- (135:6);
            \draw[webline] (-45:6) -- ($(-45:6)+(-4,0)$) -- (-3,-1) -- (-135:6);
            \draw[wline] (-3,1) -- (-3,-1);
            \draw[blue] (60:6) -- (-60:6);
            \draw[blue] (180:6) to[bend right=5] (90:6);
        }
        }
        \approx
        \mbox{
        \tikz[baseline=-.6ex, scale=.1]{
            \draw[dashed, fill=white] (0,0) circle [radius=6];
            \draw[webline, rounded corners] (45:6) -- ($(45:6)+(-4,0)$) -- (-3,0) -- (-135:6);
            \draw[webline, rounded corners] (-45:6) -- ($(-45:6)+(-4,0)$) -- (-3,0) -- (135:6);
            \draw[fill=pink, thick] (-3,0) circle [radius=20pt];
            \draw[blue] (60:6) -- (-60:6);
            \draw[blue] (180:6) to[bend right=5] (90:6);
        }
        }
        \rightsquigarrow
        \mbox{
        \tikz[baseline=-.6ex, scale=.1]{
            \draw[dashed, fill=white] (0,0) circle [radius=6];
            \draw[webline] (45:6) -- (3,0);
            \draw[webline] (-45:6) -- (3,0);
            \draw[webline] (135:6) -- (-3,0);
            \draw[webline] (-135:6) -- (-3,0);
            \draw[wline] (-3,0) -- (3,0);
            \draw[blue] (70:6) -- (-70:6);
            \draw[blue] (180:6) to[bend right=5] (90:6);
        }
        },\\
        &\mbox{
        \tikz[baseline=-.6ex, scale=.1]{
            \draw[dashed, fill=white] (0,0) circle [radius=6];
            \coordinate (LU) at ($(150:6)!.3!(30:6)$);
            \coordinate (LD) at ($(-150:6)!.3!(-30:6)$);
            \coordinate (RU) at ($(150:6)!.7!(30:6)$);
            \coordinate (RD) at ($(-150:6)!.7!(-30:6)$);
            \coordinate (L) at (-180:6);
            \coordinate (R) at (0:6);
            \draw[webline] (160:6) -- (LU) -- (30:6);
            \draw[webline] (-150:6) -- (LD);
            \draw[wline] (LD) -- (RD);
            \draw[webline] (120:6) -- (LU) -- (LD);
            \draw[webline] (RD) -- (R);
            \draw[webline] (RD) -- (-30:6);
            \draw[blue] (80:6) -- (-80:6);
            \draw[fill=pink, thick] (LU) circle [radius=20pt];
            \draw[blue, rounded corners] (180:6) -- (LU) -- (90:6);
        }
        }
        \approx
        \mbox{
        \tikz[baseline=-.6ex, scale=.1]{
            \draw[dashed, fill=white] (0,0) circle [radius=6];
            \coordinate (LU) at ($(150:6)!.3!(30:6)$);
            \coordinate (LD) at ($(-150:6)!.3!(-30:6)$);
            \coordinate (RU) at ($(150:6)!.7!(30:6)$);
            \coordinate (RD) at ($(-150:6)!.7!(-30:6)$);
            \coordinate (L) at (-180:6);
            \coordinate (R) at (0:6);
            \draw[webline] ($(LU)+(1,-1)$) -- (30:6);
            \draw[webline] (-150:6) -- (LD);
            \draw[wline] (LD) -- (RD);
            \draw[webline] ($(LU)+(1,-1)$) -- (LD);
            \draw[webline] (RD) -- (R);
            \draw[webline] (RD) -- (-30:6);
            \draw[webline] (160:6) -- ($(LU)+(-1,1)$) -- (120:6);
            \draw[wline] ($(LU)+(-1,1)$) -- ($(LU)+(1,-1)$);
            \draw[blue] (80:6) -- (-80:6);
            \draw[blue, rounded corners] (180:6) -- (LU) -- (90:6);
        }
        }
        \rightsquigarrow
        \mbox{
        \tikz[baseline=-.6ex, scale=.1]{
            \draw[dashed, fill=white] (0,0) circle [radius=6];
            \coordinate (LU) at ($(150:6)!.3!(30:6)$);
            \coordinate (LD) at ($(-150:6)!.3!(-30:6)$);
            \coordinate (RU) at ($(150:6)!.7!(30:6)$);
            \coordinate (RD) at ($(-150:6)!.7!(-30:6)$);
            \coordinate (L) at (-180:6);
            \coordinate (R) at (0:6);
            \draw[webline] (RU) -- (30:6);
            \draw[webline] (-150:6) -- (RD);
            \draw[wline] (RD) -- ($(RD)+(2,0)$);
            \draw[webline] (RU) -- (RD);
            \draw[webline] ($(RD)+(2,0)$) -- (R);
            \draw[webline] ($(RD)+(2,0)$) -- (-30:6);
            \draw[webline] (160:6) -- ($(LU)+(-1,1)$) -- (120:6);
            \draw[wline] ($(LU)+(-1,1)$) -- (RU);
            \draw[blue] (80:6) -- (-80:6);
            \draw[blue, rounded corners] (180:6) -- (LU) -- (90:6);
        }
        }
        \approx
        \mbox{
        \tikz[baseline=-.6ex, scale=.1]{
            \draw[dashed, fill=white] (0,0) circle [radius=6];
            \coordinate (LU) at ($(150:6)!.3!(30:6)$);
            \coordinate (LD) at ($(-150:6)!.3!(-30:6)$);
            \coordinate (RU) at ($(150:6)!.7!(30:6)$);
            \coordinate (RD) at ($(-150:6)!.7!(-30:6)$);
            \coordinate (L) at (-180:6);
            \coordinate (R) at (0:6);
            \draw[webline] (RU) -- (30:6);
            \draw[webline] (-150:6) -- ($(RD)+(1,0)$);
            \draw[webline] (RU) -- ($(RD)+(1,0)$);
            \draw[webline] ($(RD)+(1,0)$) -- (R);
            \draw[webline] ($(RD)+(1,0)$) -- (-30:6);
            \draw[webline] (160:6) -- ($(LU)+(-1,1)$) -- (120:6);
            \draw[wline] ($(LU)+(-1,1)$) -- (RU);
            \draw[blue] (80:6) -- (-80:6);
            \draw[fill=pink, thick] ($(RD)+(1,0)$) circle [radius=20pt];
            \draw[blue, rounded corners] (180:6) -- (LU) -- (90:6);
        }
        }.
    \end{align*}
    In the above list of $\gamma'$, we exclude configurations obtained by sliding $\gamma'$ slightly away from the crossroad because these diagrams come back to the former diagram by an intersection reduction move.
\end{proof}

\begin{cor}\label{cor:elliptic-minimal}
    For a non-elliptic $\fsp_4$-diagram $D$ and an ideal arc $\gamma$. The following are equivalent:
    \begin{enumerate}
        \item $D$ is $\gamma$-minimal.
        \item $D$ has no elliptic bordered face along $\gamma$ up to $H$-moves along $\gamma$.
    \end{enumerate}
\end{cor}
\begin{proof}
    (1) implies (2) because any intersection reduction move decreases the $\gamma$-degree of $D$. Conversely, If $D$ is not $\gamma$-minimal, then $D$ have at least one elliptic bordered face along $\gamma$ up to $H$-moves along $\gamma$ by \cref{prop:minimal-position}. Hence (2) implies (1).
\end{proof}

\begin{cor}\label{cor:exist-good}
    Let $S$ be a disjoint collection of ideal arcs in $\Sigma$ and $L\in\Blad_{\fsp_4,\bSigma}$.
    \begin{enumerate}
        \item Any representative $D$ of $L$ in a good position with respect to $S^{\mathrm{split}}$ is $S$-minimal.
        \item Any $S$-minimal representative $D\in\diag_{\fsp_4,\bSigma}$ of $L$ can be deformed into a good position with respect to $S^{\mathrm{split}}$ by a finite sequence of $H$-moves along $S^{\mathrm{split}}$.
    \end{enumerate}
\end{cor}
\begin{proof}
    We first prove (1). Suppose that $D$ is in a good position and not $S$-minimal. Then $D$ should have an elliptic bordered face along some $\gamma\in S$. It contradicts the $S$-minimality by \cref{cor:elliptic-minimal}.
    
    Next, we prove (2). \cref{cor:elliptic-minimal} claims that $D$ has only bordered $H$-faces along  any $\gamma\in S$. Hence, a finite sequence of $H$-moves along $S^{\mathrm{split}}$ deforms $D$ into $D'$ so that $D'\cap \Sigma'$ is reduced for any connected component $\Sigma'\in t(S^{\mathrm{split}})$. 
\end{proof}

\begin{dfn}[crossbar pass]\label{dfn:crossbar-pass}
    Let $S$ be a set of disjoint ideal arcs in a marked surface $\bSigma$. For $S^{\mathrm{split}}$ and an $S^{\mathrm{split}}$-transverse $\fsp_4$-diagram $D\in\diag_{\fsp_4,\bSigma}$, a \emph{crossbar pass} of $D$ is a deformation of $D$ that transports a bordered $H$-face along $\gamma_{\pm}$ in $B_{\gamma}\in b(S^{\mathrm{split}})$ to another in an adjacent biangle along parallel arcs, see \cref{fig:crossbar-pass}.
\end{dfn}

\begin{figure}
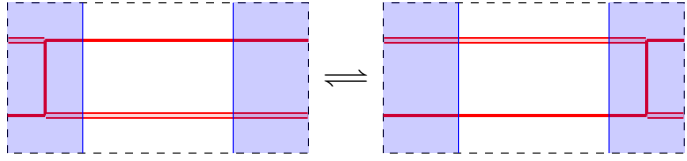

    \begin{align*}
        \mbox{
        \tikz[baseline=-.6ex, scale=.1, yshift=-10cm]{
            \draw[dashed, fill=white] (0,0) rectangle (40,20);
            \draw[wline] (0,15) -- (5,15);
            \draw[webline] (0,5) -- (5,5);
            \draw[webline] (5,5) -- (5,15);
            \draw[wline] (5,5) -- (40,5);
            \draw[webline] (5,15) -- (40,15);
            \fill[blue, opacity=.2] (0,0) rectangle (10,20);
            \fill[blue, opacity=.2] (30,0) rectangle (40,20);
            \draw[blue] (10,0) -- (10,20);
            \draw[blue] (30,0) -- (30,20);
        }
        } 
        \scalebox{1.5}[1]{$\rightleftharpoons$}
        \mbox{
        \tikz[baseline=-.6ex, scale=.1, yshift=-10cm]{
            \draw[dashed, fill=white] (0,0) rectangle (40,20);
            \draw[wline] (0,15) -- (35,15);
            \draw[webline] (0,5) -- (35,5);
            \draw[webline] (35,5) -- (35,15);
            \draw[wline] (35,5) -- (40,5);
            \draw[webline] (35,15) -- (40,15);
            \fill[blue, opacity=.2] (0,0) rectangle (10,20);
            \fill[blue, opacity=.2] (30,0) rectangle (40,20);
            \draw[blue] (10,0) -- (10,20);
            \draw[blue] (30,0) -- (30,20);
        }
        } 
    \end{align*}
    \caption{Crossbar pass: blue shaded regions belong to biangles (might be the same biangle) in $b(S^{\mathrm{split}})$.}
    \label{fig:crossbar-pass}
\end{figure}

\begin{thm}\label{thm:unique-S-good}
    Let $S$ be a finite set of disjoint ideal arcs and let $D_1,D_2\in\diag_{\fsp_4,\bSigma}$ are $\fsp_4$-diagrams in good positions with respect to $S^{\mathrm{split}}$. If $D_1\approx D_2$, then $D_2$ is connected to $D_1$ by a finite sequence of loop parallel-moves, arc parallel-moves, and crossbar passes along $S^{\mathrm{split}}$.
\end{thm}

\begin{proof}
    From \cref{lem:equivalence-generator} and the independence of loop and arc components with respect to moves (D1), (D2), and (D4) in \cref{subsec:diagram-definition}, we assume that loop and arc components are arranged in $D_1$ and $D_2$ similarly. 
    We will prove it by induction on the number of ideal arcs in $S=\{\gamma_1,\dots,\gamma_n\}$.
    
    For any marked surface $\bSigma$ with an ideal arc $\gamma$, $D_1$ and $D_2$ are $\{\gamma_{-}\}$-minimal by definition. Then, by \cref{prop:minimal-position}, they are connected by a sequence of $H$-moves along $\gamma_{-}$. These $H$-moves may or may not pass through $\gamma_{+}$; however, it turned out that they must be crossbar passes from $B_\gamma$ to itself because any bordered $H$-face along $\gamma_{\pm}$ should be contained in $B_{\gamma}$.
    
    For $n>1$, assume that the statement holds for any marked surface and any disjoint collection of $n-1$ ideal arcs on the surface.  
    Under this assumption, let us show that the statement also holds for a marked surface $\bSigma$ with $S=\{\gamma_1,\dots,\gamma_n\}$. $S$-minimal $\fsp_4$-diagrams $D_1$ and $D_2$ are in particular $\gamma_{n}$-minimal. 
    For with $D_1\approx D_2$, we will show that $D_2$ can be deformed into reduced $S$-minimal $\fsp_4$-diagrams $D_1$ by a sequence of crossbar passes between biangles in $b(S^{\mathrm{split}})$ so that $D_{1}\approx_{\{\gamma_{n,+},\gamma_{n,-}\}} D_{2}'$. If one can prove it, then the induction assumption can be applied to $\fsp_4$-diagrams $D_1\cap\Sigma'$ and $D_2'\cap\Sigma'$ on $\Sigma'\coloneqq \Sigma\setminus B_{\gamma_{n}}$ with the set of ideal arcs $S'\coloneqq\{\gamma_1,\ldots,\gamma_{n-1}\}$. Consequently, $D_{1}$ and $D_{2}'$ are related by a sequence of crossbar passes between $b({S'}^{\mathrm{split}})$, and hence the same holds for $D_{1}$ and $D_{2}$.

    Let us construct $D'_2$ using the same method as \cref{prop:minimal-position}. We denote a sequence of deformations from $D_1$ to $D_2$ by $\psi$ based on $D_1\approx D_2$. Then the biangle $B_{\gamma_{i}}$ on $D_2$ can be described as a biangle, denoted by $\psi^{\ast}B_{\gamma_{i}}$, on $D_1$ through the deformation $\psi$ for $i=1,\ldots, n$. We will construct sequences $\phi$ of deformations of $\fsp_4$-diagrams with good position from $\sqcup_{i} \psi^{\ast}B_{\gamma_{i}}$ such that this deformation bring $\psi^{\ast}B_{\gamma_{n}}$ to $B_{\gamma_{n}}$. In addition, $\phi$ must be composed only of deformations corresponding to crossbar passes.
    We will consider biangles bounded by $\psi^{\ast}\gamma_{n,\pm}$ in $\Sigma$ with $S^{\mathrm{split}}$. 
    For $p$ and $q$ on an arc $\alpha$, $\alpha_{pq}$ means subarc between $p$ and $q$.
    First, let us consider the following biangle $R$ between $\psi^{\ast}\gamma_{n,\epsilon'}$ and $\gamma_{i,\epsilon}$:
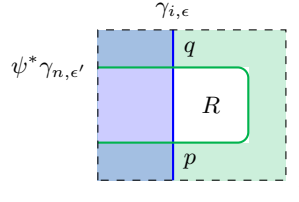
    \begin{align*}
        \mbox{
            \tikz[baseline=-.6ex, scale=.1]{
                \fill[mygreen, opacity=.2] (0,0) rectangle (25,20);
                \fill[white] (0,5) rectangle (20,15);
                \fill[blue, opacity=.2] (0,0) rectangle (10,20);
                \draw[blue, thick] (10,0) -- (10,20);
                \draw[mygreen, thick, rounded corners] (0,5) -- (20,5) -- (20,15) -- (0,15);
                \draw[dashed] (0,0) rectangle (25,20);
                \node at (10,5) [below right]{\scriptsize $p$};
                \node at (10,15) [above right]{\scriptsize $q$};
                \node at (10,20) [above]{\scriptsize $\gamma_{i,\epsilon}$};
                \node at (0,15) [left]{\scriptsize $\psi^{\ast}\gamma_{n,\epsilon'}$};
                \node at (15,10) {\scriptsize $R$};
            }
        },
    \end{align*}
    where the blue and green regions are $B_{i}$ and $\psi^{\ast}B_{n}$, respectively. 
    If $R$ has a trivial $\fsp_4$-diagram, then the biangle can be resolved by isotopy between $\psi^{\ast}\gamma_{n,\epsilon'}$ and $\gamma_{i,\epsilon}$. This deformation of $\psi^{\ast}\gamma_{n,\epsilon'}$ does not affect the $\fsp_4$-diagram $D_2$. Thus, we assume $R$ has a non-trivial diagram. The minimality of ideal arcs implies that two arcs from $p$ to $q$ have the same degree, and it implies that the right and left sides of $R$ are connected by a flat braid diagram (or the corresponding $\fsp_4$-diagram) by \cref{lem:biangle_blad}. However, the good position property claims that the flat braid diagram in $R$ should be contained in some $\psi^{\ast}B_{\gamma_j}$ for $j=1,2,\dots,n$. Hence $R$ becomes
    \begin{align*}
        \mbox
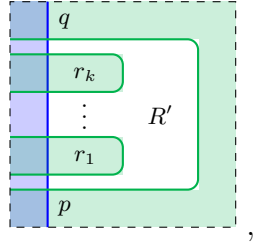
{
            \tikz[baseline=-.6ex, scale=.1]{
                \fill[mygreen, opacity=.2] (0,0) rectangle (30,30);
                \fill[white] (0,5) rectangle (25,25);
                \fill[blue, opacity=.2] (0,0) rectangle (5,30);
                \draw[blue, thick] (5,0) -- (5,30);
                \fill[mygreen, opacity=.2] (0,7) rectangle (15,12);
                \fill[mygreen, opacity=.2] (0,18) rectangle (15,23);
                \draw[mygreen, thick, rounded corners] (0,5) -- (25,5) -- (25,25) -- (0,25);
                \draw[mygreen, thick, rounded corners] (0,7) -- (15,7) -- (15,12) -- (0,12);
                \draw[mygreen, thick, rounded corners] (0,18) -- (15,18) -- (15,23) -- (0,23);
                \draw[dashed] (0,0) rectangle (30,30);
                \node at (10,7) [above]{\scriptsize $r_1$};
                \node at (10,18) [above]{\scriptsize $r_k$};
                \node at (10,15) [rotate=90]{\scriptsize $\cdots$};
                \node at (5,5) [below right]{\scriptsize $p$};
                \node at (5,25) [above right]{\scriptsize $q$};
                \node at (20,15) {\scriptsize $R'$};
            }
        },
    \end{align*}
    where biangles $r_{1},\ldots,r_{k}$ are intersections of $R$ and $\sqcup_{i}\psi^{\ast}B_{\gamma_i}$, and $R'$ is $R\setminus\sqcup_{j=1}^{k} r_{j}$.
    The right side of $r_j$. $r_{j}$ is connected by a flat braid diagram for any $j=1,\ldots,k$ and left and right sides of $R'$ are connected by a trivial $\fsp_4$-diagram. Then, the right side of $r_j$ and the right side of $R'$ can jump a crossing of the flat braid in $r_j$ simultaneously. Repeating such an operation removes $r_{j}$ from $R$ and finally reduces the biangle $R$. We remark that a crossing of a flat braid represents an ``$H$'', hence this deformation of $\sqcup_{i}\psi^{\ast}B_{\gamma_i}$ corresponds to crossbar passes.

    Next, let us consider a biangle $R$ appearing as an intersection of $B_{\gamma_{i}}$ and $\psi^{\ast}B_{\gamma_{n}}$ for some $i$, which is illustrated as below
    \begin{align*}
        \mbox
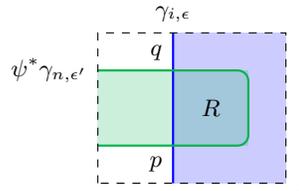
{
            \tikz[baseline=-.6ex, scale=.1]{
                \fill[blue, opacity=.2] (10,0) rectangle (25,20);
                \fill[mygreen, opacity=.2] (0,5) rectangle (20,15);
                \draw[blue, thick] (10,0) -- (10,20);
                \draw[mygreen, thick, rounded corners] (0,5) -- (20,5) -- (20,15) -- (0,15);
                \draw[dashed] (0,0) rectangle (25,20);
                \node at (10,5) [below left]{\scriptsize $p$};
                \node at (10,15) [above left]{\scriptsize $q$};
                \node at (10,20) [above]{\scriptsize $\gamma_{i,\epsilon}$};
                \node at (0,15) [left]{\scriptsize $\psi^{\ast}\gamma_{n,\epsilon'}$};
                \node at (15,10) {\scriptsize $R$};
            }
        }.
    \end{align*}
    The ideal arcs in $\psi^{\ast}S\coloneqq\{\psi^{\ast}\gamma_{j,\pm}\mid j=1,\ldots,n\}$ are mutually disjoint except at their endpoints. It implies that there exists an ideal arc $\alpha$ in $\psi^{\ast}S$ such that $\alpha$ passes through $\psi^{\ast}B_{\gamma_{i}}$ running from below $p$ to above $q$, and is the one closest to $\psi^{\ast}\gamma_{n,\epsilon'}$. We remark that $\alpha$ might be $\psi^{\ast}\gamma_{i,\epsilon}$. We illustrate it below;
    \begin{align*}
        \mbox
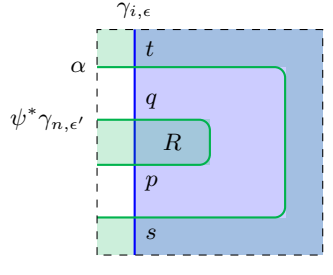
{
            \tikz[baseline=-.6ex, scale=.1]{
                \fill[mygreen, opacity=.2] (0,0) rectangle (30,30);
                \fill[white] (0,5) rectangle (25,25);
                \fill[blue, opacity=.2] (5,0) rectangle (30,30);
                \draw[blue, thick] (5,0) -- (5,30);
                \fill[mygreen, opacity=.2] (0,12) rectangle (15,18);
                \draw[mygreen, thick, rounded corners] (0,5) -- (25,5) -- (25,25) -- (0,25);
                \draw[mygreen, thick, rounded corners] (0,12) -- (15,12) -- (15,18) -- (0,18);
                \draw[dashed] (0,0) rectangle (30,30);
                \node at (10,15) {\scriptsize $R$};
                \node at (5,5) [below right]{\scriptsize $s$};
                \node at (5,25) [above right]{\scriptsize $t$};
                \node at (5,12) [below right]{\scriptsize $p$};
                \node at (5,18) [above right]{\scriptsize $q$};
                \node at (0,25) [left]{\scriptsize $\alpha$};
                \node at (0,18) [left]{\scriptsize $\psi^{\ast}\gamma_{n,\epsilon'}$};
                \node at (5,30) [above]{\scriptsize $\gamma_{i,\epsilon}$};
            }
        },
    \end{align*}
    where $s$ (resp.~$t$) is the intersection point in $\alpha\cap\gamma_{i,\epsilon}$ that is nearest to $p$ (resp.~$q$).
    By the minimality of $\gamma_{i,\epsilon}$, $\psi^{\ast}\gamma_{n,\epsilon'}$, and $\alpha$, the degree of $(\gamma_{i,\epsilon})_{sp}\ast(\psi^{\ast}\gamma_{n,\epsilon'})_{pq}\ast(\gamma_{i,\epsilon})_{qt}$ coincides with that of $\alpha_{st}$. Hence, these two arcs are connected by a flat braid diagram by \cref{lem:biangle_blad}. Moreover, it follows from the good position property that this flat braid diagram is trivial.
    Therefore, a sequence of crossbar passes between $\psi^{\ast}B_{\gamma_{n}}$ and the biangle bounded by $\alpha$ reduces $R$ in a similar way as in the first case.

    Now we assume that all biangles in the first and second cases are removed from $\sqcup_{i}\psi^{\ast}B_{\gamma_{i}}$. Denote the corresponding deformation of $D_{2}$ by $\bar{\phi}\colon D_2\leadsto \bar{D}_{2}$, and define $\bar{\psi}\coloneqq\bar{\phi}\circ\psi\colon D_{1}\to \bar{D}_{2}$.
    Finally, we consider a biangle $R$ below
    \begin{align*}
        \mbox
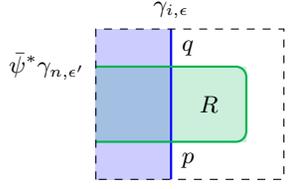
{
            \tikz[baseline=-.6ex, scale=.1]{
                \fill[mygreen, opacity=.2] (0,5) rectangle (20,15);
                \fill[blue, opacity=.2] (0,0) rectangle (10,20);
                \draw[blue, thick] (10,0) -- (10,20);
                \draw[mygreen, thick, rounded corners] (0,5) -- (20,5) -- (20,15) -- (0,15);
                \draw[dashed] (0,0) rectangle (25,20);
                \node at (10,5) [below right]{\scriptsize $p$};
                \node at (10,15) [above right]{\scriptsize $q$};
                \node at (10,20) [above]{\scriptsize $\gamma_{i,\epsilon}$};
                \node at (0,15) [left]{\scriptsize $\bar{\psi}^{\ast}\gamma_{n,\epsilon'}$};
                \node at (15,10) {\scriptsize $R$};
            }
        }.
    \end{align*}
    Similar to the above cases, a flat braid diagram in $R$ should be contained in some biangle $B_{\gamma_{j}}$. However, now this kind of intersection between $R$ and $B_{\gamma_{j}}$ is removed. Hence $R$ only has a trivial $\fsp_4$-diagram and one can remove $R$ by an isotopy without deforming $\bar{D}_{2}$.
    Repeating these operations, we can deform $\psi^{\ast}B_{\gamma_{n}}$ into $B_{n}$.
    This deformation of the biangle $\psi^{\ast}B_{\gamma_{n}}$ gives a sequence $\phi$ of crossbar passes that deforms $D_{2}$ into $D'_{2}$.
\end{proof}

\begin{cor}\label{cor:good-position}
    Let $\tri$ be an ideal triangulation of $\Sigma$ and $\tri^{\mathrm{split}}$ its split triangulation. 
    For any $L\in\Blad_{\fsp_4,\bSigma}$ and $\splittri$-transverse representative $D$ of $L$, $D$ can be deformed into a good position $D'$ with respect to $\splittri$ by a sequence of intersection reduction moves along $\splittri$, $H$-moves along $\splittri$, and ladder-equivalence relations relative to $\splittri$. Furthermore, $D'\cap T$ is uniquely determined up to arc parallel-moves of corner arcs for every $T\in t(\splittri)$.
\end{cor}
\begin{proof}
    By \cref{prop:exist-S-minimal}, $D$ can be deformed into a $\splittri$-minimal $\fsp_4$-diagram $D_{\mathrm{min}}$ through a sequence of intersection reduction moves along $\splittri$, $H$-moves along $\splittri$, and ladder-equivalence relations relative to $\splittri$. This $\splittri$-minimal representative can be deformed into a good position by a sequence of $H$-moves along $\splittri$, as stated in \cref{cor:exist-good}. By \cref{thm:unique-S-good}, such a good position is uniquely determined up to loop parallel-moves, arc parallel-moves, and crossbar passes among $b(\splittri)$. Moreover, a crossbar pass can be regarded as arc parallel-moves in triangles in $t(\splittri)$. Hence, the assertion follows.
\end{proof}

\end{document}